\documentclass[10pt,reqno]{amsart}
\usepackage{amsmath}
\usepackage{amsfonts}
\usepackage{amssymb}
\usepackage{graphicx}
\usepackage{amsthm,graphicx,color,yfonts}
\usepackage[letterpaper, margin=4.1cm]{geometry}
\usepackage{pdfsync}
\usepackage{epstopdf}%
\usepackage{pifont}
\usepackage{bbm}
\usepackage{graphicx}
\usepackage{subcaption}
\usepackage{caption}

\usepackage[colorlinks=true]{hyperref}
\hypersetup{linkcolor=red,citecolor=blue,filecolor=dullmagenta,urlcolor=blue} 
\definecolor{forestgreen}{cmyk}{0.91,0,0.88,0.12}

\usepackage{wrapfig}
\usepackage{tikz}
\usetikzlibrary{arrows,calc,decorations.pathreplacing}
\definecolor{light-gray1}{gray}{0.90}
\definecolor{light-gray2}{gray}{0.80}
\definecolor{light-gray3}{gray}{0.60}

\newcommand{\B}{\beta }

\newcommand{\R} {\mathbb R}
\newcommand{\cuad}{{\sqcap\kern-.68em\sqcup}}

\newcommand{\ve}{\varepsilon}

\definecolor{darkgreen}{rgb}{0.2,0.7,0.1}

\newcommand{\sech}{\mathop{\mbox{\normalfont sech}}\nolimits}

\newcommand{\px}{\partial_x}
\newcommand{\pt}{\partial_t}

\newcommand{\bd}[1]{\boldsymbol{#1}}

\renewcommand{\Im}{\mathrm{Im}}

\newcommand{\al}{\alpha}
\newcommand{\bt}{\beta}

\def\bm{\left( \begin{array}{cc}}
\def\endm{\end{array}\right)}
\providecommand{\abs}[1]{|#1 |}
\providecommand{\norm}[1]{\left\| #1 \right\|}

\newcommand{\be}{\begin{equation}}
\newcommand{\ee}{\end{equation}}
\newcommand{\bp}{\begin{pmatrix}}
\newcommand{\ep}{\end{pmatrix}}
\newcommand{\ba}{\begin{aligned}}
\newcommand{\ea}{\end{aligned}}

\newcommand{\qq}{\qquad}
\newcommand{\bee}{\begin{eqnarray*}}
\newcommand{\eee}{\end{eqnarray*}}
\newcommand{\ben}{\begin{enumerate}}
\newcommand{\een}{\end{enumerate}}

\numberwithin{equation}{section}

\newtheorem{theorem}{Theorem}[section]
\newtheorem{proposition}[theorem]{Proposition}

\newtheorem{lemma}[theorem]{Lemma}
\newtheorem{definition}[theorem]{Definition}

\newtheorem*{theorem*}{Theorem}
\theoremstyle{remark}
\newtheorem{remark}{Remark}[section]

\title[Variable bottom $abcd$ solitary wave dynamics]{Dynamics of generalized abcd Boussinesq solitary waves under a slowly variable bottom}

\author[de Laire]{Andr\'e de Laire}
\address{Univ.\ Lille, CNRS, Inria, UMR 8524 - Laboratoire Paul Painlev\'e, F-59000 Lille, France.}
\email{andre.de-laire@univ-lille.fr}
\thanks{}

\author[Goubet]{Olivier Goubet}
\address{Univ.\ Lille, CNRS, Inria, UMR 8524 - Laboratoire Paul Painlev\'e, F-59000 Lille, France.}
\email{olivier.goubet@univ-lille.fr}
\thanks{A.\ d.L. and O.\ G.\  acknowledge the support of the CDP C2EMPI, as well as the French State under the France-2030 programme, the University of Lille, the Initiative of Excellence of the University of Lille, the European Metropolis of Lille for their funding and support of the R-CDP-24-004-C2EMPI project.
 O.\ G. acknowledges the support of the University of Lille AAS-Internationalisation 2025, project LICHI}

\author[Mart\'inez Martini]{M. Eugenia Mart\'inez Martini}
\address{Departamento de Ingenier\'{\i}a Matem\'atica and Centro
de Modelamiento Matem\'atico (UMI 2807 CNRS), Universidad de Chile, Casilla
170 Correo 3, Santiago, Chile.}
\email{maria.martinez.m@uchile.cl}
\thanks{M.E.M.M.'s work was partly funded by Inria Lille PANDA, Chilean research grant Centro de Modelamiento Matemático (CMM) BASAL fund FB210005 for center of excellence from ANID-Chile and the project CRISIS (ANR-20-CE40-0020-01), operated by ANR.}

\author[Mu\~noz]{Claudio Mu\~noz*}
\address{CNRS and Departamento de Ingenier\'{\i}a Matem\'atica and Centro
de Modelamiento Matem\'atico (UMI 2807 CNRS), Universidad de Chile, Casilla
170 Correo 3, Santiago, Chile.}
\email{cmunoz@dim.uchile.cl}
\thanks{* Corresponding author. C.\ M.'s work was partly funded by Inria Lille PANDA and Chilean research grants FONDECYT 1231250, Centro de Modelamiento Matemático (CMM) BASAL fund FB210005 for center of excellence from ANID-Chile, and Grant PID2022-137228OB-I00 funded by the Spanish Ministerio de Ciencia, Innovaci\'on y Universidades, MICIU/AEI/10.13039/501100011033.}

\author[Poblete]{Felipe Poblete}
\address{Universidad Austral de  Chile, Facultad de Ciencias, Instituto de Ciencias F\'isicas y Matem\'aticas,  Valdivia, CHILE.}
\email {felipe.poblete@uach.cl}
\thanks{F.\ P.'s was partially supported by ANID 2022 Exploration project 13220060, ANID project FONDECYT 1221076, MathAmSud WAFFLE 23-MATH-18 and Inria PANDA}

\thanks{C.\ M.,\, M.\ E.\ M.\ and F.\ P.\ would like to thank Inria Lille and Univ.\ Lille, where part of this work was written.
 Part of this work was done while M.\ E. \ M. and C.\ M. were present at BIRS New Synergies in Partial Differential Equations (25w5403) workshop at Banff, whose support is greatly acknowledged}


\begin{document}

\begin{abstract}
The Boussinesq $abcd$ system is a 4-parameter set of equations posed in $\R_t\times\R_x$, originally derived by Bona, Chen and Saut \cite{BCS1,BCS2} as first-order 2-wave approximations of the incompressible and irrotational, two-dimensional water wave equations in the shallow water wave regime, in the spirit of the original Boussinesq derivation \cite{Bous}. Among the various particular regimes, each determined by the values of the parameters $(a, b, c, d)$ appearing in the equations, the \emph{generic} regime is characterized by the conditions $b, d > 0$ and $a, c < 0$. If additionally $b=d$, the $abcd$ system is Hamiltonian.

In this paper, we investigate the existence of generalized solitary waves and the corresponding collision problem in the physically relevant \emph{variable bottom regime}, introduced by M.\ Chen \cite{MChen}.  More precisely, the bottom is represented by a smooth space-time dependent function $h=\varepsilon h_0(\varepsilon t,\varepsilon x)$, where $\varepsilon$ is a small parameter and $h_0$ is a fixed smooth profile. This formulation allows for a detailed description of weak long-range interactions and the evolution of the solitary wave without its destruction. We establish this result by constructing a new approximate solution that captures the interaction between the solitary wave and the slowly varying bottom.
\end{abstract}

\maketitle
\tableofcontents

\section{Introduction and Main Results}

\medskip

\subsection{Setting of the problem} This paper concerns the study of the physically motivated problem of interaction of solitary waves of the one-dimensional $abcd$ system in the presence of variable bottom:
\begin{equation}\label{boussinesq_0}
\begin{cases}
(1- b\,\partial_x^2)\partial_t \eta  + \partial_x\!\left( a\, \partial_x^2 u +u + (\eta +h) u \right) =  (-1 +a_1 \partial_x^2) \partial_th  , \quad (t,x)\in \R\times\R, \\
(1- d\,\partial_x^2)\partial_t u  + \partial_x\! \left( c\, \partial_x^2 \eta + \eta  + \frac12 u^2 \right) =c_1  \partial_t^2 \partial_x h.
\end{cases}
\end{equation}
Here, $u=u(t,x)$ and $\eta=\eta(t,x)$ are real-valued scalar functions. The variable bottom is represented by the function $h=h(t,x)$, which is assumed to have dependence on time and space. Equation \eqref{boussinesq_0} assumes that the local pressure on the free surface $\eta$ is constant for simplicity. In the case where we have a fixed bottom, the equation was originally derived by Bona, Chen, and Saut \cite{BCS1,BCS2} as a first-order, one-dimensional asymptotic regime model of the water waves equation, in the vein of the Boussinesq original derivation \cite{Bous}, but maintaining all possible equivalences between the involved physical variables, and taking into account the shallow water regime. The physical perturbation parameters under which the expansion is performed are
\[
\al := \frac Ah \ll 1, \quad \bt := \frac{h^2}{\ell^2} \ll 1, \quad \al \sim \bt,
\]
where $A$ and $\ell$ are typical wave amplitude and wavelength, respectively, and $h$ is the constant depth.

The constants in \eqref{boussinesq_0} are not arbitrary and follow the conditions \cite{BCS1}
\be\label{Conds0}
\begin{aligned}
& a=\frac12 \left(\theta^2-\frac13\right)\lambda, \quad b=\frac12 \left(\theta^2-\frac13\right)(1-\lambda), \\
& c=\frac12 \left(1- \theta^2 \right)\mu, \quad d=\frac12 \left(1-\theta^2\right)(1-\mu), \\
& a_1 =\frac12 \left((1-\lambda)\left(\theta^2-\frac13\right)+1-2\theta\right), \quad c_1=1-\theta, \\
\end{aligned}
\ee
for some $\theta\in [0,1]$ and $\lambda,\mu\in\mathbb R$. Notice that  $a+b =\frac12\left(\theta^2-\frac13\right)$ and $c+d =\frac12 (1-\theta^2)\geq 0$. Moreover, $a+b+c+d = \frac13$ is independent of $\theta$. (This case is referred to as the regime without surface tension $\tau\geq 0$, otherwise we have parameters $(a,b,c,d)$ such that $a+b+c+d = \frac13-\tau$.)

As mentioned earlier, the Boussinesq system was developed to address the need for describing shallow water, small-amplitude models. Indeed, it was observed early on by Russell \cite{Russell} that the length of a water wave increases directly with the depth of the fluid, but not with the height of the wave. In fact, as the height of the wave increases, its length tends to decrease. This extension in length is accompanied by a reduction in height, and vice versa, indicating that changes in wave length and height reflect changes in water depth. As a consequence, when studying ``long" surface wave of finite amplitude ($\alpha\ll 1$ and $\beta\ll 1$), it is important to distinguish three physical conditions \cite{Ursell}:
\[
\frac\alpha\beta \ll 1, \quad \frac\alpha \beta \gg 1 \quad \text{or} \quad \frac\alpha \beta \sim 1.
\]

These key physical parameters help characterize the relative importance of dispersion and nonlinearity. While dispersion dominates in the deep ocean, nonlinearity becomes more significant in shallow coastal areas \cite{Wu}. The model introduced by Boussinesq \cite{Bous} deals with the case $\alpha\sim \beta$, and many similar versions have since been developed, commonly referred to as Boussinesq models. A full justification of these models in the one-dimensional case, assuming flat bottoms and small initial data, was first provided in \cite{Craig, Kano}. Since then, Boussinesq-type models have received considerable attention; see, for instance, \cite{BLS,BCL,LPS,Saut} and references therein for a detailed overview of the existing results. In this context, system \eqref{boussinesq_0}, as derived by Bona, Chen, and Saut in \cite{BCS1,BCS2}, involves the parameters $a, b, c, d$, which represent the interplay between dispersion and the nonlinear behavior of waves. A rigorous justification for \eqref{boussinesq_0} with flat bottom from the free surface Euler equations (as well as extensions to higher dimensions) was given by Bona, Colin, and Lannes \cite{BCL}. Later, Alvarez-Samaniego and Lannes provided further improvements \cite{ASL}. For a more than detailed account on Boussinesq models, see Klein-Saut \cite{KS2}.

The generalization of \emph{abcd} Boussinesq system introduced in \cite{BCS1,BCS2} to a variable bottom topography (uneven or moving bottoms), precisely equation \eqref{boussinesq_0}, was given by Chen in \cite{MChen}. Chen carried out a formal analysis of the water waves problem over uneven bottoms with small amplitude variations in one spatial dimension and derived a class of asymptotic models, drawing inspiration from the work of Bona, Chen, and Saut. Later, Chazel provided a rigorous justification of the model derived in \cite{MChen} for bottoms with small spatial variations in \cite{Chazel}. Generalizations to variable bottom topography for the two-dimensional case can be found in \cite{Mitsotakis, Mitsotakis2}, where the authors proposed a variable bottom \emph{abcd} Boussinesq model that allows for the conservation of suitable energy functionals in some cases and enables the description of water waves in closed basins with well-justified slip-wall boundary conditions. Other Bousinesq-type models with variable bottoms can be found in \cite{DD, Iguchi, MMS, MS, Nwogu, Peregrine}.

Ocean surface waves cover a wide range of scales, from tiny capillary waves to long waves like tsunamis with wavelengths up to thousands of kilometers. Tsunamis, often caused by tectonic events (e.g., earthquakes, landslides, volcanic eruptions), are especially significant due to their destructive potential. Although their initial amplitude is small, tsunamis carry massive energy and travel across oceans at high speeds, gradually evolving due to weak dispersion. As tsunamis approach coastal regions, their amplitude increases significantly due to shoaling effects, and the impact is strongly influenced by the shape of the coastline. After striking land, tsunamis can reflect and propagate back across the ocean with slow attenuation \cite{MaMe, Wu}. This highlights the importance of understanding the interaction between surface waves and the shape of the bottom of the fluid (see \cite{Johnson} for early results and \cite{Mu1, Mu2, Mu3, Mu4} for late developments). Although there are studies in the literature addressing tsunami generation using the Water Waves equations and the \emph{abcd} Boussinesq system, they primarily focus on modeling the limiting cases  \cite{Iguchi2, Mitsotakis}, rather than exploring the role the changing medium plays in the wave dynamics. It is our goal to address the interaction between a variable bottom and a solitary wave for the \emph{abcd} Boussinesq system \eqref{boussinesq_0}.

\subsection{The Cauchy problem}
As for the low regularity Cauchy problem associated with \eqref{boussinesq_0} and its generalizations to higher dimensions, Saut et.\ al.\ \cite{SX,SWX} studied in great detail the long time existence problem by focusing in the small physical parameter $\ve$ appearing from the asymptotic expansions. They showed well-posedness (on a time interval of order $1/\ve$) for \eqref{boussinesq_0}. Previous results by Schonbek \cite{Schonbek} and Amick \cite{Amick} considered the case $a=c=b=0$, $d=1$, a version of the original Boussinesq system, proving global well-posedness under a non-cavitation condition, and via parabolic regularization. Linares, Pilod and Saut \cite{LPS} considered existence in higher dimensions for time of order $O(\ve^{-1/2})$, in the KdV-KdV regime $(b=d=0)$. Additional low-regularity well-posedness results can be found in the work by Burtea \cite{Burtea}. On the other hand, ill-posedness results and blow-up criteria (for the case $b=1$, $a=c=0$), are proved in \cite{CL}, following the ideas in Bona and Tzvetkov \cite{BT}.

In a previous work \cite{KMPP} dealing with the flat bottom case, the \emph{Hamiltonian generic} case \cite[p. 932]{BCS2}
\begin{equation}
   \label{Conds2}
b,d >0, \quad a,~c<0, \quad b=d.
\end{equation}
was considered, and it was shown decay of small solutions on compact sets. Later, this result was improved in \cite{KM20} to consider even more $abcd$ models. See also \cite{Munoz_Rivas} for a detailed decay proof for higher powers of the nonlinearity using weighted norms. The result in \cite{KM20} is sharp in the sense that it reaches the values of $(a,b,c)$ where one starts having nonzero frequency nondecaying linear waves. Always in the flat bottom case, under \eqref{Conds2} it is well-known \cite{BCS2} that \eqref{boussinesq_0} is globally well-posed in $H^s \times H^s$, $s\geq 1$, for small data, thanks to the preservation of the energy/hamiltonian and momentum
\be\label{Energy}
\begin{aligned}
H[\eta , u](t):= &~{} \frac12\int \left( -a (\partial_x u)^2 -c (\partial_x \eta)^2  + u^2+ \eta^2 + u^2\eta \right)(t,x)dx,\\
P[\eta , u](t):= &~{} \int \left( \partial_x \eta \partial_x u + b\eta u \right)(t,x)dx.
\end{aligned}
\ee
\par
Since now, we will identify $H^1 \times H^1$ as the standard \emph{energy space} for \eqref{boussinesq_0}. However, in our case, the energy \eqref{Energy} is not conserved anymore. Indeed, consider the modified energy for the variable bottom case
\be\label{Energy_new}
H_h[\eta,u ](t):= \frac12\int \left( -a (\partial_x u)^2 -c (\partial_x \eta)^2  + u^2+ \eta^2 + u^2(\eta + h) \right)(t,x)dx.
\ee
Then, at least formally, the following is satisfied: from \eqref{Energy_new} and \eqref{boussinesq_0},
\be\label{Energy_new_11}
\begin{aligned}
\frac{d}{dt} H_h[\eta,u ](t) = &~{} c_1 \int \left( a   \partial_{x}^2 u  + u   +  u  (\eta + h)  \right) (1- \partial_x^2)^{-1}  \partial_t^2 \partial_x h \\
&~{} + \int \left( c \partial_{x}^2  \eta  +  \eta  + \frac12 u^2 \right)(1- \partial_x^2)^{-1} (-1 +a_1 \partial_x^2) \partial_th   + \frac12 \int u^2  \partial_t h.
\end{aligned}
\ee
Notice that the influence of the uneven bottom is important in the long-time behavior through \eqref{Energy_new_11}, meaning that  previous results proved in \cite{KMPP} are not easily translated to the uneven bottom regime. In the simpler case where $h$ only depends on $x$, we can see that $H_h$ is conserved; however, we shall consider the more general case where $h$ non-trivially depends on time as well. In that sense, we will choose data and conditions on $h$ that ensure globally well-defined solutions with bounded in time energy, a naturally physical condition. See also the recent work \cite{DGMMP}, where the decay properties established in \cite{KMPP,KM20} are extended to the case of variable bottom.

Assume \eqref{Conds0} and \eqref{Conds2}. Coming back to the general model \eqref{boussinesq_0}, by considering the new stretching of variables $ u(t/\sqrt{b},x/\sqrt{b})$, $\eta(t/\sqrt{b},x/\sqrt{b})$, and $h(t/\sqrt{b},x/\sqrt{b})$, we can assume from now on that  $b=d=1$ and rewrite \eqref{boussinesq_0} as the slightly simplified model
\begin{equation}\label{boussinesq}
\begin{cases}
(1- \partial_x^2)\partial_t \eta  + \partial_x \! \left( a\, \partial_x^2 u +u + u (\eta+ h) \right) = (-1 +a_1 \partial_x^2) \partial_th  , \quad (t,x)\in \R\times\R, \\
(1- \partial_x^2)\partial_t u  + \partial_x \! \left( c\, \partial_x^2 \eta + \eta  + \frac12 u^2 \right) = c_1  \partial_t^2 \partial_x h,
\end{cases}
\end{equation}
and where $a_1,c_1 $ have been properly redefined. Precisely, \eqref{boussinesq} will be the model worked in this paper.

\subsection{Solitary waves} Many $abcd$ Boussinesq models are characterized by having solitary waves, namely special solutions describing a fixed profile moving with a fixed speed. More precisely, for any $\omega, x_0\in\mathbb R$, we look for solutions in  $H^1\times H^1$  of the form
\begin{equation}\label{sol_wave}
\begin{pmatrix} \eta \\ u \end{pmatrix}:= {\bd Q}_\omega (x-\omega t-x_0), \quad {\bd Q}_\omega(s) := \begin{pmatrix} R_{\omega} \\ Q_{\omega} \end{pmatrix} (s).
\end{equation}
This amounts to consider solutions $(R_{\omega }, Q_{\omega })\in H^1\times H^1$ to
\begin{equation}\label{exact0}
\begin{aligned}
& 0=  - \omega (1-\partial^2_x)R_{\omega }  + a Q''_{\omega }+Q_{\omega }+R_{\omega }Q_{\omega } \\
& 0=  -\omega  (1-\partial^2_x)Q_{\omega } + c R''_{\omega }+R_{\omega }+\frac12Q_{\omega }^2.
\end{aligned}
\end{equation}
In \cite{BCL0}, the authors investigate the existence of solutions  ${\bd Q}_\omega$ to \eqref{exact0},
 where ${\bd Q}_\omega$ obeys a variational characterization. Among other results, they prove the existence of nontrivial ground states ${\bd Q}_\omega \in H^\infty\times H^\infty$ as long as
\[
a,\, c<0,  \quad |\omega|<  \min\{1,\sqrt{ac}\},
\]
which is the so-called subsonic regime. Note that the speed $\omega$ never reaches the sonic speed, equal to $1$.\footnote{The case $\omega=1$ is particularly interesting, see \cite{DGM} for a recent development in this direction.} The construction in \cite{BCL0} is based on a minimization approach in a Nehari manifold. Assuming moreover that $1\leq \min\{\abs{a},\abs{c}\}$, these solutions are even, up a translation. See also \cite{CNS1,CNS2,Olivera} for further results on the existence of solitary waves for \eqref{boussinesq}. Although not explicit in general, some solitary waves profiles (not necessarily ground states) are sometimes explicit. For instance, M. Chen found in \cite{MChen0} several exact solutions of \eqref{exact0}.
In the particular case  $a=c=-1$, a family of solutions indexed by a parameter $\alpha  \in (-3,\infty)\setminus\{0\} $,
is given by
 \begin{equation}
 \label{TW-chen}
 \begin{aligned}
 & R_\omega(x)= \alpha \sech^2\Big (\frac{x}2\Big), \\ & Q^{\pm}_\omega(x)=\pm \alpha \sqrt{\frac{3}{3+\alpha}} \sech^2\Big (\frac{x}2\Big),
 \text{ with  }
 \omega=\pm \frac{ 3+2\alpha }{ \sqrt{ 3(3+\alpha)} }.
 \end{aligned}
 \end{equation}
 In addition, Hakkaev, Stanislavova, Stefanov proved in \cite{HSS} that the solutions in \eqref{TW-chen} are spectrally stable if $\alpha \in(-9/4,0)$,
 which corresponds to speeds in the subsonic regime $\abs{\omega}<1$.
%

\subsection{Main result}
In this work, our main objective is to understand the weak interaction between stable $abcd$ solitary waves and a slowly varying bottom. Specifically, we construct and analyze a solitary-wave-like solution to the modified $abcd$ system that evolves in a regime characterized by small amplitude and slow variation in both space and time of the bottom.
In order to state our main results, we present the main hypotheses required for the perturbation of the bottom: 

\medskip

\noindent
{\bf Hypotheses on the bottom.} Let $\varepsilon>0$ be a small parameter. Let $h_0: \mathbb R^2 \to \mathbb R$, $h_0=h_0(s,y)$, be a fixed function in $C^\infty(\mathbb R^2)$ such that there are constants $C_{k,l}>0$ and  $k_0,l_0>0$ such that 
\begin{equation}\label{hypoH}
\begin{aligned}
|\partial_s^k \partial_y^l h_0(s,y)| \leq C_{k,l} e^{-k_0 |s|}e^{-l_0 |y|}, \quad  h(t,x) = \varepsilon h_0(\varepsilon t,\varepsilon x), \quad 
\text{ for all } k,l\geq 0. 
\end{aligned}
\end{equation}
This means that the bottom variation is small amplitude and varies slowly in both space and time. The exponential decay of $h_0$ in both variables is not sharp, and can be replaced by a suitable polynomial decay. In addition, our results should hold if \eqref{hypoH} is relaxed to $k\geq 1$, to include, for instance, a small bump bottom, with some minor modifications. However, we prefer  
to use hypothesis \eqref{hypoH} to simplify the computations in the paper.

In any case, the important physical case described by a compactly supported in space and time perturbation of the bottom, modeling a temporal modification of the bottom in a particular region of space, is included in \eqref{hypoH}. Additionally, \eqref{hypoH} is coherent with the exponential decay of unperturbed $abcd$ solitary waves, preserving that property along calculations.  

\medskip

\noindent
{\bf Hypotheses on the solitary wave.}
From now on, we assume that $a,c<0$, $b=d=1$ and that the speed satisfies
\begin{equation}\label{velo}
\omega \in (0,\omega^*),\quad \text{ where }\omega^*:=\min\{ \sqrt{ac},1\},
\end{equation}
We consider an \emph{even} solitary-wave solution ${\bd Q}_\omega $ to $abcd$ \eqref{boussinesq} in $H^1\times H^1$. In this manner, $\bd Q_\omega$ is smooth and decays exponentially, as well as its derivatives (see Lemma~\ref{lem:decay:solitons}), and
the linear operator
$\mathcal L=\mathcal L(\omega)$ given by
\begin{equation}\label{def_L}
\begin{aligned}
\mathcal L
 := &~{}
 \begin{pmatrix}
 c \partial_x^2 +1  &   -\omega (1-\partial^2_x) +Q_\omega   \\
  -\omega (1-\partial^2_x)  +Q_\omega & a \partial_x^2 +1 +R_\omega
\end{pmatrix}.
\end{aligned}
\end{equation}
is unbounded self-adjoint in $L^2\times L^2$, with dense domain $H^2\times H^2$.
We also assume that $\bd Q_\omega$ is \emph{stable}, in the following sense:
\begin{enumerate}
\item[$(i)$] (Nondegeneracy of the kernel) The function $\bd{Q}_\omega':=(R_{\omega}',Q_{\omega}')^T$ generates the kernel of $\mathcal L$, i.e.
 \begin{equation}\label{kernel}
\textup{ker}\{\mathcal L \} =\textup{span}\{\bd{Q}_\omega'\}.
 \end{equation}
\item[$(ii)$] (Negative eigenvalue). The operator $\mathcal L$ has a unique negative eigenvalue $-\mu_0$ with associated eigenfunction generated by a given function
$\Phi_{-1}\in H^2\times H^2$.
\item[$(iii)$] (Slope condition)
There is an open neighborhood $\Omega\subset (0,\omega^*)$ of $w$ such that
the map $w \in\Omega\to \bd Q_{w}\in H^2\times H^2$ is differentiable and the derivative of the momentum satisfies:
\begin{equation}\label{mom}
 \frac{d}{d\omega }P[\bd Q_\omega]<0,
 \end{equation}
for all $\omega \in \Omega$.  This sign condition is commonly referred to as the Vakhitov--Kolokolov condition \cite{vakhitov}.

\end{enumerate}
Let us make some comments on these hypotheses.
Condition \eqref{velo} implies that the operator with constant coefficients:
\begin{equation}\label{def_L0}
\begin{aligned}
{\mathcal L}_0
 := &~{}
 \begin{pmatrix}
 c \partial_x^2 +1  &   -\omega (1-\partial^2_x)     \\
  -\omega (1-\partial^2_x)   & a \partial_x^2 +1
\end{pmatrix},
\end{aligned}
\end{equation}
is coercive in $H^1\times H^1$.
Indeed, using the elementary identity,
\begin{equation}
\label{ineq:abcd}
\alpha x_1^2 + \beta x_2^2 -2\kappa x_1x_2=\left( 1-\frac\kappa{\sqrt{\alpha\beta}}\right) \big( \alpha x_1^2+\beta x_2^2\big)+
\frac{\kappa}{\sqrt{\alpha\beta}}\big( \sqrt{\alpha}x_1-\sqrt{\beta}x_2 \big)^2,
\end{equation}
for all $x_1,x_2\in\R$, $\alpha,\beta$, $\kappa\in\R$, we deduce that for all $\bd \eta=(\eta,u)\in H^2\times H^2,$
$$
\langle \mathcal L_0 \eta,\eta\rangle \geq \left( 1-\frac\omega{\sqrt{\alpha\beta}} \right)(\norm{\eta'}_{L^2}^2+\norm{u'}_{L^2}^2)+(1-\omega)
(\norm{\eta}_{L^2}^2+\norm{u}_{L^2}^2).
$$
Combining with the exponential decay of $\bd Q_\omega$, this also implies that the infimum of the essential spectrum of $\mathcal L$ is strictly positive.

Condition $(i)$ provides the existence of $\mathcal L^{-1}$ on orthogonal to $\textup{ker}\{\mathcal L \}$, as stated in Lemma~\ref{lem:L-1}.
\begin{lemma}
\label{lem:L-1}
Let {\rm Ker}$\mathcal{L}=\textup{span}\{ \bd Q_\omega'\}$.
For any ${\bf \eta}=(\eta,u)$ in R$(\mathcal{L})= {\rm Ker}\mathcal{L}^\perp$
there exists a unique $(\tilde \eta, \tilde u)$ in $H^2\times H^2 \cap {\rm Ker}\mathcal{L}^\perp$
such that $\mathcal{L}(\tilde \eta, \tilde u)=(\eta, u)$. We set $\mathcal{L}^{-1}(\eta,u)=(\tilde \eta, \tilde u)$.
\end{lemma}

\begin{proof}
According to \cite{Kato},  the following orthogonal decomposition holds
$$L^2\times L^2=\R \Phi_{-1}\oplus\R {\bf}Q'_\omega \oplus H_+,$$
 and for all $z\in  H_+$, we have $\langle \mathcal{L}z,z\rangle \geq c||z||^2_{L^2}$.
Hence, we deduce that $\mathcal{L}: D(\mathcal{L})\cap H_+ \mapsto H_+$ is one-to-one. Therefore, for ${\bf \eta}=a\Phi_{-1}+p$
in ${\rm Ker}\mathcal{L}^\perp$, we can set $\mathcal{L}^{-1}{\bf \eta}=-\frac{a}{\mu_0}\Phi_{-1}+\mathcal{L}^{-1}z$.
\end{proof}
Noticing that,
$$
\left\langle   (1-\partial^2_x)  \bd Q_\omega
  ,  \bd Q_\omega'
   \right\rangle=0,
 $$
it follows that $\mathcal L^{-1} (1-\partial^2_x) \bd Q_\omega  $ is well-defined,
and condition \eqref{mom} can be recast as
\be\label{coer_10}
\left\langle  J(1-\partial^2_x) \bd Q_\omega   ,   \mathcal L^{-1} J(1-\partial^2_x) \bd Q_\omega  \right\rangle < 0,
\ee
where
\be
\label{def:J}
J=  \begin{pmatrix} 0 & 1  \\  1 & 0 \end{pmatrix}.
\ee
 Indeed, denoting $\Lambda \bd Q_\omega=\frac{\partial}{\partial \omega} \bd Q_\omega$ and
differentiating \eqref{exact0} with respect to $\omega$, we deduce that
\begin{equation*}
\mathcal{L} \Lambda {\bd Q}_\omega = J(1-\partial^2_x)\bd Q_\omega.
\end{equation*}
Therefore, since the map $w \in\Omega\to \bd Q_{w}\in H^2\times H^2$ is differentiable, we get
\[
\left\langle J(1-\partial^2_x) \bd Q_\omega   ,   \mathcal L^{-1} J(1-\partial^2_x) \bd Q_\omega  \right\rangle =
 \frac{d}{d\omega } \int R_\omega  (1-\partial^2_x)  Q_\omega.
\]
Finally, integrating by parts, we obtain 
$$
     \int R_\omega  (1-\partial^2_x)  Q_\omega=\int R_\omega   Q_\omega+\partial_x R_\omega   \partial_x Q_\omega=P[R_\omega,Q_\omega],
$$
and thus we conclude that  \eqref{mom} and \eqref{coer_10}  are equivalent.
The slope condition \eqref{mom} also appears in the classical Grillakis-–Shatah-–Strauss conditions,
 and is essential to prove the strong coercivity property, stated in Lemma~\ref{coercivitydetailled}.

Let us highlight that the solitary waves in \eqref{TW-chen} with positive speeds in $(0,1)$ satisfy all hypotheses stated above. More precisely, defining
 $$G_{\pm} (\omega)=\frac38\left( \omega^2 -4\pm \omega \sqrt{\omega^2+8}\right), \text{ for }w\in \R,
 $$
 we see that the functions $G_\pm$ are smooth, $G_+$ is increasing, $G_-$ is decreasing, and $\Im (G_\pm)=(-3,\infty)$.
Thus, we can recast the two branches of solitary waves in  \eqref{TW-chen}  as:
\begin{equation*}
 \label{TW:w}
 \bd Q^{\pm}_\omega(x)=\left(
  G_{\pm}(\omega)
   \sech^2\Big (\frac{x}2\Big), \pm G_{\pm}(\omega) \sqrt{\frac{3}{3+G_{\pm}(\omega)}} \sech^2\Big (\frac{x}2\Big)
   \right),
\end{equation*}
for all $\omega \in \R$, except $\omega = 1$ for the plus sign, and $\omega = -1$ for the minus sign, which correspond to the trivial zero functions.
 From \eqref{TW:w}, we deduce that the maps $ \omega\in (0,1) \mapsto \bd Q^\pm_\omega\in H^2$ are differentiable, and that the energy and momentum are given by
\begin{align*}
E[\bd Q_\omega^\pm]&=\frac{18}5\big(1+\omega^2( G_\pm(\omega)-1\big),\\
P[\bd Q_\omega^+]&=\frac{16}5G_+(\omega)^2\left(1+ \frac{G_+(\omega)}{3} \right)^{-1/2}, \text{ and }\\
P[\bd Q_\omega^-]&=-\frac{32}5G_-(\omega)^2 \left(1+ \frac{G_+(\omega)}{3}\right)^{1/2},
\end{align*}
which we illustrate in Figure~\ref{fig:solitons}.  Also,  because of the Hamilton group relation:
$$\frac{d}{d\omega}{E[\bd Q_\omega^\pm]}=\omega\frac{d}{d\omega}{P[\bd Q_\omega^\pm]},$$
we observe that the slope of the curve in the right plot of Figure~\ref{fig:solitons} corresponds to the speed~$\omega$. Finally, from the expression for the momenta, we deduce that $\frac{d}{d\omega}P[\bd Q_\omega^\pm]<0$, for all $\omega \in (0,1)$, so that condition (iii) is satisfied.
In addition, in Proposition~2 in \cite{HSS}, it was established that the solitary waves $\bd Q^{\pm}_\omega$ also fulfill conditions (i) and (ii), for $\omega\in (0,1)$,  so that our main theorem applies, in particular, to these solutions.

\begin{figure}[h!]
    \centering
    \begin{subfigure}[t]{0.48\textwidth}
        \centering
        \includegraphics[width=\textwidth,height=5cm]{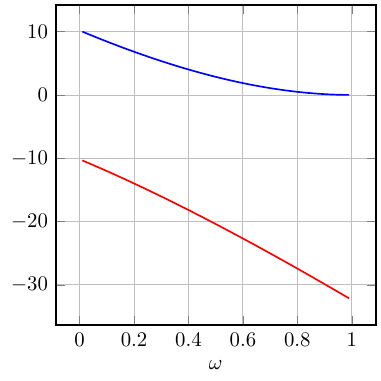}
    \end{subfigure}
    \hfill
    \begin{subfigure}[t]{0.48\textwidth}
        \centering
        \includegraphics[width=\textwidth,height=5cm]{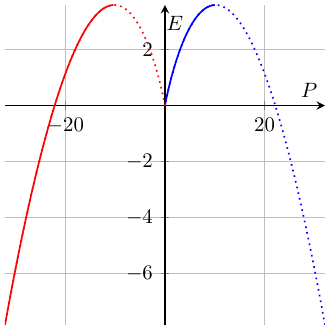}
    \end{subfigure}

    \caption{Left: The momenta $P(\bd Q^-_\omega)$ (in red) and $P(\bd Q^+_\omega)$ (in blue), as functions of $\omega\in (0,1)$. Right: Parametric curves $(P(\bd Q^-_\omega), E(\bd Q^-_\omega))$ (in red) and $(P(\bd Q^+_\omega), E(\bd Q^+_\omega))$
    (in blue), plotted in solid lines for  $\omega \in (0, 1)$, and in dotted lines for $\omega \in (-1,0)$.}
    \label{fig:solitons}
\end{figure}


To state the theorem, we fix a small constant $\delta_0 > 0$ and define $T_\varepsilon := \varepsilon^{-1 - \delta_0}$. Our main result is as follows.
\begin{theorem}\label{MT}
Let $a,c<0$, $\omega>0$ and $x_0\in \mathbb R$ be fixed parameters satisfying \eqref{velo}, and consider a solitary wave ${\bd Q}_\omega$ of the flat bottom $abcd$ system \eqref{boussinesq}, which is stable in the sense of $(i)$-$(iii)$ above.
There exist $C_0,\nu_0,\varepsilon_0>0$ such that, for all $\varepsilon\in (0,\varepsilon_0)$,  if  $h$ satisfies  \eqref{hypoH}, then
 there exists $\rho(t) \in\mathbb R$ such that the following hold.
\begin{enumerate}
\item Existence of a generalized solitary wave. There exists
\[
(\eta,u) \in C(\mathbb R, H^1\times H^1)\cap L^\infty(\mathbb R, H^1\times H^1)
\]
a solution  to \eqref{boussinesq} such that
\begin{equation}\label{Construction}
\lim_{t\to-\infty} \| (\eta,u) (t)- {\bd Q}_\omega (\cdot -\omega t-x_0) \|_{H^1\times H^1} =0.
\end{equation}
Moreover, one has
\item Pre-interaction regime. For all times $t\leq  -T_\varepsilon$,
\begin{equation}\label{PreInteraction}
\left\| (\eta,u) (t)- {\bd Q}_\omega (\cdot - \omega t-x_0) \right\|_{H^1\times H^1} \leq C_0  e^{\nu_0 \varepsilon t}.
\end{equation}
\item Interaction regime. At the time $t= T_\varepsilon$, for some $\rho_\varepsilon\gg \frac1\varepsilon,$
\begin{equation}\label{Interaction}
\left\| (\eta,u) (T_\varepsilon)- {\bd Q}_\omega (\cdot -\rho_\varepsilon) \right\|_{H^1\times H^1} \leq C_0 \varepsilon^{\frac12}.
\end{equation}
\item Exit regime. For all time $t\geq  T_\varepsilon$, 
\begin{equation}\label{Exit}
\sup_{t\geq T_\varepsilon} \left\| (\eta,u) (t)- {\bd Q}_\omega (\cdot -\rho(t)) \right\|_{H^1\times H^1} \leq C_0 \varepsilon^{\frac12},
\end{equation}
and
\begin{equation}\label{est_rho}
| \rho'(t) -\omega | \leq C C_0 \varepsilon^{\frac12}.
\end{equation}
\end{enumerate}
\end{theorem}
%
%

Estimate \eqref{Exit} is the main result of this paper: the solitary wave collides and survives the interaction, with an error of order $\varepsilon^{1/2}$. We believe that this strange order is universal, by natural reasons. Physically, \eqref{Exit} states that the variable bottom will have an important influence on the solitary wave, even if in the model its influence is much smaller. We expect that this result is universal if one modifies the $abcd$ model by other fluid models such as Serre-Green-Naghdi, Boussinesq-Peregrine, and other models of water waves. If the solitary wave is unstable, it is highly probable that the nonlinear object will not survive after the interaction, therefore, the hypotheses are in some sense necessary and sufficient. However, if the solitary wave is mildly unstable, we believe that this set of techniques can be extended with suitable changes to other models. 

\subsection{Novelty and scope of the present work}
The main contribution of this article is to provide a comprehensive (dispersive) analysis for the $abcd$ equation in the presence of an inhomogeneous background, a regime that has not been previously addressed in the literature, at least from the point of view of soliton dynamics. It can be recast as the understanding of long-time behavior of nonlinear waves in natural shallow dynamics. See \cite{Saut0, Saut1,Saut2} for recent works dealing with the Cauchy problem in highly nontrivial $abcd$ moving bottom models. While the homogeneous case has been extensively studied using Fourier methods, pseudo-differential tools, and virial-type identities, extending these ideas to coefficients depending on the spatial variable requires several new ingredients. In particular, the loss of translation invariance destroys the usual spectral decomposition and precludes the direct use of standard dispersive techniques, especially in the case of soliton dynamics. Here we need different techniques, more related to the interaction of nonlinear waves. Our approach overcomes these difficulties by developing a background-adapted set of energy and virial identities tailored to the $abcd$ model, establishing a family of monotonicity formulas suited to the variable-coefficient structure, and deriving refined energy estimates that capture the interaction between the solution and the inhomogeneous medium. Taken together, these techniques yield the first global solitary wave dynamics statements for the $abcd$ model with variable background, and show that the dispersive behavior persists under a broad class of nontrivial spatial profiles. The analysis developed in this work suggests several natural directions for further research. We believe that these methods may be applicable to a wider family of non-homogeneous dispersive equations arising in optics, fluid mechanics, and relativistic field theory, such as the Serre-Green-Naghdi, and shallow water waves under small dispersive perturbations.

\subsection*{Idea of proof}

The proof, as the statement of the theorem establishes, proceeds via a three-step description of the dynamics, composed of a first part where a generalized solitary wave for the uneven bottom resembling the one $\bd{Q}_\omega$ of the flat-bottom $abcd$ system \eqref{boussinesq} is constructed, a second part where the constructed generalized solitary wave experiments a strong adiabatic interaction that dominates the dynamics, leading to a nontrivial change in shifts and scaling; and finally, a third regime where the influence of the nontrivial bottom decreases and one can find a stability property that establishes that the perturbed solitary wave survives for all time. The slowly varying bottom, given by $h$ and satisfying  \eqref{hypoH}, induces small perturbations parameterized by $\varepsilon > 0$. This allows the use of modulated approximations and asymptotic expansions. This three-step construction follows similar approaches established first by Martel and Merle while studying the collisions of generalized KdV solitons \cite{Martel,MMcol,MMcol2}, expanded later to the case of solitary waves passing through adiabatic linear and nonlinear potentials \cite{Mu1,Mu2,Mu3,Mu4}.

More precisely, in the first regime, the main drawback is the lack of conserved quantities (Lemma \ref{Lem2p3}), a natural consequence of the variable bottom, which poses several problems along the full proof. Since the evolution is infinite in time, we need a reasonable way to measure the lack of conservation. This is obtained by decomposing the solution into two parts: the soliton one, and the error part, and performing bootstrap estimates (Subsection \ref{sub3p3}). The second regime is the interaction. Here there are several differences with respect to previous works on collision. Indeed, the most important part is the existence of a fixed error term appearing in the collision dynamics
\[
\partial_z^{-1} \partial_s h_0 (\varepsilon t, \varepsilon (\rho +z) )  \begin{pmatrix}
0 \\
 1
 \end{pmatrix}
\]
that interacts at a long distance with the solitary wave (see \eqref{Syst_2}). Therefore, even if we are able to construct an approximate solution with a high degree of accuracy, such a property is lost due to the significant influence of the fixed-in-space error term that modifies the error and makes the problem have a similar error as in previous works \cite{Mu1}. Therefore, the classical first-order approximate solution is not good enough in our case.

Probably the most difficult part in the proof of Theorem~\ref{MT} is to find the correct energy functional that justifies the approximate dynamics for a large amount of time. This is not simple, and the fact that the equations change with time makes this problem particularly challenging, especially from the perspective of modulation analysis. Since modulations modify the solution, and the equation depends on time, we need precise corrections on the Lyapunov functional that take into account this uneven behavior. The new functional presented in
\eqref{def_F2}
\[
\begin{aligned}
\mathfrak F_2(\tau)= &~{} \frac12\int \left( -a (\partial_x u_2)^2 -c (\partial_x \eta_2)^2  + u_2^2+ \eta_2^2 \right) (\tau,x)dx \\
&~{} + \frac12\int \left( 2U_2 \eta_2 u_2 +  U_1 u_2^2 \right)(\tau,x) dx + \frac12\int u_2^2  \left( \eta_2 + h\right)(\tau,x)dx\\
&~{} - \omega  \int \left( \partial_x \eta_2 \partial_x u_2 + \eta_2 u_2 \right)(\tau,x)dx \\
&~{} -m_0(\tau) \int Q_\omega u_2(\tau,x)dx. 
\end{aligned}
\]
with $m_0$ given, is of proper interest, and leads to the key estimate \eqref{boussinesq_final_25}. Notice that this functional contains the additional term $-m_0(\tau) \int Q_\omega u_2(\tau,x)dx$ only depending on the  variable $u_2$, and not $\eta_2$.

The final step in the proof is a stability step, and it is carried out using standard arguments. The only problem is to ensure that almost conservation laws are sufficiently small in variation, and this is ensured by the hypotheses presented in the main result. More effective bottom variations may lead to destroying the solitary wave, and will be treated elsewhere.

\subsection*{Organization of this paper}
This paper is organized as follows. Section~\ref{Sec:0} presents preliminary results on energy estimates that will be used throughout the paper. Section~\ref{2} is devoted to the construction of the solitary wave solution, namely equation~\eqref{Construction}. Section~\ref{Sec:3} addresses the construction of an approximate solution in the interaction region. In Section~\ref{Sec:4}, we rigorously justify the interaction regime. Section~\ref{Sec:5} contains the end of proof of Theorem~\ref{MT}, including estimates~\eqref{Interaction} and~\eqref{Exit}.


\section{Preliminaries}\label{Sec:0}

\subsection{Properties of the operator $(1-\px^2)^{-1}$} The following results are well-known in the literature, see El Dika \cite{ElDika2005-2} for further details and proofs.

\begin{definition}[Canonical variable]\label{Can_Var}
Let $u=u(x) \in L^2$ be a fixed function. We say that $f$ is canonical variable for $u$ if $f$ uniquely solves the equation
\be\label{Canonical}
f- \partial_x^2 f  = u, \quad f\in H^2(\R).
\ee
In this case, we simply denote $f=  (1-\partial_x^2)^{-1} u.$
\end{definition}

Canonical variables are standard in equations where the operator $(1-\partial_x^2)^{-1}$ appears; one of the well-known example is given by the Benjamin-Bona-Mahony BBM equation, see e.g. \cite{ElDika2005-2}.
%


\begin{lemma}[Equivalence of local $L^2$ and $H^1$ norms, \cite{KMPP}]\label{lem:L2 comparable}
Let $\phi$ be a smooth, bounded positive weight satisfying $|\phi''| \le \lambda \phi$ for some small but fixed $0 < \lambda \ll1$.  Let $f$ be a canonical variable for $u$, as introduced in Definition \ref{Can_Var} and \eqref{Canonical}.  The following are satisfied:

\begin{itemize}
\item If $u\in L^2$, then for any $d_1,d_2,d_3 > 0$, there  exist $c_0, C_0 >0$, depending on $d_j$ ($j=1,2,3$) and $\lambda >0$, such that
\begin{equation}\label{eq:L2_est}
c_0  \int \phi \, u^2 \le \int \phi\left(d_1f^2+d_2(\partial_x f)^2+d_3(\partial_x^2 f)^2\right) \le C_0 \int \phi \, u^2.
\end{equation}
\item If $u\in H^1$, then for any $d_1,d_2,d_3 > 0$, there  exist $c_0, C_0 >0$ depending on $d_j$, $j=1,2,3$, and $\lambda >0$ such that
\begin{equation}\label{eq:H1_est}
c_0\int \phi (\partial_x u)^2 \le  \int \phi\left(d_1(\partial_x f)^2+d_2(\partial_x^2 f)^2+d_3(\partial_x^3 f)^2\right) \le C_0 \int \phi (\partial_x u)^2.
\end{equation}
\end{itemize}
\end{lemma}

Lemma \ref{lem:L2 comparable} states that canonical variables can be translated into standard variables even in the presence of weights. To estimate nonlinear terms, we shall need the following set of estimates.

\begin{lemma}[\cite{ElDika2005-2,KMPP}]\label{Dika1}
The operator $ (1-\px^2)^{-1}$ satisfies the following properties:
\begin{itemize}
\item Comparison principle: for any $u,v\in H^1$,
\begin{equation}\label{eq:inverse op1}
v \le w \quad  \Longrightarrow \quad (1-\px^2)^{-1} v \le (1-\px^2)^{-1} w.
\end{equation}
\item Suppose that $\phi =\phi(x)$ is such that
\[
(1-\px^2)^{-1}\phi(x) \lesssim \phi(x), \quad x\in \R,
\]
for $\phi(x) > 0$ satisfying $|\phi^{(n)}(x)| \lesssim \phi(x)$, $n \ge 0$. Then, for $v,w,h \in H^1$, we have
\begin{equation}\label{eq:nonlinear1-1}
\int \phi^{(n)} v (1-\px^2)^{-1}(wh)_x ~\lesssim ~ \norm{v}_{H^1} \int \phi (w^2 + w_x^2 +h^2 + h_x^2),
\end{equation}
and
\begin{equation}\label{eq:nonlinear1-2}
\int\phi^{(n)} v (1-\px^2)^{-1}(wh) ~\lesssim ~\norm{v}_{H^1} \int \phi(w^2 +h^2).
\end{equation}
\item Under the previous conditions, we have
\begin{equation}\label{eq:nonlinear1-3}
\int (\phi v_x)_x (1-\px^2)^{-1}(wh) \lesssim \norm{v}_{H^1} \int \phi (w^2 + w_x^2 +h^2 + h_x^2).
\end{equation}
and
\begin{equation}\label{eq:nonlinear1-4}
\int \phi v_x (1-\px^2)^{-1} (wh)_x \lesssim \norm{v}_{H^1} \int \phi (w^2 + w_x^2 +h^2 + h_x^2).
\end{equation}
\end{itemize}
\end{lemma}
Later, Lemma \ref{Dika1} will be useful to prove an energy estimate allowing one to prove \eqref{PreInteraction}.

\subsection{Exponential decay of solitons and the  linearized system}
As mentioned in the introduction, condition \eqref{velo} guarantees that the solutions of the system \eqref{exact0} have exponential decay, as well as their derivatives. Although this fact may be known to specialists, we provide here a self-contained proof, following the strategy used in Theorem~8.1.1 in  \cite{Caz} and in \cite{Olivera}, which will also be useful for establishing the decay of solutions of the linearized operator.

\begin{lemma}
\label{lem:decay:solitons}
If  $\bd Q_\omega=(R_\omega,Q_\omega)^T \in H^1\times H^1$ is a solution of the system \eqref{exact0},
then $Q_\omega$ and $R_\omega$ belong to  $H^\infty=\cap_k H^k$ and have exponential decay, as well as their derivatives of any order. In particular,
there exist $C_0,\mu_0(\omega)>0$ such that for all $x\in \mathbb R$,
\begin{equation}\label{decay_Q}
|\partial_x^{k}\bd{Q}_\omega(x)| \leq  C_0 e^{-\mu_0 |x|}, \quad k=0,1,2,3.
\end{equation}
\end{lemma}
\begin{proof}
For simplicity, we denote $Q=Q_\omega$ and $R=R_\omega$.
To establish the regularity, it is enough to notice that the matrix $A=\begin{pmatrix} a & w\\w &
c\end{pmatrix}$ is negative-definite, since $\omega $ satisfies  \eqref{velo}. Therefore, we can diagonalize the
$A$ as $P^T D P$, with $D=\textup{diag}(-\lambda_1,-\lambda_2)$, with $\lambda_1,\lambda_2>0$. Hence, setting the new variable $(\tilde Q, \tilde R)^T=P(\tilde Q, \tilde R)$, system \eqref{exact0} can be recast in the simpler  form
\[
(1-\lambda_1\partial_{x}^2)\tilde Q=F_1,\\
 (1-\lambda_2\partial_{x}^2)\tilde R=F_2,
\]
where $F_1$ and $F_2$ are linear combinations of $Q,R$, $RQ$ and $Q^2$, and therefore $F_1$ and $F_2$ belong to $H^1$.
Hence, from the elliptic regularity of scalar equations, we deduce that $\tilde Q$ and $\tilde R$ belong to $H^3$. A bootstrap argument yields that  $\tilde Q,\tilde R\in H^\infty$, which proves the regularity of solitons.

To prove the exponential decay, let us set the positive function
\begin{equation}
\label{def:phi}
\phi_{\ve,\delta}(x)=\exp\left( \frac{\ve \abs{x}}{1+\delta \abs{x}}\right), \quad \text{ for all } x\in \R,
\end{equation}
where $\ve,\delta >0$ are small constants.
     In this manner, $\phi_{\ve,\delta}$ belongs to $L^\infty(\R)\cap H^1_{\textup{loc}}(\R),$ and $\abs{\phi'_{\ve,\delta}(x)}\leq \ve \phi_{\ve,\delta}(x)$, for all $x\in \R\setminus \{0\}$. For the sake of simplicity, we now drop the subscripts $\ve$ and $\delta$, and denote the function by $\phi.$
Multiplying the first equation in \eqref{exact0} by $\phi Q$,  the second one by  $\phi R$, and  integrating by parts, we obtain:
\begin{align*}
-\omega \int Q R \phi -\omega \int R'(Q'\phi+Q\phi')-a\int Q'(Q'\phi+Q\phi')+\int Q^2\phi+\int  R Q^2\phi=0,\\
-c\int R'^2 \phi-c\int R'R \phi'+\int R^2\phi-\omega \int Q R \phi-\omega \int Q'(R' \phi +R \phi')+\frac12 \int  R Q^2\phi=0.
\end{align*}
Adding these equations and  recalling that  $\abs{\phi'}\leq \ve \phi$, we get
\begin{equation}
\begin{aligned}
&\int (R^2\phi + Q^2\phi-2\omega  Q R \phi  ) \leq
\int (cR'^2 \phi+
a Q'^2 )\phi
\\
  &
+2\omega \int Q'R' \phi
+\ve \int (\abs{c}\abs{R'}\abs{R}
+ \abs{a} \abs{Q'}\abs{Q} )\phi
+\ve \omega \int (\abs{Q'}\abs{R} +\abs{Q}\abs{R'})\phi\\
\label{eq:dem2}
&+\frac32
\int Q^2\abs{R} \phi.
\end{aligned}
\end{equation}
The last term is easy to estimate.  Indeed, since $Q\in H^1$,   for every $\epsilon>0$, there is $r_\epsilon>0$ such that $\abs{Q(x)}\leq \epsilon$, for all
$\abs{x}\geq r_\epsilon$, so that
\begin{equation*}
    \int_\R Q^2\abs{R} \phi\leq \epsilon \int_{\abs{x}\geq r_\epsilon}\abs{R}\abs{Q}\phi +\int_{\abs{x}\leq r_\epsilon} Q^2\abs{R}\phi
\leq \frac{\epsilon}{2} \int_{\R}({R}^2+{Q}^2)\phi +\int_{\abs{x}\leq r_\epsilon} Q^2\abs{R}\phi,
\end{equation*}
where we used the Cauchy inequality $2x_1x_2\leq x_1^2+x_2^2$, for $x_1,x_2\in\R$.
Invoking again  the Cauchy inequality for the other terms in \eqref{eq:dem2}, we can gather all the terms depending on $Q$ and $R$ on the left-hand side, and leave  the terms depending on $Q'$, $R'$ on the right-hand side, as follows:
\begin{align}
\label{eq:dem3}
\int (C_1R^2 + C_2 Q^2)\phi \leq
-\int (C_3R'^2 + C_4 Q^2 -2\omega R'Q')\phi+\int_{\abs{x}\leq r_\epsilon} Q^2\abs{R}.
\end{align}
where $C_1=1-\omega-\abs{c}\ve/2-\ve \omega/2-3\epsilon/4$,
$C_2=1-\omega-\abs{a}\ve/2-\ve \omega/2-3\epsilon/4$,
$C_3=\abs{c}-\abs{c}\ve/2-\ve \omega/2$,
$C_4=\abs{a}-\abs{a}\ve/2-\ve \omega/2$.
Since $\omega \in [0,1)$, we  can take $\ve$ and $\epsilon$ small so that  $C_1,C_2,C_3$ and $C_4$ are strictly positive. In addition, since $\omega<\sqrt{\abs{a}\abs{c}}$, decreasing the value of $\ve$, we can also assume that $w<\sqrt {C_3C_4}$.
Thus, using the identity \eqref{ineq:abcd}, with $\alpha=C_3$, $\beta=C_4$ and $\kappa =\omega$,
we conclude that the first integral on the right-hand side of \eqref{eq:dem3} is nonnegative, so that
\begin{align*}
\int (C_1R^2 + C_2 Q^2)\exp\Big(\frac{\ve \abs{x}}{1+\delta \abs{x}}\Big) \leq
\frac32\int_{\abs{x}\leq r_\epsilon} Q^2\abs{R}\exp(\ve \abs{r_\epsilon}) ,
\end{align*}
for all $\delta>0$. Therefore, by the Fatou lemma, we can pass to the limit as $\delta\to0$ to deduce that
$(C_1 R^2 + C_2 Q^2)e^{\ve \abs{x}} \in L^1(\R).$
Since $Q$ and $R$ belong to $C^{0,1/2}(\R)$, by the Sobolev embedding theorem, this implies that $\eta$ and $u$ decay exponentially, as desired.
The exponential decay of derivatives can be deduced by differentiating the equation and using a similar reasoning. We give the details in the next lemma, in a more general setting.
\end{proof}
 In the proof of the main theorem, we will need the exponential decay of solutions of the linearized problem, as established in the next result.
\begin{lemma}
\label{lem:exp:decay}
Let $F, G \in L^\infty(\R)$ be functions satisfying that $\abs{F(x)}+\abs{G(x)}\leq Ce^{-\alpha \abs{x}}$, for all $x\in \R$, for some constants $C,\alpha>0$.
Assume that $(\eta,u) \in H^1\times H^1$ is a solution 	of the linear system
\begin{equation}
\label{system:linear0}
\mathcal L ( \eta , u )^T = ( F, G )^T,
\end{equation}
with $\omega$ satisfying \eqref{velo}.
Then, $(\eta,u) \in
H^2\times H^2$ and there are some constants $\beta \in (0,\alpha]$ and $\tilde C>0$ such that $\abs{\eta(x)}+\abs{u(x)}\leq \tilde Ce^{-\beta \abs{x}}$, for all $x\in \R$.
\end{lemma}

\begin{proof}
Notice that $\eta$ and $u$ are bounded by the Sobolev embedding theorem. Bearing in mind the exponential decay of $Q_\omega$ and $R_\omega$ in \eqref{decay_Q},
we can recast as \eqref{system:linear0} as
\begin{align}
\label{system:linear}
c\eta''+\eta-\omega u+\omega u''=\tilde F,\\
\label{system:linear2}
-\omega \eta+\omega \eta''+au''+u=\tilde G,
\end{align}
with $\tilde F=F-Q_\omega\eta$ and $\tilde G=G-Q_\omega\eta-R_\omega u$, which satisfy  $\abs{\tilde F(x)}+\abs{\tilde G(x)}\leq K_0 e^{-\tilde \alpha \abs{x}}$, for all $x\in \R$, with $ \tilde \alpha=\min\{\alpha,\mu_0\},$  and $K_0>0$.

The regularity statement follows as in Lemma~\ref{lem:decay:solitons}.
To prove the exponential decay, we use the function $\phi= \phi_{\ve,\delta}$ in \eqref{def:phi}. Indeed, multiplying \eqref{system:linear} by $\phi\eta$, and \eqref{system:linear2} by  $\phi u$, and  integrating by parts, and adding these equations as in the proof of Lemma~\ref{lem:decay:solitons}, we obtain:
\begin{equation}
\begin{aligned}
&\int (\eta^2\phi + u^2\phi-2\omega  u \eta \phi  ) \leq
\int (c\eta'^2 \phi+
a u'^2 )\phi
\\
  &
+2\omega \int u'\eta' \phi
+\ve \int (\abs{c}\abs{\eta'}\abs{\eta}
+ \abs{a} \abs{u'}\abs{u} )\phi
+\ve \omega \int (\abs{u'}\abs{\eta} +\abs{u}\abs{\eta'})\phi\\
\label{eq:dem22}
&+K_1
\int e^{-\tilde \alpha \abs{x}} \phi,
\end{aligned}
\end{equation}
where $K_1=K_0 (\norm{\eta}_{L^\infty  }+\norm{u}_{L^\infty}).$
The last term is easy to estimate.  Indeed,
assuming  that $\ve\leq \tilde \alpha/4$, we deduce that
$$e^{-\tilde \alpha \abs{x}} \phi =e^{-\tilde \alpha \abs{x}/2} \phi \leq e^{-\tilde \alpha \abs{x}/2} e^{(-\tilde \alpha \abs{x}/4-\delta\tilde \alpha\abs{x}^2)/(1+\delta\abs{x})}\leq e^{-\tilde \alpha \abs{x}/2},$$
so that the last integral term in \eqref{eq:dem22} can be bounded by $K_1/\tilde \alpha.$ Handling the other terms  as in the proof of Lemma~\ref{lem:decay:solitons}, we conclude that
\[
\int (C_1\eta^2 + C_2 u^2)\phi \leq
-\int (C_3\eta'^2 + C_4 u^2 -2\omega u'\eta')\phi +K_1/\tilde \alpha.
\]
where $C_1=1-\omega-\abs{c}\ve/2-\ve \omega/2$,
$C_2=1-\omega-\abs{a}\ve/2-\ve \omega/2$,
$C_3=\abs{c}-\abs{c}\ve/2-\ve \omega/2$,
$C_4=\abs{a}-\abs{a}\ve/2-\ve \omega/2$.
Taking $\ve>0$ small enough, we deduce that
\begin{align*}
\int (C_1\eta^2 + C_2 u^2)\exp\Big(\frac{\ve \abs{x}}{1+\delta \abs{x}}\Big) \leq \tilde K_1/\tilde \alpha,
\end{align*}
for all $\delta>0$. Therefore, by the Fatou lemma, we can pass to the limit as $\delta\to0$ to conclude that
$(C_1\eta^2 + C_2 u^2)e^{\ve \abs{x}} \in L^1(\R)$,
which implies that $\eta$ and $u$ decay exponentially, as desired.
\end{proof}

\subsection{Local well-posedness}

In this paragraph we discuss the well-posedness of the model \eqref{boussinesq}. It is well-known (see \cite{BCS1,BCS2}) that in the case of flat bottom, namely $h\equiv 0$, there is local well-posedness in the generic case $a,c<0$, $b,d>0$.
Global well-posedness is ensured at least in the case of small data thanks to the conservation of the Hamiltonian \eqref{Energy}. The case of global well-posedness for large data remains open, except in the vicinity of solitary waves.

Following the ideas in \cite{BCS2}, we will prove local well-posedness for \eqref{boussinesq}, under the assumptions \eqref{hypoH}.

\begin{lemma}
The system \eqref{boussinesq} is locally well-posed in $H^s\times H^s$, $s\geq 0$.
\end{lemma}

\begin{proof}
Following the proof in \cite[Theorem 2.5]{BCS2}, define the variables $v$ and $w$ as follows:
\[
\eta =\mathcal H(v+w), \quad u=v-w,
\]
where $\mathcal H$ is the Fourier multiplier  given by
\[
\mathcal F(\mathcal H g)(k)=h(k) \mathcal F(g)(k), \quad h(k):=\left(\frac{1-ak^2}{1-ck^2}\right)^\frac{1}{2}.
\]
Here we have used that $b=d$ in the considered case. Then, \eqref{boussinesq} is written as
\be\label{eqn bold v}
\partial_t {\bf v} +\mathcal B {\bf v} = \mathcal N({\bf v}, h)
\ee
where ${\bf v}:=(v,w)^T$, $\mathcal B$ is the skew-adjoint operator with symbol
\[
ik\begin{pmatrix} \sigma(k) & 0 \\ 0 & \sigma(k)\end{pmatrix}, \quad \sigma(k):=\left( \frac{(1-ak^2)(1-ck^2)}{(1+k^2)^2}\right)^{\frac12},
\]
and $\mathcal N({\bf v}, h)$ is given as follows:
\[
\mathcal N({\bf v}, h) = -\mathcal P^{-1}(1-\partial_x^2)^{-1} \begin{pmatrix} \partial_x((v-w)(\mathcal H(v+w)+h)) +(-1+a_1\partial_x^2)\partial_t h\\ (v-w)\partial_x(v-w)+c_1\partial_t^2 \partial_x h \end{pmatrix},
\]
and finally,
\[
\mathcal F(\mathcal P^{-1}):= \frac12 \begin{pmatrix} h(k)^{-1} & 1 \\ h(k) & -1\end{pmatrix}.
\]
Via Duhamel's formula, \eqref{eqn bold v} is equivalent to
\[
{\bf v} =S(t){\bf v}_0 +\int_0^t S(t-s)\mathcal N({\bf v}, h) ds =: \mathcal J[{\bf v}] ,
\]
where
\[
{\bf v}_0:= \frac12\begin{pmatrix}
    \mathcal H^{-1}\eta_0 + u_0\\ \mathcal H^{-1}\eta_0 - u_0
\end{pmatrix} \in H^s.
\]
Now we will prove that for $\varepsilon>0$ small, $\mathcal J$ is a contraction on a sufficiently (but fixed) large ball of radius $R>0$ and time $T$ small enough. Following the proof in \cite[Theorem 2.5]{BCS2}, it is enough to check the size of $\mathcal N({\bf v}, h)$ and the difference $\mathcal N({\bf v}_1, h)-\mathcal N({\bf v}_2, h)$. In the first case,
\[
\|\mathcal N({\bf v}, h)\|_{H^s} \lesssim \|{\bf v}\|_{H^s}^2 + \varepsilon^2,
\]
and in the second case,
\[
\|\mathcal N({\bf v}_1, h)-\mathcal N({\bf v}_2, h)\|_{H^s} \lesssim R(1+\varepsilon)\|{\bf v}\|_{H^s},
\]
where $R$ is the size of the ball in the $H^s\times H^s$ topology such that ${\bf v}_1,{\bf v}_2\in B(0,R)$.
\end{proof}


\subsection{Modified Energy and Momentum} Recall the energy \eqref{Energy} in the constant case $h\equiv 0$. Now we shall prove the following variations in energy and momentum. First, recall the more convenient version of
 \eqref{boussinesq}:
\begin{equation}\label{eq:abcd}
\begin{aligned}
\pt \eta  = &~{} a \px u -(1+a)(1-\partial_x^2)^{-1}\px u - (1-\partial_x^2)^{-1}\px(u(\eta +h)) \\
& ~{} + (1-\partial^2_x)^{-1}\left(-1+ a_1 \partial_x^2\right) \partial_t h \\
\pt u  =  &~{} c \px \eta -(1+c)(1-\partial_x^2)^{-1}\px \eta - \frac12(1-\partial_x^2)^{-1}\px(u^2)\\
& ~{}+ c_1  (1-\partial^2_x)^{-1} \partial_t^2 \partial_x h.
\end{aligned}
\end{equation}

\begin{lemma}\label{Virial_bousH}
 Consider the modified energy
\be\label{Energy_new000}
H_h[\eta,u ](t):= \frac12\int \left( -a (\partial_x u)^2 -c (\partial_x \eta)^2  + u^2+ \eta^2 + u^2(\eta + h) \right)(t,x)dx.
\ee
Then the following is satisfied:
\be\label{Energy_new_1}
\begin{aligned}
\frac{d}{dt} H_h[\eta,u ](t) = &~{} -a c_1 \int  u  \partial_t^2 \partial_x h \\
&~{} +c_1 \int \left( 1+a    +   \eta + h  \right)u (1- \partial_x^2)^{-1}  \partial_t^2 \partial_x h \\
&~{}  + c\int \eta  \partial_th + c a_1 \int \partial_x \eta \partial_x \partial_th \\
&~{}   +(a_1- 1) \int \left(  (1+c) \eta  + \frac12 u^2 \right)(1- \partial_x^2)^{-1}  \partial_th  \\
&~{}  -a_1 \int \left(  (1+c) \eta  + \frac12 u^2 \right) \partial_th  + \frac12 \int u^2  \partial_t h.
\end{aligned}
\ee
\end{lemma}
\begin{proof}
We have from \eqref{Energy_new000},
\[
\begin{aligned}
\frac{d}{dt} H_h[\eta,u ](t) = &~{} \int \left( -a \partial_{x} u \partial_{tx} u  -c  \partial_x \eta \partial_{tx} \eta  + u \partial_t u + \eta \partial_t \eta  +  \frac12 \partial_t(u^2(\eta+ h)) \right)\\
= &~{} \int \left( a   \partial_{x}^2 u  + u   +  u  (\eta+ h)  \right) \partial_t u \\
&~{}  + \int \left( c \partial_{x}^2  \eta  +  \eta  + \frac12 u^2 \right)\partial_{t} \eta   + \frac12 \int u^2  \partial_t h.
\end{aligned}
\]
Using \eqref{eq:abcd},
\[
\begin{aligned}
\frac{d}{dt} H_h[\eta,u ](t) = &~{} -\int \left( a   \partial_{x}^2 u  + u   +  u  (\eta+ h)  \right) (1- \partial_x^2)^{-1} \partial_x \! \left( c\, \partial_x^2 \eta + \eta  + \frac12 u^2 \right) \\
&~{} + c_1 \int \left( a   \partial_{x}^2 u  + u   +  u  (\eta +h)  \right) (1- \partial_x^2)^{-1}  \partial_t^2 \partial_x h \\
&~{} - \int \left( c \partial_{x}^2  \eta  +  \eta  + \frac12 u^2 \right) (1- \partial_x^2)^{-1} \partial_x \! \left( a\, \partial_x^2 u +u + u (\eta +h) \right) \\
&~{} + \int \left( c \partial_{x}^2  \eta  +  \eta  + \frac12 u^2 \right)(1- \partial_x^2)^{-1} (-1 +a_1 \partial_x^2) \partial_th   + \frac12 \int u^2  \partial_t h.
\end{aligned}
\]
A further simplification directly yields \eqref{Energy_new_1}.
\end{proof}

Now we consider the momentum
\begin{equation}\label{P}
P[\eta,u](t):= \int (\eta u +\partial_x \eta \partial_x u)(t,x)dx. 
\end{equation}

\begin{lemma}\label{Virial_bousP}
Let $P$ be as in \eqref{P}. Then for any $t\geq0$,
\be\label{derP}
\begin{aligned}
\frac{d}{dt} P[\eta,u](t) = &~  {}   - \frac12 \int  \partial_x h u^2 - \int u (1-  a_1 \partial_x^2)\partial_t h - c_1 \int \partial_x\eta  \partial_t^2 h.
\end{aligned}
\ee
\end{lemma}

\begin{proof}
We compute:
\[
\begin{aligned}
\frac{d}{dt} P[\eta,u](t) =&~  \int (\partial_t \eta u + \eta \partial_t u + \partial_{tx} u \partial_x \eta + \partial_x u \partial_{tx} \eta)\\
=&~ {} \int (\partial_t \eta -\partial_x^2\partial_t  \eta) u + \int  \varphi (\partial_t u  - \partial_x^2\partial_t  u ) \eta  =:  I_1+I_2.
\end{aligned}
\]
Replacing \eqref{boussinesq}, and integrating by parts, we get
\[
\begin{aligned}
I_1 =&~  \int \partial_x u (a \partial_x^2 u +u + u(\eta+ h) ) - \int u (1- a_1 \partial_x^2)\partial_t h \\
=&~{} \int \partial_x u u(\eta+ h) - \int u (1- a_1 \partial_x^2)\partial_t h,
\end{aligned}
\]
and
\[
\begin{aligned}
I_2 =&~   \int   \partial_x \eta \left(c \partial_x^2 \eta +\eta + \frac12 u^2\right)+  c_1 \int \eta  \partial_t^2\partial_x h =   \frac12 \int   \partial_x \eta  u^2 - c_1 \int \partial_x \eta  \partial_t^2 h.
\end{aligned}
\]
Adding both identities, we conclude \eqref{derP}.
\end{proof}

We shall now use Lemmas \ref{Virial_bousH} and \ref{Virial_bousP} to estimate the evolution of the modified energy and momentum.

\begin{lemma}\label{Lem2p3}
The following estimates hold:
\be\label{est_E}
\begin{aligned}
\left| \frac{d}{dt} H_h[\eta,u ](t) \right| \lesssim  &~{}  \varepsilon^2  e^{-k_0\varepsilon |t| }   \int (u^2+\eta^2+\abs{u}+\abs{\eta}) e^{-l_0\varepsilon |x| },
\end{aligned}
\ee
\be\label{est_P}
\begin{aligned}
\left| \frac{d}{dt} P[\eta,u ](t) \right| \lesssim  &~{}  \varepsilon^2  e^{-k_0\varepsilon |t| }   \int (u^2+ |\eta | +|u|) e^{-l_0\varepsilon |x| }.
\end{aligned}
\ee
Additionally, assuming that $\omega$ does not vary on time, we have for all $t_1,t_2\in\R,$
\be\label{est_E_Qc}
\left| H_h[\bd{Q}_\omega ](t_2) -H_h[\bd{Q}_\omega ](t_1)  \right| \leq  \frac12\int  Q_\omega^2(x) | h (t_2,x)-h(t_1,x)|dx.
\ee
\end{lemma}

\begin{proof}
From \eqref{Energy_new_1} and \eqref{hypoH},
\[
\begin{aligned}
\frac{d}{dt} H_h[\eta,u ](t) = &~{} -a c_1 \varepsilon^4 \int  u  \partial_s^2 \partial_y h_0(\varepsilon t, \varepsilon x) \\
&~{} +c_1  \int \left( 1+a    +   \eta + \varepsilon h_0(\varepsilon t, \varepsilon x)  \right) u (1- \partial_x^2)^{-1}  \partial_t^2 \partial_x h  \\
&~{}  + c \varepsilon^2 \int \eta  \partial_s h_0(\varepsilon t, \varepsilon x) - c a_1 \varepsilon^4 \int  \eta \partial_y^2 \partial_s h_0(\varepsilon t, \varepsilon x) \\
&~{}   +(a_1- 1) \int \left(  (1+c) \eta  + \frac12 u^2 \right)(1- \partial_x^2)^{-1}  \partial_t h  \\
&~{}  -a_1 \varepsilon^2 \int \left(  (1+c) \eta  + \frac12 u^2 \right) \partial_s h_0(\varepsilon t, \varepsilon x)  + \frac12 \varepsilon^2  \int u^2  \partial_s h_0(\varepsilon t, \varepsilon x).
\end{aligned}
\]
Also, using \eqref{eq:inverse op1}, we get
\[
\begin{aligned}
\abs{(1- \partial_x^2)^{-1}  \partial_t^2 \partial_x h }\lesssim &~{} \varepsilon^{4} e^{-k_0\varepsilon |t|} (1- \partial_x^2)^{-1} e^{-l_0\varepsilon |x|}\\
\lesssim &~{} \varepsilon^{4} e^{-k_0\varepsilon |t|}  \int e^{-|x-y|}  e^{-l_0\varepsilon |y|}dy \lesssim  \varepsilon^{4} e^{-k_0\varepsilon |t|}e^{-l_0\varepsilon |x|}.
\end{aligned}
\]
A similar bound allows us to conclude that
\[
|(1- \partial_x^2)^{-1}  \partial_t^2 \partial_x h | \lesssim \varepsilon^{4} e^{-k_0\varepsilon |t|}e^{-l_0\varepsilon |x|}
\ \text { and } \
|(1- \partial_x^2)^{-1}  \partial_t h( t, x) | \lesssim  \varepsilon^{2} e^{-k_0\varepsilon |t|}e^{-l_0\varepsilon |x|}.
\]
Therefore, we deduce that
\[
\begin{aligned}
\left| \frac{d}{dt} H_h[\eta,u ](t) \right| \lesssim  &~{}  \varepsilon^4 e^{-k_0\varepsilon |t| } \int  |u| e^{-l_0\varepsilon |x| } +   \varepsilon^4 e^{-k_0\varepsilon |t|} \int \left( 1  +  | \eta | \right) |u| e^{-l_0\varepsilon |x|}  \\
&~{} + \varepsilon^2 e^{-k_0\varepsilon |t| } \int  |\eta | e^{-l_0\varepsilon |x| }    +\varepsilon^2 e^{-k_0\varepsilon |t|} \int \left(   |\eta | +  u^2 \right) e^{-l_0\varepsilon |x|}   \\
&~{}  +  \varepsilon^2  e^{-k_0\varepsilon |t| }  \int u^2 e^{-l_0\varepsilon |x| }   \\
\lesssim &~{}   \varepsilon^2  e^{-k_0\varepsilon |t| }   \int (u^2+\eta^2) e^{-l_0\varepsilon |x| }  +  \varepsilon^{2}  e^{-k_0\varepsilon |t| } \int ( |\eta | +|u| ) e^{-l_0\varepsilon |x| }.
\end{aligned}
\]
This proves \eqref{est_E}.
We prove now \eqref{est_P}. From \eqref{derP} and \eqref{hypoH}, we have
\[
\begin{aligned}
\frac{d}{dt} P[\eta,u](t) = &~{}  - \frac12 \varepsilon^2 \int  \partial_y h_0(\varepsilon t,\varepsilon x) u^2 \\
&~{} - \varepsilon^2 \int u (1-  a_1 \varepsilon^2 \partial_y^2)\partial_s h_0 (\varepsilon t,\varepsilon x) + c_1 \varepsilon^4 \int \eta \partial_y \partial_s^2 h_0(\varepsilon t,\varepsilon x).
\end{aligned}
\]
Bounding terms,
\[
\begin{aligned}
\left| \frac{d}{dt} P[\eta,u](t) \right| \lesssim  &~{}  \varepsilon^2 e^{-k_0\varepsilon |t| }  \int e^{-l_0\varepsilon |x| }  u^2 \\
&~{} +  \varepsilon^2e^{-k_0\varepsilon |t| }  \int |u|   e^{-l_0\varepsilon |x| } + \varepsilon^4e^{-k_0\varepsilon |t| }  \int |\eta|  e^{-l_0\varepsilon |x| }.
\end{aligned}
\]
This concludes the proof of \eqref{est_P}. Finally, we prove \eqref{est_E_Qc}. From \eqref{Energy_new}, it is clear that
\begin{equation}
\label{der:Hh}
H_h[\bd{Q}_\omega ] = \frac12\int \left( -a (\partial_x Q_\omega)^2 -c (\partial_x R_\omega)^2  + Q_\omega^2+ R_\omega^2 + Q_\omega^2(R_\omega + h) \right).
\end{equation}
Consequently, since $\omega$ does not depend on time,
\begin{equation*}
H_h[\bd{Q}_\omega ](t_2) -H_h[\bd{Q}_\omega ](t_1) = \frac12\int  Q_\omega^2( h (t_2)-h(t_1)) .
\end{equation*}
Therefore, \eqref{est_E_Qc} is satisfied.
\end{proof}

Let us remark that, similarly to \eqref{der:Hh},  we have from \eqref{P},
\begin{equation}\label{P_ind}
P[\bd{Q}_\omega ]  = \int (R_\omega Q_\omega +\partial_xR_\omega \partial_x Q_\omega ),
\end{equation}
a quantity only depending on $\omega$. 

\subsection{Modulated $abcd$ waves in uneven media} Let us fix $h$ as in the hypotheses of Theorem \ref{MT}. Let us define the nonlinear mapping operator
\begin{equation}\label{S0}
{\bf S}_h(\eta,u):= \begin{pmatrix}
(1- \partial_x^2)\partial_t \eta  + \partial_x\!\left( a\, \partial_x^2 u +u + (\eta+ h) u \right) + (1 - a_1 \partial_x^2) \partial_th  \\
(1- \partial_x^2)\partial_t u  + \partial_x\! \left( c\, \partial_x^2 \eta + \eta  + \frac12 u^2 \right) - c_1  \partial_t^2 \partial_x h.
\end{pmatrix}.
\end{equation}
The equation ${\bf S}_h(\eta,u)=(0,0)^T$ represents an exact solution to \eqref{boussinesq}.  Consider smooth modulation parameters $(\omega(t),\rho(t)) \in (0,\infty) \times \mathbb R$ to be defined later. Let
\begin{equation}\label{mod_Q}
\begin{pmatrix} \eta \\ u \end{pmatrix}(t,x) = {\bd Q}_{\omega(t),\rho(t)} (x)  = \begin{pmatrix} R_{\omega(t)} \\ Q_{\omega(t)}\end{pmatrix} (x-\rho(t))
\end{equation}
be modulated solitary waves. Notice that for each $(\omega,x_0)\in (0,\infty)\times \mathbb R$ fixed, we have that ${\bd Q}_\omega$ defined in \eqref{sol_wave} is an exact solution to \eqref{boussinesq} with $h=0$. Indeed, as expressed in \eqref{exact0}, one has
\begin{equation}\label{exact}
{\bf S}_0(R_{\omega },Q_{\omega })= \partial_x \begin{pmatrix}
 - \omega (1-\partial^2_x)R_{\omega }  + a Q''_{\omega }+Q_{\omega }+R_{\omega }Q_{\omega } \\
 -\omega  (1-\partial^2_x)Q_{\omega } + c R''_{\omega }+R_{\omega }+\frac12Q_{\omega }^2
 \end{pmatrix}=\begin{pmatrix} 0 \\ 0 \end{pmatrix}.
\end{equation}
We denote $\Lambda R_\omega= \frac{\partial}{\partial\omega} R_\omega$, $\omega=\omega(t)$, $Q_\omega:=Q_{\omega(t)}(x-\rho(t))$ and  $R_\omega:=R_{\omega(t)}(x-\rho(t))$.
 We have
\begin{equation*}
{\bf S}_h(R_\omega,Q_\omega) =\begin{pmatrix} {S}_{h,1}  \\ {S}_{h,2} \end{pmatrix},
\end{equation*}
where
\[
\begin{aligned}
{S}_{h,1}:=  &~{} (1-\partial^2_x)(\Lambda R_\omega \omega'-(\rho'-\omega)R'_\omega)   \\
&~{} -\omega(1-\partial^2_x)R'_\omega+\partial_x(a Q''_\omega +Q_\omega+R_\omega Q_\omega+ h Q_\omega) + (1 - a_1\partial^2_x)\partial_t h,
\end{aligned}
\]
and
\[
\begin{aligned}
{S}_{h,2} = &~{} (1-\partial^2_x)(\Lambda Q_\omega \omega '-(\rho'-\omega )Q'_\omega)-\omega (1-\partial^2_x)Q'_\omega \\
&~{}   +\partial_x\left( c R''_\omega +R_\omega +\frac12Q_\omega^2\right)-c_1\partial^2_t\partial_xh.
\end{aligned}
 \]
Since $(R_\omega,Q_\omega)$ satisfy \eqref{exact0}, one has
\begin{equation}\label{S_hQ}
\begin{aligned}
{\bf S}_h(R_\omega,Q_\omega)= &~{} \omega '  (1-\partial^2_x)\begin{pmatrix}
\Lambda R_\omega \\
\Lambda Q_\omega
\end{pmatrix}
-(\rho'-\omega)(1-\partial^2_x)\partial_x
\begin{pmatrix}
R_\omega  \\
 Q_\omega
 \end{pmatrix}
\\
&~{} +
 \begin{pmatrix}
 \partial_x(hQ_\omega)  + (1 - a_1\partial^2_x)\partial_t h \\
 -c_1\partial^2_t\partial_xh
 \end{pmatrix} .
 \end{aligned}
\end{equation}
The last term above represents the contribution of the uneven bottom to the solitary wave modified dynamics.

Notice that \eqref{exact} naturally leads to the identity
\begin{equation}\label{AA}
\partial_x J\mathcal L \begin{pmatrix} R_\omega\\ Q_\omega \end{pmatrix}=
\begin{pmatrix}
   -\omega (1-\partial^2_x)R_{\omega}' +  a Q_{\omega}''' +Q_{\omega}'+ Q_{\omega}R_{\omega}' +Q_{\omega}' R_{\omega} \\
  -\omega (1-\partial^2_x) Q_{\omega}' + c R_{\omega}''' + R_{\omega}' + Q_{\omega}Q_{\omega}'
   \end{pmatrix} =\begin{pmatrix} 0 \\ 0 \end{pmatrix},
\end{equation}
and where $\mathcal L$ and $J$ were introduced in \eqref{def_L} and \eqref{coer_10}.
Let us consider the associated linearized dynamics, represented by the unbounded operator $\partial_x J\mathcal L$, described around a solitary wave.
From \eqref{AA}, we have that $(R_{\omega}',Q_{\omega}')^T$ belongs to the kernel of $\mathcal L$ and by the stability hypothesis $(ii)$, this is the generator of the kernel of $\mathcal L$. Recall that $a,c<0.$ 

\subsection{Coercivity} From \eqref{Energy_new}, we see that if $\bd{\eta_1}=(\eta_1,u_1)$ is a given perturbation of the solitary wave, then 
\[
\begin{aligned}
 &H_h[{\bd Q}_\omega +  \bd{\eta_1} ](t)
 \\& \quad = H_h[{\bd Q}_\omega] + \int \left( -a Q_\omega' \partial_x u_1  -c R_\omega'\partial_x \eta_1   + Q_\omega u_1 + R_\omega \eta_1 + \frac12 Q_\omega^2 \eta_1 + Q_\omega u_1 (R_\omega +h)   \right) \\
& \qquad \quad  + \frac12\int \left( -a (\partial_x u_1)^2 -c (\partial_x \eta_1)^2  + u_1^2+ \eta_1^2 + 2Q_\omega \eta_1 u_1 + u_1^2(R_\omega + \eta_1 + h) \right).
\end{aligned}
\]
Integrating by parts,
\be\label{Energy_new_new}
\begin{aligned}
& H_h[{\bd Q}_\omega +  \bd{\eta_1} ](t) \\
 &~{} =  H_h[{\bd Q}_\omega] + \int \left( a Q_\omega'' u_1  + c R_\omega'' \eta_1   + Q_\omega u_1 + R_\omega \eta_1 + \frac12 Q_\omega^2 \eta_1 + Q_\omega u_1 (R_\omega +h)   \right) \\
&~{} \quad + \frac12\int \left( -a (\partial_x u_1)^2 -c (\partial_x \eta_1)^2  + u_1^2+ \eta_1^2 + 2Q_\omega \eta_1 u_1 + u_1^2(R_\omega + \eta_1 + h) \right) .
\end{aligned}
\ee
Similarly,
\[
\begin{aligned}
P[{\bd Q}_\omega + \bd{\eta_1} ](t) = &~{} P[{\bd Q}_\omega] \\
&~{} + \int \left( R_\omega' \partial_x u_1  + Q_\omega '\partial_x \eta_1   + R_\omega u_1 + Q_\omega \eta_1  \right) + \int \left( \eta_1 u_1 +\partial_x\eta_1 \partial_x u_1\right) \\
=&~{} P [{\bd Q}_\omega]  + \int \left( -R_\omega''  u_1  - Q_\omega'' \eta_1   + R_\omega u_1 + Q_\omega \eta_1  \right)
+ \int \left( \eta_1 u_1 +\partial_x\eta_1 \partial_x u_1\right).
\end{aligned}
\]
Combing with  \eqref{exact0} and \eqref{def_L}, we conclude that
\begin{equation*}
\begin{aligned}
& H_h[{\bd Q}_\omega + \bd{\eta_1} ](t) -\omega P[{\bd Q}_\omega +  \bd{\eta_1}](t)  \\
 &~{}  = H_h[{\bd Q}_\omega] -\omega P[{\bd Q}_\omega]  + \int  Q_\omega u_1 h \\
 &~{}\quad + \frac12\int \left( -a (\partial_x u_1)^2 -c (\partial_x \eta_1)^2  + u_1^2+ \eta_1^2- 2\omega (\eta_1 u_1 +\partial_x\eta_1 \partial_x u_1)  \right) \\
  &~{}\quad + \frac12\int \left( 2Q_\omega \eta_1 u_1 +  R_\omega u_1^2 \right)  + \frac12\int u_1^2  \left( \eta_1 + h\right),
  \end{aligned}
\end{equation*}
so that
\be\label{Energy_new2}
H_h[{\bd Q}_\omega + \bd{\eta_1} ](t) -\omega P[{\bd Q}_\omega +  \bd{\eta_1}](t)
  = H_h[{\bd Q}_\omega] -\omega P[{\bd Q}_\omega]  + \left\langle  \bd{\eta_1}  ,  \mathcal L \bd{\eta_1}  \right\rangle   + \int  Q_\omega u_1 h+ \frac12\int u_1^2  \left( \eta_1 + h\right).
\ee
%
Finally, notice that from Lemma~\ref{coercivitydetailled} below, and  under the orthogonality condition $\langle \bd{\eta_1} ,\bd{Q}_\omega' \rangle =0$, it follows that the term $\left\langle  \bd{\eta_1}  ,  \mathcal L \bd{\eta_1}  \right\rangle$ satisfies the coercivity estimate:
\be\label{coer_1}
 \left\langle  \bd{\eta_1}  ,  \mathcal L \bd{\eta_1}  \right\rangle \geq c_0 \| \bd{\eta_1}\|_{H^1\times H^1}^2 -\frac1{c_0} \left| \langle \bd{\eta_1} , J(1-\partial_x^2)\bd{Q}_\omega \rangle \right|^2,
\ee
for some $c_0>0$.

 \begin{lemma}\label{coercivitydetailled}
  There exists $c_0>0$ such that, for all $\bd{\eta}\in H^1\times H^1$ satisfying the orthogonality conditions
\begin{equation}
\label{perp:cond}
\langle \bd{\eta} , \bd{Q}_\omega' \rangle = \langle \bd{\eta} , J(1-\partial_x^2)\bd{Q}_\omega\rangle =0,\quad
\end{equation}
we have
\be\label{coer_0}
\langle \mathcal L  \bd{\eta} , \bd{\eta} \rangle  \geq c_0 \| \bd{\eta} \|_{H^1\times H^1}^2.
\ee
\end{lemma}
\begin{proof}
We begin by proving a weaker property, i.e.\ the coercivity in $L^2\times L^2$.
Set $y=J(1-\partial_x^2)\bd{Q}_\omega$. Using the decomposition introduced in
the proof of Lemma \ref{lem:L-1}, we have $y=a_0\Phi_{-1}+z_0 \in \R \Phi_{-1}\oplus (D(\mathcal{L})\cap H_+)$.
Let us observe that condition \eqref{coer_10} yields $a_0\neq 0$.
Introduce ${\bf \eta}=a\Phi_{-1}+z$, that is orthogonal to $y$; i.e.
$0= aa_0+\langle z,z_0\rangle$. We have
\begin{equation}\label{mur0}
a^2 a_0^2 \leq \langle \mathcal{L} z,z \rangle \langle \mathcal{L}^{-1} z_0,z_0\rangle.
\end{equation}
This yields
\begin{align}\label{mur1}
\langle \mathcal L  \bd{\eta} , \bd{\eta} \rangle =-\mu_0 a^2+ \langle \mathcal L  z , z \rangle
&\geq \langle \mathcal L  z , z \rangle \left(-\frac{\mu_0 \langle \mathcal{L}^{-1} z_0,z_0\rangle}{a_0^2}  +1\right)\\
&\geq  \frac{-\mu_0}{a_0^2} \langle \mathcal L  z , z \rangle \langle\mathcal{L}^{-1} y,y\rangle.
\end{align}
Gathering \eqref{mur0} and \eqref{mur1} yields $\langle \mathcal L  \bd{\eta} , \bd{\eta} \rangle\geq c||\bd{\eta}||^2_{L^2\times L^2}$.
We now prove \eqref{coer_0} arguing by contradiction. Consider a sequence $\bd{\eta}_m$ such that $||{\bd \eta}_m||_{H^1\times H^1}=1$
and $\langle \mathcal L  \bd{\eta}_m , \bd{\eta}_m \rangle < \frac 1 m$.
Due to the $L^2\times L^2$ coercivity, $\bd{\eta}_m$ converges to $0$.
Recall that the bilinear form $\langle {\mathcal L}_0\cdot , \cdot \rangle$ introduced in \eqref{def_L0}
defines a scalar product on $H^1\times H^1$. Since $\mathcal{L}-\mathcal{L}_0$ is a continuous operator in $L^2\times L^2$, we have $\langle {\mathcal L}_0 \bd{\eta}_m, \bd{\eta}_m\rangle$
goes to $0$, which is a contradiction. Therefore, the proof is complete.
\end{proof}

\section{Construction of the generalized solitary wave}\label{2}

In this section, our objective is to prove the existence of a solution $\bd{\eta}$ to \eqref{boussinesq} such that \eqref{Construction} and \eqref{PreInteraction} are satisfied. Clearly, \eqref{PreInteraction} implies \eqref{Construction}. To construct the exact solitary-wave solution, we shall follow standard procedures, see \cite{Martel,Mu1} for early developments in the pure gKdV case, and in the variable medium cases, respectively. Recall the definition of ${\bf S}_h$ introduced in \eqref{S0}. In this section, we will provide a detailed description of the interaction between the constructed solitary wave and the moving bottom.
Let $\omega>0$ be a fixed parameter. Consider the exact solitary wave
\begin{equation}\label{SW_fixed}
{\bd Q}_{\omega,x_0}(t,x):= \begin{pmatrix} R_{\omega}(x-\omega t -x_0)\\ Q_{\omega}(x-\omega t -x_0) \end{pmatrix},
\end{equation}
exact solution to \eqref{boussinesq} with $h\equiv 0$. Denote
\begin{equation}\label{profile}
{\bd Q}_{\omega}(x):= {\bd Q}_{\omega,0}(0,x),
\end{equation}
as the profile associated with the solitary wave \eqref{SW_fixed}. Let $T_n>0$ be an increasing sequence of times tending to infinity, such that always $T_0\geq \frac12T_\varepsilon$. Let $(\eta_n,u_n)$ be the solution to the Cauchy problem associated with  \eqref{boussinesq} with an exact profile \eqref{profile} at $t=-T_n$, of the form
\begin{equation}\label{Cauchy}
{\bf S}_h(\eta_n,u_n)=(0,0)^T, \quad \hbox{such that}\quad (\eta_n,u_n)(-T_n,x)={\bd Q}_{\omega}(x + \omega T_n).
\end{equation}
Let us denote $I_n:= (-T_{n,-}, T_{n,+}) \ni -T_n$ the maximal interval of existence of each $(\eta_n,u_n)$. Notice that it is not clear if $I_n$ is bounded or unbounded.  However, one easily has $\|{\bd Q}_{\omega,x_0}(-T_n)\|_{H^1\times H^1} \leq C$ uniform in $n$. As usual, we shall establish uniform estimates at a fixed time $t<- \frac12T_\varepsilon$.

\subsection{Uniform estimates}
Let us introduce the notation  $$\bd{\eta}_n(t):=(\eta_n,u_n)\in C(I_n, H^1\times H^1)$$ for  the solution to \eqref{Cauchy} with initial data $(\eta_n,u_n)(-T_n)={\bd Q}_{\omega}(-T_n)$.
We will establish the following proposition.
\begin{proposition}\label{unif_est}
There exist $C_0,\mu_1,\varepsilon_0>0$ such that, for all $0<\varepsilon<\varepsilon_0$, and all $n\geq 0$, we have the inclusion $[-T_n,-\frac12T_\varepsilon] \subseteq I_n$ and the estimate
\begin{equation}\label{unif_est0}
\left\| (\eta_n,u_n)(t) - {\bd Q}_{\omega}(\cdot -\omega t)\right\|_{H^1 \times H^1} \leq C_0  e^{\mu_1 \varepsilon t},
\end{equation}
for all $t \in [-T_n,-\frac12T_\varepsilon]$.
\end{proposition}

In order to prove Proposition \ref{unif_est}, we shall use a bootstrap argument. For all $n\geq 0$, the following is true. If for some $T_{n,*} \in [\frac12T_\varepsilon, T_n]$ and for all $t \in [-T_n,-T_{n,*}]$, we have
\begin{equation}\label{unif_est0_0}
\left\| (\eta_n,u_n)(t) - {\bd Q}_{\omega}(\cdot -\omega t)\right\|_{H^1 \times H^1} \leq C_0 e^{\mu_1 \varepsilon t},
\end{equation}
then for all $t \in [-T_n,-T_{n,*}]$,
\begin{equation}\label{unif_est0_1}
\left\| (\eta_n,u_n)(t) - {\bd Q}_{\omega}(\cdot -\omega t)\right\|_{H^1 \times H^1} \leq \frac12C_0   e^{\mu_1 \varepsilon t}.
\end{equation}
A classical argument reveals that $ (\eta_n,u_n)(-T_n)={\bd Q}_{\omega,0}(-T_n)$, and \eqref{unif_est0_1} prove \eqref{unif_est0}. Let us assume \eqref{unif_est0_0} for $t \in [-T_n,-T_{n,*}]$.

\subsection{Modulation} Thanks to \eqref{unif_est0_0} we are in a regime where the solution is close to the exact flat-bottom $abcd$ solitary wave. Using this fact, we shall now prove a modulation result.

\begin{lemma}\label{mod1}
There exist $C_0>1$, $\mu_1,\varepsilon_0>0$ such that, if $T_0$ is sufficiently large and \eqref{unif_est0_0} holds  for all $0<\varepsilon<\varepsilon_0$, and all $n\geq 0$, then
there exist $C>0$ and  a $C^1$-modulation shift $\rho_{1,n}:[-T_n,-T_{n,*}] \to \mathbb R$ such that
\begin{equation}\label{z1}
\bd{\eta}_{1,n} := \bd{\eta}_n(t) - {\bd Q}_{\omega}(\cdot - \omega t -\rho_{1,n}(t)),
\end{equation}
satisfies, for all $t \in [-T_n,-T_{n,*}]$,
\begin{equation}\label{Ortho1}
\langle \bd{\eta}_{1,n}, {\bd Q}'_{\omega}(\cdot -\omega t- \rho_{1,n}(t)) \rangle =0 \ \text{ and }\ \|\bd{\eta}_{1,n}\|_{H^1\times H^1}\leq CC_0  e^{\mu_1 \varepsilon t}.
\end{equation}
Moreover, $\bd{\eta}_{1,n} = (\eta_{1,n},u_{1,n})$ in \eqref{z1} satisfies the system
\begin{equation}\label{eq_z1}
\begin{aligned}
&(1- \partial_x^2)\partial_t \eta_{1,n}  + \partial_x \! \left( a\, \partial_x^2 u_{1,n} +u_{1,n} + u_{1,n} R_\omega + \eta_{1,n} Q_\omega + (Q_\omega + u_{1,n}) h \right) \\
& \qquad = (-1 +a_1 \partial_x^2) \partial_th  + \rho_{1,n}' (1- \partial_x^2)R_\omega' , \\
& (1- \partial_x^2)\partial_t u_{1,n}  + \partial_x \! \left( c\, \partial_x^2 \eta_{1,n} + \eta_{1,n} +  Q_\omega u_{1,n} + \frac12 u_{1,n}^2 \right)
\\
& \qquad = c_1  \partial_t^2 \partial_x h + \rho_{1,n}' (1- \partial_x^2)Q_\omega' ,
\end{aligned}
\end{equation}
and
\begin{equation}\label{der_rho1}
| \rho_{1,n}'(t)| \leq C \|\bd{\eta}_{1,n}\|_{H^1\times H^1} + C \varepsilon e^{\varepsilon ( k_0 t + \frac9{10}l_0 t )}.
\end{equation}
\end{lemma}
Note that a standard output from $\bd \eta_n(-T_n)=(\eta_n,u_n)(-T_n)={\bd Q}_{\omega}(-T_n)$ is that $\bd{\eta}_{1,n}(-T_n) = 0$, $\rho_{1,n}(-T_n)=0$ and ${\bd Q}_{\omega}(\cdot + \omega T_n -\rho_{1,n}(-T_n)) ={\bd Q}_{\omega}(\cdot + \omega T_n)$. These facts will be used several times below.

\begin{proof}[Proof of Lemma \ref{mod1}]
The proof is classical by now. We only sketch the main steps, see \cite{MM,MM1,MM2} for detailed proofs. Let $t \in [-T_n,-T_{n,*}]$ be a fixed time. Under \eqref{unif_est0_0}, and $T_0$ larger if necessary, we can apply the Implicit Function Theorem on the orthogonality defined in \eqref{Ortho1}. For the sake of simplicity, we drop the dependence on $n$. Indeed, let
\[
\begin{aligned}
& H^1\times H^1\times \mathbb R \ni (\bd{\eta},\rho_1) = (\eta,u,\rho_1)
\\
& \qquad \longrightarrow \Gamma( \bd{\eta},\rho_1) := \langle \bd{\eta} - {\bd Q}_{\omega}(\cdot -\omega t- \rho_1 ), {\bd Q}'_{\omega}(\cdot -\omega t- \rho_1) \rangle \in \mathbb R.
\end{aligned}
\]
It is clear that this defines a $C^1$ functional in the above variables, and for all $t$, we have the identity $ \Gamma( {\bd Q}_{\omega}(\cdot -\omega t) , 0)  =0$. Additionally, $ \partial_{\rho_1}\Gamma|_{( {\bd Q}_{\omega}(\cdot -\omega t) , 0)} =-\|{\bd{Q}}_{\omega}' \|_{H^1\times H^1}^2\neq 0.$ Therefore, for each $t$ and $\bd{\eta}(t)$ small enough, there is $\rho_1(t)$ satisfying $ \Gamma(\bd{\eta}(t), \rho_1(t))  =0$. This proves the first part of \eqref{Ortho1}. The second part follows easily from the first part. The proof of \eqref{eq_z1} is direct after replacing \eqref{z1} in \eqref{boussinesq} and using the fact that ${\bd Q}'_{\omega}(\cdot -\omega t - x_0)$ solves \eqref{boussinesq} for any fixed $x_0$. Finally, to prove \eqref{der_rho1}, testing \eqref{eq_z1} against ${\bd Q}'_{\omega}(\cdot -\omega t- \rho_{1,n}(t))$ and using \eqref{Ortho1} one gets
\[
| \rho_{1,n}'(t)| \| \bd{Q}_\omega' \|_{L^2}^2  \leq C \|\bd{\eta}_{1,n}\|_{H^1\times H^1} + C \| \bd{Q}_{\omega} (|h| +| \partial_th | + |\partial_t^2 h| )\|_{L^1\times L^1}.
\]
Using \eqref{hypoH}, one finally gets \eqref{der_rho1}.
\end{proof}

\subsection{Energy and momentum estimates}\label{sub3p3} Suppose $C_0>1.$ Since $\eta_1$ satisfies \eqref{Ortho1},
\begin{equation}
\label{est:eta1}
\|\bd{\eta}_{1}\|_{L^\infty\times L^\infty }\leq  \|\bd{\eta}_{1}\|_{H^1\times H^1}\leq CC_0  e^{\mu_1 \varepsilon t}\leq C C_0\varepsilon,
\end{equation}
for $t\leq -T_\varepsilon.$
Recalling  that
$$
 \langle \bd{\eta}_1, J(1-\partial_x^2)\bd{Q}_\omega \rangle  =\int \left( (R_{\omega }-R_{\omega }'')  u_1  + (Q_{\omega }-Q_{\omega }'')  \eta_1  \right),
$$
and using \eqref{exact0}, it follows from  \eqref{Energy_new_new} that
\[
\begin{aligned}
& H_h[{\bd Q}_\omega +  \bd{\eta}_1 ](t)
\\
 &~{}  = H_h[{\bd Q}_\omega] + \omega  \langle \bd{\eta}_1, J(1-\partial_x^2)\bd{Q}_\omega \rangle + \int  Q_\omega u_1 h  \\
&~{} \quad + \frac12\int \left( -a (\partial_x u_1)^2 -c (\partial_x \eta_1)^2  + u_1^2+ \eta_1^2 + 2Q_\omega \eta_1 u_1 + u_1^2(R_\omega + \eta_1 + h) \right) .
\end{aligned}
\]
Since $\omega>0$, we deduce that
\[
\begin{aligned}
&\omega \left| \langle \bd{\eta}_1, J(1-\partial_x^2)\bd{Q}_\omega \rangle (t) \right| \\
&~{}  \leq \omega \left| \langle \bd{\eta}_1 , J(1-\partial_x^2)\bd{Q}_\omega \rangle (-T_n) \right|  + \left| H_h[{\bd Q}_\omega +  \bd{\eta}_1 ](t) -H_h[{\bd Q}_\omega +  \bd{\eta}_1 ](-T_n)\right|\\
&~{} \quad  +\left|\int Q_\omega u_1 h(t)\right|+\left|\int Q_\omega u_1 h(-T_n)\right|             + \left| H_h[{\bd Q}_\omega](-T_n)-  H_h[{\bd Q}_\omega](t)\right| \\
&~{} \quad  +\left|  \frac12\int \left( -a (\partial_x u_1)^2 -c (\partial_x \eta_1)^2  + u_1^2+ \eta_1^2 + 2Q_\omega \eta_1 u_1 + u_1^2(R_\omega + \eta_1 + h) \right)(t) \right| \\
&~{} \quad  + \left| \frac12\int \left( -a (\partial_x u_1)^2 -c (\partial_x \eta_1)^2  + u_1^2+ \eta_1^2 + 2Q_\omega \eta_1 u_1 + u_1^2(R_\omega + \eta_1 + h) \right)(-T_n) \right| .
\end{aligned}
\]
Since $\bd{\eta}_1(-T_n)=0$, the first and last terms in the right-hand side of this inequality are equal to zero. In addition, using \eqref{est:eta1}, we conclude that
\be\label{est11}
\begin{aligned}
& \left| \langle \bd{\eta}_1, J(1-\partial_x^2)\bd{Q}_\omega \rangle (t) \right| \\
&~{} \leq  \left| H_h[{\bd Q}_\omega +  \bd{\eta}_1 ](t) -H_h[{\bd Q}_\omega +  \bd{\eta}_1 ](-T_n)\right|\\
&~{} \quad + \left| H_h[{\bd Q}_\omega](-T_n)-  H_h[{\bd Q}_\omega](t)\right| + C\|\bd{\eta}_1(t)\|_{H^1\times H^1}^2+ C  \int u_1^2 h(t).
\end{aligned}
\ee
In view of, \eqref{Ortho1},
Since $-T_n<t<-T_\varepsilon,$ we deduce from \eqref{decay_Q} and \eqref{est_E},
\[
\begin{aligned}
\left| \frac{d}{dt} H_h[{\bd Q}_\omega +  \bd{\eta}_1  ](t) \right| \lesssim  &~{}  \varepsilon^2  e^{-k_0\varepsilon |t| }   \int ((Q_\omega + u_1)^2 + (R_\omega+ \eta_1)^2) e^{-l_0\varepsilon |x| }  \\
&~{} +  \varepsilon^{2}  e^{-k_0\varepsilon |t| } \int ( |R_\omega+ \eta_1 | +|Q_\omega + u_1 | ) e^{-l_0\varepsilon |x| } \\
\lesssim &~{}  \varepsilon^{2}  e^{-k_0\varepsilon |t| } \int ( |R_\omega| + |\eta_1 | +\abs{Q_\omega} + |u_1 | ) e^{-l_0\varepsilon |x| } \\
\lesssim &~{}  \varepsilon^{2}  e^{-k_0\varepsilon |t| } \int  e^{-\mu_0 |x -\omega t-\rho_1(t)|}   e^{-l_0\varepsilon |x| } + \varepsilon^{3/2}  e^{-k_0\varepsilon |t| } \| \bd{\eta}\|_{H^1\times H^1}, 
\end{aligned}
\]
since $\norm{e^{-l_0 \ve {x}}}\lesssim \ve^{-1/2}.$
From \eqref{der_rho1}, and the fact that $t<-T_\varepsilon$, we get
\[
\omega t + \rho_1(t) < \frac{9}{10} \omega t.
\]
Then, using \eqref{unif_est0_0},
\[
\left| \frac{d}{dt} H_h[{\bd Q}_\omega +  \bd{\eta}_1  ](t) \right| \lesssim \varepsilon^{2}  e^{-(k_0 + \frac{9}{10} l_0 \omega)\varepsilon |t| } + C_0 \varepsilon^{5/2}  e^{-(k_0+ \mu_1)\varepsilon |t| },
\]
and choosing $\mu_1:= \min\{k_0, \frac12(k_0 + \frac9{10}l_0 \omega)\}$, we get by integration,
\begin{equation}\label{est_H_1}
\left| H_h[{\bd Q}_\omega +  \bd{\eta}_1 ](t) -H_h[{\bd Q}_\omega +  \bd{\eta}_1 ](-T_n)\right| \lesssim (1+C_0  \varepsilon^{\frac12})\varepsilon e^{ 2 \mu_1 \varepsilon t } \lesssim \varepsilon e^{ 2 \mu_1 \varepsilon t },
\end{equation}
 for $-T_n<t<-T_\varepsilon$. Following similar computations, from \eqref{est_P} and \eqref{unif_est0_0}, one gets
\begin{equation}\label{est_P_1}
\left| P[{\bd Q}_\omega +  \bd{\eta}_1 ](t) - P[{\bd Q}_\omega +  \bd{\eta}_1 ](-T_n)\right| \lesssim (1+C_0  \varepsilon^{\frac12})\varepsilon e^{ 2 \mu_1 \varepsilon t } \lesssim \varepsilon e^{ 2 \mu_1 \varepsilon t }.
\end{equation}
Additionally, using \eqref{est_E_Qc},
\[
\begin{aligned}
\left| H_h[{\bd Q}_\omega](-T_n)-  H_h[{\bd Q}_\omega](t)\right|   \lesssim  &~{} \int  Q_\omega^2 (\abs{ h (-T_n)}+\abs{h(t)}\\
 \lesssim  &~{}  \varepsilon e^{k_0 \varepsilon t } \int  Q_\omega^2(x-\rho_1(t)) e^{-l_0 \varepsilon |x| } \\
 \lesssim &~{} \varepsilon  e^{k_0 \varepsilon t } \int  e^{-2\mu_0 |x -\rho_1(t)|}   e^{-l_0\varepsilon |x| }.
\end{aligned}
\]
Therefore,
 \begin{equation}\label{dif_H}
\left| H_h[{\bd Q}_\omega](-T_n)-  H_h[{\bd Q}_\omega](t)\right|  \lesssim  \varepsilon e^{-(k_0 + \frac{9}{5} l_0 \omega)\varepsilon |t| } \lesssim   \varepsilon e^{ 2 \mu_1 \varepsilon t }.
\end{equation}
Finally, \eqref{unif_est0_0} gives
 \begin{equation}\label{dif_NL}
\left|  \int u_1^2 h(t) \right| \lesssim \varepsilon e^{ k_0 \varepsilon t } \int u_1^2 e^{- l_0 \varepsilon |x| } \lesssim C_0^2  \varepsilon e^{ (k_0 +2\mu_1) \varepsilon t } \lesssim C_0^2  \varepsilon e^{ 3\mu_1 \varepsilon t }.
\end{equation}
Collecting the previous four estimates, namely \eqref{est11}, \eqref{est_H_1}, \eqref{est_P_1} and \eqref{dif_NL}, we have
\[
\left| \langle \bd{\eta}_1, J(1-\partial_x^2)\bd{Q}_\omega \rangle (t) \right| \lesssim C_0^2 \varepsilon e^{ 3\mu_1 \varepsilon t } + \varepsilon e^{ 2 \mu_1 \varepsilon t } + C_0^2 \varepsilon^2 e^{ 2\mu_1 \varepsilon t } \lesssim  \varepsilon e^{ 2\mu_1 \varepsilon t } .
\]
Now, for a fixed constant $C>0$, \eqref{coer_1} and  \eqref{Energy_new2} lead to
\[
\begin{aligned}
&\| \bd{\eta}_1(t) \|_{H^1\times H^1}^2  \\
& \quad \leq C \left\langle  \bd{\eta}_1 ,  \mathcal L \bd{\eta}_1  \right\rangle (t) + C \left| \langle \bd{\eta}_1 , J(1-\partial_x^2)\bd{Q}_\omega \rangle (t) \right|^2\\
& \quad \leq C \left( \left\langle  \bd{\eta}_1   ,  \mathcal L \bd{\eta}_1  \right\rangle(t)  - \left\langle  \bd{\eta}_1  ,  \mathcal L \bd{\eta}_1  \right\rangle (-T_n)  \right) \\
& \quad \quad  +C \left\langle  \bd{\eta}_1 ,  \mathcal L \bd{\eta}_1  \right\rangle (-T_n)    +C C_0^4 \varepsilon^2 e^{4\mu_1 \varepsilon t } \\
& \quad \leq C\left| H_h[{\bd Q}_\omega + \bd{\eta}_1 ](t) -\omega P[{\bd Q}_\omega +  \bd{\eta}_1 ](t)  -H_h[{\bd Q}_\omega + \bd{\eta}_1 ](-T_n) + \omega P[{\bd Q}_\omega +  \bd{\eta}_1](-T_n)  \right| \\
& \quad \quad + C\left| H_h[{\bd Q}_\omega](-T_n) -\omega P[{\bd Q}_\omega] (-T_n) -H_h[{\bd Q}_\omega](t) + \omega P[{\bd Q}_\omega] (t) \right|  \\
& \quad \quad + C\left|     \int  Q_\omega u_1 h (-T_n)  -  \int  Q_\omega u_1 h (t) \right|  \\
& \quad \quad + C \left| \int u_1^2  \left( \eta_1 + h\right)(-T_n) -\int u_1^2  \left( \eta_1 + h\right)(t) \right| \\
& \quad \quad  +C \left\langle  \bd{\eta}_1 ,  \mathcal L \bd{\eta}_1  \right\rangle (-T_n)    +C  \varepsilon^2 e^{4\mu_1 \varepsilon t }.
\end{aligned}
\]
Using that $\bd{\eta}_1(-T_n)=0$ and that $P[{\bd Q}_\omega] (-T_n) =P[{\bd Q}_\omega] (t)$, by \eqref{P_ind}, this simplifies to
\[
\begin{aligned}
&\| \bd{\eta}_1(t) \|_{H^1\times H^1}^2  \\
& \quad \leq C\left| H_h[{\bd Q}_\omega + \bd{\eta}_1 ](t) -\omega P[{\bd Q}_\omega +  \bd{\eta}_1](t)  -H_h[{\bd Q}_\omega + \bd{\eta}_1 ](-T_n) + \omega P[{\bd Q}_\omega +  \bd{\eta}_1](-T_n)  \right| \\
& \quad \quad + C\left| H_h[{\bd Q}_\omega](-T_n) -H_h[{\bd Q}_\omega](t) \right|  \\
& \quad \quad + C\left|  \int  Q_\omega u_1 h (t) \right| + C \left| \int u_1^2  \left( \eta_1 + h\right)(t) \right|   +C  \varepsilon^2 e^{4\mu_1 \varepsilon t }.
\end{aligned}
\]
Also, from \eqref{est_H_1}-\eqref{est_P_1}, if $-T_n<t<-T_\varepsilon$,
\[
\begin{aligned}
&\| \bd{\eta}_1(t) \|_{H^1\times H^1}^2  \\
& \quad \leq  C\left| H_h[{\bd Q}_\omega](-T_n) -H_h[{\bd Q}_\omega](t) \right|  \\
& \quad \quad + C\left|  \int  Q_\omega u_1 h (t) \right| + C \left| \int u_1^2  \left( \eta_1 + h\right)(t) \right|   +C( \varepsilon^2 +\varepsilon C_0) e^{2\mu_1\varepsilon t }.
\end{aligned}
\]
Finally, from \eqref{dif_H}
\[
\left| H_h[\bd{Q}_\omega ](-T_n) -H_h[\bd{Q}_\omega ](t)  \right| \leq  C \varepsilon e^{2\mu_1 \varepsilon t } .
\]
Hence, using also \eqref{dif_NL},
\[
\begin{aligned}
&\| \bd{\eta}_1(t) \|_{H^1\times H^1}^2  \\
& \quad \lesssim  C\left|  \int  Q_\omega u_1 h (t) \right| + C \left| \int u_1^2  \eta_1 (t) \right|   +C(\varepsilon C_0 +\varepsilon + \varepsilon C_0^2 e^{\mu_1\varepsilon t }) e^{2\mu_1\varepsilon t }.
\end{aligned}
\]
Now, we get
\[
 \left| \int u_1^2  \eta_1 (t) \right| \lesssim CC_0^3  e^{3\mu_1 \varepsilon t},
\]
and from the $H^1(\R) \hookrightarrow L^\infty(\R)$  embedding,
\[
\begin{aligned}
\left|  \int  Q_\omega u_1 h (t) \right| \lesssim &~{}  \varepsilon  e^{-k_0\varepsilon |t| } \int  e^{-\mu_0 |x -\rho_1(t)|}   e^{-l_0\varepsilon |x| } |u_1| \\
\lesssim  &~{} C_0 \varepsilon  e^{(k_0 + \frac9{10} l_0\omega + \mu_1)\varepsilon t } \lesssim C_0 \varepsilon e^{3\mu_1 \varepsilon t}.
\end{aligned}
\]
We conclude, for $-T_n<t<-T_\varepsilon,$
\[
\begin{aligned}
&\| \bd{\eta}_1(t) \|_{H^1\times H^1}^2  \\
& \quad  \le C( \varepsilon C_0 +\varepsilon + (C_0^3+ \varepsilon C_0^2 )e^{\mu_1\varepsilon t }) e^{2\mu_1\varepsilon t } <\frac1{16C^2}  C_0^2  e^{2\mu_1\varepsilon t },
\end{aligned}
\]
if $C_0$ is chosen large enough, and $\varepsilon$ sufficiently small, and where $C$ is the constant appearing in \eqref{der_rho1} multiplied by the size of the solitary wave. Finally, using \eqref{der_rho1} and the fact that $\rho_1(-T_n)=0$, we get
\[
|\rho_1(t) | \le \frac{C_0  e^{\mu_1\varepsilon t }}{4\|\bd{Q}_\omega\|_{H^1\times H^1}} .
\]
Then, gathering the two previous estimates, \eqref{unif_est0_1} is proved. Consequently, \eqref{unif_est0} is proved. Moreover, the fact that the maximal interval of existence $I_n$ of each $(\eta_n,u_n)$ obeys the inclusion $[-T_n,-\frac12T_\varepsilon] \subseteq I_n$ follows directly from \eqref{unif_est0}.

\subsection{Proof of existence} Proof of \eqref{Construction}. Now we are ready to prove the first part of Theorem \ref{MT}, dealing with the existence part. Let us come back to the dependence on $n$ for the solution $\bd{\eta}$. As a consequence of \eqref{unif_est0}, we have $\left\| (\eta_n,u_n)(t) \right\|_{H^1 \times H^1} \leq C_0$. Moreover, notice from Proposition \ref{unif_est} that the maximal interval of existence $I_n$ of each $(\eta_n,u_n)$ obeys the inclusion $[-T_n,-\frac12T_\varepsilon] \subseteq I_n$. We claim that for each $\delta>0$, there exists $R_0>0$ such that if $0<\varepsilon<\varepsilon_0$,
\begin{equation}\label{compa}
\int_{|x|>R_0} (|\eta_n|^2 +|\partial_x \eta_n|^2 +|u_n|^2 +|\partial_x u_n|^2)\left( -\frac12T_\varepsilon ,x \right)dx \leq \delta.
\end{equation}
Since $\left\| (\eta_n,u_n)(-\frac12T_\varepsilon) \right\|_{H^1 \times H^1} \leq C_0$, as $n$ tends to infinity, it follows that $(\eta_n,u_n)(-\frac12T_\varepsilon) $ weakly converges in $H^1(\mathbb R)^2$ to a $(\eta_{*,0},u_{*,0})$, and $(\eta_n,u_n)(-\frac12T_\varepsilon) $ strongly converges to $(\eta_{*,0},u_{*,0})$ in $L^2_\text{loc} (\mathbb R)^2.$  Thanks to \eqref{compa}, we have strong convergence in $H^1(\mathbb R)^2$.

Let $(\eta_{*},u_{*}) =(\eta_{*},u_{*})(t)$ be the solution to $abcd$ system \eqref{boussinesq} with initial data $(\eta_{*,0},u_{*,0})$ at time $t= -\frac12T_\varepsilon$. From local well-posedness, we have $(\eta_{*},u_{*})$ well-defined in the interval $(-T_{*,-}, T_{*,+})$ containing $-\frac12T_\varepsilon$. Now we use the continuous dependence of the initial value problem. First of all, we have $(-T_{*,-}, T_{*,+}) \subseteq I_n$ for each $n$, as a classical argument shows. Also, for each $-T_{*,-} < t\leq -\frac12T_\varepsilon$, $(\eta_{n},u_{n})(t)$ strongly converges to $(\eta_{*},u_{*})(t)$ in $H^1(\mathbb R)^2$. Passing to the limit in  \eqref{unif_est0} we obtain
\[
\left\| (\eta_*,u_*)(t) - {\bd Q}_{\omega}(\cdot -\omega t)\right\|_{H^1 \times H^1} \leq C_0  e^{\mu_0 \varepsilon t}.
\]
Since the solution at time $t=-T_{*,-}$ is bounded, necessarily $T_{*,-} =\infty$. The proof of Theorem \ref{MT}, property \eqref{Construction} and estimate \eqref{PreInteraction} is complete.

\medskip

We finally prove \eqref{compa}. This follows from a virial estimate proved in \cite{DGMMP}. Let $\psi = \psi(t,x)$ be a smooth, nonnegative and bounded function, to be chosen in the sequel. Again, we drop the dependence on $n$, since it is not important. We consider the localized energy functional defined by
\begin{equation}\label{eq:local energy}
E_\text{loc}(t) = \frac12 \int \psi(t,x) \left( - a(\px u)^2 - c(\px \eta)^2 + u^2 + \eta^2 + u^2(\eta+h)\right)(t,x)dx.
\end{equation}

The following results, proved in \cite{DGMMP}, provide the time derivative of this local energy.
\begin{lemma}[Variation of local energy $E_\text{loc}$]\label{lem:energy1}
Let $(u,\eta)$ be a solution  of \eqref{boussinesq}, and set  $f=(1-\partial_x^2)^{-1} \eta$ and $g=(1-\partial_x^2)^{-1} u$. The time derivative of the local energy in  \eqref{eq:local energy} is given by:
\begin{equation}\label{eq:energy1000}
\begin{aligned}
\frac{d}{dt} E_\emph{loc}(t) =&~{} \int \psi' fg + (1-2(a+c)) \int \psi' \partial_x f \partial_x g  \\
&+ (3ac-2(a+c))\int \psi'  \partial_x^2 f \partial_x^2 g + 3ac\int \psi'  \partial_x^3 f\partial_x^3 g\\
&+ \emph{SNL}_0(t) +\emph{SNL}_1(t),
\end{aligned}
\end{equation}
where the small nonlinear parts $\emph{SNL}_0$ and $\emph{SNL}_1$ are given by
\[
\emph{SNL}_0(t) := \frac12 \int \partial_t \psi \left( - a(\px u)^2 - c(\px \eta)^2 + u^2 + \eta^2 + u^2(\eta+h)\right),
\]
\begin{equation}\label{eq:E1E2E3E4_new2}
\begin{aligned}
& \emph{SNL}_1(t)\\
&~{}  := a(c-1)\int\psi''\partial_x^2 f \partial_x g +c(a-1)\int\psi''\partial_x f \partial_x^2 g\\
&\qquad -a\int \psi''\partial_x f  g -c\int \psi'' f \partial_x g  -a\int \psi''\partial_x^2 f \partial_x g  -c\int \psi''\partial_x f \partial_x^2 g \\
&\qquad +\frac{a}{2} \int \psi'\partial_x^2 f  (1-\px^2)^{-1}(u^2) +\frac12\int \psi'f(1-\px^2)^{-1}(u^2)\\
&\qquad +c\int \psi'\partial_x^2 g(1-\px^2)^{-1}(u(\eta+h)) + \int \psi'g(1-\px^2)^{-1}(u(\eta+h)) \\
&\qquad +\frac12\int \psi'(1-\px^2)^{-1}(u(\eta+h))(1-\px^2)^{-1}(u^2)\\
&\qquad +\frac{a}{2} \int \psi'\partial_x^3 f (1-\px^2)^{-1}\partial_x (u^2) +\frac12\int \psi'\partial_x f (1-\px^2)^{-1}\partial_x (u^2)\\
&\qquad +c\int \psi'\partial_x^3 g(1-\px^2)^{-1}\partial_x (u(\eta+h)) + \int \psi'\partial_x g (1-\px^2)^{-1}\partial_x (u(\eta+h)) \\
&\qquad +\frac12\int \psi'(1-\px^2)^{-1}\partial_x (u(\eta+h)) (1-\px^2)^{-1}\partial_x (u^2)\\
&\qquad  +\frac12 a\int \partial_x(\psi' \partial_xu)(1-\px^2)^{-1}(u^2)   + c\int \partial_x(\psi' \partial_x\eta)(1-\px^2)^{-1}(u(\eta+h))  \\
&\qquad  + ac_1 \int \psi' \partial_xu  (1-\partial^2_x)^{-1}  \partial_t^2 \partial_x h  + c\int \psi'   \partial_x \eta (1-\partial^2_x)^{-1}\left(-1+ a_1 \partial_x^2\right) \partial_t h .
\end{aligned}
\end{equation}
\end{lemma}
 We shall use Lemma \ref{lem:energy1} as follows. First of all, let us fix $\delta>0$ and $T_0>0$ such that $C\varepsilon e^{-\mu_1\varepsilon T_0}<\delta.$ Choose $\psi $ and $L$ large such that
\[
\psi=\psi_0 \left( \frac{|x| -R_0}{L}\right), \quad \psi_0 \in C^\infty(\mathbb R), \quad \psi_0 (s\leq 1)=0, \quad  \psi_0 (s\geq 2)=1, \quad \psi'_0\geq 0.
\]
Then from \eqref{eq:energy1000} and \eqref{eq:L2_est}-\eqref{eq:H1_est}, and since $\psi$ does not depend on time,
\begin{equation}\label{eq:energy1}
\begin{aligned}
\left| \frac{d}{dt} E_\emph{loc}(t)\right|  \lesssim  \frac1{L}\int \psi_0' (|\eta|^2 +|\partial_x \eta|^2 +|u|^2 +|\partial_x u|^2) +|\emph{SNL}_1(t)| .
\end{aligned}
\end{equation}
Following \eqref{eq:L2_est}-\eqref{eq:H1_est}, and estimates \eqref{eq:nonlinear1-1}, \eqref{eq:nonlinear1-2}, \eqref{eq:nonlinear1-3} and \eqref{eq:nonlinear1-4}, we bound $|\text{SNL}_1(t)|$ in \eqref{eq:E1E2E3E4_new2} as follows:
\begin{equation}\label{eq:E1E2E3E4_new2_0}
\begin{aligned}
& |\text{SNL}_1(t)| \\
& \qquad  \lesssim \frac1{L}\int \psi_0' (|\eta|^2 +|\partial_x \eta|^2 +|u|^2 +|\partial_x u|^2 +|h|^2 +|\partial_t h|^2 +|\partial_x h|^2 +|\partial_t^2 \partial_xh|^2).
\end{aligned}
\end{equation}
Indeed, we have from \eqref{eq:L2_est},
\begin{equation}\label{eq:E1E2E3E4_new2_1}
\begin{aligned}
&~{}  \Big| a(c-1)\int\psi''\partial_x^2 f \partial_x g +c(a-1)\int\psi''\partial_x f \partial_x^2 g \Big|  \\
& \qquad \lesssim   \frac1{L^2}\int \psi_0' (|\eta|^2  +|u|^2 ).
\end{aligned}
\end{equation}
Second, using again \eqref{eq:H1_est},
\begin{equation}\label{eq:E1E2E3E4_new2_2}
\begin{aligned}
&~{}  \Big|  -a\int \psi''\partial_x f  g -c\int \psi'' f \partial_x g  -a\int \psi''\partial_x^2 f \partial_x g  -c\int \psi''\partial_x f \partial_x^2 g \Big|  \\
& \qquad \lesssim  \frac1{L^2}\int \psi_0' (|\eta|^2  +|u|^2 ).
\end{aligned}
\end{equation}
Third, using \eqref{eq:nonlinear1-2},
\begin{equation}\label{eq:E1E2E3E4_new2_3}
\begin{aligned}
&~{}  \Big|\frac{a}{2} \int \psi'\partial_x^2 f  (1-\px^2)^{-1}(u^2) +\frac12\int \psi'f(1-\px^2)^{-1}(u^2) \Big|   \lesssim \frac1{L} \|\eta\|_{H^1}\int \psi_0' |u|^2.
\end{aligned}
\end{equation}
Fourth, using again \eqref{eq:nonlinear1-2},
\begin{equation}\label{eq:E1E2E3E4_new2_4}
\begin{aligned}
&~{}  \Big|c\int \psi'\partial_x^2 g(1-\px^2)^{-1}(u(\eta+h)) + \int \psi'g(1-\px^2)^{-1}(u(\eta+h))  \Big|  \\
& \qquad \lesssim \frac1{L} \| u \|_{H^1}\int \psi_0' (|u|^2 +|\eta|^2 + |h|^2). 
\end{aligned}
\end{equation}
Fifth, using again \eqref{eq:nonlinear1-2},
\begin{equation}\label{eq:E1E2E3E4_new2_5}
\begin{aligned}
&~{}  \Big| \frac12\int \psi'(1-\px^2)^{-1}(u(\eta+h))(1-\px^2)^{-1}(u^2) \Big|  \\
& \qquad \lesssim  \frac1{L} \| (1-\px^2)^{-1}(u^2) \|_{H^1}\int \psi_0' (|u|^2 +|\eta|^2 +|h|^2)\\
& \qquad \lesssim  \frac1{L} \| u \|_{L^2}^2 \int \psi_0' (|u|^2 +|\eta|^2 +|h|^2).
\end{aligned}
\end{equation}
Now, using \eqref{eq:nonlinear1-4},
\begin{equation}\label{eq:E1E2E3E4_new2_6}
\begin{aligned}
&~{}  \Big| \frac{a}{2} \int \psi'\partial_x^3 f (1-\px^2)^{-1}\partial_x (u^2) +\frac12\int \psi'\partial_x f (1-\px^2)^{-1}\partial_x (u^2) \Big| \\
& \qquad \lesssim \frac1{L} \| \eta \|_{H^1}\int \psi_0' (|u|^2  + |\partial_x u |^2).
\end{aligned}
\end{equation}
Similarly,
\begin{equation}\label{eq:E1E2E3E4_new2_7}
\begin{aligned}
&~{}  \Big| c\int \psi'\partial_x^3 g(1-\px^2)^{-1}\partial_x (u(\eta+h)) + \int \psi'\partial_x g (1-\px^2)^{-1}\partial_x (u(\eta+h))\Big| \\
& \qquad \lesssim \frac1{L} \| u \|_{H^1}\int \psi_0' (|u|^2  + |\partial_x u |^2 + |\eta |^2  + |\partial_x \eta |^2 +|h|^2  + |\partial_x h |^2).
\end{aligned}
\end{equation}
Similar to \eqref{eq:E1E2E3E4_new2_5}, and using estimates from \cite{KMM},
\begin{equation}\label{eq:E1E2E3E4_new2_8}
\begin{aligned}
&~{}  \Big|\frac12\int \psi'(1-\px^2)^{-1}\partial_x (u(\eta+h)) (1-\px^2)^{-1}\partial_x (u^2) \Big| \\
& \qquad \lesssim  \frac1{L} \| (1-\px^2)^{-1} \partial_x (u^2) \|_{H^1}\int \psi_0' (|u|^2 +|\partial_xu|^2 +|\eta|^2 +|\partial_x\eta |^2+|h|^2+|\partial_xh|^2)\\
& \qquad \lesssim  \frac1{L} \| u \|_{H^1}^2 \int \psi_0' (|u|^2 +|\partial_xu|^2 +|\eta|^2 +|\partial_x\eta |^2+|h|^2+|\partial_xh|^2).
\end{aligned}
\end{equation}
Similar as in \eqref{eq:E1E2E3E4_new2_7},
\begin{equation}\label{eq:E1E2E3E4_new2_9}
\begin{aligned}
&~{}  \Big| c\int \psi'\partial_x^3 g(1-\px^2)^{-1}\partial_x (u(\eta+h)) + \int \psi'\partial_x g (1-\px^2)^{-1}\partial_x (u(\eta+h)) \Big| \\
& \qquad \lesssim \frac1{L} \| u \|_{H^1}\int \psi_0' (|u|^2  + |\partial_x u |^2 + |\eta |^2  + |\partial_x \eta |^2 +|h|^2  + |\partial_x h |^2).
\end{aligned}
\end{equation}
Now we use \eqref{eq:nonlinear1-3},
\begin{equation}\label{eq:E1E2E3E4_new2_10}
\begin{aligned}
&~{}  \Big|\frac12 a\int \partial_x(\psi' \partial_xu)(1-\px^2)^{-1}(u^2)   + c\int \partial_x(\psi' \partial_x\eta)(1-\px^2)^{-1}(u(\eta+h))   \Big| \\
& \qquad \lesssim \frac1L  (\|u\|_{H^1} +\|\eta\|_{H^1}) \int \psi_0' (|u|^2  + |\partial_x u |^2 +|\eta|^2  + |\partial_x \eta |^2 + |h|^2  + |\partial_x h |^2).
\end{aligned}
\end{equation}
Finally, using H\"older and \cite{KMM}, 
\begin{equation}\label{eq:E1E2E3E4_new2_11}
\begin{aligned}
&~{}  \Big| ac_1 \int \psi' \partial_xu  (1-\partial^2_x)^{-1}  \partial_t^2 \partial_x h  + c\int \psi'   \partial_x \eta (1-\partial^2_x)^{-1}\left(-1+ a_1 \partial_x^2\right) \partial_t h\Big| \\
& \qquad \lesssim \frac1L  \int \psi_0' ( |\partial_x u |^2  + |\partial_x \eta |^2 +  |(1-\partial^2_x)^{-1}  \partial_t^2 \partial_xh|^2  + |\partial_t h |^2 + | (1-\partial^2_x)^{-1}\partial_t h|^2)\\
& \qquad \lesssim \frac1L  \int \psi_0' ( |\partial_x u |^2  + |\partial_x \eta |^2 +  |\partial_t^2 \partial_xh|^2  + |\partial_t h |^2 ).
\end{aligned}
\end{equation}
By gathering \eqref{eq:E1E2E3E4_new2_1}--\eqref{eq:E1E2E3E4_new2_11}, we obtain \eqref{eq:E1E2E3E4_new2_0}.
Thanks to \eqref{eq:E1E2E3E4_new2_0}, we deduce from   \eqref{eq:energy1} that
\begin{equation}\label{eq:energy1_2}
\begin{aligned}
&\left| \frac{d}{dt} E_\text{loc}(t)\right|
\\
& \quad  \lesssim  \frac1{L}\int \psi_0' (|\eta|^2 +|\partial_x \eta|^2 +|u|^2 +|\partial_x u|^2 +|h|^2 +|\partial_t h|^2 +|\partial_x h|^2 +|\partial_t^2 \partial_xh|^2).
\end{aligned}
\end{equation}
Integrating again \eqref{eq:energy1_2} in $\left[ -T_0, -\frac12 T_\varepsilon \right]$, we obtain
\[
\left| E_\text{loc}\left(-\frac12T_\varepsilon \right) \right| < \delta.
\]
Finally, making the decomposition \eqref{z1}, we establish \eqref{compa}. For full details, see \cite{Mu1} and references therein.


\section{The interaction regime}\label{Sec:3}

Recall that the interaction regime is defined as $[-T_\varepsilon, T_\varepsilon]$, with $T_\varepsilon = \varepsilon^{-1-\delta}.$

\subsection{Preliminaries}
Consider the modulated solitary wave $\bd{Q}_\omega$ introduced in \eqref{mod_Q}. Denote $z=x-\rho(t)$. If we introduce this object into \eqref{boussinesq}, we shall obtain \eqref{S_hQ}. Using the hypotheses on $h$ in \eqref{hypoH}, the system \eqref{S_hQ} reduces to:
\[
\begin{aligned}
{\bf S}_h(R_\omega,Q_\omega)= &~{} \omega '  (1-\partial^2_x)\begin{pmatrix}
\Lambda R_\omega \\
\Lambda Q_\omega
\end{pmatrix}
-(\rho'-\omega)(1-\partial^2_x)\partial_x
\begin{pmatrix}
R_\omega  \\
 Q_\omega
 \end{pmatrix}
\\
&~{} +
\varepsilon  \begin{pmatrix}
 \partial_x(h_0 Q_\omega)   \\
 0
 \end{pmatrix}
 +\varepsilon^2   \begin{pmatrix}
\partial_s h_0 \\
 0
 \end{pmatrix}
 - \varepsilon^4  \begin{pmatrix}
  a_1\partial^2_y\partial_s h_0 \\
 c_1 \partial^2_s\partial_y h_0
 \end{pmatrix} .
 \end{aligned}
\]
A further development reveals the modified structure of the error terms appearing by the interaction of the solitary wave and the variable bottom:
\begin{equation}\label{S_hQ_new}
\begin{aligned}
{\bf S}_h(R_\omega,Q_\omega)= &~{} \omega '  (1-\partial^2_z)\begin{pmatrix}
\Lambda R_\omega \\
\Lambda Q_\omega
\end{pmatrix}
-(\rho'-\omega)(1-\partial^2_z)\partial_z
\begin{pmatrix}
R_\omega  \\
 Q_\omega
 \end{pmatrix}
\\
&~{} + \varepsilon  h_0(\varepsilon t, \varepsilon \rho(t)) \partial_z \begin{pmatrix} Q_\omega   \\ 0 \end{pmatrix}
\\
&~{}   + \varepsilon^2  \left(  \partial_y h_0(\varepsilon t, \varepsilon \rho(t))  \partial_z \begin{pmatrix}  zQ_\omega \\ 0 \end{pmatrix}
+  \partial_s h_0 (\varepsilon t, \varepsilon x)  \begin{pmatrix}
1\\
 0
 \end{pmatrix}  \right)\\
&~{}  + \varepsilon^3 \left( \frac12  \partial_y^2 h_0(\varepsilon t, \varepsilon \rho(t))  \partial_z \begin{pmatrix}  z^2 Q_\omega  \\ 0 \end{pmatrix}  \right)
\\
&~{} +  \varepsilon^4 \left( -   \begin{pmatrix}
  a_1\partial^2_y\partial_s h_0 \\
 c_1 \partial^2_s\partial_y h_0
 \end{pmatrix} (\varepsilon t,\varepsilon x)  +\frac16   \partial_y^3 h_0(\varepsilon t, \varepsilon \rho(t)) \partial_z \begin{pmatrix}  z^3 Q_\omega \\ 0 \end{pmatrix} \right) \\
 &~{} +  \varepsilon^5 \partial_x \begin{pmatrix}  \tilde h_0(\varepsilon t, \varepsilon \xi_1(t,x)) z^4 Q_\omega   \\ 0 \end{pmatrix} .
 \end{aligned}
\end{equation}
Here $\xi_1$ represents a mean value function depending on $t$ and $x$ with values in the interval $x-\rho(t)$ and $x$. The third, fourth and fifth terms in \eqref{S_hQ_new} will represent the influence of the bottom on the dynamics, and they are divided in three different terms: a first one directly related with the interaction of the solitary wave with the varying bottom, a second one related to space variations of the bottom (no solitary wave influence) and a final one dealing with time corrections in the bottom. Other terms are essentially small in the slowly varying regime.

\subsection{Linear correction}
Following previous works \cite{Mu1,Mu2,Mu3,Mu4} dealing with the scalar gKdV and NLS dynamics, now we consider a vector correction term of the form
\begin{equation}\label{def_W0}
\bd{W} (t,x) : = \begin{pmatrix} W_1 \\ W_2 \end{pmatrix} (t,x)= \varepsilon \begin{pmatrix} A_1 \\ B_1 \end{pmatrix} (t, \omega(t),z)+ \varepsilon^2 \begin{pmatrix} A_2 \\  B_2\end{pmatrix}(t,\omega(t),z).
\end{equation}
Notice that we have separated the dependences in $W_1$ and $W_2$: one dealing with the scaling parameter $\omega(t)$, and another representing the rest of the time dependences. Then \eqref{S0} becomes
\begin{equation}\label{S0_new}
{\bf S}_h(\bd{Q}_\omega + \bd{W} ) = {\bf S}_h(\bd{Q}_\omega) + {\bf S}_h' (\bd{Q}_\omega)\bd{W} +\bd{R}_\text{aux},
\end{equation}
with ${\bf S}_h(\bd{Q}_\omega) $ given by \eqref{S_hQ_new},
\[
\begin{aligned}
 {\bf S}_h' (\bd{Q}_\omega)\bd{W}= &~{}
 \begin{pmatrix}
(1- \partial_z^2)\partial_t W_1  + \partial_z\!\left( a\, \partial_z^2 W_2 + W_2 +   Q_\omega W_1  + R_\omega W_2 \right)   \\
(1- \partial_z^2)\partial_t W_2  + \partial_z\! \left( c\, \partial_z^2 W_1  + W_1  + Q_\omega W_2 \right)
\end{pmatrix},
\end{aligned}
\]
and
\[
\begin{aligned}
\bd{R}_\text{aux} = &~{}   \begin{pmatrix}
 \partial_x\!\left( (W_1 + h) W_2 \right)   \\
\frac12  \partial_x( W_2^2 )
\end{pmatrix}.
\end{aligned}
\]
We will modify \eqref{S0_new} to better represent the influence of \eqref{def_W0} in the dynamics. First of all, we have
\be\label{dtW1}
\begin{aligned}
& (1- \partial_x^2)\partial_t W_1\\
& \quad = \varepsilon (1- \partial_z^2) \partial_t A_1 + \varepsilon \omega ' (1- \partial_z^2) \Lambda A_1\\
& \quad \quad  - \varepsilon (\rho'-\omega)(1- \partial_z^2) \partial_z A_1 - \varepsilon \omega (1- \partial_z^2) \partial_z A_1 \\
& \quad \quad + \varepsilon^2 (1- \partial_z^2) \partial_t A_2 + \varepsilon^2 \omega ' (1- \partial_z^2) \Lambda A_2 \\
& \quad \quad- \varepsilon^2 (\rho'-\omega)(1- \partial_z^2) \partial_z A_2 - \varepsilon^2 \omega (1- \partial_z^2) \partial_z A_2.
\end{aligned}
\ee
A completely similar output is obtained in the case of $(1- \partial_x^2)\partial_t W_2$:
\be\label{dtW2}
\begin{aligned}
& (1- \partial_x^2)\partial_t W_2\\
& \quad = \varepsilon (1- \partial_z^2) \partial_t B_1 + \varepsilon \omega ' (1- \partial_z^2) \Lambda B_1 \\
& \quad \quad- \varepsilon (\rho'-\omega)(1- \partial_z^2) \partial_z B_1 - \varepsilon \omega (1- \partial_z^2) \partial_z B_1 \\
& \quad \quad + \varepsilon^2 (1- \partial_z^2) \partial_t B_2 + \varepsilon^2 \omega ' (1- \partial_z^2) \Lambda B_2 \\
& \quad \quad- \varepsilon^2 (\rho'-\omega)(1- \partial_z^2) \partial_z B_2 - \varepsilon^2 \omega (1- \partial_z^2) \partial_z B_2.
\end{aligned}
\ee
On the other hand,
\[
\begin{aligned}
& \partial_x\!\left( (W_1 + h) W_2 \right)  \\
& \quad =\partial_x\!\left(\varepsilon hB_1+ \varepsilon^2 A_1B_1 +\varepsilon^2 hB_2 + \varepsilon^3 (A_1B_2+A_2B_1) +\varepsilon^4 A_2B_2 \right)\\
& \quad =\partial_z\!\left(\varepsilon^2 h_0(\varepsilon t, \varepsilon \rho(t)) B_1 + \varepsilon^2 A_1B_1  \right)\\
& \quad \quad +\partial_x\!\left( \varepsilon^3 \partial_y h_0(\varepsilon t, \varepsilon \xi_3(t,x)) z B_1 +\varepsilon^3 h_0 (\varepsilon t, \varepsilon x) B_2 + \varepsilon^3 (A_1B_2+A_2B_1) +\varepsilon^4 A_2B_2 \right),
\end{aligned}
\]
and
\[
\begin{aligned}
& \frac12  \partial_x( W_2^2 )  = \frac12  \partial_z( \varepsilon^2 B_1^2 +2\varepsilon^3 B_1 B_2 +\varepsilon^4 B_2^2).
\end{aligned}
\]
Therefore, gathering the previous results, we have the modified representation of \eqref{S0_new}:
\begin{equation}\label{S0_new_new}
{\bf S}_h(\bd{Q}_\omega + \bd{W} ) = {\bf S}_h^{\#}(\bd{Q}_\omega) +\partial_x J \mathcal L \bd{W} +\bd{R},
\end{equation}
where
\begin{equation}\label{S_hQ_new_new}
\begin{aligned}
 {\bf S}_h^{\#}(\bd{Q}_\omega) = &~{} \omega '  (1-\partial^2_z)\begin{pmatrix}
\Lambda (R_\omega +\varepsilon A_1 + \varepsilon^2 A_2)\\
\Lambda (Q_\omega+\varepsilon B_1 + \varepsilon^2 B_2)
\end{pmatrix} \\
&~{}
-(\rho'-\omega)(1-\partial^2_z)\partial_z
\begin{pmatrix}
R_\omega  +\varepsilon A_1 + \varepsilon^2 A_2\\
 Q_\omega+\varepsilon B_1 + \varepsilon^2 B_2
 \end{pmatrix}
\\
&~{} + \varepsilon  h_0(\varepsilon t, \varepsilon \rho(t)) \partial_z \begin{pmatrix} Q_\omega   \\ 0 \end{pmatrix}
\\
&~{}   + \varepsilon^2  \left(  \partial_y h_0(\varepsilon t, \varepsilon \rho(t))  \partial_z \begin{pmatrix}  zQ_\omega \\ 0 \end{pmatrix}
+  \partial_s h_0 (\varepsilon t, \varepsilon x )  \begin{pmatrix}
1\\
 0
 \end{pmatrix}  \right)
 \\
 & ~{}  + \varepsilon (1- \partial_z^2) \partial_t \begin{pmatrix} A_1 \\ B_1 \end{pmatrix}  \\
 &~{}+ \varepsilon^2 \partial_z \begin{pmatrix}  h_0(\varepsilon t, \varepsilon \rho(t)) B_1 + A_1B_1 \\   \frac12 B_1^2 \end{pmatrix};
 \end{aligned}
\end{equation}
from \eqref{def_L},
\begin{equation}\label{def_L_new}
\begin{aligned}
\partial_x J \mathcal L \bd{W} = &~{}
\begin{pmatrix}
-\omega (1- \partial_z^2)\partial_y W_1  + \partial_z\!\left( a\, \partial_z^2 W_2 + W_2 +   Q_\omega W_1  + R_\omega W_2 \right)   \\
-\omega  (1- \partial_z^2)\partial_y W_2  + \partial_z\! \left( c\, \partial_z^2 W_1  + W_1  + Q_\omega W_2 \right)
\end{pmatrix}
\\
=&~{}
\partial_z  \begin{pmatrix} 0 & 1  \\  1 & 0 \end{pmatrix}
\begin{pmatrix}
 -\omega  (1- \partial_z^2)W_2  +  c\, \partial_z^2 W_1  + W_1  + Q_\omega W_2   \\
-\omega (1- \partial_z^2) W_1  +  a\, \partial_z^2 W_2 + W_2 +   Q_\omega W_1  + R_\omega W_2
\end{pmatrix}
\\
=&~{}
\partial_z  \begin{pmatrix} 0 & 1  \\  1 & 0 \end{pmatrix}
\begin{pmatrix}
 c \partial_z^2 +1  &   -\omega (1-\partial^2_z) +Q_\omega   \\
  -\omega (1-\partial^2_z)  +Q_\omega & a \partial_z^2 +1 +R_\omega
\end{pmatrix}
\begin{pmatrix} W_1 \\ W_2 \end{pmatrix},
\end{aligned}
\end{equation}
and
\begin{equation}\label{R_new}
\begin{aligned}
\bd{R} =&~{} \frac12   \varepsilon^3   \partial_y^2 h_0(\varepsilon t, \varepsilon \rho(t))  \partial_z \begin{pmatrix}  z^2 Q_\omega  \\ 0 \end{pmatrix}
\\
 &~{} +\varepsilon^3 \partial_x\!\left(  \partial_y h_0(\varepsilon t, \varepsilon \xi_3(t,x))  \begin{pmatrix} z B_1 \\ 0 \end{pmatrix}  +  h_0 (\varepsilon t, \varepsilon x) \begin{pmatrix}  B_2 \\ 0 \end{pmatrix}  \right) \\
 &~{} +  \varepsilon^3 \partial_z \begin{pmatrix} A_1B_2+A_2B_1  \\ B_1 B_2 \end{pmatrix}
 \\
&~{} +  \varepsilon^4 \left( -   \begin{pmatrix}
  a_1\partial^2_y\partial_s h_0 \\
 c_1 \partial^2_s\partial_y h_0
 \end{pmatrix} (\varepsilon t,\varepsilon x)  +\frac16   \partial_y^3 h_0(\varepsilon t, \varepsilon \rho(t)) \partial_z \begin{pmatrix}  z^3 Q_\omega \\ 0 \end{pmatrix}+ \partial_z \begin{pmatrix} A_2B_2 \\ \frac12B_2^2 \end{pmatrix} \right)  \\
 &~{} +  \varepsilon^5 \partial_x \begin{pmatrix}  \tilde h_0(\varepsilon t, \varepsilon \xi_1(t,x)) z^4 Q_\omega   \\ 0 \end{pmatrix}.
 \end{aligned}
\end{equation}
Our next objective will be to reduce the error from $O(\varepsilon)$ to order $O(\varepsilon^3).$

\subsection{Resolution of linear systems} From \eqref{S_hQ_new_new} and \eqref{def_L_new} first we shall solve
\begin{equation}\label{Syst_1}
\mathcal L \begin{pmatrix} A_1 \\ B_1 \end{pmatrix}=
\begin{aligned}
 \begin{pmatrix}
 c \partial_z^2 +1  &   -\omega (1-\partial^2_z) +Q_\omega   \\
  -\omega (1-\partial^2_z)  +Q_\omega & a \partial_z^2 +1 +R_\omega
\end{pmatrix}
\begin{pmatrix} A_1 \\ B_1 \end{pmatrix} = - h_0(\varepsilon t, \varepsilon \rho(t))  \begin{pmatrix} 0 \\ Q_\omega    \end{pmatrix}.
\end{aligned}
\end{equation}

Since $(0,Q_\omega)^T$ is orthogonal to the $\bd Q_\omega'$,
by Lemma~\ref{lem:L-1} there is a unique \textit{even} function $(A_{0,\omega} , B_{0,\omega})$ orthogonal to $\textup {span}\{\bd Q_\omega'\}$,  solution to
 \begin{equation}\label{H2}
  \mathcal L ( A_{0,\omega} ,B_{0,\omega})^T  =   ( 0, Q_\omega)^T,
 \end{equation}
which, by Lemma~\ref{lem:exp:decay},  has an exponential decay:
 \[
 (|A_{0,\omega}| + | B_{0,\omega}|)(z) \lesssim e^{-\tilde \mu_0 |z|},
 \]
for some $\tilde \mu_0>0$. For simplicity, we denote $\tilde \mu_0$ by $\mu_0$.
In this manner,
\begin{equation}\label{Sol_1}
 \begin{pmatrix} A_1 \\ B_1 \end{pmatrix} (t,\omega(t),z)=   - h_0(\varepsilon t, \varepsilon \rho(t))  \begin{pmatrix} A_{0,\omega} \\ B_{0,\omega} \end{pmatrix}(z).
 \end{equation}
 is a solution to \eqref{Syst_1}, satisfying
%

\begin{equation}\label{est_A1B1 L2}
\begin{aligned}
\left\| \begin{pmatrix} A_1 \\ B_1 \end{pmatrix} (t) \right\|_{L^2_x} \lesssim e^{-k_0 \varepsilon |t| +l_0 \varepsilon |\rho(t)|},
\end{aligned}
\end{equation}
together with all its spatial partial derivatives.
We follow the convention that the partial time derivative does not consider $\omega(t)$,  since this is considered separately.
In order to solve a second linear system, we first need from \eqref{S_hQ_new_new} an estimate for $(1- \partial_z^2) \partial_t (A_1 , B_1)^T$. We deduce  from \eqref{Sol_1} that
\be\label{dtA1B1}
\begin{aligned}
(1- \partial_z^2) \partial_t \begin{pmatrix} A_1 \\ B_1 \end{pmatrix} = &~{} - \varepsilon \partial_s h_0(\varepsilon t, \varepsilon \rho(t)) (1- \partial_z^2) \begin{pmatrix} A_{0,\omega} \\ B_{0,\omega} \end{pmatrix}\\
&~{}    - \varepsilon \omega (t) \partial_y h_0(\varepsilon t, \varepsilon \rho(t)) (1- \partial_z^2) \begin{pmatrix} A_{0,\omega} \\ B_{0,\omega} \end{pmatrix}  \\
&~{} - \varepsilon (\rho'(t)-\omega (t)) \partial_y h_0(\varepsilon t, \varepsilon \rho(t)) (1- \partial_z^2) \begin{pmatrix} A_{0,\omega} \\ B_{0,\omega} \end{pmatrix} .
\end{aligned}
\ee
Having this last estimate into account, and the solution \eqref{Sol_1}, we rewrite \eqref{S0_new_new} in an updated from as follows:
\begin{equation}\label{S0_new_2000}
{\bf S}_h(\bd{Q}_\omega + \bd{W} ) = {\bf S}_h^{\#}(\bd{Q}_\omega) +\partial_x J \mathcal L \bd{W} +\bd{R},
\end{equation}
where \eqref{S_hQ_new_new} becomes now
\begin{equation}\label{S_hQ_new_2}
\begin{aligned}
& {\bf S}_h^{\#}(\bd{Q}_\omega)
\\
 &~{} =  \omega '  (1-\partial^2_z)\begin{pmatrix}
\Lambda (R_\omega +\varepsilon A_1 + \varepsilon^2 A_2)\\
\Lambda (Q_\omega+\varepsilon B_1 + \varepsilon^2 B_2)
\end{pmatrix} \\
&~{}
\quad -(\rho'-\omega)(1-\partial^2_z) \left( \partial_z
\begin{pmatrix}
R_\omega  +\varepsilon A_1 + \varepsilon^2 A_2\\
 Q_\omega+\varepsilon B_1 + \varepsilon^2 B_2
 \end{pmatrix}  - \varepsilon^2  \partial_y h_0(\varepsilon t, \varepsilon \rho(t)) \begin{pmatrix} A_{0,\omega} \\ B_{0,\omega} \end{pmatrix}  \right)
\\
&~{} \quad   + \varepsilon^2  \left(  \partial_y h_0(\varepsilon t, \varepsilon \rho(t))  \partial_z \begin{pmatrix}  zQ_\omega \\ 0 \end{pmatrix}
+  \partial_s h_0 (\varepsilon t, \varepsilon x)  \begin{pmatrix}
1\\
 0
 \end{pmatrix}  \right)  \\
 &~{} \quad + \varepsilon^2  h_0^2 (\varepsilon t, \varepsilon \rho(t)) \partial_z \begin{pmatrix}  - B_{0,\omega} + A_{0,\omega}B_{0,\omega} \\   \frac12 B_{0,\omega}^2 \end{pmatrix} \\
 &~{} \quad  -  \varepsilon^2 h_0(\varepsilon t, \varepsilon \rho(t))  \left( \partial_s h_0(\varepsilon t, \varepsilon \rho(t)) + \omega (t) \partial_y h_0(\varepsilon t, \varepsilon \rho(t)) \right) (1- \partial_z^2) \begin{pmatrix} A_{0,\omega} \\ B_{0,\omega} \end{pmatrix}; 
 \end{aligned}
\end{equation}
and from \eqref{def_L_new} is given now as
\begin{equation}\label{def_L_new_2}
\begin{aligned}
\partial_x J \mathcal L \bd{W} = &~{}
\partial_z  \begin{pmatrix} 0 & 1  \\  1 & 0 \end{pmatrix}  \begin{pmatrix}
 c \partial_z^2 +1  &   -\omega (1-\partial^2_z) +Q_\omega   \\
  -\omega (1-\partial^2_z)  +Q_\omega & a \partial_z^2 +1 +R_\omega
\end{pmatrix}
\begin{pmatrix} \varepsilon^2 A_2 \\ \varepsilon^2 B_2 \end{pmatrix}.
\end{aligned}
\end{equation}
Finally, \eqref{R_new} remains unchanged. From \eqref{S_hQ_new_2} and \eqref{def_L_new_2} we must solve now  
\begin{equation}\label{Syst_2}
\begin{aligned}
\mathcal L
\begin{pmatrix} A_2 \\ B_2 \end{pmatrix}
& = - \partial_y h_0(\varepsilon t, \varepsilon \rho(t)) \begin{pmatrix}  0 \\
zQ_\omega \end{pmatrix}
-  \partial_z^{-1} \partial_s h_0 (\varepsilon t, \varepsilon x )
\begin{pmatrix}
0 \\
 1
 \end{pmatrix}\\
 &  -h_0^2 (\varepsilon t, \varepsilon \rho(t))  \begin{pmatrix}    \frac12 B_{0,\omega}^2 \\ - B_{0,\omega} + A_{0,\omega}B_{0,\omega}  \end{pmatrix}
  \\
&   +  \varepsilon^2 h_0(\varepsilon t, \varepsilon \rho(t))  \Big( \partial_s h_0(\varepsilon t, \varepsilon \rho(t))
+ \omega (t) \partial_y h_0(\varepsilon t, \varepsilon \rho(t)) \Big)
(1- \partial_z^2) \partial_z^{-1} \begin{pmatrix} A_{0,\omega} \\ B_{0,\omega} \end{pmatrix}
\\
\end{aligned}
\end{equation}
Recall that $\partial_z^{-1}$ denotes the antiderivative operator $\int_z^\infty$, so that for $F \in \mathcal S$, $\mathcal S$ the Schwartz class, we have $\partial_z^{-1} F$ converging to zero as $z\to +\infty$, but only in $L^\infty$ if $z\to -\infty$.
In this manner, the terms $ \partial_z^{-1} \partial_s h_0 (\varepsilon t, \varepsilon x ) $ and $(1- \partial_z^2) \partial_z^{-1} ( B_{0,\omega}, A_{0,\omega})^T$ are not necessarily in $(L^2)^2$. The first term reveals a strong influence of the bottom in the interaction dynamics of the solitary wave at the second order in $\varepsilon$, making the analysis more difficult than in previous cases \cite{Mu1}. In order to correct this error, we will perturb \eqref{Syst_2} as follows: we rewrite \eqref{S0_new_2000} as
\begin{equation}\label{S0_new_2}
{\bf S}_h(\bd{Q}_\omega + \bd{W} ) = {\bf S}_h^{\dagger}(\bd{Q}_\omega) +\partial_x J \mathcal L \bd{W} +\bd{R}^\dagger,
\end{equation}
where \eqref{S_hQ_new_2} becomes now
\begin{equation}\label{S_hQ_new_3}
\begin{aligned}
& {\bf S}_h^{\dagger}(\bd{Q}_\omega)
\\
 &~{} = (\omega ' - \varepsilon^2 f_1(t)) (1-\partial^2_z)\begin{pmatrix}
\Lambda (R_\omega +\varepsilon A_1 + \varepsilon^2 A_2)\\
\Lambda (Q_\omega+\varepsilon B_1 + \varepsilon^2 B_2)
\end{pmatrix}
\\
&~{}
\quad -(\rho'-\omega -  \varepsilon^2 f_2(t))(1-\partial^2_z) \left( \partial_z
\begin{pmatrix}
R_\omega  +\varepsilon A_1 + \varepsilon^2 A_2\\
 Q_\omega+\varepsilon B_1 + \varepsilon^2 B_2
 \end{pmatrix}  \right)
\\
&~{}
\quad -(\rho'-\omega -  \varepsilon^2 f_2(t))(1-\partial^2_z) \left( -  \varepsilon^2  \partial_y h_0(\varepsilon t, \varepsilon \rho(t)) \begin{pmatrix} A_{0,\omega} \\ B_{0,\omega} \end{pmatrix}    \right)
\\
&~{} \quad   + \varepsilon^2  \left(  \partial_y h_0(\varepsilon t, \varepsilon \rho(t))  \partial_z \begin{pmatrix}  zQ_\omega \\ 0 \end{pmatrix}
+  \partial_s h_0 (\varepsilon t, \varepsilon x)  \begin{pmatrix}
1\\
 0
 \end{pmatrix} \right)
 \\
&~{}  \quad  + \varepsilon^2  \left(  f_1(t) (1-\partial^2_z)\begin{pmatrix}
\Lambda R_\omega \\
\Lambda Q_\omega
\end{pmatrix}  -  f_2(t)(1-\partial^2_z)  \partial_z
\begin{pmatrix} R_\omega  \\ Q_\omega  \end{pmatrix}  \right)  \\
 &~{} \quad + \varepsilon^2  h_0^2 (\varepsilon t, \varepsilon \rho(t)) \partial_z \begin{pmatrix}  - B_{0,\omega} + A_{0,\omega}B_{0,\omega} \\   \frac12 B_{0,\omega}^2 \end{pmatrix} \\
 &~{} \quad  -  \varepsilon^2 h_0(\varepsilon t, \varepsilon \rho(t))  \left( \partial_s h_0(\varepsilon t, \varepsilon \rho(t)) + \omega (t) \partial_y h_0(\varepsilon t, \varepsilon \rho(t)) \right) (1- \partial_z^2) \begin{pmatrix} A_{0,\omega} \\ B_{0,\omega} \end{pmatrix};
 \end{aligned}
\end{equation}
the linear system \eqref{def_L_new_2} on $\bd{W}$ remains the same, and \eqref{R_new} becomes now
\begin{equation}\label{R_new_2}
\begin{aligned}
& \bd{R}^\dagger \\
&~{} = \frac12   \varepsilon^3   \partial_y^2 h_0(\varepsilon t, \varepsilon \rho(t))  \partial_z \begin{pmatrix}  z^2 Q_\omega  \\ 0 \end{pmatrix}
\\
 &~{} \quad +\varepsilon^3 \partial_x\!\left(  \partial_y h_0(\varepsilon t, \varepsilon \xi_3(t,x))  \begin{pmatrix} z B_1 \\ 0 \end{pmatrix}  +  h_0 (\varepsilon t, \varepsilon x) \begin{pmatrix}  B_2 \\ 0 \end{pmatrix}  \right) \\
 &~{} \quad  +  \varepsilon^3 \partial_z \begin{pmatrix} A_1B_2+A_2B_1  \\ B_1 B_2 \end{pmatrix}
 \\
 &~{} \quad  + \varepsilon^3  f_1(t) (1-\partial^2_z)\begin{pmatrix}
\Lambda ( A_1 + \varepsilon A_2)\\
\Lambda (B_1 + \varepsilon B_2)
\end{pmatrix}
 \\
 &~{}\quad  -  \varepsilon^3  f_2(t)(1-\partial^2_z) \left( \partial_z
\begin{pmatrix}
 A_1 + \varepsilon A_2\\
 B_1 + \varepsilon B_2
 \end{pmatrix} - \varepsilon  \partial_y h_0(\varepsilon t, \varepsilon \rho(t)) \begin{pmatrix} A_{0,\omega} \\ B_{0,\omega} \end{pmatrix}  \right) \\
&~{}\quad  +  \varepsilon^4 \left( -   \begin{pmatrix}
  a_1\partial^2_y\partial_s h_0 \\
 c_1 \partial^2_s\partial_y h_0
 \end{pmatrix} (\varepsilon t,\varepsilon x)  +\frac16   \partial_y^3 h_0(\varepsilon t, \varepsilon \rho(t)) \partial_z \begin{pmatrix}  z^3 Q_\omega \\ 0 \end{pmatrix}+ \partial_z \begin{pmatrix} A_2B_2 \\ \frac12B_2^2 \end{pmatrix} \right)  \\
 &~{} \quad +  \varepsilon^5 \partial_x \begin{pmatrix}  \tilde h_0(\varepsilon t, \varepsilon \xi_1(t,x)) z^4 Q_\omega   \\ 0 \end{pmatrix}.
 \end{aligned}
\end{equation}
 The new coefficients $f_1$ and $f_2$ will be chosen such that \eqref{Syst_2}-\eqref{S_hQ_new_3} now become
\begin{equation}\label{Syst_3}
 \begin{aligned}
\mathcal L
\begin{pmatrix} A_2 \\ B_2 \end{pmatrix}
&~{} = - \partial_y h_0(\varepsilon t, \varepsilon \rho(t)) \begin{pmatrix}  0 \\ zQ_\omega \end{pmatrix}
-  \partial_z^{-1} \partial_s h_0 (\varepsilon t, \varepsilon x )  \begin{pmatrix}
0 \\
 1
 \end{pmatrix}\\
 &~{} \quad -h_0^2 (\varepsilon t, \varepsilon \rho(t))  \begin{pmatrix}    \frac12 B_{0,\omega}^2 \\ - B_{0,\omega} + A_{0,\omega}B_{0,\omega}  \end{pmatrix} \\
 &~{}\quad +   h_0(\varepsilon t, \varepsilon \rho(t))  \left( \partial_s h_0(\varepsilon t, \varepsilon \rho(t)) + \omega (t) \partial_y h_0(\varepsilon t, \varepsilon \rho(t)) \right) (1- \partial_z^2) \partial_z^{-1} \begin{pmatrix} B_{0,\omega} \\ A_{0,\omega}  \end{pmatrix}  \\
 &~{}  \quad   -  f_1(t) (1-\partial^2_z) \partial_z^{-1} \begin{pmatrix}
\Lambda Q_\omega \\
\Lambda R_\omega
\end{pmatrix}  +   f_2(t)(1-\partial^2_z)
\begin{pmatrix} Q_\omega  \\ R_\omega  \end{pmatrix}  .
\end{aligned}
\end{equation}

 Let us observe that on the right-hand side of \eqref{Syst_3}, the second, fourth, and fifth terms are just in $L^\infty(\mathbb{R})^2$. Let us observe that on the right-hand side of \eqref{Syst_3}, the second, fourth, and fifth terms are just in $L^\infty(\mathbb{R})^2$. Then Lemma 2.9 does not apply straightforwardly. We then proceed as follows.

\smallskip

\noindent {\bf Step 0}. First solving

\begin{equation}\label{Syst_3bis}
 \begin{aligned}
\mathcal L
\begin{pmatrix} A_{2,1}\\ B_{2,1} \end{pmatrix}=
 &~{}  -h_0^2 (\varepsilon t, \varepsilon \rho(t))  \begin{pmatrix}    \frac12 B_{0,\omega}^2 \\ - B_{0,\omega} + A_{0,\omega}B_{0,\omega}  \end{pmatrix}  
 +   f_2(t)(1-\partial^2_z)
\begin{pmatrix} Q_\omega  \\ R_\omega  \end{pmatrix} ,
\end{aligned}
\end{equation}
is straightforward since the right-hand side belongs
to (Ker${\mathcal L})^\perp$. We then seek
$$
\begin{pmatrix} A_{2,2}\\ B_{2,2} \end{pmatrix}=
\begin{pmatrix} A_2-A_{2,1}\\ B_2-B_{2,1} \end{pmatrix}. $$
Let us observe that
\[
\mathcal L= \begin{pmatrix} 1 & -\omega\\ -\omega & 1  \end{pmatrix}
+\begin{pmatrix} c\partial^2_z & \omega \partial^2_z +Q_\omega \\ \omega \partial^2_z + Q_\omega & a\partial^2_zR_\omega\end{pmatrix}.
\]
Set $M(\omega)=\begin{pmatrix} 1 & -\omega\\ -\omega & 1\end{pmatrix}^{-1}.$ 
Solving \eqref{Syst_3} reads now
$$\mathcal L \begin{pmatrix} A_{2,2} \\ B_{2,2} \end{pmatrix}= \begin{pmatrix} F \\ G\end{pmatrix},  $$
\noindent amounts to solving, setting 
$$ \begin{pmatrix} \tilde A_{2,2} \\ \tilde B_{2,2} \end{pmatrix}
=\begin{pmatrix} A_{2,2}\\ B_{2,2}\end{pmatrix} - M(\omega) \begin{pmatrix} F \\ G\end{pmatrix}$$
\begin{equation}\label{newsun1}
\mathcal L \begin{pmatrix} \tilde A_{2,2}\\ \tilde B_{2,2}\end{pmatrix} = - \begin{pmatrix} c\partial^2_z & \omega \partial^2_z +Q_\omega \\ \omega \partial^2_z + Q_\omega & a\partial^2_zR_\omega\end{pmatrix}
M(\omega) \begin{pmatrix} F \\ G\end{pmatrix}.
\end{equation}
This can be solved by appealing to Lemma 2.9 since the right-hand side of
\eqref{newsun1} belongs to $L^2(\mathbb{R})^2$. For this purpose,
we now choose $f_1(t)$ in order to ensure that the right-hand side
of \eqref{newsun1} belongs to (Ker${\mathcal L})^\perp$. 
To simplify the computations, we claim 
\begin{equation}\label{newsun2}
{\rm If} \: \left\langle \begin{pmatrix} F \\ G\end{pmatrix}, {\bd Q'_\omega} \right\rangle=0 \quad {\rm then} \quad \left\langle\begin{pmatrix} c\partial^2_z & \omega \partial^2_z +Q_\omega \\ \omega \partial^2_z + Q_\omega & a\partial^2_zR_\omega\end{pmatrix}
M(\omega) \begin{pmatrix} F \\ G\end{pmatrix}, {\bd  Q'_\omega} \right\rangle=0.
\end{equation}
 The proof of \eqref{newsun2} reads
\begin{equation}\label{newsun3}
\left\langle ({\mathcal L}-M(\omega)^{-1})M(\omega)\begin{pmatrix} F \\ G\end{pmatrix}, {\bd  Q'_\omega} \right\rangle=
\left\langle M(\omega)\begin{pmatrix} F \\ G\end{pmatrix}, {\mathcal L}{\bd  Q'_\omega} \right\rangle=0.
\end{equation}

We summarize the previous computations in the following statement

\begin{lemma}\label{relax} Consider ${\bd F}\in L^\infty(\R)^2$ such that $\partial^l_z {\bd F}$ has exponential decay for any $l\geq 1$ and such that $\langle{\bd F,  \bd Q'}\rangle=0$. Then there exists a unique solution
to ${\mathcal L}{\bd A}={\bd F}$ in $L^\infty(\mathbb{R})^2$ such that $\langle {\bd A,  Q'}\rangle=0$. Moreover, $\partial^l_z {\bd A}$
has exponential decay for any $l\geq 1$.
\end{lemma}

\begin{proof}
For existence, we essentially carbon-copy the arguments in
\eqref{newsun1}-\eqref{newsun3}. For the extra decay
of the derivatives, simply differentiate the equation
and prove the result recursively on $l$.
\end{proof}

\noindent We now solve $\left\langle \begin{pmatrix} F \\ G\end{pmatrix}, {\bd  Q'_\omega} \right\rangle=0$ in the next step.
Additionally, it will be fixed such that \eqref{Syst_3} has a unique solution $\bd{W} \in L^\infty \times L^\infty$
satisfying  $ \langle \bd{Q}_\omega ,  \bd{W} \rangle =\langle \bd{Q}_\omega' ,  \bd{W} \rangle  =0$.

\medskip

\noindent {\bf Step 1:} In order to obtain solvability, we require
\begin{equation}\label{Syst_4}
\begin{aligned}
0 = &~{} -  \partial_y h_0(\varepsilon t, \varepsilon \rho(t))  \left\langle  \bd{Q}_\omega' , \begin{pmatrix}  0 \\ zQ_\omega \end{pmatrix} \right\rangle
-  \left\langle  \bd{Q}_\omega' ,  \partial_z^{-1} \partial_s h_0 (\varepsilon t, \varepsilon x )  \begin{pmatrix} 0 \\ 1 \end{pmatrix} \right\rangle \\
 &~{} \quad - h_0^2 (\varepsilon t, \varepsilon \rho(t))  \left\langle  \bd{Q}_\omega' ,  \begin{pmatrix}    \frac12 B_{0,\omega}^2 \\ - B_{0,\omega} + A_{0,\omega}B_{0,\omega}  \end{pmatrix} \right\rangle \\
 &~{} \quad   + h_0(\varepsilon t, \varepsilon \rho(t))  \left( \partial_s h_0(\varepsilon t, \varepsilon \rho(t)) + \omega (t) \partial_y h_0(\varepsilon t, \varepsilon \rho(t)) \right) \left\langle  \bd{Q}_\omega' ,  (1- \partial_z^2) \partial_z^{-1} \begin{pmatrix} B_{0,\omega} \\ A_{0,\omega}  \end{pmatrix}  \right\rangle \\
 &~{}  \quad   -  f_1(t)  \left\langle  \bd{Q}_\omega' ,  (1-\partial^2_z) \partial_z^{-1} \begin{pmatrix}
\Lambda Q_\omega \\
\Lambda R_\omega
\end{pmatrix}  \right\rangle +   f_2(t)
 \left\langle  \bd{Q}_\omega' , (1-\partial^2_z) \begin{pmatrix} Q_\omega  \\ R_\omega  \end{pmatrix} \right\rangle .
\end{aligned}
\end{equation}
A further simplification in \eqref{Syst_4} that uses the exact value of $\bd{Q}_\omega$, parity properties, and integration by parts gives
\begin{equation}\label{Syst_5}
\begin{aligned}
0 = &~{}  \frac12  \partial_y h_0(\varepsilon t, \varepsilon \rho(t)) \int Q_\omega^2
+ \int  \partial_s h_0 (\varepsilon t, \varepsilon x )  Q_\omega (z) \\
 &~{}  + h_0(\varepsilon t, \varepsilon \rho(t))  \left( \partial_s h_0(\varepsilon t, \varepsilon \rho(t)) + \omega (t) \partial_y h_0(\varepsilon t, \varepsilon \rho(t)) \right) \\
 &~{} \quad \times \int \left( R_\omega (1- \partial_z^2) B_{0,\omega} + Q_\omega (1-\partial_z^2 )A_{0,\omega}  \right) \\
 &~{}   +  f_1(t)  \partial_\omega \int R_\omega  (1-\partial^2_z)  Q_\omega . 
\end{aligned}
\end{equation}
Bearing in mind \eqref{mom}, we define
\[
\begin{aligned}
d_0(\omega):= &~{} \left(\partial_\omega \int R_\omega  (1-\partial^2_z)  Q_\omega \right)^{-1}, \\
d_2(\omega) := &~{} \int \left( R_\omega (1- \partial_z^2) B_{0,\omega} + Q_\omega (1-\partial_z^2 )A_{0,\omega}  \right),
\end{aligned}
\]
so  we have a unique $f_1$ of \eqref{Syst_5} given by:
one has
\begin{equation}\label{f1}
\begin{aligned}
f_1(t) = &~{} - d_0(\omega) \Bigg( \frac12  \partial_y h_0(\varepsilon t, \varepsilon \rho(t)) \int Q_\omega^2 + \int  \partial_s h_0 (\varepsilon t, \varepsilon (z+\rho(t)) )  Q_\omega (z)dz  \\
&~{} \qquad \qquad + d_2(\omega)  h_0(\varepsilon t, \varepsilon \rho(t))  \left( \partial_s h_0(\varepsilon t, \varepsilon \rho(t)) + \omega (t) \partial_y h_0(\varepsilon t, \varepsilon \rho(t)) \right)  \Bigg).
\end{aligned}
\end{equation}

\noindent  With this definition of $f_1$
applying Lemma \ref{relax} we have the existence and uniqueness of $\begin{pmatrix} A_2 \\ B_2 \end{pmatrix}$
that solves \eqref{Syst_3}.

Under the bootstrap assumption $|\omega(t)-\omega_0|\leq \frac1{100}$, one can get rid of terms involving $\omega$. Later we will improve this estimate by showing in Lemma \ref{DS_lem} that $|\omega(t)-\omega_0|\leq C\varepsilon$. Now we perform some estimates on $f_1$ in \eqref{f1} using the 1/100 assumption. First, using \eqref{hypoH},
\begin{equation}\label{f1_1}
\begin{aligned}
\left| \frac12  \partial_y h_0(\varepsilon t, \varepsilon \rho(t)) \int Q_\omega^2 \right| \lesssim e^{-k_0 \varepsilon |t|}e^{-l_0 \varepsilon |\rho(t)|}.
\end{aligned}
\end{equation}
Second,
\begin{equation}\label{f1_2}
\begin{aligned}
\left| \int  \partial_s h_0 (\varepsilon t, \varepsilon x )  Q_\omega (z)\right| \lesssim e^{-k_0 |\varepsilon t|} \int  Q_\omega (z) e^{-l_0 \varepsilon x} dx \lesssim e^{-k_0 \varepsilon |t|} e^{-l_0 \varepsilon |\rho(t)|}.
\end{aligned}
\end{equation}
Now, using the exponential decay of $Q_\omega$,
\begin{equation}\label{f1_4}
\begin{aligned}
& \left|   h_0(\varepsilon t, \varepsilon \rho(t))  \left( \partial_s h_0(\varepsilon t, \varepsilon \rho(t)) + \omega (t) \partial_y h_0(\varepsilon t, \varepsilon \rho(t)) \right) \right|
\\
& \quad \times \left|  \int \left( R_\omega (1- \partial_z^2) B_{0,\omega} + Q_\omega (1-\partial_z^2 )A_{0,\omega}  \right) \right| \lesssim e^{-2k_0 \varepsilon |t|} e^{-2l_0 \varepsilon |\rho(t)|}.
\end{aligned}
\end{equation}
Gathering in \eqref{Syst_5} the estimates \eqref{f1_1}, \eqref{f1_2}, and \eqref{f1_4}, we get
\begin{equation}\label{cota_f1}
|f_1(t)| \lesssim e^{-k_0 \varepsilon |t| -l_0 \varepsilon |\rho(t)|}.
\end{equation}

\noindent {\bf Step 2}. We have now
\begin{equation}\label{Syst_3_A2B2} 
\begin{aligned}
\begin{pmatrix} A_2 \\ B_2 \end{pmatrix} &~{} = \mathcal L^{-1}  \left[  - \partial_y h_0(\varepsilon t, \varepsilon \rho(t)) \begin{pmatrix}  0 \\ zQ_\omega \end{pmatrix} \right.
-  \partial_z^{-1} \partial_s h_0 (\varepsilon t, \varepsilon x )  \begin{pmatrix}
0 \\
 1
 \end{pmatrix}\\
 &~{}\qquad \qquad  +  h_0(\varepsilon t, \varepsilon \rho(t))  \left( \partial_s h_0(\varepsilon t, \varepsilon \rho(t)) + \omega (t) \partial_y h_0(\varepsilon t, \varepsilon \rho(t)) \right) (1- \partial_z^2) \partial_z^{-1} \begin{pmatrix} B_{0,\omega} \\ A_{0,\omega}  \end{pmatrix}  \\
 &~{}  \qquad\qquad   \left.  -  f_1(t) (1-\partial^2_z) \partial_z^{-1} \begin{pmatrix}
\Lambda Q_\omega \\
\Lambda R_\omega
\end{pmatrix}   \right]
\\
&~{} \quad +   f_2(t) \mathcal L^{-1} \left( (1-\partial^2_z)
\begin{pmatrix} Q_\omega  \\ R_\omega  \end{pmatrix} \right) \\
 &~{} \quad -h_0^2 (\varepsilon t, \varepsilon \rho(t))  \mathcal L^{-1} \begin{pmatrix}    \frac12 B_{0,\omega}^2 \\ - B_{0,\omega} + A_{0,\omega}B_{0,\omega}  \end{pmatrix} .
\end{aligned}
\end{equation}
Notice that $(A_2,B_2)$ are not in $L^2\times L^2$. We will see below that $\mathcal L^{-1}$ is composed of a term in $L^\infty\times L^\infty$ plus a term in $L^2\times L^2$. In that sense, $\mathcal L^{-1}$ must be understood as a generalized inverse of $\mathcal L$. In fact, there are three terms that are not in $L^2$
in the right-hand side of \eqref{Syst_3_A2B2}, namely the second, the third, and the fourth one. Let us observe also that we do not know
$f_2(t)$; we will chose $f_2$ in the sequel.

\medskip

\noindent We now specify the bounds on the second, third, and fourth terms in the right-hand side of \eqref{Syst_3_A2B2}. 

Let us define

\begin{equation}\label{A20}
\begin{aligned}
& \begin{pmatrix} A_{2,0} \\ B_{2,0} \end{pmatrix}
\\
&  :=   h_0(\varepsilon t, \varepsilon \rho(t))  \left( \partial_s h_0(\varepsilon t, \varepsilon \rho(t)) + \omega (t) \partial_y h_0(\varepsilon t, \varepsilon \rho(t)) \right)  M(\omega) (1- \partial_z^2) \partial_z^{-1} \begin{pmatrix} B_{0,\omega} \\ A_{0,\omega}  \end{pmatrix}
\\
& \qquad -  \frac1{1-\omega^2}\partial_z^{-1} \partial_s h_0 (\varepsilon t, \varepsilon x )  \begin{pmatrix} \omega \\ 1 \end{pmatrix}  -  f_1(t)  M(\omega) (1-\partial^2_z) \partial_z^{-1} \begin{pmatrix}
\Lambda Q_\omega \\
\Lambda R_\omega
\end{pmatrix}.
\end{aligned}
\end{equation}

An important point to be emphasized is that both $A_{2,0}$ and $B_{2,0}$ are just bounded functions in the $z$ variable. Special care requires the function $\partial_z^{-1} \partial_s h_0 (\varepsilon t, \varepsilon x )$. Indeed, one has
\begin{equation}\label{mejorada0}
\begin{aligned}
& |\partial_z^{-1} \partial_s h_0 (\varepsilon t, \varepsilon x )| \\
& \quad \lesssim   e^{-k_0 \varepsilon |t|} \int_{z+ \rho(t)}^\infty e^{-l_0 \varepsilon |\sigma|} d\sigma
\lesssim \begin{cases}  \frac1{\varepsilon}  e^{-k_0 \varepsilon |t|} e^{-l_0 \varepsilon (z+\rho(t))} , & z > -\rho(t) \\  \frac1{\varepsilon}  e^{-k_0 \varepsilon |t|}, & z\leq -\rho(t) \end{cases}
\\
& \quad \lesssim \frac1{\varepsilon} e^{-k_0 \varepsilon |t|}e^{-l_0 \varepsilon (z+\rho(t))_+ }.
\end{aligned}
\end{equation}
Similarly, for $\ell=1,2,3,\ldots$
\begin{equation}\label{mejorada}
| \partial_z^\ell \partial_z^{-1} \partial_s h_0 (\varepsilon t, \varepsilon x )| \lesssim   \varepsilon^{\ell-1} e^{-k_0 \varepsilon |t|} e^{-l_0 \varepsilon |x|}. 
\end{equation}
Gathering these estimates, we have from \eqref{A20} and \eqref{f1}-\eqref{cota_f1},
\[
 \lim_{z\to +\infty} \left( |A_{2,0}(z; \omega(t),\rho(t))| +|B_{2,0}(z; \omega(t),\rho(t))|  \right)=0,\\
\]
and more precisely, using \eqref{H2}, \eqref{hypoH}, \eqref{decay_Q} and \eqref{cota_f1},
\begin{equation}\label{ErrorA20}
\begin{aligned}
& |A_{2,0}(z; \omega(t),\rho(t))| +|B_{2,0}(z; \omega(t),\rho(t))| \\
& \quad \lesssim  \left|   h_0(\varepsilon t, \varepsilon \rho(t))  \left( \partial_s h_0(\varepsilon t, \varepsilon \rho(t)) + \omega (t) \partial_y h_0(\varepsilon t, \varepsilon \rho(t)) \right)  M(\omega) (1- \partial_z^2) \partial_z^{-1} \begin{pmatrix} B_{0,\omega} \\ A_{0,\omega}  \end{pmatrix}\right|
\\
& \quad \quad +\left|  \frac1{1-\omega^2}\partial_z^{-1} \partial_s h_0 (\varepsilon t, \varepsilon x )  \begin{pmatrix} \omega \\ 1 \end{pmatrix}  \right| + \left|  f_1(t)  M(\omega) (1-\partial^2_z) \partial_z^{-1} \begin{pmatrix}
\Lambda Q_\omega \\
\Lambda R_\omega
\end{pmatrix} \right| \\
& \quad \lesssim \frac1{\varepsilon} e^{-k_0 \varepsilon |t|}e^{-l_0 \varepsilon (z+\rho(t))_+ } + e^{-k_0 \varepsilon |t| -l_0 \varepsilon |\rho(t)|} e^{- \frac12\mu_0 z_+}.
\end{aligned}
\end{equation}
Similarly, for $\ell=1,2,3,\ldots$ we use \eqref{mejorada} to get the better estimate
\begin{equation}\label{ErrorA201}
\begin{aligned}
& | \partial_z^\ell A_{2,0}(z; \omega(t),\rho(t))| +| \partial_z^\ell B_{2,0}(z; \omega(t),\rho(t))|
\\
&  \lesssim  \left|  h_0(\varepsilon t, \varepsilon \rho(t))  \left( \partial_s h_0(\varepsilon t, \varepsilon \rho(t)) + \omega (t) \partial_y h_0(\varepsilon t, \varepsilon \rho(t)) \right)  M(\omega) (1- \partial_z^2) \partial_z^{\ell-1} \begin{pmatrix} B_{0,\omega} \\ A_{0,\omega}  \end{pmatrix}\right|
\\
&  \quad +\left|  \frac1{1-\omega^2}\partial_z^{\ell-1} \partial_s h_0 (\varepsilon t, \varepsilon x )  \begin{pmatrix} \omega \\ 1 \end{pmatrix}  \right| + \left|  f_1(t)  M(\omega) (1-\partial^2_z) \partial_z^{\ell-1} \begin{pmatrix}
\Lambda Q_\omega \\
\Lambda R_\omega
\end{pmatrix} \right| \\
&  \lesssim   \varepsilon^{\ell-1} e^{-k_0 \varepsilon |t|}e^{-l_0 \varepsilon |x|} +e^{-k_0 \varepsilon |t| -l_0 \varepsilon |\rho(t)|} e^{- \frac12\mu_0  |z|}.
\end{aligned}
\end{equation}

\medskip

{\bf Step 3.}
From \eqref{A20}, we have
\begin{equation}\label{LA20}
\begin{aligned}
& \mathcal L \begin{pmatrix} A_{2,0} \\ B_{2,0} \end{pmatrix}
\\
&\quad  = h_0(\varepsilon t, \varepsilon \rho(t))  \left( \partial_s h_0(\varepsilon t, \varepsilon \rho(t)) + \omega (t) \partial_y h_0(\varepsilon t, \varepsilon \rho(t)) \right) (1- \partial_z^2) \partial_z^{-1} \begin{pmatrix} B_{0,\omega} \\ A_{0,\omega}  \end{pmatrix}
\\
& \quad \qquad -  \frac1{1-\omega^2}\partial_z^{-1} \partial_s h_0 (\varepsilon t, \varepsilon x )  \begin{pmatrix} 0 \\ 1 \end{pmatrix}  -  f_1(t)  (1-\partial^2_z) \partial_z^{-1} \begin{pmatrix}
\Lambda Q_\omega \\
\Lambda R_\omega
\end{pmatrix}.
\end{aligned}
\end{equation}
Finally, following \eqref{Syst_3_A2B2} we decompose 
\begin{equation}\label{decoA2}
\begin{aligned}
\begin{pmatrix} A_{2} \\ B_{2} \end{pmatrix} = &~{} \mathcal L^{-1} F +   f_2(t) \mathcal L^{-1} \left( (1-\partial^2_z)  \begin{pmatrix} Q_\omega  \\ R_\omega  \end{pmatrix} \right)
\\
&~{}  -h_0^2 (\varepsilon t, \varepsilon \rho(t))  \mathcal L^{-1} \begin{pmatrix}    \frac12 B_{0,\omega}^2 \\ - B_{0,\omega} + A_{0,\omega}B_{0,\omega}  \end{pmatrix} ,
\end{aligned}
\end{equation}
where
\begin{equation}\label{decoA22}
\mathcal L^{-1} F  := \begin{pmatrix} A_{2,0} \\ B_{2,0} \end{pmatrix} + \begin{pmatrix} A_{2,1} \\ B_{2,1} \end{pmatrix},
\end{equation}
with $A_{2,0},B_{2,0}$ given in \eqref{A20} and $A_{2,1},B_{2,1}$ to be found.

\medskip

{\bf Step 4}. Then, replacing in \eqref{Syst_3_A2B2}, one gets from \eqref{decoA22} the better behaved problem for $A_{2,1},B_{2,1}$:
\[
\mathcal L  \begin{pmatrix} A_{2,1} \\ B_{2,1} \end{pmatrix} = F - \mathcal L \begin{pmatrix} A_{2,0} \\ B_{2,0} \end{pmatrix},
\]
which reads after using \eqref{LA20},
\begin{equation}\label{Syst_3_nuevo}
\begin{aligned}
& \begin{pmatrix}
 c \partial_z^2 +1  &   -\omega (1-\partial^2_z) +Q_\omega   \\
  -\omega (1-\partial^2_z)  +Q_\omega & a \partial_z^2 +1 +R_\omega
\end{pmatrix}
\begin{pmatrix} A_{2,1} \\ B_{2,1} \end{pmatrix} \\
&~{} = - \partial_y h_0(\varepsilon t, \varepsilon \rho(t)) \begin{pmatrix}  0 \\ zQ_\omega \end{pmatrix}  -  \begin{pmatrix}
 c \partial_z^2   &   \omega \partial^2_z +Q_\omega   \\
  \omega \partial^2_z  +Q_\omega & a \partial_z^2  +R_\omega
\end{pmatrix} \begin{pmatrix} A_{2,0} \\ B_{2,0} \end{pmatrix}.
\end{aligned}
\end{equation}
By construction, the right-hand side in \eqref{Syst_3_nuevo} is orthogonal to $\bd{Q}'_\omega.$  Therefore, the existence of $A_{2,1},B_{2,1}$ is guaranteed. Moreover, now the right-hand side in \eqref{Syst_3_nuevo} is exponentially decreasing. This is a consequence of the fact that $A_{2,0},B_{2,0}$ are merely bounded (see \eqref{ErrorA20}), with partial derivatives localized in space by \eqref{ErrorA201}. More precisely, thanks to \eqref{ErrorA20} and \eqref{ErrorA201},
\begin{equation}\label{CotaF}
\begin{aligned}
& \left|  - \partial_y h_0(\varepsilon t, \varepsilon \rho(t)) \begin{pmatrix}  0 \\ zQ_\omega \end{pmatrix}  -  \begin{pmatrix}
 c \partial_z^2   &   \omega \partial^2_z +Q_\omega   \\
  \omega \partial^2_z  +Q_\omega & a \partial_z^2  +R_\omega
\end{pmatrix} \begin{pmatrix} A_{2,0} \\ B_{2,0} \end{pmatrix} \right| \\
& \quad \lesssim e^{-k_0 \varepsilon |t| -l_0 \varepsilon |\rho(t)|} e^{- \frac12\mu_0  |z|}  + \left| \partial_z^2   \begin{pmatrix} A_{2,0} \\ B_{2,0} \end{pmatrix} \right| + \left| e^{-\mu_0 |z|} \begin{pmatrix} A_{2,0} \\ B_{2,0} \end{pmatrix} \right|  \\
& \quad \lesssim e^{-k_0 \varepsilon |t| -l_0 \varepsilon |\rho(t)|} e^{- \frac12\mu_0  |z|} + \varepsilon e^{-k_0 \varepsilon |t|}e^{-l_0 \varepsilon |x|}  
 +\frac1{\varepsilon} e^{-k_0 \varepsilon |t|} e^{- \mu_0  |z|}.
\end{aligned}
\end{equation}
By using Lemma~\ref{lem:exp:decay}, we conclude that for
\begin{equation}\label{CotaA2}
\begin{aligned}
& \left|  \begin{pmatrix} A_{2,1} \\ B_{2,1} \end{pmatrix} \right| (z; \omega(t),\rho(t))\\
& \quad \lesssim e^{-k_0 \varepsilon |t| -l_0 \varepsilon |\rho(t)|} e^{- \frac12\tilde \mu_0  |z|} + \varepsilon e^{-k_0 \varepsilon |t|}e^{-\tilde  l_0 \varepsilon |x|}  
 +\frac1{\varepsilon} e^{-k_0 \varepsilon |t|} e^{- \mu_0  |z|}.
\end{aligned}
\end{equation}
for some constants $\tilde \mu>0$, $\tilde l_0>0$. {For simplicity, we will continue to denote these constants by $\mu_0$ and $l_0$.}
Using that
\[
\mathcal L \partial_z  \begin{pmatrix} A_{2,1} \\ B_{2,1} \end{pmatrix} = \partial_z F - \mathcal L \partial_z \begin{pmatrix} A_{2,0} \\ B_{2,0} \end{pmatrix} - (\partial_z \mathcal L) \begin{pmatrix} A_{2,0} \\ B_{2,0} \end{pmatrix} -(\partial_z\mathcal L)  \begin{pmatrix} A_{2,1} \\ B_{2,1} \end{pmatrix},
\]
we get using the value of $F$,
\begin{equation}\label{cota_Fz}
\begin{aligned}
& |\partial_z F|(z,\varepsilon x; \omega(t),\rho(t)) \lesssim  e^{-k_0 \varepsilon|t| -l_0 \varepsilon |\rho(t)|}  e^{-\frac12 \mu_0 |z|}  + e^{-k_0 \varepsilon|t| -l_0 \varepsilon |x|} .
\end{aligned}
\end{equation}

{\bf Step 5.}
Estimate \eqref{cota_Fz} reveals that the $z$ derivatives of $F$ have better decay estimates than the original $F$. This translates to the estimates on the derivatives of $A_{2,1}$ and $B_{2,1}$ as follows:
\begin{equation}\label{CotaA2_2}
\begin{aligned}
& \left| \partial_z^{\ell}  \begin{pmatrix} A_{2,1} \\ B_{2,1} \end{pmatrix} \right| (z; \omega(t),\rho(t))\\
& \quad \lesssim e^{-k_0 \varepsilon |t| -l_0 \varepsilon |\rho(t)|} e^{- \frac12\mu_0  |z|} + \varepsilon^{\ell +1} e^{-k_0 \varepsilon |t|}e^{-l_0 \varepsilon |x|}  
 +\frac1{\varepsilon} e^{-k_0 \varepsilon |t|} e^{- \mu_0  |z|}.
\end{aligned}
\end{equation}
Finally, we recover from \eqref{decoA2} and \eqref{decoA22},
\begin{equation}\label{Sol_2}
\begin{aligned}
\begin{pmatrix} A_{2} \\ B_{2} \end{pmatrix} = &~{}  m_2  \bd{Q}'_\omega +  \begin{pmatrix} A_{2,0} \\ B_{2,0} \end{pmatrix} + \begin{pmatrix} A_{2,1} \\ B_{2,1} \end{pmatrix}  +   f_2(t) \mathcal L^{-1} \left( (1-\partial^2_z)  \begin{pmatrix} Q_\omega  \\ R_\omega  \end{pmatrix} \right)
\\
&~{}  -h_0^2 (\varepsilon t, \varepsilon \rho(t))  \mathcal L^{-1} \begin{pmatrix}    \frac12 B_{0,\omega}^2 \\ - B_{0,\omega} + A_{0,\omega}B_{0,\omega}  \end{pmatrix} ,
\end{aligned}
\end{equation}
with $A_{2,0}, B_{2,0}$ given in \eqref{A20}, and $A_{2,1}, B_{2,1}$ given by \eqref{Syst_3_nuevo} and \eqref{CotaA2}.
We choose now $m_2$ such that the solution $(A_2,B_2)$ is orthogonal to $\bd{Q}'_\omega $. This gives $m_2=0$.

\medskip

{\bf Step 6.} Now we choose $f_2(t)$ such that the solution is orthogonal to $J(1-\partial_x^2)\bd{Q}_\omega$, exactly as in \eqref{coer_10}. This uniquely determines $f_2$:
\begin{equation}\label{Sol_2_ortho}
\begin{aligned}
& \left\langle \begin{pmatrix} A_{2} \\ B_{2} \end{pmatrix}, (1-\partial^2_z)  \begin{pmatrix} Q_\omega  \\ R_\omega  \end{pmatrix}\right\rangle
\\
 &~{} =   \left\langle \begin{pmatrix} A_{2,0} \\ B_{2,0} \end{pmatrix}, (1-\partial^2_z)  \begin{pmatrix} Q_\omega  \\ R_\omega  \end{pmatrix}\right\rangle  +\left\langle \begin{pmatrix} A_{2,1} \\ B_{2,1} \end{pmatrix},  (1-\partial^2_z)  \begin{pmatrix} Q_\omega  \\ R_\omega  \end{pmatrix}\right\rangle
 \\
 &~{} \quad  +   f_2(t) \left\langle \mathcal L^{-1} \left( (1-\partial^2_z)  \begin{pmatrix} Q_\omega  \\ R_\omega  \end{pmatrix} \right) , (1-\partial^2_z)  \begin{pmatrix} Q_\omega  \\ R_\omega  \end{pmatrix}\right\rangle
\\
&~{}  \quad  -h_0^2 (\varepsilon t, \varepsilon \rho(t)) \left\langle \mathcal L^{-1} \begin{pmatrix}    \frac12 B_{0,\omega}^2 \\ - B_{0,\omega} + A_{0,\omega}B_{0,\omega}  \end{pmatrix} , (1-\partial^2_z)  \begin{pmatrix} Q_\omega  \\ R_\omega  \end{pmatrix}\right\rangle =0.
\end{aligned}
\end{equation}
Indeed, from \eqref{coer_10} we have
\[
\left\langle (1-\partial^2_z)  \begin{pmatrix} Q_\omega  \\ R_\omega  \end{pmatrix} ,   \mathcal L^{-1} \left( (1-\partial^2_z)  \begin{pmatrix} Q_\omega  \\ R_\omega  \end{pmatrix} \right) \right\rangle < 0.
\]
This fact allows us to isolate $f_2$. We get, using \eqref{CotaA2} and \eqref{Sol_2_ortho},
\begin{equation}\label{cota_f2}
|f_2(t)| \lesssim e^{-k_0 \varepsilon |t| -l_0 \varepsilon |\rho(t)|}.
\end{equation}

\medskip

{\bf Step 7.}
Finally, from \eqref{Sol_2}, \eqref{ErrorA20} and \eqref{CotaA2} we obtain the pointwise bound
\begin{equation}\label{Sol_2_final}
\begin{aligned}
& \left| \begin{pmatrix} A_{2} \\ B_{2} \end{pmatrix} \right| (z; \omega(t),\rho(t))\\
 &~{} \quad \lesssim   \frac1{\varepsilon} e^{-k_0 \varepsilon |t|}e^{-l_0 \varepsilon (z+\rho(t))_+ } + e^{-k_0 \varepsilon |t| -l_0 \varepsilon |\rho(t)|} e^{- \frac12\mu_0 z_+}  \\
 &~{} \quad \qquad + e^{-k_0 \varepsilon |t| -l_0 \varepsilon |\rho(t)|} e^{- \frac12\mu_0  |z|} + \varepsilon e^{-k_0 \varepsilon |t|}e^{-l_0 \varepsilon |x|}  
 +\frac1{\varepsilon} e^{-k_0 \varepsilon |t|} e^{- \mu_0  |z|}
 \\
 &~{} \quad \lesssim   \frac1{\varepsilon} e^{-k_0 \varepsilon |t|}e^{-l_0 \varepsilon (z+\rho(t))_+ } + e^{-k_0 \varepsilon |t| -l_0 \varepsilon |\rho(t)|} e^{- \frac12\mu_0 z_+}  \\
 &~{} \quad \qquad + \varepsilon e^{-k_0 \varepsilon |t|}e^{-l_0 \varepsilon |x|}  +\frac1{\varepsilon} e^{-k_0 \varepsilon |t|} e^{- \mu_0  |z|} .
\end{aligned}
\end{equation}
This estimate tells us that both $A_2$ and $B_2$ are exponentially decreasing on the right, but just bounded as $z\to -\infty$. Fortunately, their amplitude decays exponentially in time.
The situation for the derivatives is hopefully better: first of all we have from \eqref{Sol_2} and $\ell=1,2,3,\ldots$
\[
\begin{aligned}
\partial_z^\ell \begin{pmatrix} A_{2} \\ B_{2} \end{pmatrix} = &~{}  \partial_z^\ell \begin{pmatrix} A_{2,0} \\ B_{2,0} \end{pmatrix} +\partial_z^\ell \begin{pmatrix} A_{2,1} \\ B_{2,1} \end{pmatrix}  +   f_2(t) \partial_z^\ell\mathcal L^{-1} \left( (1-\partial^2_z)  \begin{pmatrix} Q_\omega  \\ R_\omega  \end{pmatrix} \right)
\\
&~{}  -h_0^2 (\varepsilon t, \varepsilon \rho(t)) \partial_z^\ell \mathcal L^{-1} \begin{pmatrix}    \frac12 B_{0,\omega}^2 \\ - B_{0,\omega} + A_{0,\omega}B_{0,\omega}  \end{pmatrix}.
\end{aligned}
\]
The estimates on the last two terms are not difficult to obtain. However, the first two require care. Therefore, from \eqref{ErrorA201}, \eqref{CotaA2_2}, and $\ell=1,2,3,\ldots$
\begin{equation}\label{Sol_2_final_2}
\begin{aligned}
& \left| \partial_z^\ell \begin{pmatrix} A_{2} \\ B_{2} \end{pmatrix} \right| (z; \omega(t),\rho(t))\\
 &~{} \quad \lesssim  \varepsilon^{\ell-1} e^{-k_0 \varepsilon |t|}e^{-l_0 \varepsilon |x|} +e^{-k_0 \varepsilon |t| -l_0 \varepsilon |\rho(t)|} e^{- \frac12\mu_0  |z|} \\
 &~{} \quad \quad + e^{-k_0 \varepsilon |t| -l_0 \varepsilon |\rho(t)|} e^{- \frac12\mu_0  |z|} + \varepsilon^{\ell +1} e^{-k_0 \varepsilon |t|}e^{-l_0 \varepsilon |x|}  
 +\frac1{\varepsilon} e^{-k_0 \varepsilon |t|} e^{- \mu_0  |z|} \\
  &~{} \quad \lesssim  \varepsilon^{\ell-1} e^{-k_0 \varepsilon |t|}e^{-l_0 \varepsilon |x|}  +\frac1{\varepsilon} e^{-k_0 \varepsilon |t|} e^{- \frac12\mu_0  |z|} .
\end{aligned}
\end{equation}
Here the only complicated terms are coming from the ones with the factor $e^{-l_0 \varepsilon |x|}$, which is a reminiscent of the existence of the non-flat bottom in the space variable.
%
From \eqref{Sol_2_final} we compute the following norms:
\begin{equation}\label{A2B2 Linfty}
\begin{aligned}
 \left\| \begin{pmatrix} A_{2} \\ B_{2} \end{pmatrix} \right\|_{L^\infty_x\times L^\infty_x }
&~{}  \lesssim  \frac1{\varepsilon} e^{-k_0 \varepsilon |t|} \left\|  e^{-l_0 \varepsilon (z+\rho(t))_+ }  \right\|_{L^\infty_x} + e^{-k_0 \varepsilon |t| -l_0 \varepsilon |\rho(t)|} \left\|  e^{- \frac12\mu_0 z_+}  \right\|_{L^\infty_x}  \\
 &~{} \quad \quad + \varepsilon e^{-k_0 \varepsilon |t|} \left\|  e^{-l_0 \varepsilon |x|} \right\|_{L^\infty_x}   +\frac1{\varepsilon} e^{-k_0 \varepsilon |t|}\left\|   e^{- \mu_0  |z|}  \right\|_{L^\infty_x}
\\
&~{}  \lesssim   \frac1{\varepsilon} e^{-k_0 \varepsilon |t|}.
\end{aligned}
\end{equation}
Also, using \eqref{Sol_2_final_2}, we have for $\ell=1,2,3,\ldots$
\begin{equation}\label{A2B2 DLinfty}
\begin{aligned}
 \left\|\partial_z^\ell  \begin{pmatrix} A_{2} \\ B_{2} \end{pmatrix} \right\|_{L^\infty_x \times L^\infty_x } \lesssim  &~{}  \varepsilon^{\ell-1} e^{-k_0 \varepsilon |t|}  \left\| e^{-l_0 \varepsilon |x|}  \right\|_{L^\infty_x}  +\frac1{\varepsilon} e^{-k_0 \varepsilon |t|}  \left\|e^{- \frac12\mu_0  |z|}  \right\|_{L^\infty_x}
\\
 \lesssim &~{}  \frac1{\varepsilon} e^{-k_0 \varepsilon |t|}.
\end{aligned}
\end{equation}

\subsection{Correction term}
Consider an even function $\chi \in C^\infty_0(\mathbb R)$ such that $0 \le \chi \le 1$, $\chi'\leq 0 $ for $s\geq 0$, and
\[
\chi(s)=
\begin{cases}
	1 & \text{if } |s| \leq 1 \\
	0 & \text{if } |s| \ge 2.
\end{cases}	
\]
Finally, for $\varepsilon>0$ consider
\begin{equation}\label{chi_ve}
\chi_\varepsilon(z):= \chi\left(\varepsilon z \right),
\end{equation}
so that $\chi_\varepsilon(z) =1$ if $|z| \leq \varepsilon^{-1}$ and $\chi_\varepsilon(z)=0$ if $|z|> 2\varepsilon^{-1}$. We define $\bd{W}^\sharp$ as the following modification function of $\bd{W}$ defined in \eqref{def_W0}:
\begin{equation}\label{def_W}
\bd{W}^\sharp (t,x) =\bd{W}^\sharp (t, \omega(t),z) = \varepsilon \begin{pmatrix} A_1 \\ B_1 \end{pmatrix} (t, \omega(t),z)+ \varepsilon^2 \begin{pmatrix} A_2 \\  B_2\end{pmatrix}(t,\omega(t),z) \chi_\varepsilon (z),
\end{equation}
with $A_1,B_1,A_2,B_2$ found in \eqref{Sol_1} and \eqref{Sol_2}. Notice that the dependence on $ (t, \omega(t),z) $ means that we separate dependences on $\omega$ and $\rho$ (through $z$) and explicit dependences on $t$. Quickly, using \eqref{Sol_2_final}, one has
\begin{equation}\label{A2B2 L2}
\begin{aligned}
& \left\|  \chi_\varepsilon \begin{pmatrix} A_{2} \\ B_{2} \end{pmatrix} \right\|_{L^2_x\times L^2_x}
\\
&~{} \lesssim   \frac1{\varepsilon} e^{-k_0 \varepsilon |t|} \left\| \chi_\varepsilon  e^{-l_0 \varepsilon (z+\rho(t))_+ }  \right\|_{L^2_x} + e^{-k_0 \varepsilon |t| -l_0 \varepsilon |\rho(t)|} \left\|  \chi_\varepsilon e^{- \frac12\mu_0 z_+}  \right\|_{L^2_x}  \\
 &~{} \quad \quad + \varepsilon e^{-k_0 \varepsilon |t|} \left\| \chi_\varepsilon  e^{-l_0 \varepsilon |x|} \right\|_{L^2_x}   +\frac1{\varepsilon} e^{-k_0 \varepsilon |t|}\left\|   \chi_\varepsilon e^{- \mu_0  |z|}  \right\|_{L^2_x}
\\
&~{} \lesssim  \frac{e^{-k_0 \varepsilon |t|} }{\varepsilon^{\frac32+\frac{\delta_0}2}} .
\end{aligned}
\end{equation}
This estimate reveals how dangerous the terms arising from the pure interaction with the variable bottom are, without any genuine interaction with the solitary wave. Similarly, for $\ell=1,2,3,\ldots$ one gets from \eqref{Sol_2_final_2},
\begin{equation}\label{A2B2 DL2}
\begin{aligned}
& \left\|  \chi_\varepsilon \partial_z^\ell \begin{pmatrix} A_{2} \\ B_{2} \end{pmatrix} \right\|_{L^2_x\times L^2_x} \\
& \quad \lesssim \varepsilon^{\ell-1} e^{-k_0 \varepsilon |t|}  \left\|  \chi_\varepsilon e^{-l_0 \varepsilon |x|}  \right\|_{L^2_x}  +\frac1{\varepsilon} e^{-k_0 \varepsilon |t|}  \left\|  \chi_\varepsilon e^{- \frac12\mu_0  |z|}  \right\|_{L^2_x} \lesssim \frac1{\varepsilon} e^{-k_0 \varepsilon |t|} .
\end{aligned}
\end{equation}
We conclude  from \eqref{def_W}, \eqref{est_A1B1 L2}, \eqref{A2B2 Linfty}-\eqref{A2B2 DLinfty} and \eqref{A2B2 L2}-\eqref{A2B2 DL2}  that
\begin{equation}\label{estW}
\begin{aligned}
& \| \bd{W}^\sharp (t) \|_{L^\infty_x\times L^\infty_x} \lesssim \varepsilon e^{-k_0 \varepsilon |t| -l_0 \varepsilon |\rho(t)|} +  \varepsilon e^{-k_0 \varepsilon |t|} \lesssim  \varepsilon e^{-k_0 \varepsilon |t|} ,\\
& \| \bd{W}^\sharp (t) \|_{L^2_x\times L^2_x} \lesssim \varepsilon e^{-k_0 \varepsilon |t| -l_0 \varepsilon |\rho(t)|} +  \varepsilon^{\frac12} e^{-k_0 \varepsilon |t|} \lesssim  \varepsilon^{\frac12} e^{-k_0 \varepsilon |t|} ,\\
& \| \partial_x^\ell \bd{W}^\sharp (t) \|_{L^2_x\times L^2_x} \lesssim \varepsilon e^{-k_0 \varepsilon |t| -l_0 \varepsilon |\rho(t)|} +  \varepsilon e^{-k_0 \varepsilon |t|} \lesssim  \varepsilon e^{-k_0 \varepsilon |t|},\quad \ell=1,2,3 .
\end{aligned}
\end{equation}
Similar estimates hold for all its spatial partial derivatives. Concerning the time derivative, we have
\begin{equation}\label{dtW}
\begin{aligned}
& \| \partial_t \bd{W}^\sharp (t) \|_{L^\infty_x\times L^\infty_x} \lesssim \varepsilon  e^{-k_0 \varepsilon |t|}, \\
& \| \partial_t \bd{W}^\sharp (t) \|_{L^2_x\times L^2_x} \lesssim  \varepsilon  e^{-k_0 \varepsilon |t|} .
\end{aligned}
\end{equation}
Let us prove this last fact. Following \eqref{def_W} and \eqref{dtW1}-\eqref{dtW2}
\be\label{dtW1s}
\begin{aligned}
& \partial_t W_1\\
& \quad = \varepsilon  \partial_t A_1 + \varepsilon \omega '  \Lambda A_1 - \varepsilon (\rho'-\omega) \partial_z A_1 - \varepsilon \omega \partial_z A_1 \\
& \quad \quad + \varepsilon^2 \chi_\varepsilon \partial_t  A_2 + \varepsilon^2 \omega ' \chi_\varepsilon \Lambda A_2 - \varepsilon^2 (\rho'-\omega)\partial_z( \chi_\varepsilon  A_2) - \varepsilon^2 \omega \partial_z(\chi_\varepsilon  A_2).
\end{aligned}
\ee
From \eqref{Sol_1}-\eqref{dtA1B1}, \eqref{DS} and the exponential decay of $A_1,B_1,$ we get
\[
\begin{aligned}
& \|\varepsilon  \partial_t A_1 + \varepsilon \omega '  \Lambda A_1 - \varepsilon \omega \partial_z A_1 \|_{L^\infty} \lesssim \varepsilon  e^{-k_0 \varepsilon |t| - \frac12l_0 \varepsilon |\rho(t)| },\\
& \|\varepsilon  \partial_t A_1 + \varepsilon \omega '  \Lambda A_1 - \varepsilon \omega \partial_z A_1 \|_{L^2} \lesssim \varepsilon  e^{-k_0 \varepsilon |t| - \frac12l_0 \varepsilon |\rho(t)| }.
\end{aligned}
\]
Similarly, using \eqref{Syst_3_A2B2}, \eqref{A2B2 Linfty}, \eqref{A2B2 DLinfty}, \eqref{A2B2 L2} and \eqref{A2B2 DL2},
\[
\begin{aligned}
 & \| \varepsilon^2 \chi_\varepsilon \partial_t  A_2 + \varepsilon^2 \omega ' \chi_\varepsilon \Lambda A_2 - \varepsilon^2 \omega \partial_z(\chi_\varepsilon  A_2)\|_{L^\infty} \lesssim  \varepsilon  e^{-k_0 \varepsilon |t|}, \\
 & \| \varepsilon^2 \chi_\varepsilon \partial_t  A_2 + \varepsilon^2 \omega ' \chi_\varepsilon \Lambda A_2  - \varepsilon^2 \omega \partial_z(\chi_\varepsilon  A_2)\|_{L^2} \lesssim  \varepsilon  e^{-k_0 \varepsilon |t|}.\\
 \end{aligned}
\]
Therefore, from \eqref{dtW1s} and the previous estimates we obtain the first half part of \eqref{dtW}. A completely similar output is obtained in the case of $\partial_t W_2$, using the fact that
\be\label{dtW2s}
\begin{aligned}
& \partial_t W_2\\
& \quad = \varepsilon \partial_t B_1 + \varepsilon \omega ' \Lambda B_1 - \varepsilon (\rho'-\omega) \partial_z B_1 - \varepsilon \omega  \partial_z B_1 \\
& \quad \quad + \varepsilon^2 \chi_\varepsilon   \partial_t B_2 + \varepsilon^2 \omega '  \chi_\varepsilon  \Lambda B_2 - \varepsilon^2 (\rho'-\omega) \partial_z(\chi_\varepsilon  B_2) - \varepsilon^2 \omega  \partial_z(\chi_\varepsilon  B_2).
\end{aligned}
\ee
Similar estimates performed on \eqref{dtW2s} complete the proof of \eqref{dtW}.

Also, we have in \eqref{S0_new_2},
\begin{equation}\label{S0_final}
{\bf S}_h(\bd{Q}_\omega + \bd{W}^\sharp) = {\bf S}_h^{\sharp}(\bd{Q}_\omega) +\partial_x J \mathcal L \bd{W}^\sharp +\bd{R}^\sharp,
\end{equation}
where
where \eqref{S_hQ_new_3} and the linear system \eqref{def_L_new_2} on $\bd{W}^\sharp$ become now
\[
\begin{aligned}
& {\bf S}_h^{\sharp}(\bd{Q}_\omega)+\partial_x J \mathcal L \bd{W}^\sharp \\
  &~{} = (\omega ' - \varepsilon^2 f_1(t)) (1-\partial^2_z)\begin{pmatrix}
\Lambda (R_\omega +\varepsilon A_1 + \varepsilon^2 A_2\chi_\varepsilon )\\
\Lambda (Q_\omega+\varepsilon B_1 + \varepsilon^2 B_2\chi_\varepsilon)
\end{pmatrix} \\
&~{} \quad
-(\rho'-\omega -  \varepsilon^2 f_2(t))(1-\partial^2_z) \left( \partial_z
\begin{pmatrix}
R_\omega  +\varepsilon A_1 + \varepsilon^2 A_2\chi_\varepsilon\\
 Q_\omega+\varepsilon B_1 + \varepsilon^2 B_2\chi_\varepsilon
 \end{pmatrix}   \right)
 \\
 &~{} \quad
-(\rho'-\omega -  \varepsilon^2 f_2(t))(1-\partial^2_z) \left( - \varepsilon  h_0(\varepsilon t, \varepsilon \rho(t)) \begin{pmatrix} A_1 \\ B_1 \end{pmatrix}  \right);
 \end{aligned}
\]
 and \eqref{R_new_2} becomes now
\begin{equation}\label{R_new_3}
\begin{aligned}
\bd{R}^\sharp =&~{}  \varepsilon^2  \left(  \partial_y h_0(\varepsilon t, \varepsilon \rho(t))  \partial_z \begin{pmatrix}  zQ_\omega \\ 0 \end{pmatrix}
+  \partial_s h_0 (\varepsilon t, \varepsilon x)  \begin{pmatrix}
1\\
 0
 \end{pmatrix} \right) (1-\chi_\varepsilon)
 \\
&~{}   + \varepsilon^2  \left(  f_1(t) (1-\partial^2_z)\begin{pmatrix}
\Lambda R_\omega \\
\Lambda Q_\omega
\end{pmatrix}  -  f_2(t)(1-\partial^2_z)  \partial_z
\begin{pmatrix} R_\omega  \\ Q_\omega  \end{pmatrix}  \right)  (1-\chi_\varepsilon)\\
 &~{}+ \varepsilon^2  h_0^2 (\varepsilon t, \varepsilon \rho(t)) \partial_z \begin{pmatrix}  - B_{0,\omega} + A_{0,\omega}B_{0,\omega} \\   \frac12 B_{0,\omega}^2 \end{pmatrix} (1-\chi_\varepsilon)\\
 &~{} + \varepsilon^2  h_0(\varepsilon t, \varepsilon \rho(t))  \left( \partial_s h_0(\varepsilon t, \varepsilon \rho(t)) + \omega (t) \partial_y h_0(\varepsilon t, \varepsilon \rho(t)) \right) (1- \partial_z^2) \begin{pmatrix} A_{0,\omega} \\ B_{0,\omega} \end{pmatrix} (1-\chi_\varepsilon)  \\
 &~{}  + \varepsilon^2 \left[ \partial_x J \mathcal L  \begin{pmatrix}
A_2 \\
B_2
\end{pmatrix}, \chi_\varepsilon  \right]\\
&~{} + \frac12   \varepsilon^3   \partial_y^2 h_0(\varepsilon t, \varepsilon \rho(t))  \partial_z \begin{pmatrix}  z^2 Q_\omega  \\ 0 \end{pmatrix}
\\
 &~{} +\varepsilon^3 \partial_x\!\left(  \partial_y h_0(\varepsilon t, \varepsilon \xi_3(t,x))  \begin{pmatrix} z B_1 \\ 0 \end{pmatrix}  +  h_0 (\varepsilon t, \varepsilon x) \begin{pmatrix}   \chi_\varepsilon B_2 \\ 0 \end{pmatrix}  \right) \\
 &~{}  +  \varepsilon^3 \partial_z \begin{pmatrix} A_1 \chi_\varepsilon  B_2+A_2 \chi_\varepsilon  B_1  \\  \chi_\varepsilon  B_1 B_2 \end{pmatrix}
 \\
 &~{} + \varepsilon^3  f_1(t) (1-\partial^2_z)\begin{pmatrix}
\Lambda ( A_1 + \varepsilon A_2 \chi_\varepsilon )\\
\Lambda (B_1 + \varepsilon B_2 \chi_\varepsilon )
\end{pmatrix}
 \\
 &~{} -  \varepsilon^3  f_2(t)(1-\partial^2_z) \left( \partial_z
\begin{pmatrix}
 A_1 + \varepsilon A_2 \chi_\varepsilon \\
 B_1 + \varepsilon B_2 \chi_\varepsilon
 \end{pmatrix}  - \varepsilon  \partial_y h_0(\varepsilon t, \varepsilon \rho(t)) \begin{pmatrix} A_{0,\omega} \\ B_{0,\omega} \end{pmatrix}  \right) \\
&~{} +  \varepsilon^4 \Bigg( -   \begin{pmatrix}
  a_1\partial^2_y\partial_s h_0 \\
 c_1 \partial^2_s\partial_y h_0
 \end{pmatrix} (\varepsilon t,\varepsilon x)  +\frac16   \partial_y^3 h_0(\varepsilon t, \varepsilon \rho(t)) \partial_z \begin{pmatrix}  z^3 Q_\omega \\ 0 \end{pmatrix} \\
 &~{} \qquad \quad + \partial_z \begin{pmatrix}  \chi_\varepsilon^2  A_2B_2 \\ \frac12 \chi_\varepsilon^2 B_2^2 \end{pmatrix} \Bigg) +  \varepsilon^5 \partial_x \begin{pmatrix}  \tilde h_0(\varepsilon t, \varepsilon \xi_1(t,x)) z^4 Q_\omega   \\ 0 \end{pmatrix}.
 \end{aligned}
\end{equation}
The long term \eqref{R_new_3} contains all the previous error terms plus the new ones appearing from the broken symmetries appearing when introducing the cut-off function $\chi_\varepsilon.$

\subsection{Dynamical system} Let $\omega_0>0$ be a fixed parameter. In what follows we shall assume the validity of the dynamical system
\begin{equation}\label{DS}
\begin{aligned}
& \omega' = \varepsilon^2 f_1(t), \quad \rho'-\omega =  \varepsilon^2 f_2(t),
\\
& (\omega,\rho)(-T_\varepsilon)=(\omega_0, -\omega_0 T_\varepsilon).
\end{aligned}
\end{equation}
Under this choice we obtain in \eqref{S0_final},
\begin{equation}\label{S0_final_final}
{\bf S}_h(\bd{Q}_\omega + \bd{W}^\sharp) = \bd{R}^\sharp.
\end{equation}
Additionally, we have
\begin{lemma}\label{DS_lem}
Let $(\omega,\rho)$ be the local solution to \eqref{DS}. Then $(\omega,\rho)$ is globally defined for all $t\geq -T_\varepsilon$ and one has
\begin{equation}\label{final_wp}
\lim_{t\to+\infty} \omega (t) = \omega_+ >0, \quad \lim_{t\to+\infty} \rho (t) = + \infty .
\end{equation}
Moreover, we have for some fixed $C>0$
\begin{equation}\label{final_wp_2}
|\omega_+ -\omega_0| \leq C \varepsilon. 
\end{equation}
\end{lemma}

\begin{proof}
Recall \eqref{cota_f1} and \eqref{cota_f2}. Let $(\omega,\rho)$ be the local solution to \eqref{DS}. Let $-T_\varepsilon <t<T_*$, where $T_* \leq +\infty$ is the maximal time of existence of the solution. Notice that
\[
|\omega(t) -\omega(-T_\varepsilon)| \le \varepsilon^2 \int_{-T_\varepsilon}^t  |f_1(s)|ds \lesssim \varepsilon.
\]
Therefore, $\omega(t)$ is globally defined. It also proves \eqref{final_wp_2} provided $\omega_+$ exists. A similar argument works for $\rho(t)$:
\[
|\rho(t)  -\rho(-T_\varepsilon) -\omega (t + T_\varepsilon)| \le \varepsilon^2 \int_{-T_\varepsilon}^t  |f_2(s)|ds \lesssim \varepsilon.
\]
This proves the second property in \eqref{final_wp}. 
The limit of $\omega(t)$ when $t$ diverges to $+\infty$ is 
$$\omega_+= \varepsilon^2 \int_{-T_\varepsilon}^{+\infty} f_1(s)ds
+\omega(-T_\varepsilon).$$
\noindent This integral converges thanks to \eqref{cota_f1}. 
The fact that $\omega_+>0$ is a consequence of \eqref{final_wp_2}
if $\varepsilon$ is small enough.
\end{proof}

\subsection{Error estimates} Now we estimate the term $\bd{R}^\sharp $ in \eqref{R_new_3}.

\begin{lemma}
We have the estimate
\begin{equation}\label{Cota_R}
\| \bd{R}^\sharp \|_{H^2\times H^2} \lesssim \varepsilon^{\frac32} e^{-k_0 \varepsilon |t|} +\varepsilon^{10}.
\end{equation}
\end{lemma}

\begin{proof}
We decompose $\bd{R}^\sharp $ in \eqref{R_new_3} as follows:
\[
\bd{R}^\sharp  = \sum_{j=1}^{12}\bd{R}^\sharp_j,
\]
where each $j$ represents a line in \eqref{R_new_3}. Using \eqref{hypoH},
\[
\begin{aligned}
 |\bd{R}^\sharp_1| = &~{} \left|
\varepsilon^2  \left(  \partial_y h_0(\varepsilon t, \varepsilon \rho(t))  \partial_z \begin{pmatrix}  zQ_\omega \\ 0 \end{pmatrix}
+  \partial_s h_0 (\varepsilon t, \varepsilon x)  \begin{pmatrix}
1\\
 0
 \end{pmatrix} \right) (1-\chi_\varepsilon) \right|
\\
 \lesssim & ~{} \varepsilon^2 e^{-k_0 \varepsilon |t| -l_0 \varepsilon |\rho(t)|} |zQ_\omega(z)| + \varepsilon^2 e^{-k_0 \varepsilon |t| -l_0 \varepsilon |x|}(1-\chi_\varepsilon) .
\end{aligned}
\]
Therefore, from \eqref{chi_ve},
\begin{equation}\label{R_new_3_1}
\begin{aligned}
 \|\bd{R}^\sharp_1\|_{L^2\times L^2} \lesssim  &~{}  \varepsilon^2 e^{-k_0 \varepsilon |t| -l_0 \varepsilon |\rho(t)|}  + \varepsilon^2 e^{-k_0 \varepsilon |t|} \| e^{ -l_0 \varepsilon |x|}(1-\chi_\varepsilon) \|_{L^2}
\\
 \lesssim & ~{}  \varepsilon^2 e^{-k_0 \varepsilon |t| -l_0 \varepsilon |\rho(t)|}  + \varepsilon^{\frac32} e^{-k_0 \varepsilon |t|}.
\end{aligned}
\end{equation}
The $H^2\times H^2$ is computed in similar terms, giving better or equal results. Now, using the exponential decay of $\Lambda R_\omega$ and $\Lambda Q_\omega$ \eqref{decay_Q}, and \eqref{cota_f1}-\eqref{cota_f2},
\begin{equation}\label{R_new_3_2}
\begin{aligned}
 \| \bd{R}^\sharp_2\|_{H^2\times H^2} =&~{} \left\|
 \varepsilon^2  \left(  f_1(t) (1-\partial^2_z)\begin{pmatrix}
\Lambda R_\omega \\
\Lambda Q_\omega
\end{pmatrix}  -  f_2(t)(1-\partial^2_z)  \partial_z
\begin{pmatrix} R_\omega  \\ Q_\omega  \end{pmatrix}  \right)  (1-\chi_\varepsilon)
\right\|_{H^2}
\\
 \lesssim & ~{} \varepsilon^2  e^{-k_0 \varepsilon |t| -l_0 \varepsilon |\rho(t)|} e^{-\frac12\mu_0 \varepsilon^{-1}} \ll \varepsilon^{10}.
\end{aligned}
\end{equation}
Similarly,
\begin{equation}\label{R_new_3_3}
\begin{aligned}
\| \bd{R}^\sharp_3 \|_{H^2\times H^2} =&~{} \left\|
 \varepsilon^2  h_0^2 (\varepsilon t, \varepsilon \rho(t)) \partial_z \begin{pmatrix}  - B_{0,\omega} + A_{0,\omega}B_{0,\omega} \\   \frac12 B_{0,\omega}^2 \end{pmatrix} (1-\chi_\varepsilon)
 \right\|_{H^2}
\\
\lesssim & ~{} \varepsilon^2  e^{-k_0 \varepsilon |t| -l_0 \varepsilon |\rho(t)|} e^{-\frac 12\mu_0 \varepsilon^{-1}} \ll \varepsilon^{10}.
\end{aligned}
\end{equation}
Since from \eqref{H2} we have $A_{0,\omega}, B_{0,\omega} \in H^\infty(\mathbb R)$,
\begin{equation}\label{R_new_3_4}
\begin{aligned}
 \| \bd{R}^\sharp_4\|_{H^2\times H^2}
 = &~{} \varepsilon^2 |  h_0(\varepsilon t, \varepsilon \rho(t))  \left( \partial_s h_0(\varepsilon t, \varepsilon \rho(t)) + \omega (t) \partial_y h_0(\varepsilon t, \varepsilon \rho(t)) \right) |\\
&~{} \quad \times   \left\|  (1- \partial_z^2) \begin{pmatrix} A_{0,\omega} \\ B_{0,\omega} \end{pmatrix} (1-\chi_\varepsilon)\right\|_{H^2}
\\
 \lesssim & ~{}  \varepsilon^2  e^{-2k_0 \varepsilon |t| -2l_0 \varepsilon |\rho(t)|} e^{-\frac12\mu_0 \varepsilon^{-1}} \ll \varepsilon^{10}.
\end{aligned}
\end{equation}
From \eqref{Syst_3_A2B2}, 
\begin{equation}\label{Syst_3_A2B2_new}
\begin{aligned}
& \partial_x \left( J \mathcal L \left( \begin{pmatrix} A_2 \\ B_2 \end{pmatrix}  \chi_\varepsilon \right)\right) -  \chi_\varepsilon \partial_x J \mathcal L \begin{pmatrix} A_2 \\ B_2 \end{pmatrix}
 = \left[ J\mathcal L  \begin{pmatrix}  A_2 \\ B_2 \end{pmatrix} \right] \partial_x \chi_\varepsilon
\\
&~{} =    - \partial_y h_0(\varepsilon t, \varepsilon \rho(t))  \partial_x \chi_\varepsilon \begin{pmatrix}   zQ_\omega \\ 0 \end{pmatrix} -  \partial_x \chi_\varepsilon \partial_z^{-1} \partial_s h_0 (\varepsilon t, \varepsilon x )  \begin{pmatrix}
1 \\ 0 \end{pmatrix}\\
 &~{}\quad  -\left( h_0(\varepsilon t, \varepsilon \rho(t)) \partial_s h_0(\varepsilon t, \varepsilon \rho(t)) + \omega (t) h_0^2 (\varepsilon t, \varepsilon \rho(t)) \right)\partial_x \chi_\varepsilon (1- \partial_z^2) \partial_z^{-1} \begin{pmatrix}  A_{0,\omega} \\B_{0,\omega}   \end{pmatrix}  \\
 &~{}  \quad   -  f_1(t) \partial_x \chi_\varepsilon(1-\partial^2_z) \partial_z^{-1} \begin{pmatrix}
\Lambda R_\omega \\
\Lambda Q_\omega
\end{pmatrix} +   f_2(t) \partial_x \chi_\varepsilon (1-\partial^2_z)
\begin{pmatrix} R_\omega  \\ Q_\omega  \end{pmatrix}\\
 &~{} \quad -h_0^2 (\varepsilon t, \varepsilon \rho(t))   \begin{pmatrix}   - B_{0,\omega} + A_{0,\omega}B_{0,\omega} \\  \frac12 B_{0,\omega}^2  \end{pmatrix}.
 \\
\end{aligned}
\end{equation}
Now we bound the terms in \eqref{Syst_3_A2B2_new} as follows: using \eqref{mejorada0}, \eqref{mejorada} and \eqref{ErrorA20}, together with \eqref{cota_f1} and \eqref{cota_f2},
\[
\begin{aligned}
&\left| \partial_x \left( J \mathcal L \left( \begin{pmatrix} A_2 \\ B_2 \end{pmatrix}  \chi_\varepsilon \right)\right) -  \chi_\varepsilon \partial_x J \mathcal L \begin{pmatrix} A_2 \\ B_2 \end{pmatrix}  \right|
\\
&~{} \lesssim   e^{-k_0\varepsilon  |t|} e^{-l_0 \varepsilon |x|}e^{-\frac12\mu_0 |z|}  +  |\chi_0'|(\varepsilon x)  e^{-k_0 \varepsilon |t|}e^{-l_0 \varepsilon (z+\rho(t))_+ } \\
 &~{}\quad  + \varepsilon   |\chi_0'|(\varepsilon x)    e^{-k_0 \varepsilon |t| -l_0 \varepsilon |\rho(t)|} e^{- \frac12\mu_0 z_+}  \\
&~{} \quad   +\left| \partial_x \left( J\left(2\partial_x \chi_\varepsilon  \begin{pmatrix} c\partial_z  & \omega \partial_z  \\ \omega \partial_z  & a \partial_z  \end{pmatrix}   \begin{pmatrix} A_2 \\ B_2 \end{pmatrix} +  \partial_x^2 \chi_\varepsilon \begin{pmatrix} cA_2 +\omega B_2 \\ \omega A_2 +a B_2 \end{pmatrix}  \right) \right) \right|.
\end{aligned}
\]
Therefore,
\begin{equation}\label{Syst_3_A2B2_new2}
\begin{aligned}
&\left\| \partial_x \left( J \mathcal L \left( \begin{pmatrix} A_2 \\ B_2 \end{pmatrix}  \chi_\varepsilon \right)\right) -  \chi_\varepsilon \partial_x J \mathcal L \begin{pmatrix} A_2 \\ B_2 \end{pmatrix}  \right\|_{H^2}
\\
&~{} \lesssim   e^{-k_0\varepsilon  |t|}  +\varepsilon^{-\frac12}  e^{-k_0 \varepsilon |t|}   + \varepsilon^{\frac12}  e^{-k_0 \varepsilon |t| -l_0 \varepsilon |\rho(t)|}.  
\end{aligned}
\end{equation}
Using \eqref{Syst_3_A2B2_new2} and the previous estimate,
\begin{equation}\label{R_new_3_5}
\begin{aligned}
 \left\| \bd{R}^\sharp_5 \right\|_{H^2\times H^2}  \lesssim &~{} \left\|
\varepsilon^2 \left[ \partial_x J \mathcal L  \begin{pmatrix}
A_2 \\
B_2
\end{pmatrix}, \chi_\varepsilon  \right]
\right\|_{H^2}  \lesssim  \varepsilon^\frac32 e^{-k_0 \varepsilon |t|}.
\end{aligned}
\end{equation}
Now, using \eqref{hypoH} and the exponential decay of $Q_\omega$,
\begin{equation}\label{R_new_3_6}
\begin{aligned}
 \| \bd{R}^\sharp_6 \|_{H^2\times H^2} =&~{} \left\|
\frac12   \varepsilon^3   \partial_y^2 h_0(\varepsilon t, \varepsilon \rho(t))  \partial_z \begin{pmatrix}  z^2 Q_\omega  \\ 0 \end{pmatrix}  \right\|_{H^2}  \lesssim  \varepsilon^3 e^{-k_0 \varepsilon |t| -l_0 \varepsilon |\rho(t)|}.
\end{aligned}
\end{equation}
We use now \eqref{hypoH} and \eqref{Sol_1} to get
\[
\begin{aligned}
 \| \bd{R}^\sharp_7\|_{H^2\times H^2} =&~{} \left\|
 \varepsilon^3 \partial_x\!\left(  \partial_y h_0(\varepsilon t, \varepsilon \xi_3(t,x))  \begin{pmatrix} z B_1 \\ 0 \end{pmatrix}  +  h_0 (\varepsilon t, \varepsilon \cdot) \begin{pmatrix}   \chi_\varepsilon B_2 \\ 0 \end{pmatrix}  \right)
 \right\|_{H^2}
\\
 \lesssim & ~{} \varepsilon^3 e^{-2k_0 \varepsilon |t|}e^{-l_0 \varepsilon |\rho(t)|} +  \varepsilon^3 \left\|
 \partial_x\!\left(  h_0 (\varepsilon t, \varepsilon \cdot)   \chi_\varepsilon B_2   \right)
 \right\|_{H^2} .
\end{aligned}
\]
Now, notice that from \eqref{Sol_2_final}-\eqref{Sol_2_final_2}, \eqref{A2B2 Linfty}-\eqref{A2B2 DLinfty}, and \eqref{A2B2 L2}-\eqref{A2B2 DL2},
\[
\begin{aligned}
 \left\|  \partial_x\!\left(  h_0 (\varepsilon t, \varepsilon \cdot)   \chi_\varepsilon B_2   \right)
 \right\|_{L^2}
 \lesssim &~{}  \varepsilon \left\|   \partial_y h_0 (\varepsilon t, \varepsilon \cdot)   \chi_\varepsilon B_2\right\|_{L^2}  + \left\|   h_0 (\varepsilon t, \varepsilon \cdot)   \partial_x \chi_\varepsilon B_2  \right\|_{L^2}
 \\
 &~{}  + \left\|  h_0 (\varepsilon t, \varepsilon \cdot)   \chi_\varepsilon \partial_x B_2  \right\|_{L^2} \\
 \lesssim &~{}  \varepsilon^{-3/2} e^{-k_0 \varepsilon |t|} +\varepsilon^{-1} e^{-k_0 \varepsilon |t|} \lesssim  \varepsilon^{-3/2} e^{-k_0 \varepsilon |t|}.
\end{aligned}
\]
The remaining two derivatives are handled in the same fashion, noting that pointwise bounds for derivatives in 
$L^2$ are, at least, better suited than those for the function itself. Therefore,
\begin{equation}\label{R_new_3_7}
\begin{aligned}
 \| \bd{R}^\sharp_7\|_{H^2\times H^2}
 \lesssim & ~{} \varepsilon^3 e^{-2k_0 \varepsilon |t|}e^{-l_0 \varepsilon |\rho(t)|} +  \varepsilon^{\frac32} e^{-k_0 \varepsilon |t|} \lesssim \varepsilon^{\frac32} e^{-k_0 \varepsilon |t|} .
\end{aligned}
\end{equation}
Now we use the exponential decay of $A_1,B_1$,  \eqref{Sol_1} and the polynomial growth of $A_2,B_2$ \eqref{A2B2 Linfty}-\eqref{A2B2 DLinfty} to get
 \begin{equation}\label{R_new_3_8}
\begin{aligned}
 \| \bd{R}^\sharp_8\|_{H^2\times H^2} =&~{} \left\|
\varepsilon^3 \partial_z \begin{pmatrix} A_1 \chi_\varepsilon  B_2+A_2 \chi_\varepsilon  B_1  \\  \chi_\varepsilon  B_1 B_2 \end{pmatrix}
\right\|_{H^2}
\\
\lesssim  &~{} \varepsilon^3 (\left\| \partial_z \chi_\varepsilon  A_1 B_2 \right\|_{H^2}  + \left\| \chi_\varepsilon \partial_z A_1 B_2 \right\|_{H^2} +\left\|  \chi_\varepsilon  A_1 \partial_z B_2 \right\|_{H^2} ) \\
&~{} + \varepsilon^3 (\left\| \partial_z \chi_\varepsilon  A_2 B_1 \right\|_{H^2}  + \left\| \chi_\varepsilon \partial_z A_2 B_1 \right\|_{H^2} +\left\|  \chi_\varepsilon  A_2 \partial_z B_1 \right\|_{H^2} )\\
&~{} + \varepsilon^3 (\left\| \partial_z \chi_\varepsilon  B_1 B_2 \right\|_{H^2}  + \left\| \chi_\varepsilon \partial_z B_1 B_2 \right\|_{H^2} +\left\|  \chi_\varepsilon  B_1 \partial_z B_2 \right\|_{H^2} )
\\
\lesssim  &~{}  \varepsilon^2 e^{-k_0 \varepsilon |t|}e^{-l_0 \varepsilon |\rho(t)|}  .
\end{aligned}
\end{equation}
Similarly, using additionally \eqref{cota_f1},
\[
\begin{aligned}
  \|  \bd{R}^\sharp_9\|_{H^2\times H^2} =&~{} \left\|
 \varepsilon^3  f_1(t) (1-\partial^2_z)\begin{pmatrix}
\Lambda ( A_1 + \varepsilon A_2 \chi_\varepsilon )\\
\Lambda (B_1 + \varepsilon B_2 \chi_\varepsilon )
\end{pmatrix}
\right\|_{H^2}
\\
 \lesssim & ~{}  \varepsilon^3 e^{-2k_0 \varepsilon |t| -2l_0 \varepsilon |\rho(t)|}( \| (1-\partial^2_z)\Lambda  A_{1,0}  \|_{H^2} +\| (1-\partial^2_z)\Lambda  B_{1,0}  \|_{H^2} )
 \\
 & ~{} + \varepsilon^4 e^{-k_0 \varepsilon |t| -l_0 \varepsilon |\rho(t)|}( \| (1-\partial^2_z)\Lambda  ( \chi_\varepsilon A_{2})  \|_{H^2} +\| (1-\partial^2_z) \Lambda  ( \chi_\varepsilon B_{2})  \|_{H^2} ).
\end{aligned}
\]
Since $A_{1,0}$ and $B_{1,0}$ are exponentially decreasing, we get
\begin{equation}\label{R_new_3_91}
\begin{aligned}
  \|  \bd{R}^\sharp_9\|_{H^2\times H^2} \lesssim & ~{}  \varepsilon^3 e^{-2k_0 \varepsilon |t| -2l_0 \varepsilon |\rho(t)|}
  \\
  &~{} + \varepsilon^4 e^{-k_0 \varepsilon |t| -l_0 \varepsilon |\rho(t)|} \\
  &~{} \quad \times ( \| (1-\partial^2_z)\Lambda  ( \chi_\varepsilon A_{2})  \|_{H^2} +\| (1-\partial^2_z) \Lambda  ( \chi_\varepsilon B_{2})  \|_{H^2} ).
\end{aligned}
\end{equation}
Since the dependence on $\omega$ in $A_2$ and $B_2$ is present through dependence on $Q_\omega$ and $R_\omega$, we readily have $|\Lambda  ( \chi_\varepsilon A_{2})| = |\chi_\varepsilon \Lambda   A_{2}| \lesssim  \chi_\varepsilon |A_2| $, and similar for $B_2$. Then from \eqref{A2B2 L2} and \eqref{A2B2 DL2},
\[
\begin{aligned}
&  \| (1-\partial^2_z)\Lambda  ( \chi_\varepsilon A_{2})  \|_{H^2} +\| (1-\partial^2_z) \Lambda  ( \chi_\varepsilon B_{2})  \|_{H^2} \\
& \quad \lesssim  \| (1-\partial^2_z) \chi_\varepsilon A_{2}  \|_{H^2} +\| (1-\partial^2_z)\chi_\varepsilon B_{2}   \|_{H^2} \lesssim  \frac{e^{-k_0 \varepsilon |t|} }{\varepsilon^{\frac32}} .
\end{aligned}
\]
We conclude from \eqref{R_new_3_91} that
\begin{equation}\label{R_new_3_9}
\begin{aligned}
  \|  \bd{R}^\sharp_9\|_{H^2\times H^2} \lesssim & ~{}  \varepsilon^3 e^{-2k_0 \varepsilon |t| -2l_0 \varepsilon |\rho(t)|}   + \varepsilon^{\frac52} e^{-2k_0 \varepsilon |t| -l_0 \varepsilon |\rho(t)|}
  \\
  \lesssim &~{}  \varepsilon^{\frac52} e^{-2k_0 \varepsilon |t| -l_0 \varepsilon |\rho(t)|}.
\end{aligned}
\end{equation}
Now we use \eqref{cota_f2}, \eqref{Sol_1} and \eqref{hypoH} to obtain
\begin{equation}\label{R_new_3_100}
\begin{aligned}
 \| \bd{R}^\sharp_{10}\|_{H^2\times H^2} =&~{} \left\|
 -  \varepsilon^3  f_2(t) (1-\partial^2_z) \left( \partial_z
\begin{pmatrix}
 A_1 + \varepsilon A_2 \chi_\varepsilon \\
 B_1 + \varepsilon B_2 \chi_\varepsilon
 \end{pmatrix}  - \varepsilon \partial_y h_0(\varepsilon t, \varepsilon \rho(t)) \begin{pmatrix} A_{0,\omega} \\ B_{0,\omega} \end{pmatrix}  \right)
 \right\|_{H^2\times H^2}
\\
 \lesssim & ~{} \varepsilon^3 e^{-k_0 \varepsilon |t| -l_0 \varepsilon |\rho(t)|} \left\| (1-\partial^2_z)  \partial_z
\begin{pmatrix}
 A_1 + \varepsilon A_2 \chi_\varepsilon \\
 B_1 + \varepsilon B_2 \chi_\varepsilon
 \end{pmatrix}  \right\|_{H^2\times H^2} \\
 &~{} +  \varepsilon^3  e^{-2k_0 \varepsilon |t| -2l_0 \varepsilon |\rho(t)|}
 \\
 \lesssim & ~{} \varepsilon^3 e^{-2k_0 \varepsilon |t| -2l_0 \varepsilon |\rho(t)|} +\varepsilon^4 e^{-k_0 \varepsilon |t| -l_0 \varepsilon |\rho(t)|} \left\| (1-\partial^2_z)  \partial_z
\begin{pmatrix}
  A_2 \chi_\varepsilon \\
  B_2 \chi_\varepsilon
 \end{pmatrix}  \right\|_{H^2 \times H^2}  \\
 &~{} +  \varepsilon^3  e^{-2k_0 \varepsilon |t| -2l_0 \varepsilon |\rho(t)|} .
\end{aligned}
\end{equation}
We estimate the second term in the last line above. We compute the $L^2$ norm, knowing that the remaining two derivatives have at least better properties. We have
\[
\begin{aligned}
& \left\| (1-\partial^2_z)  \partial_z
\begin{pmatrix}
  A_2 \chi_\varepsilon \\
  B_2 \chi_\varepsilon
 \end{pmatrix}  \right\|_{L^2\times L^2}
\\
&\quad  \lesssim \left\| \partial_z(A_2 \chi_\varepsilon)  \right\|_{L^2} + \left\| \partial_z(  B_2 \chi_\varepsilon ) \right\|_{L^2}  + \left\|  \partial_z^3 ( A_2 \chi_\varepsilon)  \right\|_{L^2} + \left\| \partial_z^3(  B_2 \chi_\varepsilon ) \right\|_{L^2}\\
&\quad  \lesssim \left\| A_2 \partial_z\chi_\varepsilon  \right\|_{L^2} +\left\| \partial_zA_2 \chi_\varepsilon \right\|_{L^2}+ \left\| B_2 \partial_z\chi_\varepsilon  \right\|_{L^2} +\left\| \partial_zB_2 \chi_\varepsilon \right\|_{L^2}\\
&\quad  \quad + \left\|  \partial_z^3 A_2 \chi_\varepsilon \right\|_{L^2} + \left\|  \partial_z^2 A_2  \partial_z\chi_\varepsilon \right\|_{L^2} + \left\|  \partial_z A_2  \partial_z^2 \chi_\varepsilon \right\|_{L^2}  + \left\|   A_2  \partial_z^3 \chi_\varepsilon \right\|_{L^2}  \\
&\quad  \quad + \left\|  \partial_z^3 B_2 \chi_\varepsilon \right\|_{L^2} + \left\|  \partial_z^2 B_2  \partial_z\chi_\varepsilon \right\|_{L^2} + \left\|  \partial_z B_2  \partial_z^2 \chi_\varepsilon \right\|_{L^2}  + \left\|   B_2  \partial_z^3 \chi_\varepsilon \right\|_{L^2} .
\end{aligned}
\]
First, notice that  %
\[
\begin{aligned}
&\left\|  \partial_z^2 A_2  \partial_z\chi_\varepsilon \right\|_{L^2} + \left\|  \partial_z A_2  \partial_z^2 \chi_\varepsilon \right\|_{L^2}  + \left\|   A_2  \partial_z^3 \chi_\varepsilon \right\|_{L^2}
\\
& \quad + \left\|  \partial_z^2 B_2  \partial_z\chi_\varepsilon \right\|_{L^2} + \left\|  \partial_z B_2  \partial_z^2 \chi_\varepsilon \right\|_{L^2}  + \left\|   B_2  \partial_z^3 \chi_\varepsilon \right\|_{L^2} \lesssim  e^{-k_0 \varepsilon |t|}.
\end{aligned}
\]
Second, using that $|\partial_x^\ell \chi_\varepsilon| \lesssim \varepsilon^\ell 1_{\{|x|\leq \varepsilon^{-1-\delta_0} \} }$, from \eqref{A2B2 L2} and \eqref{A2B2 DL2} we have
\[
\begin{aligned}
& \left\| A_2 \partial_z\chi_\varepsilon  \right\|_{L^2} +\left\| \partial_zA_2 \chi_\varepsilon \right\|_{L^2}+ \left\| B_2 \partial_z\chi_\varepsilon  \right\|_{L^2} +\left\| \partial_zB_2 \chi_\varepsilon \right\|_{L^2}
\\
& \quad \lesssim \frac{1}{\varepsilon^{\frac12}} e^{-k_0 \varepsilon |t|}  +\frac1{\varepsilon} e^{-k_0 \varepsilon |t|}  \lesssim \frac1{\varepsilon} e^{-k_0 \varepsilon |t|},
\end{aligned}
\]
and
\[
\begin{aligned}
&\left\|  \partial_z^3 A_2 \chi_\varepsilon \right\|_{L^2} + \left\|  \partial_z^3 B_2 \chi_\varepsilon \right\|_{L^2}   \lesssim \frac1{\varepsilon} e^{-k_0 \varepsilon |t|}.
\end{aligned}
\]
We conclude that
\[
\left\| (1-\partial^2_z)  \partial_z
\begin{pmatrix}
  A_2 \chi_\varepsilon \\
  B_2 \chi_\varepsilon
 \end{pmatrix}  \right\|_{L^2\times L^2}  \lesssim \frac1{\varepsilon} e^{-k_0 \varepsilon |t|}.
\]
The remaining estimates for the two derivatives in $L^2$ are similar and have at least better behavior. Therefore, we get in \eqref{R_new_3_100}
\begin{equation}\label{R_new_3_10}
\begin{aligned}
 \| \bd{R}^\sharp_{10}\|_{H^2\times H^2}  \lesssim & ~{} \varepsilon^3 e^{-2k_0 \varepsilon |t| -2l_0 \varepsilon |\rho(t)|} +\varepsilon^3 e^{-2k_0 \varepsilon |t| -l_0 \varepsilon |\rho(t)|}  +  \varepsilon^3  e^{-2k_0 \varepsilon |t| -2l_0 \varepsilon |\rho(t)|}
 \\
  \lesssim & ~{} \varepsilon^3 e^{-2k_0 \varepsilon |t| -l_0 \varepsilon |\rho(t)|} .
\end{aligned}
\end{equation}
We use again \eqref{hypoH} and \eqref{decay_Q} to obtain
\[
\begin{aligned}
& \|\bd{R}^\sharp_{11}\|_{H^2\times H^2} \\
&~{} \leq \left\|
\varepsilon^4 \left( -   \begin{pmatrix}
  a_1\partial^2_y\partial_s h_0 \\
 c_1 \partial^2_s\partial_y h_0
 \end{pmatrix} (\varepsilon t,\varepsilon x)  +\frac16   \partial_y^3 h_0(\varepsilon t, \varepsilon \rho(t)) \partial_z \begin{pmatrix}  z^3 Q_\omega \\ 0 \end{pmatrix}  \right)
 \right\|_{H^2\times H^2} \\
&~{} \qquad + \left\|
\varepsilon^4  \partial_z \begin{pmatrix}  \chi_\varepsilon^2  A_2B_2 \\ \frac12 \chi_\varepsilon^2 B_2^2 \end{pmatrix}
 \right\|_{H^2\times H^2}
\\
 & ~{}  \lesssim \varepsilon^4 e^{-k_0 \varepsilon |t|} \| e^{-l_0 \varepsilon |x|} \|_{H^2} +\varepsilon^4 e^{-k_0 \varepsilon |t|}e^{-l_0 \varepsilon |\rho(t)|}  \| |z|^3 Q_\omega \|_{H^2}
 \\
 &~{} \qquad +\varepsilon^4 \left\|  \partial_z \begin{pmatrix}  \chi_\varepsilon^2  A_2B_2 \\ \frac12 \chi_\varepsilon^2 B_2^2 \end{pmatrix}  \right\|_{H^2\times H^2}
 \\
  & ~{}  \lesssim \varepsilon^{7/2} e^{-k_0 \varepsilon |t|}  +\varepsilon^4 e^{-k_0 \varepsilon |t|}e^{-l_0 \varepsilon |\rho(t)|} +\varepsilon^4 \left\|  \partial_z \begin{pmatrix}  \chi_\varepsilon^2  A_2B_2 \\ \frac12 \chi_\varepsilon^2 B_2^2 \end{pmatrix}  \right\|_{H^2\times H^2} .
\end{aligned}
\]
The last term above is bounded as follows:
\[
\begin{aligned}
& \left\|  \partial_z \begin{pmatrix}  \chi_\varepsilon^2  A_2B_2 \\ \frac12 \chi_\varepsilon^2 B_2^2 \end{pmatrix}  \right\|_{H^2\times H^2}
\\
&~{} \lesssim  \left\|  \partial_z (\chi_\varepsilon^2  A_2B_2)  \right\|_{H^2}  + \left\|  \partial_z (\chi_\varepsilon^2 B_2^2) \right\|_{H^2} \\
&~{} \lesssim    \left\| | \partial_z \chi_\varepsilon| \chi_\varepsilon  (A_2^2 +B_2^2) \right\|_{H^2} + \left\|  \chi_\varepsilon^2 ( |\partial_z  A_2 ||B_2| +|A_2 ||\partial_z   B_2|)     \right\|_{H^2}  + \left\| \chi_\varepsilon^2 B_2 \partial_z B_2  \right\|_{H^2}.
\end{aligned}
\]
Recall \eqref{A2B2 Linfty}. From this estimate, we get
\[
 \left\| | \partial_z \chi_\varepsilon| \chi_\varepsilon  (A_2^2 +B_2^2) \right\|_{L^2} \lesssim  \frac1{\varepsilon^2} e^{-2k_0 \varepsilon |t|} \left\| | \partial_z \chi_\varepsilon| \right\|_{L^2} \lesssim  \frac1{\varepsilon^{\frac32}} e^{-2k_0 \varepsilon |t|}.
\]
A similar estimate holds for higher derivatives. Now, using \eqref{A2B2 DLinfty} and \eqref{A2B2 Linfty},
\[
\begin{aligned}
& \left\|  \chi_\varepsilon^2 ( |\partial_z  A_2 ||B_2| +|A_2 ||\partial_z   B_2|)     \right\|_{L^2} + \left\| \chi_\varepsilon^2 B_2 \partial_z B_2  \right\|_{L^2}
\\
&\quad \lesssim  \frac1{\varepsilon^2} e^{-2k_0 \varepsilon |t|} \left\| \chi_\varepsilon \right\|_{L^2} \lesssim  \frac1{\varepsilon^{\frac52}} e^{-2k_0 \varepsilon |t|}.
\end{aligned}
\]
Finally,
\begin{equation}\label{R_new_3_11}
\begin{aligned}
& \|\bd{R}^\sharp_{11}\|_{H^2\times H^2} \\
& ~{} \lesssim \varepsilon^{\frac72} e^{-k_0 \varepsilon |t|}  +\varepsilon^4 e^{-k_0 \varepsilon |t|}e^{-l_0 \varepsilon |\rho(t)|} +\varepsilon^{\frac32} e^{-2k_0 \varepsilon |t|} \lesssim \varepsilon^{\frac32} e^{-k_0 \varepsilon |t|}.
\end{aligned}
\end{equation}
Now
\begin{equation}\label{R_new_3_12}
\begin{aligned}
 \|\bd{R}^\sharp_{12}\|_{H^2} =&~{} \left \|
\varepsilon^5 \partial_x \begin{pmatrix}  \tilde h_0(\varepsilon t, \varepsilon \xi_1(t,x)) z^4 Q_\omega   \\ 0 \end{pmatrix}
\right\|_{H^2\times H^2}  \lesssim \varepsilon^5 \| z^4 Q_\omega \|_{H^2} \lesssim \varepsilon^5.
 \end{aligned}
\end{equation}
Finally, gathering \eqref{R_new_3_1}, \eqref{R_new_3_2}, \eqref{R_new_3_3}, \eqref{R_new_3_4}, \eqref{R_new_3_5}, \eqref{R_new_3_6}, \eqref{R_new_3_7}, \eqref{R_new_3_8}, \eqref{R_new_3_9}, \eqref{R_new_3_10}, \eqref{R_new_3_11} and \eqref{R_new_3_12}, we get \eqref{Cota_R}. The proof is complete.
\end{proof}

In the next section, our main objective is to make the previous construction rigorous, in the sense that the actual solution constructed in the pre-interaction region will be matched with the function $\bd{Q}_\omega +\bd{W}^\sharp$  defined above.

\section{Stability estimates}\label{Sec:4}

\subsection{Preliminaries} Recall $ \bd{Q}_{\omega} (z)$ and $\bd{W}^\sharp$ defined in \eqref{sol_wave} and \eqref{def_W}. Let us define 
\begin{equation}\label{defU}
\bd{U}(t,x) := \begin{pmatrix} U_1\\ U_2\end{pmatrix} (t,x) := \bd{Q}_{\omega(t)} (z) +  \bd{W}^\sharp(t,\omega(t),z) , \quad z= x-\rho(t),
\end{equation}
such that \eqref{DS}, \eqref{S0_final_final} and \eqref{Cota_R} are satisfied. Then Lemma \ref{DS_lem} also holds. We have from \eqref{S0},
 for $\bd{\eta_2}=( \eta , u)^T$ that will be specified in the sequel,
 
\[
\begin{aligned}
& 0= {\bf S}_h(\bd{U} + \bd{\eta_2} ) \\
 &~{} = \begin{pmatrix}
(1- \partial_x^2)\partial_t (U_1+ \eta)  + \partial_x\!\left( a\, \partial_x^2 (U_2+ u_2) + (U_2+ u) + (U_1+ \eta+ h) (U_2+ u) \right)   \\
(1- \partial_x^2)\partial_t (U_2+ u)  + \partial_x\! \left( c\, \partial_x^2 (U_1+ \eta_2 ) + U_1+ \eta  + \frac12 (U_2 +u )^2 \right)
\end{pmatrix} \\
 &~{}  \quad + \begin{pmatrix}
 (1 - a_1 \partial_x^2) \partial_th  \\
- c_1  \partial_t^2 \partial_x h
\end{pmatrix} \\
&~{} = {\bf S}_h(\bd{U}  ) \\
&~{}
\quad + \begin{pmatrix}
(1- \partial_x^2)\partial_t  \eta  + \partial_x\!\left( a\, \partial_x^2  u + u +  (U_1+h) u +  U_2\eta   \right)  \\
(1- \partial_x^2)\partial_t u  + \partial_x\! \left( c\, \partial_x^2 \eta  + \eta  + U_2 u  \right)
\end{pmatrix}
+ \begin{pmatrix}
 \partial_x\!\left( \eta u  \right)  \\
 \partial_x\!\left( \frac12 u^2 \right)
\end{pmatrix}
\\
 & ~{} = \bd{R}^\dagger +  {\bf S}_h'(\bd{U} )\bd{\eta_2}  +  \bd{N}(\bd{U},\bd{\eta_2}).
\end{aligned}
\]
Here, following \eqref{def_L_new},
 \begin{equation}\label{def_L_U}
\begin{aligned}
{\bf S}_h'(\bd{U} )\bd{\eta_2}   = &~{}
 \begin{pmatrix}
(1- \partial_x^2)\partial_t  \eta  + \partial_x\!\left( a\, \partial_x^2  u + u +  (U_1+h) u +  U_2\eta   \right)  \\
(1- \partial_x^2)\partial_t u  + \partial_x\! \left( c\, \partial_x^2 \eta  + \eta  + U_2 u  \right) ,
\end{pmatrix}
\end{aligned}
\end{equation}
and
\begin{equation}\label{NN}
 \bd{N}(\bd{U},\bd{\eta_2}) =  \begin{pmatrix}
 \partial_x\!\left( \eta u  \right)  \\
 \partial_x\!\left( \frac12 u^2 \right)
\end{pmatrix}.
\end{equation}
From \eqref{S0_new_2}, \eqref{def_L_U} and \eqref{NN}, it is clear that $\bd{\eta}$ satisfies the equation
\begin{equation}\label{eq_eta}
{\bf S}_h'(\bd{U} )\bd{\eta_2} = - \bd{N}(\bd{U},\bd{\eta_2})  -\bd{R}^\dagger.
\end{equation}
Equation \eqref{eq_eta} reveals the main linear dependence for perturbations of the approximate solutions already constructed.

\subsection{Modulation}

Let $(\omega,\rho)=(\omega,\rho)(t)$ be given by Lemma \ref{DS_lem} and $\bd{U}$ defined in \eqref{defU}, depending on the variables $(t,\omega,x-\rho)$. Notice that $(\omega,\rho)$ are globally defined for $t\geq -T_\varepsilon$. Let $t\in [-T_\varepsilon, T^*]$, with $T^*$ the maximal time of existence of the solution $\bd{\eta}$ constructed in the pre-interaction regime. Recall that \eqref{PreInteraction} is satisfied at time $t=-T_\varepsilon.$ For $K_2>1$ to be fixed later, let us define
\begin{equation}\label{T2}
\begin{aligned}
  T_2(K_2) &~{} :=  \sup \Big\{ T \in \left(- T_\varepsilon, 3T^*/4 \right) ~ : ~ \hbox{for all } t\in [-T_\varepsilon, T],  \hbox{ there is a $C^1$ shift $\widetilde\rho_2(t)\in\mathbb R$} \\
&~{} \qquad \qquad   \hbox{ such that }    \|  \bd{\eta}(t, \cdot ) - {\bf U}(t + \widetilde \rho_2, \omega(t + \widetilde \rho_2), \cdot-\rho(\cdot + \widetilde \rho_2)) \|_{H^1\times H^1} \leq K_2 \varepsilon^{\frac12} \\
&~{} \qquad  \qquad \hbox{ and  }  |\widetilde\rho_2'(t)| \leq 1/4 \Big\}.
\end{aligned}
\end{equation}
By continuity of the flow one has $T_2>-T_\varepsilon.$ The objective is to show that for $K_2$ large but fixed, $0<\varepsilon <\varepsilon_2<\varepsilon_1$ sufficiently small ($\varepsilon_1$ coming from the pre-interaction part), we have $T_2= T_\varepsilon.$ This will prove that $T^*>T_\varepsilon.$ Let us assume, by contradiction, that for all $K_2>0$ large, $\varepsilon>0$ small, we have $T_2<T_\varepsilon.$

\begin{lemma}\label{LemMod2}
There exists $C,\overline\varepsilon>0$ such that the following holds. For $0<\varepsilon< \overline \varepsilon$, there exists a unique time-dependent function $\rho_2 \in C^1([-T_\varepsilon, T_2(K_2)])$ such that, for all $t \in [-T_\varepsilon, T_2(K_2)]$,
\begin{equation}\label{ModulationTesis1}
 \left\langle
 \bd{\eta} (t) - {\bf U}( t + \rho_2(t), \omega(t+ \rho_2(t)), \cdot -\rho(t +\rho_2(t))) ,  {\bd Q}'_{\omega (t+ \rho_2(t))}(\cdot -\rho(t+\rho_2(t)))\right\rangle= 0.
\end{equation}
\end{lemma}	

\begin{proof}
Let $\bd{\eta} (t)$ be the solution constructed in the pre-interaction region. Let $\tilde \rho_2(t)$ be such that the ``tubular neighborhood'' property \eqref{T2} is satisfied. We will invoke the Implicit Function Theorem to modify $\widetilde \rho_2(t)$ by a new $\rho_2(t)$ in order to prove \eqref{ModulationTesis1}  for any fixed time $t \in [-T_\varepsilon, T_2(K_2)]$. Indeed, fix $t \in [-T_\varepsilon, T_2(K_2)]$. Consider the functional $\Phi=\Phi_t$,
$\Phi: H^{1}(\R)\times H^1(\R)\times \R$ given by
\begin{equation}\label{ModAux1}
\begin{aligned}
& \Phi(\bd{\overline{\eta}} ,\overline \rho_2)
\\
&\quad := \left\langle  	
 \bd{\overline{\eta}} \! \left(x\right)-{\bf U}\left( t+ \overline \rho_2, \omega(t+\overline \rho_2), x-\rho\right(t+\overline \rho_2)) ,  {\bd Q}_{\omega(t+\overline \rho_2)}'\left(x-\rho(t+\overline \rho_2)\right) \right\rangle.
\end{aligned}
\end{equation}
Notice that for any $t\in [-T_\varepsilon, T_2(K_2)]$ fixed, approximate solution ${\bf U}(t,\omega(t),x-\rho(t))$ as above and shift given as $\overline \rho_2=\tilde \rho_2(t)$, we have
\begin{equation}\label{ModAux1b}
\Phi({\bf U}\left( t+ \widetilde \rho_2(t), \omega(t+\widetilde \rho_2(t)), \cdot -\rho(t+\widetilde \rho_2(t))\right),\widetilde \rho_2(t)) =0.
\end{equation}
Now, taking into account \eqref{T2}, the idea is to work in a vicinity of
\begin{equation}\label{ModAux2}
(\overline {\bd \eta},\overline \rho_2) = ({\bf U}\left( t+ \tilde \rho_2(t), \omega(t+\tilde \rho_2(t)), \cdot -\rho(t+\tilde \rho_2(t))\right),\tilde \rho_2(t)).
\end{equation}
We compute the functional derivative with respect to $\overline \rho_2$ at the point $(\overline {\bd \eta},\overline \rho_2) $ defined as in \eqref{ModAux2}.  We get from \eqref{ModAux1}, \begin{align*}
&(D_{\overline \rho_2}\Phi)|_{ \left( {\bf U}\left( t+ \tilde \rho_2(t), \omega(t+\tilde \rho_2(t)), \cdot -\rho(t+\tilde \rho_2(t))\right),\tilde \rho_2(t)\right)} \\
&=- \left\langle (\partial_t {\bf U})(t+ \widetilde \rho_2, \omega(t + \widetilde \rho_2), \cdot - \rho(t + \widetilde \rho_2)) ,   {\bd Q}_{\omega(t+\widetilde \rho_2)}'(\cdot -\rho(t+\widetilde \rho_2)) \right\rangle \\
&\quad - \omega'(t+\widetilde \rho_2)  \left\langle ( \Lambda  {\bf U})(t+\widetilde \rho_2, \omega(t+\widetilde \rho_2), \cdot - \rho(t+\widetilde \rho_2)) , {\bd Q}_{\omega(t+\widetilde \rho_2)}'(\cdot -\rho(t+\widetilde \rho_2)) \right\rangle \\
&\quad + \rho'\left(t +\widetilde \rho_2\right)  \left\langle (\partial_x {\bf U})(t+\widetilde \rho_2, \omega(t+\widetilde \rho_2), \cdot - \rho(t+\widetilde \rho_2)) , {\bd Q}_{\omega(t+\widetilde \rho_2)}'(\cdot -\rho(t+\widetilde \rho_2)) \right\rangle. 
\end{align*}
Taking into account the definition of $\bd U$ in \eqref{defU}, parity properties, and integrating by parts, we have that
\begin{align*}
&D_{\overline \rho_2}\Phi\left( {\bf U}\left( t+ \tilde \rho_2(t), \omega(t+\tilde \rho_2(t)), \cdot -\rho(t+\tilde \rho_2(t))\right),\tilde \rho_2(t)\right) \\
&= -  \left\langle (\partial_t {\bf W}^\sharp)(t+ \widetilde \rho_2, \omega(t + \widetilde \rho_2), \cdot - \rho(t + \widetilde \rho_2)) , {\bd Q}_{\omega(t+\widetilde \rho_2)}'(\cdot -\rho(t+\widetilde \rho_2)) \right\rangle \\
& \quad - \omega'(t+\widetilde \rho_2) \left\langle (\partial_\omega {\bf W}^\sharp)(t+ \widetilde \rho_2, \omega(t + \widetilde \rho_2), \cdot - \rho(t + \widetilde \rho_2)),  {\bd Q}_{\omega(t+\widetilde \rho_2)}'(\cdot -\rho(t+\widetilde \rho_2)) \right\rangle \\
&\quad +\omega\left(t +\widetilde \rho_2\right)  \left\langle (\partial_x {\bf W}^\sharp)(t+\widetilde \rho_2, \omega(t+\widetilde \rho_2), \cdot  - \rho(t+\widetilde \rho_2)) , {\bd Q}_{\omega(t+\widetilde \rho_2)}'(\cdot  -\rho(t+\widetilde \rho_2)) \right\rangle \\
& \quad + \rho'\left( t+\widetilde \rho_2\right) \left\langle \bd{Q}_{\omega(t+\widetilde \rho_2)}' , \bd{Q}_{\omega(t+\widetilde \rho_2)}'  \right\rangle.
\end{align*}
Notice that $\left\langle \bd{Q}_{\omega}' , \bd{Q}_{\omega}'  \right\rangle \neq 0$. Then, for $\varepsilon$ small enough, thanks to the estimations for ${\bf W}^\sharp$ and $\partial_x {\bf W}^\sharp$ \eqref{estW}, and since $\rho' = \omega + \varepsilon ^2 f_2$ (see \eqref{DS}), with $f_2$ satisfying the exponential decay in time given by \eqref{cota_f2}, 
\[
|D_{\overline \rho_2}\Phi  \left( {\bf U}\left( t+ \tilde \rho_2(t), \omega(t+\tilde \rho_2(t)), \cdot -\rho(t+\tilde \rho_2(t))\right),\tilde \rho_2(t)\right) |>0,
\]
uniformly in time $t \in [-T_\varepsilon, T_2(K_2)]$. Consequently, since \eqref{ModAux1b} holds, using the Implicit Function Theorem (IFT) and the definition of $T_2(K_2)$ (see \eqref{T2}), we can conclude the existence of a smooth perturbation $ \rho_2(t)$ of $\widetilde \rho_2(t)$. The fact that $\rho_2$ is bounded is also a consequence of the IFT. The fact that $\rho_2$ is of class $C^1$ comes from the fact that the functional $\Phi$ itself is of class $C^1$. The proof is complete.
\end{proof}	

Recall that $\rho_2'(t)$ satisfies the estimate
\begin{equation}\label{boot}
|\rho_2'(t)| \leq \frac14, \quad  t\in [-T_\varepsilon, T_2(K_2)].
\end{equation}
Later, we shall improve this bootstrap estimate. Under this assumption, if we consider
\[
 \tau: [-T_\varepsilon, T_2(K_2)] \longmapsto  \tau([-T_\varepsilon, T_2(K_2)]), \quad   \tau(t):= t+\rho_2(t),
\]
then $\tau$ defines an increasing bijection. Then, there exists some $\widehat \rho_2$ such that we can write $t=:\tau + \widehat \rho_2(\tau)$. We also have
\begin{equation}\label{rhotau}
|\widehat \rho_2'(\tau)| = |1-t'(\tau)| = \left| \frac{\rho_2'(t(\tau))}{1+\rho_2' (t(\tau))}\right| \leq \frac13.
\end{equation}
Redefining
\begin{equation}\label{eta_2}
{\bd \eta}_2(\tau, x):= {\bd \eta}(\tau + \widehat \rho_2(\tau))- {\bf U}(\tau, \omega(\tau), x-\rho(\tau)), \quad \tau \in \tau([-T_\varepsilon, T_2(K_2)]),
\end{equation}
we can consider ${\bd \eta}_2$ in \eqref{eta_2} as a perturbation defined in an interval of the form $[-\widehat T_\varepsilon, \widehat T_2(K_2)]$.  Moreover, as a direct consequence of the definition of $T_2(K_2)$ \eqref{T2} and  Lemma \ref{LemMod2}, we have that ${\bd \eta}_2$ satisfies the bound
\[
\Vert {\bd \eta}_2(\tau , \cdot )\Vert_{H^1(\R)\times H^1(\R)}\le C K_2\varepsilon^{\frac12}, \quad \tau \in [-\widehat T_\varepsilon, \widehat T_2(K_2)].
\]
Now we discuss the meaning and validity of the interval $[-\widehat T_\varepsilon, \widehat T_2(K_2)]$ in terms of the local existence of ${\bd \eta}$.First, thanks to the bound on the derivative of $\widehat \rho_2$, we have
\[
\frac23T_\varepsilon \leq \widehat T_\varepsilon \leq \frac43 T_\varepsilon,  \quad \widehat T_2(K_2)\leq \frac43 T_2(K_2) < T^*.
\]
These estimates will be improved below, but now are sufficient to ensure that  ${\bd \eta}$ is still well-defined in the interval $[-\widehat T_\varepsilon, \widehat T_2(K_2)]$. If there is no confusion, from now on, we will drop the hat on $\widehat\rho_2$. Notice that \eqref{ModulationTesis1} reads now
\begin{equation}\label{ModulationTesis1a}
\begin{aligned}
& \left\langle   \bd{\eta}_2 (\tau, \cdot) , {\bd Q}'_{\omega (\tau)}(\cdot -\rho(\tau))\right\rangle
\\
& \quad =\left\langle \bd{\eta} (\tau + \rho_2(\tau)) - {\bf U}(\tau, \omega(\tau), \cdot -\rho(\tau)) , {\bd Q}'_{\omega (\tau)}(\cdot -\rho(\tau))\right\rangle= 0.
\end{aligned}
\end{equation}
Before differentiating \eqref{ModulationTesis1a}, we need an expression for the equation satisfied by $\bd{\eta}_2$ in \eqref{eta_2} in the $\tau$ variable. This is simple but somehow cumbersome. First, one has
\begin{equation}\label{deco_eta2}
 {\bd \eta}(\tau +  \rho_2(\tau)) = \begin{pmatrix} \eta \\ u  \end{pmatrix} (\tau +\rho_2(\tau)) = {\bf U}(\tau, \omega(\tau), x-\rho(\tau)) +{\bd \eta}_2(\tau, x).
\end{equation}
We will just write $ {\bf U}$ if no confusion is present. Since ${\bd \eta}$ is an exact solution to \eqref{boussinesq}, we will have
\[
\begin{aligned}
&(1-\partial_x^2 ) \partial_\tau [ {\bd \eta}(\tau +  \rho_2(\tau))]  \\
&\quad = (1+ \rho_2'(\tau))(1-\partial_x^2 ) (\partial_\tau{\bd \eta})(\tau +  \rho_2(\tau)) \\
&\quad = -(1+ \rho_2'(\tau))\partial_x \begin{pmatrix} a\, \partial_x^2 u +u + u (\eta+ h)  \\ c\, \partial_x^2 \eta + \eta  + \frac12 u^2 \end{pmatrix} (\tau +  \rho_2(\tau))  \\
&\quad\quad +(1+ \rho_2'(\tau)) \begin{pmatrix} (-1 +a_1 \partial_x^2) \partial_\tau h \\ c_1  \partial_\tau^2 \partial_x h \end{pmatrix}(\tau +  \rho_2(\tau)) .
\end{aligned}
\]
Using the decomposition for $\bd{\eta}_2$,
\[
\begin{aligned}
&(1-\partial_x^2 ) \partial_\tau {\bf U} +  (1-\partial_x^2 ) \partial_\tau {\bd \eta}_2  \\
&\quad = -(1+ \rho_2'(\tau))\partial_x \begin{pmatrix} a\, \partial_x^2 U_2 + U_2 + U_2 (U_1 + h (\tau +  \rho_2(\tau)))  \\ c\, \partial_x^2 U_1 + U_1  + \frac12 U_2^2 \end{pmatrix}  \\
&\quad \quad  -(1+ \rho_2'(\tau))\partial_x \begin{pmatrix} a\, \partial_x^2 u_2 +u_2 +U_2 \eta_2  + u_2 (U_1 + h (\tau +  \rho_2(\tau))) + u_2\eta_2  \\ c\, \partial_x^2 \eta_2 + \eta_2  +  U_2 u_2 \end{pmatrix}  \\
&\quad\quad +(1+ \rho_2'(\tau)) \begin{pmatrix} (-1 +a_1 \partial_x^2) \partial_\tau h \\ c_1  \partial_\tau^2 \partial_x h \end{pmatrix}(\tau +  \rho_2(\tau)),
\end{aligned}
\]
and therefore, from  \eqref{S0_final_final} and \eqref{def_L_U},
\begin{equation}\label{boussinesq_final}
\begin{aligned}
& \bd{R}^\sharp +(1+ \rho_2'(\tau)) \left(  {\bf S}_h'(\bd{U} )\bd{\eta}_2 + {\bf N}\left( {\bd \eta}_2\right) \right)- \rho_2'(\tau)   (1-\partial_x^2 ) \partial_\tau {\bd \eta}_2  \\
&\quad = - \left( 1+ \rho_2'(\tau)\right)\partial_x\left( \begin{matrix} U_2 (h(\tau+\rho_2(\tau))-h (\tau) )  \\ 0\end{matrix}\right) \\
& \quad \quad - \rho_2'(\tau) \partial_x \begin{pmatrix} a\, \partial_x^2 U_2 + U_2 + U_2 (U_1 + h (\tau))  \\ c\, \partial_x^2 U_1 + U_1  + \frac12 U_2^2 \end{pmatrix}  \\
&\quad \quad  -(1+ \rho_2'(\tau))\partial_x \begin{pmatrix}  u_2  (h (\tau +  \rho_2(\tau)) - h (\tau))  \\ 0 \end{pmatrix}  \\
&\quad \quad + (1+  \rho_2'(\tau)) \left(  \begin{pmatrix} a_1 \partial_x^2 \partial_\tau h \\ c_1  \partial_\tau^2 \partial_x h \end{pmatrix} (\tau +  \rho_2(\tau)) - \begin{pmatrix} a_1 \partial_x^2 \partial_\tau h \\ c_1  \partial_\tau^2 \partial_x h \end{pmatrix} (\tau)\right),
\end{aligned}
\end{equation}
where
\[
{\bf N}\left( {\bd \eta}_2\right)= \partial_x \begin{pmatrix}  u_2\eta_2  \\ \frac12u_2^2 \end{pmatrix} .
\]

Recalling \eqref{def_L_U} and \eqref{boussinesq_final}, we find that
\begin{equation}\label{boussinesq_final_22}
\begin{aligned}
&\partial_\tau
 \begin{pmatrix}
  \eta_2   \\
u_2
\end{pmatrix}
\\
&\quad  = - (1+ \rho_2'(\tau))  (1- \partial_x^2)^{-1}\partial_x  J \begin{pmatrix}
  c\, \partial_x^2 \eta_2  + \eta_2  + U_2 u_2 + \frac12u_2^2 \\
a\, \partial_x^2  u_2 + u_2 +  (U_1+h) u_2 +  U_2\eta_2   + u_2\eta_2  \end{pmatrix}    \\
&\quad \quad  - \left( 1+ \rho_2'(\tau)\right)  (1- \partial_x^2)^{-1} \partial_x\left( \begin{matrix} U_2 (h(\tau+\rho_2(\tau))-h (\tau) )  \\ 0\end{matrix}\right) \\
& \quad \quad - \rho_2'(\tau)  (1- \partial_x^2)^{-1}\partial_x \begin{pmatrix} a\, \partial_x^2 U_2 + U_2 + U_2 (U_1 + h (\tau))  \\ c\, \partial_x^2 U_1 + U_1  + \frac12 U_2^2 \end{pmatrix}  \\
&\quad \quad  -(1+ \rho_2'(\tau))  (1- \partial_x^2)^{-1}\partial_x \begin{pmatrix}  u_2  (h (\tau +  \rho_2(\tau)) - h (\tau))  \\ 0 \end{pmatrix}  \\
&\quad \quad + (1+  \rho_2'(\tau))  (1- \partial_x^2)^{-1} \left(  \begin{pmatrix} a_1 \partial_x^2 \partial_\tau h \\ c_1  \partial_\tau^2 \partial_x h \end{pmatrix} (\tau +  \rho_2(\tau)) - \begin{pmatrix} a_1 \partial_x^2 \partial_\tau h \\ c_1  \partial_\tau^2 \partial_x h \end{pmatrix} (\tau)\right)
\\
& \quad\quad -  (1- \partial_x^2)^{-1}\bd{R}^\sharp .
\end{aligned}
\end{equation}
%
Denote
\begin{equation}\label{def_M}
M := \begin{pmatrix}
  c\, \partial_x^2 \eta_2  + \eta_2  + U_2 u_2 + \frac12u_2^2 \\
a\, \partial_x^2  u_2 + u_2 +  (U_1+h) u_2 +  U_2\eta_2   + u_2\eta_2  \end{pmatrix} .
\end{equation}
First, we perform some estimates on $M$ introduced in \eqref{def_M}. We have
\[
\begin{aligned}
& \left| (1- \partial_x^2)^{-1} M \right|
\\
&\quad  = \left| \begin{pmatrix}
  -c \eta_2  + (1+c) (1- \partial_x^2)^{-1}\eta_2  + (1- \partial_x^2)^{-1}\left(U_2 u_2 + \frac12u_2^2 \right)\\
- a  u_2 + (1+a) (1- \partial_x^2)^{-1} u_2 + (1- \partial_x^2)^{-1}\left(  (U_1+h) u_2 +  U_2\eta_2   + u_2\eta_2 \right) \end{pmatrix} \right|.
\end{aligned}
\]
From this identity, the boundedness of $\bd{U}$ and the $L^\infty$ bound on $H^1$ solutions we get
\begin{equation}\label{boundM}
\begin{aligned}
& \left\| (1- \partial_x^2)^{-1} M \right\|_{L^2\times L^2} \lesssim \| \bd{\eta}_2 \|_{L^2\times L^2} + \| \bd{\eta}_2 \|_{H^1\times H^1}^2.
\end{aligned}
\end{equation}
Also, notice that from \eqref{defU}, \eqref{def_W} and \eqref{exact0},
\begin{equation}\label{decoeqU}
\begin{aligned}
\begin{pmatrix} a\, \partial_x^2 U_2 + U_2 + U_1U_2  \\ c\, \partial_x^2 U_1 + U_1  + \frac12 U_2^2 \end{pmatrix} = &~{} \omega (1-\partial_x^2) \begin{pmatrix} R_\omega  \\ Q_\omega \end{pmatrix}
\\
&~{} + \begin{pmatrix} a\, \partial_x^2 W_2 + W_2 + W_1 Q_\omega + W_2 R_\omega  +W_1W_2  \\ c\, \partial_x^2 W_1 + W_1 + Q_\omega W_2  + \frac12 W_2^2 \end{pmatrix} .
\end{aligned}
\end{equation}
This decomposition reveals that at first order part of $M$ in \eqref{def_M} is deeply related to $(1-\partial_x^2)\bd{Q}_\omega$. This fact will later be used to get better estimates on the evolution of related energy functionals.

\begin{lemma}\label{ModulationInTime}
There exists $\varepsilon_2>0$ such that, if $0<\varepsilon<\varepsilon_2$, the following are satisfied:
\begin{enumerate}
\item[\emph{(i)}] We have 
\begin{equation}\label{ModulationTesis2}
	\Vert {\bd\eta }_2(\tau) \Vert_{H^1\times H^1}\lesssim K_2\varepsilon^{\frac12}, \quad \text{for}\quad t \in [-T_\varepsilon, T_2(K_2)],
\end{equation}
and
\begin{equation}\label{ModulationTesis3}
\vert \rho_2 (\tau(-T_\varepsilon))\vert +\Vert {\bd \eta}_2(\tau(- T_\varepsilon)) \Vert_{H^1\times H^1}\le C\varepsilon^{\frac12},
\end{equation}	
where $C>0$ is independent of $K_2$.
\smallskip
\item[\emph{(ii)}] There exist $C>0$ such that
\begin{equation}\label{estimationForRho1}
	\vert \rho'_2(\tau)\vert\le C\left( \varepsilon + \Vert {\bd \eta}_2(\tau)\Vert_{H^1\times H^1} + \Vert {\bd \eta}_2 (\tau)\Vert_{H^1\times H^1}^2  \right),
\end{equation}
for all $t \in [-T_\varepsilon, T_2(K_2)]$.
\end{enumerate}
\end{lemma}

\begin{remark}
Notice that \eqref{ModulationTesis2}-\eqref{estimationForRho1} strictly improve \eqref{rhotau}, and by consequence \eqref{boot}:
\[
|\rho_2'(t)| =|1-\tau'(t)|  = \left| \frac{\hat\rho_2'(\tau(t))}{1+\hat\rho_2'(\tau(t))}\right| \leq 2K \varepsilon^{\frac12}.
\]
(Notice that we came back to the hat notation to avoid confusion.)
\end{remark}

\begin{proof}[Proof of Lemma \ref{ModulationInTime}]
The proof of \eqref{ModulationTesis3} is classical and it is related to the fact that at time $t=-T_\varepsilon$ one has independent bounds coming from the pre-interaction region. This is precisely what one gets in \eqref{PreInteraction} after choosing $t= -T_\varepsilon$, and the dynamical system \eqref{DS} was chosen to respect this estimate. On the other hand, \eqref{ModulationTesis2} is a direct consequence of \eqref{T2}.

Let us prove the more involved estimate \eqref{estimationForRho1}. Directly from Lemma \ref{LemMod2} and since the operator $\left( 1-\partial_x^2\right)$ is self-adjoint, taking derivative in \eqref{ModulationTesis1a}, we have that
\[
 \left\langle   \partial_\tau \bd{\eta}_2 , {\bd Q}'_{\omega (\tau)}(\cdot -\rho(\tau))\right\rangle + \left\langle   \bd{\eta}_2 ,  \partial_\tau {\bd Q}'_{\omega (\tau)}(\cdot -\rho(\tau))\right\rangle  =0.
\]
Using \eqref{boussinesq_final_22}, 
\begin{equation}\label{estRho2}
\begin{aligned}
0 = &~{}  - (1+ \rho_2'(\tau)) \left\langle (1- \partial_x^2)^{-1}\partial_x  JM , {\bd Q}'_{\omega (\tau)}(\cdot -\rho(\tau))\right\rangle \\
&~{}- \left( 1+ \rho_2'(\tau)\right)  \left\langle   (1- \partial_x^2)^{-1} \partial_x\left( \begin{matrix} U_2 (h(\tau+\rho_2(\tau))-h (\tau) )  \\ 0\end{matrix}\right) , {\bd Q}'_{\omega (\tau)}(\cdot -\rho(\tau))\right\rangle \\
&~{} - \rho_2'(\tau)   \left\langle  (1- \partial_x^2)^{-1}\partial_x \begin{pmatrix} a\, \partial_x^2 U_2 + U_2 + U_2 (U_1 + h (\tau))  \\ c\, \partial_x^2 U_1 + U_1  + \frac12 U_2^2 \end{pmatrix} , {\bd Q}'_{\omega (\tau)}(\cdot -\rho(\tau))\right\rangle \\
&~{}-(1+ \rho_2'(\tau))  \left\langle  (1- \partial_x^2)^{-1}\partial_x \begin{pmatrix}  u_2  (h (\tau +  \rho_2(\tau)) - h (\tau))  \\ 0 \end{pmatrix}  , {\bd Q}'_{\omega (\tau)}(\cdot -\rho(\tau))\right\rangle \\
&~{}+  (1+  \rho_2'(\tau))  \Bigg\langle   (1- \partial_x^2)^{-1} \left(  \begin{pmatrix} a_1 \partial_x^2 \partial_\tau h \\ c_1  \partial_\tau^2 \partial_x h \end{pmatrix} (\tau +  \rho_2(\tau)) - \begin{pmatrix} a_1 \partial_x^2 \partial_\tau h \\ c_1  \partial_\tau^2 \partial_x h \end{pmatrix} (\tau)\right) \\
&~{} \qquad  \qquad \qquad \quad , {\bd Q}'_{\omega (\tau)}(\cdot -\rho(\tau)) \Bigg\rangle \\
&~{} - \left\langle (1- \partial_x^2)^{-1}\bd{R}^\sharp , {\bd Q}'_{\omega (\tau)}(\cdot -\rho(\tau))\right\rangle \\
&~{} + \omega' \left\langle   \bd{\eta}_2 , \Lambda {\bd Q}'_{\omega (\tau)}(\cdot -\rho(\tau))\right\rangle - \rho'\left\langle   \bd{\eta}_2 , {\bd Q}''_{\omega (\tau)}(\cdot -\rho(\tau))\right\rangle=: \sum_{j=1}^8 \text{M}_j. 
\end{aligned}
\end{equation}
Now, we estimate each term $\text{M}_j$ in \eqref{estRho2}, starting with the first term in the RHS. Following the estimates \eqref{estW}, we have that for $l=0, 1,2$, the norms $\Vert \partial_x^l {\bf U}\Vert_{L^2\times L^2}$ and $\Vert {\bf} {\bf U}\Vert_{L^\infty\times L^\infty}$ are uniformly bounded. Then, we see that
\begin{equation}\label{estRho2.1}
\begin{aligned}
& |\text{M}_1| +|\text{M}_4| +|\text{M}_5|
 \lesssim \Vert {\bd \eta}_2\Vert_{H^1\times H^1} + \Vert {\bd \eta}_2\Vert_{H^1\times H^1}^2.
\end{aligned}
\end{equation}
Next, we have that
\begin{equation}\label{estRho2.3}
\begin{aligned}
|\text{M}_2| \lesssim &~{}  \left\langle  \partial_x \left( U_2 (h(\tau+\rho_2(\tau))-h (\tau) ) \right) , (1-\partial_x^2)^{-1}{\bd Q}_{\omega(t+\rho_2)}'\right\rangle\\
 \lesssim &~{} \|h(\tau+\rho_2(\tau))-h (\tau) )  \|_{L^\infty} \lesssim  \varepsilon.
\end{aligned}
\end{equation}
Similarly, we also have from \eqref{hypoH} that
\begin{equation}\label{estRho2.2}
|\text{M}_6| \lesssim \varepsilon^4. 
\end{equation}
Now we deal with $\text{M}_3$. Using \eqref{decoeqU}, 
\[
\begin{aligned}
& - \rho_2'(\tau)   \left\langle  (1- \partial_x^2)^{-1}\partial_x \begin{pmatrix} a\, \partial_x^2 U_2 + U_2 + U_2 (U_1 + h (\tau))  \\ c\, \partial_x^2 U_1 + U_1  + \frac12 U_2^2 \end{pmatrix} , {\bd Q}'_{\omega (\tau)}(\cdot -\rho(\tau))\right\rangle
\\
& = - \rho_2'(\tau)\rho'  \left\langle  \bd{Q}'_\omega , {\bd Q}'_{\omega}  \right\rangle \\
& \quad - \rho_2'(\tau)   \left\langle  (1- \partial_x^2)^{-1}\partial_x\begin{pmatrix} a\, \partial_x^2 W_2 + W_2 + W_1 Q_\omega + W_2 R_\omega  +W_1W_2  \\ c\, \partial_x^2 W_1 + W_1 + Q_\omega W_2  + \frac12 W_2^2 \end{pmatrix}  , {\bd Q}'_{\omega}\right\rangle\\
& \quad - \rho_2'(\tau)   \left\langle  (1- \partial_x^2)^{-1}\partial_x\begin{pmatrix} (Q_\omega + W_2) h  \\ 0 \end{pmatrix}  , {\bd Q}'_{\omega}\right\rangle.
\end{aligned}
\]
Therefore, thanks to \eqref{estW} and \eqref{hypoH},
\begin{equation}\label{estRho2.4}
\left| \text{M}_3 +\rho_2'(\tau)  \omega \|\bd{Q}_\omega' \|^2_{L^2\times L^2} \right| \lesssim \varepsilon^{\frac12} e^{-k_0 \varepsilon |t|} |\rho_2'(\tau)| .
\end{equation}
From \eqref{Cota_R}, we have that
\begin{equation}\label{estRho2.6}
|\text{M}_7| = \left\vert \left\langle (1- \partial_x^2)^{-1}\bd{R}^\sharp , {\bd Q}'_{\omega (\tau)}(\cdot -\rho(\tau))\right\rangle \right\vert \le \varepsilon^{\frac32}e^{-k_0\varepsilon \vert t \vert} + \varepsilon^{10}.
\end{equation}
Finally,
\begin{equation}\label{estRho2.7}
\begin{aligned}
& |\text{M}_8| +|\text{M}_9|  \lesssim \Vert {\bd \eta}_2\Vert_{H^1\times H^1}.
\end{aligned}
\end{equation}
Then, gathering \eqref{estRho2.1}, \eqref{estRho2.3}, \eqref{estRho2.2}, \eqref{estRho2.4}, \eqref{estRho2.6} and \eqref{estRho2.7}, we conclude \eqref{estimationForRho1}.

\end{proof}

\subsection{Energy and Momentum estimates}
Let us consider the following functional in the variable $\tau$:
\begin{equation}\label{def_F2}
\begin{aligned}
\mathfrak F_2(\tau):= &~{} \frac12\int \left( -a (\partial_x u_2)^2 -c (\partial_x \eta_2)^2  + u_2^2+ \eta_2^2 \right) (\tau,x)dx \\
&~{} + \frac12\int \left( 2U_2 \eta_2 u_2 +  U_1 u_2^2 \right)(\tau,x) dx + \frac12\int u_2^2  \left( \eta_2 + h\right)(\tau,x)dx\\
&~{} - \omega  \int \left( \partial_x \eta_2 \partial_x u_2 + \eta_2 u_2 \right)(\tau,x)dx  -m_0(\tau) \int Q_\omega u_2(\tau,x)dx. 
\end{aligned}
\end{equation}
This functional is reminiscent of the Hamiltonian and Momentum functionals described in \eqref{Energy}. Here $U_1,U_2, Q_\omega$ are functions evaluated at the variable $z=x-\rho(\tau)$, and $m_0(\tau)$ is the coefficient
\be\label{def_m0}
m_0(\tau) : =  - \varepsilon^2   \rho_2(\tau)  \int_0^1   (\partial_s h_0) (\varepsilon(\tau+\sigma \rho_2(\tau)) , \varepsilon \rho(\tau)) d \sigma,
\ee
with small time variation:
\be\label{def_m0p}
\begin{aligned}
m_0'(\tau) = &~{} - \varepsilon^2   \rho_2'(\tau)  \int_0^1   (\partial_s h_0) (\varepsilon(\tau+\sigma \rho_2(\tau)) , \varepsilon \rho(\tau)) d \sigma \\
&~{} - \varepsilon^3 \rho_2(\tau)  \int_0^1  (1 +\sigma \rho_2'(\tau) ) (\partial_s^2 h_0) (\varepsilon(\tau+\sigma \rho_2(\tau)) , \varepsilon \rho(\tau)) d \sigma \\
 &~{} - \varepsilon^3 \rho'(\tau)  \rho_2(\tau)  \int_0^1   (\partial_y \partial_s h_0) (\varepsilon(\tau+\sigma \rho_2(\tau)) , \varepsilon \rho(\tau)) d \sigma.
\end{aligned}
\ee
The term $m_0$ is chosen to cancel out some bad terms in \eqref{muerete1} below.  Using  \eqref{estimationForRho1} we obtain, assuming $K_2$ large but to be fixed later,
\[
\vert \rho_2(\tau) -\rho_2(-\hat T_\varepsilon)  \vert\le C\left( \varepsilon + K_2  \varepsilon^{\frac12}  +K_2^2  \varepsilon  \right) (\tau + \hat T_\varepsilon) \leq C K_2 (1+K_2\varepsilon^{\frac12}) \varepsilon^{-\frac12-\delta}.
\]
Therefore, thanks to \eqref{ModulationTesis3}, a crude estimate for $\rho_2(\tau) $ is
\be\label{cotas_rho2}
\begin{aligned}
&|\rho_2(\tau) | \lesssim C K_2 (1+K_2\varepsilon^{\frac12}) \varepsilon^{-\frac12-\delta},\\
&\varepsilon^2 \left( 1+\varepsilon | \rho_2(\tau)| +| \rho_2(\tau)| \| \bd{\eta}_2 \|_{L^2\times L^2} \right) \lesssim CK_2 \varepsilon^{2-\delta},
\end{aligned}
\ee
provided $\varepsilon$ is chosen small, depending on a fixed $K_2$. Finally, using \eqref{hypoH} and \eqref{cotas_rho2} in \eqref{def_m0} and \eqref{def_m0p},
\be\label{cotas_m0}
\begin{aligned}
|m_0(\tau) | \lesssim  &~{}  K_2 (1+K_2\varepsilon^{\frac12}) \varepsilon^{\frac32-\delta} e^{-k_0 \varepsilon |\tau+\sigma \rho_2(\tau)| -l_0 \varepsilon |\rho_2(\tau)|} \\
\lesssim &~{}  K_2\varepsilon^{\frac32-\delta} e^{- \frac12k_0 \varepsilon |\tau|}, \\
|m_0'(\tau) | \lesssim &~{}  K_2\varepsilon^{\frac52-\delta} e^{- \frac12k_0 \varepsilon |\tau|} .
\end{aligned}
\ee
Notice that $\mathfrak F_2$ stays bounded in $\tau$. Indeed, we have in \eqref{def_F2},

\begin{lemma}[Boundedness and coercivity]
Let $\mathfrak F_2(\tau)$ be defined as in \eqref{def_F2}. There exist $\varepsilon_2,c_2>0$ such that, for all  $\varepsilon \in(0,\varepsilon_2)$, we have the follwong estiamates
\begin{gather}
    \label{Bound_F2}
|\mathfrak F_2(\tau)| \lesssim \| \bd{\eta}_2(\tau)\|_{H^1\times H^1}^2 + \frac1{c_2} K_2^2\varepsilon^{2-\delta},
\\
\label{Coer_F2}
\mathfrak F_2(\tau) \geq c_2\| \bd{\eta}_2(\tau)\|_{H^1\times H^1}^2- \frac1{c_2} \left|  \langle \bd{\eta}_2(\tau) , J(1-\partial_x^2)\bd{Q}_\omega\rangle \right|^2- \frac1{c_2} K_2^2\varepsilon^{2-\delta} .
\end{gather}
\end{lemma}

\begin{proof}
The proof of \eqref{Bound_F2} is direct from the definition of $\mathfrak F_2$ given in \eqref{def_F2} and the boundedness in $L^\infty$ in time of $\bd{U}$. Now we prove \eqref{Coer_F2}. This is direct from the identity
\[
\begin{aligned}
\mathfrak F_2(\tau) \geq &~{} \frac12\int \left( -a (\partial_x u_2)^2 -c (\partial_x \eta_2)^2  + u_2^2+ \eta_2^2 \right) \\
&~{} + \frac12\int \left( 2Q_\omega \eta_2 u_2 +  R_\omega u_2^2 \right) dx + \frac12\int u_2^2  \left( \eta_2 + h\right)\\
&~{} - \omega  \int \left( \partial_x \eta_2 \partial_x u_2 + \eta_2 u_2 \right) - C \|\bd{W}^\sharp(\tau) \|_{L^\infty}\| \bd{\eta}_2(\tau)\|_{H^1\times H^1}^2\\
&~{} - C|m_0(\tau)| \|\bd{\eta}_2\|_{L^2\times L^2}\\
= &~{} \frac12 \langle \mathcal L  \bd{\eta}_2 , \bd{\eta}_2 \rangle- C \|\bd{W}^\sharp(\tau) \|_{L^\infty}\| \bd{\eta}_2(\tau)\|_{H^1\times H^1}^2- C|m_0(\tau)| \|\bd{\eta}_2\|_{L^2\times L^2}.
\end{aligned}
\]
From here, \eqref{ModulationTesis2}, \eqref{coer_0} and \eqref{estW} give
\[
\begin{aligned}
\mathfrak F_2(\tau) \geq &~{} \left( c_0 - C\varepsilon \right) \| \bd{\eta} (\tau)\|_{H^1\times H^1}^2 - \frac1{c_0} \left|  \langle \bd{\eta}_2 (\tau), J(1-\partial_x^2)\bd{Q}_\omega\rangle \right|^2 - C K_2^2\varepsilon^{2-\delta}
\\
\geq &~{} c_2  \| \bd{\eta}(\tau) \|_{H^1\times H^1}^2- \frac1{c_2} \left|  \langle \bd{\eta}_2 (\tau), J(1-\partial_x^2)\bd{Q}_\omega\rangle \right|^2 - C K_2^2\varepsilon^{2-\delta} .
\end{aligned}
\]
Above, we have used \eqref{cotas_m0} to bound $m_0$.  The proof is complete.
\end{proof}

Now we perform some estimates on the term $ \left|  \langle \bd{\eta}_2 (\tau), J(1-\partial_x^2)\bd{Q}_\omega\rangle \right|^2$. First of all, recall the decomposition on $\bd{\eta}_2$ given in \eqref{deco_eta2}. Recall the energy from \eqref{Energy_new},
\[
H_h[\bd{\eta}(t)] = \frac12\int \left( -a (\partial_x u)^2 -c (\partial_x \eta)^2  + u^2+ \eta^2 + u^2(\eta + h) \right)(t,x)dx.
\]
Notice that following \eqref{est_E}, we have for $t=  \tau + \rho_2(\tau)),$
\be\label{est_E_new2}
\begin{aligned}
& \left| \frac{d}{d\tau} H_h[\bd{\eta}( \tau + \rho_2(\tau)) ] \right|
\\
& ~{} \lesssim (1+|\rho_2'(\tau)|) \left| \frac{d}{dt} H_h[\eta,u ]( \tau + \rho_2(\tau))) \right|
\\
  &~{} \lesssim \varepsilon^2  e^{-k_0\varepsilon | \tau + \rho_2(\tau))| }   \int ( (U_2+ u_2)^2 +(U_1+\eta_2)^2) e^{-l_0\varepsilon |x| }  \\
&~{} \quad +  \varepsilon^{2}  e^{-k_0\varepsilon | \tau + \rho_2(\tau)) | } \int ( |U_1+\eta_2 | +|U_2+ u_2 | ) e^{-l_0\varepsilon |x| } \\
&~{} \lesssim  \varepsilon^2  e^{- \frac12 k_0\varepsilon |\tau| }  \int ( U_2^2 +U_1^2) e^{-l_0\varepsilon |x| }  + \varepsilon^2  e^{-\frac12 k_0\varepsilon |\tau | }  \| \bd{\eta}_2(\tau ) \|_{L^2\times L^2}^2 \\
&~{} \quad    + \varepsilon^2  e^{- \frac12k_0\varepsilon |\tau | }  \int  \left(  |U_1| +|U_2| \right) e^{-l_0\varepsilon |x| }   +  \varepsilon^{\frac32}  e^{-k_0\varepsilon |\tau | } \| \bd{\eta}_2 (\tau )\|_{L^2\times L^2}.
\end{aligned}
\ee
Now we perform the remaining estimates on $U_1$ and $U_2$. First, using \eqref{estW} and the fact that $\rho' = \omega + \varepsilon ^2 f_2$ (see \eqref{DS}), with $f_2$ satisfying the exponential decay in time given by \eqref{cota_f2}, by choosing $k_0$ smaller if necessary,
\[
\begin{aligned}
 \int  \left(  |U_1| +|U_2| \right) e^{-l_0\varepsilon |x| }  \lesssim & ~{} \int  \left(  Q_\omega  +| R_\omega | \right) e^{-l_0\varepsilon |x| }   + \int  \left(  |W_1| +|W_2| \right) e^{-l_0\varepsilon |x| } \\
  \lesssim & ~{}   e^{-k_0 \varepsilon |\rho(\tau)|} +  \varepsilon^{\frac12} e^{-k_0 \varepsilon |\tau |} \lesssim  e^{-\frac12k_0 \varepsilon |\tau |},
\end{aligned}
\]
and
\[
\begin{aligned}
 \int  \left(  U_1^2 + U_2^2 \right) e^{-l_0\varepsilon |x| }  \lesssim & ~{} \int  \left(  Q_\omega^2  + R_\omega^2 \right) e^{-l_0\varepsilon |x| }   + \int  \left(  W_1^2 + W_2^2 \right) e^{-l_0\varepsilon |x| } \\
  \lesssim & ~{}  e^{- k_0 \varepsilon |\rho(\tau)|} +  \varepsilon^{\frac32} e^{-k_0 \varepsilon |\tau |} \lesssim  e^{-\frac12k_0 \varepsilon |\tau |}.
\end{aligned}
\]
Coming back to \eqref{est_E_new2},
\be\label{est_E_new30}
\begin{aligned}
& \left| \frac{d}{d\tau} H_h[\bd{\eta}( \tau + \rho_2(\tau)) ] \right|
\\
&~{} \lesssim  \varepsilon^2  e^{-k_0\varepsilon |\tau| }  + \varepsilon^2  e^{-\frac12 k_0\varepsilon |\tau | }  \| \bd{\eta}_2(\tau ) \|_{L^2\times L^2}^2  +  \varepsilon^{\frac32}  e^{-k_0\varepsilon |\tau | } \| \bd{\eta}_2 (\tau )\|_{L^2\times L^2}.
\end{aligned}
\ee
As above defined, let $\hat T_\varepsilon$ be such that $-\hat T_\varepsilon + \rho_2(-\hat T_\varepsilon) =-T_\varepsilon.$ It is clear that $-\hat T_\varepsilon  \sim T_\varepsilon$, with a minor  relative error.  Now we use \eqref{ModulationTesis2} and integrate in time $-T_\varepsilon \leq \tau + \rho_2(\tau) \leq T^* <T_\varepsilon$ to conclude in \eqref{est_E_new30} that
\be\label{est_E_new3}
\begin{aligned}
&  \left|  H_h[\bd{\eta}( \tau + \rho_2(\tau) )] -  H_h[\bd{\eta}(- T_\varepsilon)]  \right|
 \lesssim  \varepsilon   + K_2^2 \varepsilon^{2} +K_2 \varepsilon. 
\end{aligned}
\ee
Now, following a similar argument as in \eqref{Energy_new_new}, we get
\be\label{Energy_new4}
\begin{aligned}
& H_h[{\bf U} +  \bd{\eta}_2 ](\tau + \rho_2(\tau)) \\
 &~{} \quad + \int \left( a \partial_x^2 U_2 u_2  + c \partial_x^2 U_1 \eta_2   + U_2 u_2 + U_1 \eta_2 + \frac12 U_2^2 \eta_2 + U_2 u_2 (U_1 +h(\tau + \rho_2(\tau)))   \right) \\
&~{} \quad + \frac12\int \left( -a (\partial_x u_2)^2 -c (\partial_x \eta_2)^2  + u_2^2+ \eta_2^2 + 2U_2 \eta_2 u_2 + u_2^2(U_1 + \eta_2 + h(\tau + \rho_2(\tau))) \right)\\
&~{}=: \text{H}_1 + \text{H}_2+ \text{H}_3 .
\end{aligned}
\ee
We treat each term $\text{H}_j$ as follows. First,
\[
\begin{aligned}
& H_h[{\bf U}(\tau) ] = \\
&~{} = \frac12\int \left( -a (\partial_x Q_\omega)^2 -c (\partial_x R_\omega)^2  + Q_\omega^2+ R_\omega^2 + Q_\omega^2(R_\omega + h(\tau + \rho_2(\tau))) \right) \\
&~{} \quad +  \int \left( a \partial_x^2 Q_\omega W_2  + c \partial_x^2 R_\omega W_1   + Q_\omega W_2 + R_\omega W_1 + \frac12 Q_\omega^2 W_1 + Q_\omega W_2 (R_\omega +h(\tau + \rho_2(\tau)))   \right) \\
&~{} \quad + \frac12\int \left( -a (\partial_x W_2)^2 -c (\partial_x W_1)^2  + W_2^2+ W_1^2 + 2 Q_\omega W_1 W_2 + W_2^2(R_\omega + W_1 + h(\tau + \rho_2(\tau))) \right).
\end{aligned}
\]
Using \eqref{estW} and \eqref{hypoH}, we obtain
\[
\begin{aligned}
 H_h[{\bf U}(\tau) ] = &~{} \frac12\int \left( -a (\partial_x Q_\omega)^2 -c (\partial_x R_\omega)^2  + Q_\omega^2+ R_\omega^2 + Q_\omega^2 R_\omega  \right)
 \\
&~{} +  O\left(  \varepsilon e^{-k_0 \varepsilon |\tau|}  +\varepsilon e^{-2k_0 \varepsilon |\tau |}  \right).
\end{aligned}
\]
Finally, we use \eqref{DS} to obtain
\[
 |\text{H}_1(\tau) -\text{H}_1(\tau(-T_\varepsilon))| \lesssim |\omega(\tau) - \omega (\tau(-T_\varepsilon))| + \varepsilon e^{-2k_0 \varepsilon |\tau |} \lesssim \varepsilon.
\]
As for the second term $\text{H}_2$, we have
\[
\begin{aligned}
&\text{H}_2 = \int \left( a \partial_x^2 U_2 u_2  + c \partial_x^2 U_1 \eta_2   + U_2 u_2 + U_1 \eta_2 + \frac12 U_2^2 \eta_2 + U_2 u_2 (U_1 +h(\tau + \rho_2(\tau)))   \right)
\\
& = \int \left( a Q_\omega'' u_2  + c R_\omega'' \eta_2   + Q_\omega u_2 + R_\omega \eta_2 + \frac12 Q_\omega^2 \eta_2 + Q_\omega R_\omega u_2   \right) \\
&\quad +\int \left( a \partial_x^2 W_2 u_2  + c \partial_x^2 W_1 \eta_2   + W_2 u_2 + W_1 \eta_2 +\left( Q_\omega W_2+ \frac12 W_2^2 \right) \eta_2\right)\\
&\quad+\int \left(    W_2 u_2 (R_\omega +W_1 +h(\tau + \rho_2(\tau))) +Q_\omega  u_2 (W_1 +h(\tau + \rho_2(\tau)))  \right) \\
& =: \text{H}_{2,1} +\text{H}_{2,2} +\text{H}_{2,3}.
\end{aligned}
\]
To estimate $\text{H}_{2,1} $, we use \eqref{exact} again to find ($\omega>0$)
\[
\text{H}_{2,1} = \omega \langle(1-\partial_x^2) J \bd{Q}_\omega ,\bd{\eta}_2\rangle.
\]
To bound $\text{H}_{2,2}$, we use \eqref{estW} as follows:
\[
|\text{H}_{2,2}| \lesssim \| \bd{W}^\sharp(\tau )\|_{H^2\times H^2}  \| \bd{\eta}_2(\tau )\|_{H^1\times H^1} \lesssim  \varepsilon^{\frac12} e^{-k_0 \varepsilon |\tau |} \| \bd{\eta}_2(\tau )\|_{H^1\times H^1}.
\]
Very similar, but also using \eqref{hypoH}, we get
\[
\begin{aligned}
|\text{H}_{2,3}| \lesssim &~{}  \left(  \| \bd{W}^\sharp(\tau )\|_{H^2\times H^2} + \|h(\tau + \rho_2(\tau))) \|_{L^\infty} \right) \| \bd{\eta}_2(\tau )\|_{H^1\times H^1}
\\
\lesssim &~{}  \varepsilon^{\frac12} e^{- \frac12 k_0 \varepsilon |\tau |} \| \bd{\eta}_2(\tau )\|_{H^1\times H^1}.
\end{aligned}
\]
Therefore,
\[
\text{H}_2 (\tau )= \omega \langle(1-\partial_x^2) J \bd{Q}_\omega ,\bd{\eta}_2(\tau )\rangle + O\left( \varepsilon^{\frac12} e^{- \frac12 k_0 \varepsilon |\tau |} \| \bd{\eta}_2(\tau )\|_{H^1\times H^1}\right).
\]
Finally, the term $\text{H}_3$ is simply bounded as follows:
\[
| \text{H}_3 (\tau)| \lesssim \| \bd{\eta}_2(\tau)\|_{H^1\times H^1}^2.
\]
Collecting the previous estimates, and \eqref{est_E_new3}, we conclude from \eqref{Energy_new4} that
\begin{equation}\label{cotaQ}
\begin{aligned}
|\langle(1-\partial_x^2) J \bd{Q}_\omega ,\bd{\eta}_2(\tau )\rangle| \lesssim &~{} |\langle(1-\partial_x^2) J \bd{Q}_\omega ,\bd{\eta}_2(\tau(-T_\varepsilon))\rangle |  \\
&~{}  + \varepsilon
  +\varepsilon^{\frac12} e^{- \frac12 k_0 \varepsilon |\tau |} \| \bd{\eta}_2(\tau )\|_{H^1\times H^1}\\
&~{}  +\varepsilon + K_2^2 \varepsilon^{2} +  K_2 \varepsilon  + \sup_{\tau }\| \bd{\eta}_2(\tau)\|_{H^1\times H^1}^2\\
 \lesssim &~{} \varepsilon^{\frac12} + (K_2+K_2^2)\varepsilon.
\end{aligned}
\end{equation}
Notice that the constant involved in the first term $ \varepsilon^{\frac12}$ on the right-hand side above does not depend on $K_2$. Now we use \eqref{cotaQ} to conclude in \eqref{Coer_F2} the improved bound
\begin{equation}\label{Coer_F2_new}
\mathfrak F_2(\tau) \geq c_2\| \bd{\eta}_2(\tau)\|_{H^1\times H^1}^2-  C(\varepsilon + (K_2+K_2^2)^2\varepsilon^{2-\delta}). 
\end{equation}
This estimate will be combined with a suitable upper bound on $\mathfrak F_2$, which is obtained from the following result.

\begin{proposition}[Bound on evolution]
There exist $\varepsilon_2,C_2>0$ such that, for  all $\varepsilon \in (0,\varepsilon_2)$, we have  
\begin{equation}\label{boussinesq_final_25}
\begin{aligned}
|\mathfrak F_2'(\tau) | \lesssim &~{} \varepsilon^{2}(1+\varepsilon | \rho_2(\tau)|) e^{- \frac12k_0 \varepsilon |\tau|} \left( \| \bd{\eta}_2 \|_{L^2\times L^2} + \| \bd{\eta}_2 \|_{H^1\times H^1}^2\right) \\
&~{} +  \varepsilon  |\rho_2'(\tau)|  e^{-k_0 \varepsilon |\tau |}\|\bd{\eta}_2\|_{H^1\times H^1}  + \varepsilon  e^{- \frac12 k_0 \varepsilon |\tau|}  \left( \| \bd{\eta}_2 \|_{H^1\times H^1}^2 + \| \bd{\eta}_2 \|_{H^1\times H^1}^3\right) \\
&~{} +\left( \varepsilon^{\frac32} e^{-k_0 \varepsilon |\tau|} +\varepsilon^{10} \right)\left(\| \bd{\eta}_2 \|_{L^2\times L^2} + \| \bd{\eta}_2 \|_{H^1\times H^1}^2 \right) .
\end{aligned}
\end{equation}
\end{proposition}

\begin{proof}
We compute: 
\[
\begin{aligned}
\mathfrak F_2'(\tau)= &~{} \int \left( -a \partial_x u_2\partial_x \partial_\tau u_2  - c \partial_x \eta_2 \partial_x \partial_\tau \eta_2  +  u_2  \partial_\tau u_2+  \eta_2 \partial_\tau \eta_2 \right)dx \\
&~{} + \frac12\int \left( 2 \partial_\tau U_2 \eta_2 u_2 + 2U_2 \partial_\tau \eta_2 u_2 + 2U_2 \eta_2 \partial_\tau u_2 + \partial_\tau U_1 u_2^2   +  2U_1 u_2 \partial_\tau u_2 \right) dx\\
&~{} + \int u_2 \partial_\tau u_2  \left( \eta_2 + h\right) + \frac12\int u_2^2  \left( \partial_\tau \eta_2 +  \partial_\tau h\right)dx\\
&~{} - \omega'  \int \left( \partial_x \eta_2 \partial_x u_2 + \eta_2 u_2 \right)dx \\
&~{} - \rho'\int \left( \partial_x\partial_\tau \eta_2 \partial_x u_2 +\partial_x \eta_2 \partial_\tau\partial_x u_2 + \partial_\tau\eta_2 u_2 +\eta_2 \partial_\tau u_2 \right)dx\\
&~{}  -m_0'(\tau) \int Q_\omega u_2dx  -m_0(\tau) \int \partial_\tau Q_\omega u_2 dx -m_0(\tau) \int Q_\omega \partial_\tau u_2 dx.
\end{aligned}
\]
After integration by parts and rearranging terms,
\[
\begin{aligned}
\mathfrak F_2'(\tau)= &~{} \int  \left(a  \partial_x^2 u_2 +u_2 + U_2 \eta_2 + U_1 u_2 + u_2 \left( \eta_2 + h\right) \right) \partial_\tau u_2
\\
&~{}  +\int  \left( c\partial_x^2 \eta_2 +\eta_2 + U_2u_2+ \frac12 u_2^2 \right) \partial_\tau \eta_2 \\
&~{} +  \int \left[\partial_\tau U_2 \eta_2 u_2 + \frac12 \left( \partial_\tau  U_1 + \partial_\tau h\right) u_2^2  \right] - \omega'  \int \left( \partial_x \eta_2 \partial_x u_2 + \eta_2 u_2 \right)dx \\
&~{} - \rho'\int \left( (1-\partial_x^2) \partial_\tau\eta_2 u_2 + (1-\partial_x^2)\partial_\tau u_2 \eta_2 \right)dx\\
&~{}  -m_0' \int Q_\omega u_2dx  -m_0(\tau) \int \partial_\tau Q_\omega u_2 dx -m_0(\tau) \int Q_\omega \partial_\tau u_2 dx \\
= &~{} \left\langle  \partial_\tau \bd{\eta}_2 , \begin{pmatrix} c\partial_x^2 \eta_2 +\eta_2 + U_2u_2+ \frac12 u_2^2 \\ a  \partial_x^2 u_2 +u_2 + U_2 \eta_2 + U_1 u_2 + u_2 \left( \eta_2 + h\right) \end{pmatrix} \right\rangle
\\
&~{} +  \int \left[\partial_\tau U_2 \eta_2 u_2 + \frac12 \left( \partial_\tau  U_1 + \partial_\tau h\right) u_2^2  \right]\\
&~{} - \omega'  \int \left( \partial_x \eta_2 \partial_x u_2 + \eta_2 u_2 \right)  - \rho'  \left\langle  (1-\partial_x^2) \partial_\tau \bd{\eta}_2 , J \bd{\eta}_2\right\rangle \\
&~{}  -m_0' \int Q_\omega u_2dx  -m_0 \int \partial_\tau Q_\omega u_2 dx -m_0 \int Q_\omega \partial_\tau u_2 dx.
\end{aligned}
\]
Therefore, replacing \eqref{boussinesq_final_22} in $\mathfrak F_2'$, taking into account the defintion of $M$ in \eqref{def_M}:
\begin{equation}\label{boussinesq_final_23}
\begin{aligned}
& \mathfrak F_2'(\tau)
\\
&  =  - (1+ \rho_2'(\tau))  \left\langle  (1- \partial_x^2)^{-1}\partial_x  JM , M\right\rangle    \\
&\quad  - \left( 1+ \rho_2'(\tau)\right)  \left\langle  (1- \partial_x^2)^{-1} \partial_x\left( \begin{matrix} U_2 (h(\tau+\rho_2(\tau))-h (\tau) )  \\ 0\end{matrix}\right) , M \right\rangle \\
& \quad - \rho_2'(\tau)  \left\langle  (1- \partial_x^2)^{-1}\partial_x \begin{pmatrix} a\, \partial_x^2 U_2 + U_2 + U_2 (U_1 + h (\tau))  \\ c\, \partial_x^2 U_1 + U_1  + \frac12 U_2^2 \end{pmatrix} , M \right\rangle  \\
&\quad  -(1+ \rho_2'(\tau))   \left\langle (1- \partial_x^2)^{-1}\partial_x \begin{pmatrix}  u_2  (h (\tau +  \rho_2(\tau)) - h (\tau))  \\ 0 \end{pmatrix}  ,M \right\rangle \\
& \quad + (1+  \rho_2'(\tau))  \left\langle  (1- \partial_x^2)^{-1} \left(  \begin{pmatrix} a_1 \partial_x^2 \partial_\tau h \\ c_1  \partial_\tau^2 \partial_x h \end{pmatrix} (\tau +  \rho_2(\tau)) - \begin{pmatrix} a_1 \partial_x^2 \partial_\tau h \\ c_1  \partial_\tau^2 \partial_x h \end{pmatrix} (\tau)\right) , M \right\rangle
\\
&\quad -  \left\langle  (1- \partial_x^2)^{-1}\bd{R}^\sharp , M\right\rangle
\\
&\quad +  \int \left[\partial_\tau U_2 \eta_2 u_2 + \frac12 \left( \partial_\tau  U_1 + \partial_\tau h\right) u_2^2  \right] \\
&\quad  - \omega'  \int \left( \partial_x \eta_2 \partial_x u_2 + \eta_2 u_2 \right)  -  \rho'\left\langle  (1-\partial_x^2) \partial_\tau \bd{\eta}_2 , J \bd{\eta}_2\right\rangle\\
&\quad  -m_0'(\tau) \int Q_\omega u_2dx  -m_0(\tau) \int \partial_\tau Q_\omega u_2 dx -m_0(\tau) \int Q_\omega \partial_\tau u_2 dx\\
& =: \sum_{j=1}^{12} \mathcal F_{j}. 
\end{aligned}
\end{equation}
Now we bound each term $\mathcal F_j$, $j=1,2,\ldots, 12$.  We consider the bound on $\mathcal F_1$. Let
\[
J_0:= \frac1{\sqrt{2}}\begin{pmatrix} i & -i \\ 1 & 1\end{pmatrix}, \quad \hbox{so that} \quad J_0^T J_0 = J.
\]
Then, setting $\bd M=\begin{pmatrix} M_1 \\ M_2 \end{pmatrix} $, 
\begin{equation}\label{mF1}
\begin{aligned}
\mathcal F_{1}  = &~{} - (1+ \rho_2'(\tau))  \left\langle  (1- \partial_x^2)^{-1}\partial_x  J\bd M ,\bd M\right\rangle \\
= &~{} - (1+ \rho_2'(\tau)) \left\langle  \partial_x  J_0 (1- \partial_x^2)^{-1/2} \bd M , J_0(1- \partial_x^2)^{-1/2} \bd M\right\rangle =0.
\end{aligned}
\end{equation}
Now we consider $\mathcal F_2.$ First of all, notice that,
\[
\begin{aligned}
\mathcal F_{2}  = &~{}  -\left( 1+ \rho_2'(\tau)\right)  \left\langle  \partial_x\left(U_2 (h(\tau+\rho_2(\tau))-h (\tau) ) \right) ,  (1- \partial_x^2)^{-1}M_1 \right\rangle
 \\
= &~{}   -  \left( 1+ \rho_2'(\tau)\right)\left\langle  (\partial_x U_2) (h(\tau+\rho_2(\tau))-h (\tau) )  ,  (1- \partial_x^2)^{-1}M_1 \right\rangle
\\
&~{} -   \left( 1+ \rho_2'(\tau)\right)  \left\langle  U_2 (  \partial_x h(\tau+\rho_2(\tau))-  \partial_xh (\tau) ) ,  (1- \partial_x^2)^{-1}M_1 \right\rangle
\\
=: &~{} \mathcal F_{2,1} + \mathcal F_{2,2}.
\end{aligned}
\]
We begin estimating $\mathcal F_{2,1}$. We need some simplifications first. Notice that
\[
\begin{aligned}
 \mathcal F_{2,1}
 = &~{}- \left( 1+ \rho_2'(\tau)\right) \left\langle  Q_\omega' (h(\tau+\rho_2(\tau))-h (\tau) )  ,  (1- \partial_x^2)^{-1}M_1 \right\rangle  \\
& ~{}+ O\left( \left| \left\langle  (\partial_x W_2) (h(\tau+\rho_2(\tau))-h (\tau) )  ,  (1- \partial_x^2)^{-1}M_1 \right\rangle \right| \right)\\
 = : &~{} \mathcal F_{2,1,1} +\mathcal F_{2,1,2}.
\end{aligned}
\]
We have from \eqref{hypoH}, \eqref{boundM} and \eqref{estW},
\[
\begin{aligned}
&\mathcal F_{2,1,1}= -  \left( 1+ \rho_2'(\tau)\right) \left\langle  Q_\omega' (h(\tau+\rho_2(\tau))-h (\tau) )  ,  (1- \partial_x^2)^{-1}M_1 \right\rangle  \\
&\quad = - \left( 1+ \rho_2'(\tau)\right) \varepsilon \int_0^1\frac{d}{d\sigma} \left\langle  Q_\omega'  h_0 (\varepsilon(\tau+\sigma \rho_2(\tau)) , \varepsilon \cdot )  ,  (1- \partial_x^2)^{-1}M_1 \right\rangle d\sigma
\\
&\quad = - \left( 1+ \rho_2'(\tau)\right)\varepsilon^2  \rho_2(\tau) \int_0^1 \left\langle  Q_\omega'  (\partial_s h_0) (\varepsilon(\tau+\sigma \rho_2(\tau)) , \varepsilon \cdot )  ,  (1- \partial_x^2)^{-1}M_1 \right\rangle d \sigma  \\
&\quad =- \left( 1+ \rho_2'(\tau)\right) \varepsilon^2   \rho_2(\tau)  \int_0^1   (\partial_s h_0) (\varepsilon(\tau+\sigma \rho_2(\tau)) , \varepsilon \rho(\tau)) d \sigma \left\langle  Q_\omega'   ,  (1- \partial_x^2)^{-1}M_1 \right\rangle  \\
&\quad \quad+  O\left( \varepsilon^3  \left| \int_0^1 \rho_2(\tau) (\partial_y \partial_s h_0) (\varepsilon(\tau+\sigma \rho_2(\tau)) , \varepsilon \rho(\tau))   \left\langle  z Q_\omega'    ,  (1- \partial_x^2)^{-1}M_1 \right\rangle d \sigma \right| \right)\\
&\quad \quad+ O\left(\varepsilon^4 | \rho_2(\tau)   | \left| \int_0^1 \left\langle  z^2 Q_\omega'   (\partial_y^2 \partial_s h_0) (\varepsilon(\tau+\sigma \rho_2(\tau)) , \varepsilon \rho_{\tau,x}) ,  (1- \partial_x^2)^{-1}M_1 \right\rangle d \sigma \right| \right)
\\
& \quad = - \left( 1+ \rho_2'(\tau)\right) \varepsilon^2   \rho_2(\tau)  \int_0^1   (\partial_s h_0) (\varepsilon(\tau+\sigma \rho_2(\tau)) , \varepsilon \rho(\tau)) d \sigma \left\langle  Q_\omega'   ,  (1- \partial_x^2)^{-1}M_1 \right\rangle \\
& \quad \quad +O\left(  \varepsilon^3 \vert \rho_2(\tau)\vert e^{- \frac12 k_0 \varepsilon \vert \tau\vert } \left( e^{-\frac12 l_0 \varepsilon |\rho(t)|} +\varepsilon \right) \left( \Vert {\bd \eta}_2 \Vert_{L^2\times L^2} +  \Vert {\bd \eta}_2 \Vert_{H^1\times H^1}^2\right) \right)\\
& \quad =  - \left( 1+ \rho_2'(\tau)\right)\varepsilon^2   \rho_2(\tau)  \int_0^1   (\partial_s h_0) (\varepsilon(\tau+\sigma \rho_2(\tau)) , \varepsilon \rho(\tau)) d \sigma \left\langle  Q_\omega'   ,  (1- \partial_x^2)^{-1}M_1 \right\rangle \\
& \quad \quad +O\left(  \varepsilon^3 \vert \rho_2(\tau)\vert e^{- \frac12 k_0 \varepsilon \vert \tau\vert }  \left( \Vert {\bd \eta}_2 \Vert_{L^2\times L^2} +  \Vert {\bd \eta}_2 \Vert_{H^1\times H^1}^2\right) \right)
\end{aligned}
\]
In the last estimates we have used that $\tau+\sigma \rho_2(\tau) \geq \frac12\tau$ and $\| Q_\omega' (z) (\partial_s \partial_y h_0) (t , \varepsilon \cdot ) \|_{L^2_x} \lesssim e^{-|t| - \frac12l_0 \varepsilon |\rho(\tau)|}$. On the other hand,
\[
\begin{aligned}
\mathcal F_{2,1,2} \lesssim &~{} |\left\langle  (\partial_x W_2) (h(\tau+\rho_2(\tau))-h (\tau) )  ,  (1- \partial_x^2)^{-1}M_1 \right\rangle| \\
\lesssim  &~{} \varepsilon e^{-k_0 \varepsilon |\tau| }\| \partial_x \bd{W}^\sharp\|_{L^2 \times L^2} \left( \| \bd{\eta}_2 \|_{L^2\times L^2} + \| \bd{\eta}_2 \|_{H^1\times H^1}^2\right) \\
\lesssim  &~{} \varepsilon^2  e^{-2k_0 \varepsilon |\tau| }  \left( \| \bd{\eta}_2 \|_{L^2\times L^2} + \| \bd{\eta}_2 \|_{H^1\times H^1}^2\right).
\end{aligned}
\]
We conclude that
\begin{equation}\label{mF21}
\begin{aligned}
\mathcal F_{2,1} = &~{} - \left( 1+ \rho_2'(\tau)\right) \varepsilon^2   \rho_2(\tau)  \int_0^1   (\partial_s h_0) (\varepsilon(\tau+\sigma \rho_2(\tau)) , \varepsilon \rho(\tau)) d \sigma \left\langle  Q_\omega'   ,  (1- \partial_x^2)^{-1}M_1 \right\rangle
\\
&~{} + O\left( \varepsilon^{2}(1+  \varepsilon | \rho_2(\tau)|) e^{- \frac12k_0 \varepsilon |\tau|}   \left( \| \bd{\eta}_2 \|_{L^2\times L^2} + \| \bd{\eta}_2 \|_{H^1\times H^1}^2\right) \right).
\end{aligned}
\end{equation}
Concerning $\mathcal F_{2,2}$, we get from \eqref{hypoH}, \eqref{boundM} and the boundedness in $H^1\times H^1$ of $\bd{U},$
\begin{equation}\label{mF22}
\begin{aligned}
 \left| \mathcal F_{2,2} \right|  \lesssim &~{} \| \partial_x h(\tau+\rho_2(\tau))- \partial_x h (\tau)  \|_{L^\infty}  \|  U_2\|_{L^2} \|  (1- \partial_x^2)^{-1}M_1 \|_{L^2} \\
 \lesssim &~{} \varepsilon^2  ( e^{-k_0 \varepsilon |\tau|} +e^{-k_0 \varepsilon |\tau + \rho_2(\tau)|}) \left( \| \bd{\eta}_2 \|_{L^2\times L^2} + \| \bd{\eta}_2 \|_{H^1\times H^1}^2\right)\\
 \lesssim &~{} \varepsilon^2  e^{- \frac12k_0 \varepsilon |\tau|}  \left( \| \bd{\eta}_2 \|_{L^2\times L^2} + \| \bd{\eta}_2 \|_{H^1\times H^1}^2\right).
\end{aligned}
\end{equation}
We conclude from \eqref{mF21} and \eqref{mF22} that
\begin{equation}\label{mF2}
\begin{aligned}
\mathcal F_{2} = &~{} \mathcal F_{2,1} +\mathcal F_{2,2}  \\
= &~{}  - \left( 1+ \rho_2'(\tau)\right) \varepsilon^2   \rho_2(\tau)  \int_0^1   (\partial_s h_0) (\varepsilon(\tau+\sigma \rho_2(\tau)) , \varepsilon \rho(\tau)) d \sigma \left\langle  Q_\omega'   ,  (1- \partial_x^2)^{-1}M_1 \right\rangle\\
  &~{} + O\left( \varepsilon^{2}(1+\varepsilon  | \rho_2(\tau)|) e^{- \frac12k_0 \varepsilon |\tau|} \left( \| \bd{\eta}_2 \|_{L^2\times L^2} + \| \bd{\eta}_2 \|_{H^1\times H^1}^2\right) \right).
\end{aligned}
\end{equation}
Now we consider $\mathcal F_3$. Using \eqref{decoeqU},
\[
\begin{aligned}
\mathcal F_3 = &~{}  - \rho_2'(\tau)  \left\langle  (1- \partial_x^2)^{-1}\partial_x \begin{pmatrix} a\, \partial_x^2 U_2 + U_2 + U_2 (U_1 + h (\tau))  \\ c\, \partial_x^2 U_1 + U_1  + \frac12 U_2^2 \end{pmatrix} , M \right\rangle  \\
 = &~{} - \omega \rho_2'(\tau)  \left\langle \partial_x\begin{pmatrix} R_\omega  \\ Q_\omega \end{pmatrix} , M \right\rangle  \\
 &~{} - \rho_2'(\tau)  \left\langle  (1- \partial_x^2)^{-1}\partial_x \begin{pmatrix} a\, \partial_x^2 W_2 + W_2 + W_1 Q_\omega + W_2 R_\omega  +W_1W_2  \\ c\, \partial_x^2 W_1 + W_1 + Q_\omega W_2  + \frac12 W_2^2 \end{pmatrix},  M \right\rangle
 \\
 =: &~{} \mathcal F_{3,1} +\mathcal F_{3,2} .
\end{aligned}
\]
Recall that $M$ is given in \eqref{def_M}. To bound $\mathcal F_{3,1}$, we proceed using the fact that $M$ is at first order the linearized operator associated with $\bd{Q}_\omega'$. Indeed, from \eqref{kernel} and \eqref{ModulationTesis1a},
\[
\begin{aligned}
\left\langle \partial_x\begin{pmatrix} R_\omega  \\ Q_\omega \end{pmatrix} , M \right\rangle  = &~{}  \left\langle \partial_x\begin{pmatrix} R_\omega  \\ Q_\omega \end{pmatrix} , \begin{pmatrix}
   c\, \partial_x^2 \eta_2  + \eta_2 +  Q_\omega u_2  \\
a\, \partial_x^2  u_2 + u_2 +  R_\omega u_2 +  Q_\omega \eta_2    \end{pmatrix} \right\rangle
\\
&~{}  + \left\langle \partial_x\begin{pmatrix} R_\omega  \\ Q_\omega \end{pmatrix} , \begin{pmatrix}
    W_2 u_2  \\
  (W_1+h) u_2 +  W_2\eta_2     \end{pmatrix} \right\rangle +\left\langle \partial_x\begin{pmatrix} R_\omega  \\ Q_\omega \end{pmatrix} , \begin{pmatrix}
   \frac12u_2^2 \\
  u_2\eta_2  \end{pmatrix} \right\rangle \\
= &~{} \left\langle \bd{Q}_\omega' , \mathcal L \begin{pmatrix} \eta_2 \\ u_2 \end{pmatrix} \right\rangle +\omega  \left\langle \bd{Q}_\omega' , (1-\partial_x^2) \begin{pmatrix} u_2 \\ \eta_2 \end{pmatrix} \right\rangle \\
&~{}  + \left\langle  \bd{Q}_\omega'  , \begin{pmatrix}
    W_2 u_2  \\
  (W_1+h) u_2 +  W_2\eta_2     \end{pmatrix} \right\rangle + \frac12 \int R_\omega' u_2^2 + \int Q_\omega' \eta_2 u_2
 \\
= &~{} \omega  \left\langle (1-\partial_x^2) J\bd{Q}_\omega' , \bd{\eta}_2 \right\rangle +  \left\langle  \bd{Q}_\omega'  , \begin{pmatrix}
    W_2 u_2  \\
  (W_1+h) u_2 +  W_2\eta_2     \end{pmatrix} \right\rangle
  \\
  &~{} + \frac12 \int R_\omega' u_2^2 + \int Q_\omega' \eta_2 u_2 .
\end{aligned}
\]
It is not difficult to see that from \eqref{estW} and \eqref{hypoH},
\[
\begin{aligned}
\left|  \left\langle  \bd{Q}_\omega'  , \begin{pmatrix}
    W_2 u_2  \\
  (W_1+h) u_2 +  W_2\eta_2     \end{pmatrix} \right\rangle \right| \lesssim &~{} \left( \| \bd{W}^\sharp\|_{L^\infty \times L^\infty}  + \|h\|_{L^\infty}\right) \|\bd{\eta}_2\|_{L^2\times L^2} \\
  \lesssim &~{} \varepsilon e^{-k_0 \varepsilon |\tau|}\|\bd{\eta}_2\|_{L^2\times L^2}.
\end{aligned}
\]
From here, we find that
\begin{equation}\label{mF31}
\begin{aligned}
 \mathcal F_{3,1}  = &~{} - \omega \rho_2'(\tau)  \left\langle \partial_x\begin{pmatrix} R_\omega  \\ Q_\omega \end{pmatrix} , M \right\rangle
\\
= &~{} -\omega^2 \rho_2'(\tau) \left\langle (1-\partial_x^2) J\bd{Q}_\omega' , \bd{\eta}_2 \right\rangle -\omega\rho'(\tau) \left( \frac12 \int R_\omega' u_2^2 + \int Q_\omega' \eta_2 u_2 \right) \\
&~{} + O\left(  |\rho_2'(\tau)|  \varepsilon e^{-k_0 \varepsilon |\tau |}\|\bd{\eta}_2\|_{L^2\times L^2}  \right).
\end{aligned}
\end{equation}
Now we deal with $\mathcal F_{3,2}$. First of all, notice that
\begin{equation}\label{cotaF3}
\begin{aligned}
& \left|   \left\langle \partial_x \begin{pmatrix} a\, \partial_x^2 W_2 + W_2 + W_1 Q_\omega + W_2 R_\omega  +W_1W_2  \\ c\, \partial_x^2 W_1 + W_1 + Q_\omega W_2  + \frac12 W_2^2 \end{pmatrix} ,  (1-\partial_x^2)^{-1} M \right\rangle \right|
\\
& \quad \leq \left|   \left\langle \begin{pmatrix} a\, \partial_x^3 W_2 +  \partial_x W_2    \\ c\, \partial_x^3 W_1 +  \partial_x W_1    \end{pmatrix} ,  (1-\partial_x^2)^{-1} M\right\rangle \right|
\\
& \quad \quad + \left|   \left\langle \begin{pmatrix}   \partial_x W_1 Q_\omega + W_1 \partial_x Q_\omega  +  \partial_xW_2 R_\omega+W_2 \partial_x R_\omega   \\   \partial_xQ_\omega W_2  + Q_\omega \partial_x W_2    \end{pmatrix} ,   (1-\partial_x^2)^{-1} M \right\rangle \right|
\\
& \quad \quad + \left|   \left\langle \begin{pmatrix}  \partial_x W_1W_2+  W_1\partial_x W_2  \\  W_2  \partial_x W_2 \end{pmatrix} ,  (1-\partial_x^2)^{-1} M \right\rangle \right| .
\end{aligned}
\end{equation}
Consequently, let us estimate $\mathcal F_{3,2}$ into three main parts as
\begin{align*}
|\mathcal F_{3,2}| & \leq | \rho_2'(\tau) | \left|   \left\langle \begin{pmatrix} a\, \partial_x^3 W_2 +  \partial_x W_2    \\ c\, \partial_x^3 W_1 +  \partial_x W_1    \end{pmatrix} ,  (1-\partial_x^2)^{-1} M\right\rangle \right|
\\ & \quad +| \rho_2'(\tau) | \left|   \left\langle \begin{pmatrix}   \partial_x W_1 Q_\omega + W_1 \partial_x Q_\omega  +  \partial_xW_2 R_\omega+W_2 \partial_x R_\omega   \\   \partial_xQ_\omega W_2  + Q_\omega \partial_x W_2    \end{pmatrix} ,   (1-\partial_x^2)^{-1} M \right\rangle \right|
\\ & \quad +| \rho_2'(\tau) |  \left|   \left\langle \begin{pmatrix}  \partial_x W_1W_2+  W_1\partial_x W_2  \\  W_2  \partial_x W_2 \end{pmatrix} ,  (1-\partial_x^2)^{-1} M \right\rangle \right|
\\ &= : \mathcal F_{3,2,1} + \mathcal F_{3,2,2} + \mathcal F_{3,2,3}.
\end{align*}
It is not difficult to see from \eqref{estW} and \eqref{boundM} that
\[
\begin{aligned}
|\mathcal F_{3,2,1}| \lesssim &~{}  | \rho_2'(\tau) |\|\partial_x \bd{W}^\sharp(\tau)\|_{H^2\times H^2} \left( \| \bd{\eta}_2 \|_{L^2\times L^2} + \| \bd{\eta}_2 \|_{H^1\times H^1}^2 \right) \\
 \lesssim &~{} \varepsilon  | \rho_2'(\tau) | e^{-k_0 \varepsilon |\tau|}\left( \| \bd{\eta}_2 \|_{L^2\times L^2} + \| \bd{\eta}_2 \|_{H^1\times H^1}^2 \right).
\end{aligned}
\]
Also,
\[
\begin{aligned}
|\mathcal F_{3,2,2}| \lesssim &~{} | \rho_2'(\tau) | \|  \bd{W}^\sharp(\tau)\|_{L^\infty \times L^\infty} \| \bd{Q}_\omega \|_{L^2\times L^2} \left( \| \bd{\eta}_2 \|_{L^2\times L^2} + \| \bd{\eta}_2 \|_{H^1\times H^1}^2 \right)\\
&~{} + \|  \partial_x \bd{W}^\sharp(\tau)\|_{L^2 \times L^2} \| \bd{Q}_\omega \|_{L^\infty\times L^\infty}\left( \| \bd{\eta}_2 \|_{L^2\times L^2} + \| \bd{\eta}_2 \|_{H^1\times H^1}^2 \right) \\
 \lesssim &~{}  \varepsilon  | \rho_2'(\tau) | e^{-k_0 \varepsilon |\tau|} \left( \| \bd{\eta}_2 \|_{L^2\times L^2} + \| \bd{\eta}_2 \|_{H^1\times H^1}^2 \right).
\end{aligned}
\]
Finally,
\[
\begin{aligned}
|\mathcal F_{3,2,3}| \lesssim &~{}   | \rho_2'(\tau) |\|\bd{W}^\sharp(\tau)\|_{L^\infty \times L^\infty} \|\partial_x \bd{W}^\sharp(\tau)\|_{L^2\times L^2} \left( \| \bd{\eta}_2 \|_{L^2\times L^2} + \| \bd{\eta}_2 \|_{H^1\times H^1}^2 \right)\\
 \lesssim &~{}  \varepsilon^2  | \rho_2'(\tau) | e^{-2k_0 \varepsilon |\tau|} \left( \| \bd{\eta}_2 \|_{L^2\times L^2} + \| \bd{\eta}_2 \|_{H^1\times H^1}^2 \right).
\end{aligned}
\]
We conclude
\begin{equation}\label{mF32}
|\mathcal F_{3,2}| \lesssim \varepsilon  | \rho_2'(\tau) | e^{-k_0 \varepsilon |\tau|}\left( \| \bd{\eta}_2 \|_{L^2\times L^2} + \| \bd{\eta}_2 \|_{H^1\times H^1}^2 \right).
\end{equation}
Gathering \eqref{mF31} and \eqref{mF32}, we obtain
\begin{equation}\label{mF3}
\begin{aligned}
 \mathcal F_{3} =&~{} \mathcal F_{3,1} + \mathcal F_{3,2} \\
= &~{}  -\omega^2 \rho_2'(\tau) \left\langle (1-\partial_x^2) J\bd{Q}_\omega' , \bd{\eta}_2 \right\rangle -\omega\rho'_2(\tau) \left( \frac12 \int R_\omega' u_2^2 + \int Q_\omega' \eta_2 u_2 \right)\\
 &~{}  + O\left(  |\rho_2'(\tau)|  \varepsilon e^{-k_0 \varepsilon |\tau |}\|\bd{\eta}_2\|_{L^2\times L^2}  \right) .
\end{aligned}
\end{equation}
Now we are concerned with $\mathcal F_4$. First,
\[
\begin{aligned}
 \left| \mathcal F_{4} \right|  \lesssim &~{} \left|   \left\langle \partial_x \begin{pmatrix}  u_2  (h (\tau +  \rho_2(\tau)) - h (\tau))  \\ 0 \end{pmatrix}  , (1- \partial_x^2)^{-1}M \right\rangle \right| \\
 \lesssim &~{}  \left|   \left\langle  \partial_x u_2  (h (\tau +  \rho_2(\tau)) - h (\tau))  , (1- \partial_x^2)^{-1}M_1 \right\rangle \right| \\
&~{}  +  \left|   \left\langle   u_2  (  \partial_x h (\tau +  \rho_2(\tau)) - \partial_x h (\tau))    , (1- \partial_x^2)^{-1}M_1 \right\rangle \right| \\
 =: &~{} \mathcal F_{4,1} + \mathcal F_{4,2}.
\end{aligned}
\]
Now we deal with $\mathcal F_{4,1}$.
\begin{equation}\label{mF41}
\begin{aligned}
 \left| \mathcal F_{4,1} \right|  \lesssim &~{} \| h(\tau+\rho_2(\tau))-h (\tau)  \|_{L^\infty}  \| \partial_x u_2\|_{L^2} \|  (1- \partial_x^2)^{-1}M_1 \|_{L^2} \\
  \lesssim &~{} \varepsilon  ( e^{-k_0 \varepsilon |\tau|} +e^{-k_0 \varepsilon |\tau + \rho_2(\tau)|}) \left( \| \bd{\eta}_2 \|_{L^2\times L^2}^2 + \| \bd{\eta}_2 \|_{H^1\times H^1}^3\right)\\
   \lesssim &~{} \varepsilon  e^{- \frac12k_0 \varepsilon |\tau|}  \left( \| \bd{\eta}_2 \|_{L^2\times L^2}^2 + \| \bd{\eta}_2 \|_{H^1\times H^1}^3\right).
\end{aligned}
\end{equation}
Concerning $\mathcal F_{4,2}$, we get
\begin{equation}\label{mF42}
\begin{aligned}
 \left| \mathcal F_{4,2} \right|  \lesssim &~{}  \| \partial_x h(\tau+\rho_2(\tau))- \partial_x h (\tau)  \|_{L^\infty}  \|  u_2\|_{L^2} \|  (1- \partial_x^2)^{-1}M_1 \|_{L^2} \\
 \lesssim &~{} \varepsilon^2  ( e^{-k_0 \varepsilon |\tau|} +e^{-k_0 \varepsilon |\tau + \rho_2(\tau)|}) \left( \| \bd{\eta}_2 \|_{L^2\times L^2}^2 + \| \bd{\eta}_2 \|_{H^1\times H^1}^3\right)\\
 \lesssim &~{} \varepsilon^2  e^{- \frac12k_0 \varepsilon |\tau|}  \left( \| \bd{\eta}_2 \|_{L^2\times L^2}^2 + \| \bd{\eta}_2 \|_{H^1\times H^1}^3\right) .
\end{aligned}
\end{equation}
Gathering \eqref{mF41} and \eqref{mF42},
\begin{equation}\label{mF4}
\begin{aligned}
 \left| \mathcal F_{4} \right|  \lesssim &~{}  |\mathcal F_{4,1}| + |\mathcal F_{4,2}| \lesssim  \varepsilon  e^{- \frac12 k_0 \varepsilon |\tau|}  \left( \| \bd{\eta}_2 \|_{L^2\times L^2}^2 + \| \bd{\eta}_2 \|_{H^1\times H^1}^3\right) .
\end{aligned}
\end{equation}
To bound $\mathcal F_5$, we proceed as follows. First,
\[
\begin{aligned}
 \left| \mathcal F_{5} \right|  \lesssim &~{}  \left| \left\langle   \left(  \begin{pmatrix} a_1 \partial_x^2 \partial_\tau h \\ c_1  \partial_\tau^2 \partial_x h \end{pmatrix} (\tau +  \rho_2(\tau)) - \begin{pmatrix} a_1 \partial_x^2 \partial_\tau h \\ c_1  \partial_\tau^2 \partial_x h \end{pmatrix} (\tau)\right) , (1- \partial_x^2)^{-1} M \right\rangle \right| \\
\lesssim &~{}  \left\|   \begin{pmatrix} a_1 \partial_x^2 \partial_\tau h \\ c_1  \partial_\tau^2 \partial_x h \end{pmatrix} (\tau +  \rho_2(\tau)) - \begin{pmatrix} a_1 \partial_x^2 \partial_\tau h \\ c_1  \partial_\tau^2 \partial_x h \end{pmatrix} (\tau) \right\|_{L^2\times L^2 } \left\| (1- \partial_x^2)^{-1} M \right\|_{L^2\times L^2}.
\end{aligned}
\]
Now we use \eqref{hypoH} and \eqref{boundM} to conclude that
\begin{equation}\label{mF5}
\begin{aligned}
\left| \mathcal F_{5} \right| \lesssim &~{}\varepsilon^{\frac72} \left( e^{-k_0 \varepsilon |\tau|} + e^{-k_0 \varepsilon |\tau +\rho_2(\tau)|} \right) \left(  \| \bd{\eta}_2 \|_{L^2\times L^2} + \| \bd{\eta}_2 \|_{H^1\times H^1}^2 \right)\\
\lesssim &~{} \varepsilon^{\frac72}  e^{- \frac12 k_0 \varepsilon |\tau|}  \left(  \| \bd{\eta}_2 \|_{L^2\times L^2} + \| \bd{\eta}_2 \|_{H^1\times H^1}^2 \right).
\end{aligned}
\end{equation}
Notice that we have used the bound $|\tau +\rho_2(\tau)| \geq \frac12|\tau|$. Now we bound $\mathcal F_6$. We have from \eqref{boundM} and \eqref{Cota_R},
\begin{equation}\label{mF6}
\begin{aligned}
 \left| \mathcal F_{6} \right| \lesssim &~{} \left|  \left\langle  \bd{R}^\sharp , (1- \partial_x^2)^{-1} M\right\rangle  \right| \\
 \lesssim &~{} \left( \varepsilon^{\frac32} e^{-k_0 \varepsilon |\tau|} +\varepsilon^{10} \right)\left(\| \bd{\eta}_2 \|_{L^2\times L^2} + \| \bd{\eta}_2 \|_{H^1\times H^1}^2 \right).
\end{aligned}
\end{equation}
We deal with $\mathcal F_7.$ First, recall that from \eqref{defU} we have 
\[
 \partial_\tau  \bd{U} = \omega' \Lambda \bd{Q}_{\omega} -\rho'\bd{Q}_{\omega}' + \partial_\tau \bd{W}^\sharp.
\]
Therefore,
\[
\begin{aligned}
\mathcal F_7 =  &~{} \omega'  \int \left[\Lambda Q_\omega \eta_2 u_2 + \frac12  \Lambda  R_\omega u_2^2  \right] -\rho' \int \left[ Q_\omega' \eta_2 u_2 + \frac12 R_\omega' u_2^2  \right]  \\
&~{}+ \int \left[\partial_\tau W_2 \eta_2 u_2 + \frac12 \left( \partial_\tau  W_1 + \partial_\tau h\right) u_2^2  \right] \\
= : &~{} \mathcal F_{7,1} +\mathcal F_{7,2} +\mathcal F_{7,3} 
\end{aligned}
\]
Note that $\mathcal F_{7,2}$ is a large term and needs to cancel out with another term. Also, from \eqref{DS} and the estimate
on $f_1(t)$,
\begin{equation}\label{mF71}
\begin{aligned}
& \left| \mathcal F_{7,1} \right| \lesssim  \varepsilon |f_1(\tau)| \| \bd{\eta}_2(\tau) \|_{H^1\times H^1}^2 \lesssim  \varepsilon e^{-k_0 \varepsilon |\tau| -l_0 \varepsilon |\rho(\tau)|}  \| \bd{\eta}_2(\tau) \|_{H^1\times H^1}^2.
\end{aligned}
\end{equation}
Additionally, from \eqref{dtW} and \eqref{hypoH},
\begin{equation}\label{mF73}
\begin{aligned}
 \left| \mathcal F_{7,3} \right| \lesssim &~{} \left( \| \partial_\tau \bd{W}(\tau) \|_{L^\infty\times L^\infty} + \| \partial_\tau h(\tau)\|_{L^\infty} \right) \| \bd{\eta}_2(\tau) \|_{H^1\times H^1}^2  \\
 \lesssim &~{} \left( \| \partial_\tau \bd{W}(\tau) \|_{L^\infty\times L^\infty} + \varepsilon^2 e^{-k_0 \varepsilon |\tau| }  \right) \| \bd{\eta}_2(\tau) \|_{H^1\times H^1}^2 \\
 \lesssim &~{}\varepsilon e^{-k_0 \varepsilon |\tau| } \| \bd{\eta}_2(\tau) \|_{H^1\times H^1}^2 .
\end{aligned}
\end{equation}
Concluding, using \eqref{mF71} and \eqref{mF73},
\begin{equation}\label{mF7}
\begin{aligned}
& \left| \mathcal F_{7} + \omega \int \left[ Q_\omega' \eta_2 u_2 + \frac12 R_\omega' u_2^2  \right]  \right| \lesssim \varepsilon e^{-k_0 \varepsilon |\tau| }  \| \bd{\eta}_2(\tau) \|_{H^1\times H^1}^2. 
\end{aligned}
\end{equation}
We deal now with $\mathcal F_8.$ Using \eqref{DS}, we have
\begin{equation}\label{mF8}
\begin{aligned}
& \left| \mathcal F_{8}  \right| = \left| \omega'  \int \left( \partial_x \eta_2 \partial_x u_2 + \eta_2 u_2 \right)  \right| \lesssim \varepsilon e^{-k_0 \varepsilon |\tau| -l_0 \varepsilon |\rho(\tau)|}  \| \bd{\eta}_2(\tau) \|_{H^1\times H^1}^2.
\end{aligned}
\end{equation}
Dealing with $\mathcal F_9$ is very similar to handling $\mathcal F_1$--$ \mathcal F_6$ above, with some important differences in the cases of $\mathcal F_1$ and $\mathcal F_3$. We have from \eqref{boussinesq_final_22},
\[
\begin{aligned}
&  \left\langle  (1-\partial_x^2) \partial_\tau \bd{\eta}_2 , J \bd{\eta}_2\right\rangle
\\
&  =  - (1+ \rho_2'(\tau))  \left\langle  \partial_x  JM , J \bd{\eta}_2 \right\rangle    \\
&\quad  - \left( 1+ \rho_2'(\tau)\right)  \left\langle   \partial_x\left( \begin{matrix} U_2 (h(\tau+\rho_2(\tau))-h (\tau) )  \\ 0\end{matrix}\right) , J \bd{\eta}_2 \right\rangle \\
& \quad - \rho_2'(\tau)  \left\langle  \partial_x \begin{pmatrix} a\, \partial_x^2 U_2 + U_2 + U_2 (U_1 + h (\tau))  \\ c\, \partial_x^2 U_1 + U_1  + \frac12 U_2^2 \end{pmatrix} , J \bd{\eta}_2 \right\rangle  \\
&\quad  -(1+ \rho_2'(\tau))   \left\langle \partial_x \begin{pmatrix}  u_2  (h (\tau +  \rho_2(\tau)) - h (\tau))  \\ 0 \end{pmatrix}  , J \bd{\eta}_2 \right\rangle \\
& \quad + (1+  \rho_2'(\tau))  \left\langle  \left(  \begin{pmatrix} a_1 \partial_x^2 \partial_\tau h \\ c_1  \partial_\tau^2 \partial_x h \end{pmatrix} (\tau +  \rho_2(\tau)) - \begin{pmatrix} a_1 \partial_x^2 \partial_\tau h \\ c_1  \partial_\tau^2 \partial_x h \end{pmatrix} (\tau)\right) , J \bd{\eta}_2 \right\rangle
\\
&\quad -  \left\langle    \bd{R}^\sharp , J \bd{\eta}_2 \right\rangle
\\
& =: \mathcal F_{9,1}+\mathcal F_{9,2}+\mathcal F_{9,3}+\mathcal F_{9,4}+\mathcal F_{9,5}+\mathcal F_{9,6}.
\end{aligned}
\]
The terms $\mathcal F_{9,2}$, $\mathcal F_{9,4}$, $\mathcal F_{9,5}$ and $\mathcal F_{9,6}$ are treated in a similar form to previous computations for $\mathcal F_{2}$, $\mathcal F_{4}$, $\mathcal F_{5}$ and $\mathcal F_{6}$, respectively. We get, following the proof of estimates \eqref{mF2}, \eqref{mF4}, \eqref{mF5} and \eqref{mF6},
\[
\begin{aligned}
& \mathcal F_{9,2} + |\mathcal F_{9,4}|+|\mathcal F_{9,5}|+|\mathcal F_{9,6}| \\
&~{} =   - \left( 1+ \rho_2'(\tau)\right) \varepsilon^2   \rho_2(\tau)  \int_0^1   (\partial_s h_0) (\varepsilon(\tau+\sigma \rho_2(\tau)) , \varepsilon \rho(\tau)) d \sigma \left\langle  Q_\omega'   ,  u_2 \right\rangle\\
&~{} \quad +O\left(\varepsilon^{2}(1+\varepsilon | \rho_2(\tau)|) e^{- \frac12k_0 \varepsilon |\tau|} \| \bd{\eta}_2 \|_{L^2\times L^2}  \right)\\
&~{} \quad  +O\left( \varepsilon  e^{- \frac12l_0 \varepsilon |\tau|}   \| {\bd\eta}_2 \|_{L^2\times L^2}^2 \right) \\
&~{} \quad +O\left(  \varepsilon^{\frac72} \left( e^{-k_0 \varepsilon |\tau|} + e^{-k_0 \varepsilon |\tau +\rho_2(\tau)|} \right)  \| \bd{\eta}_2 \|_{L^2\times L^2} \right)
 \\
& ~{} \quad  +O\left(  \left( \varepsilon^{\frac32} e^{-k_0 \varepsilon |\tau|} +\varepsilon^{10} \right)\| \bd{\eta}_2 \|_{L^2\times L^2} \right) \\
&~{} =  - \left( 1+ \rho_2'(\tau)\right) \varepsilon^2   \rho_2(\tau)  \int_0^1   (\partial_s h_0) (\varepsilon(\tau+\sigma \rho_2(\tau)) , \varepsilon \rho(\tau)) d \sigma \left\langle  Q_\omega'   ,  u_2 \right\rangle\\
&~{} \quad +O\left( \varepsilon^2 \left( 1+  \varepsilon \vert \rho_2(\tau)\vert \right) e^{- \frac12 k_0 \varepsilon |\tau|}    \| \bd{\eta}_2 \|_{L^2\times L^2}  +  \left( \varepsilon^{\frac32} e^{-k_0 \varepsilon |\tau|} +\varepsilon^{10} \right)\| \bd{\eta}_2 \|_{L^2\times L^2}\right)
\\
&~{} \quad + O\left( \varepsilon  e^{- \frac12 k_0 \varepsilon |\tau|}   \| {\bd\eta}_2 \|_{L^2\times L^2}^2 \right).
\end{aligned}
\]
Finally, we use \eqref{cotas_rho2} to get the simplified expression
\begin{equation}\label{mF9conj}
\begin{aligned}
& \mathcal F_{9,2} + |\mathcal F_{9,4}|+|\mathcal F_{9,5}|+|\mathcal F_{9,6}| \\
&~{} =  - \left( 1+ \rho_2'(\tau)\right) \varepsilon^2   \rho_2(\tau)  \int_0^1   (\partial_s h_0) (\varepsilon(\tau+\sigma \rho_2(\tau)) , \varepsilon \rho(\tau)) d \sigma \left\langle  Q_\omega'   ,  u_2 \right\rangle\\
&~{} \quad +O\left( \varepsilon^2 \left( 1+  K_2 \varepsilon^{\frac12 -\delta}  \right) e^{- \frac12 k_0 \varepsilon |\tau|}    \| \bd{\eta}_2 \|_{L^2\times L^2}  +  \left( \varepsilon^{\frac32} e^{-k_0 \varepsilon |\tau|} +\varepsilon^{10} \right)\| \bd{\eta}_2 \|_{L^2\times L^2}\right)
\\
&~{} \quad + O\left( \varepsilon  e^{- \frac12 k_0 \varepsilon |\tau|}   \| {\bd\eta}_2 \|_{L^2\times L^2}^2 \right)\\
&~{} =  - \left( 1+ \rho_2'(\tau)\right) \varepsilon^2   \rho_2(\tau)  \int_0^1   (\partial_s h_0) (\varepsilon(\tau+\sigma \rho_2(\tau)) , \varepsilon \rho(\tau)) d \sigma \left\langle  Q_\omega'   ,  u_2 \right\rangle\\
&~{} \quad +O\left( \left( \varepsilon^{\frac32} e^{-\frac12k_0 \varepsilon |\tau|} +\varepsilon^{10} \right)\| \bd{\eta}_2 \|_{L^2\times L^2} +\varepsilon  e^{- \frac12 k_0 \varepsilon |\tau|}   \| {\bd\eta}_2 \|_{L^2\times L^2}^2 \right).
\end{aligned}
\end{equation}
Now we deal with $\mathcal F_{9,1}$. We have from \eqref{def_M}
\[
\begin{aligned}
\left\langle  \partial_x  J \bd M , J \bd{\eta}_2 \right\rangle = &~{} \left\langle  \partial_x  J  \begin{pmatrix}   c\, \partial_x^2 \eta_2  + \eta_2 +  Q_\omega u_2  \\  a\, \partial_x^2  u_2 + u_2 +  R_\omega u_2 +  Q_\omega \eta_2    \end{pmatrix}  , J \bd{\eta}_2 \right\rangle \\
&~{} +\left\langle  \partial_x  J  \begin{pmatrix}    W_2 u_2  \\    (W_1 +h)u_2 + W_2 \eta_2    \end{pmatrix}  , J \bd{\eta}_2 \right\rangle\\
= &~{}  -\left\langle     \begin{pmatrix}     a\, \partial_x^2  u_2 + u_2 +  R_\omega u_2 +  Q_\omega \eta_2  \\ c\, \partial_x^2 \eta_2  + \eta_2 +  Q_\omega u_2   \end{pmatrix}  ,   \partial_x  \begin{pmatrix}     u_2  \\ \eta_2   \end{pmatrix}  \right\rangle \\
&~{} + O\left( \left(\|\bd{W}^\sharp\|_{L^\infty\times L^\infty} +\|h\|_{L^\infty}  \right) \| \bd{\eta}_2\|_{H^1\times H^1}^2 \right)\\
= &~{} -\int \left( R_\omega u_2 \partial_x u_2 +  Q_\omega \partial_x (\eta_2  u_2 )\right) + O\left( \varepsilon e^{-l_0 \varepsilon |\tau|} \| \bd{\eta}_2\|_{H^1\times H^1}^2 \right) \\
= &~{} \int \left( \frac12R_\omega' u_2^2 + Q_\omega' \eta_2u_2 \right) + O\left( \varepsilon e^{-l_0 \varepsilon |\tau|} \| \bd{\eta}_2\|_{H^1\times H^1}^2 \right).
\end{aligned}
\]
Therefore,
\begin{equation}\label{mF91}
\begin{aligned}
\mathcal F_{9,1}  = &~{} -(1+ \rho_2'(\tau))  \int \left( \frac12R_\omega' u_2^2 + Q_\omega' \eta_2u_2 \right) + O\left( \varepsilon e^{-l_0 \varepsilon |\tau|} \| \bd{\eta}_2\|_{H^1\times H^1}^2 \right).
\end{aligned}
\end{equation}
Notice that the largest part of this term, multiplied by $-\rho'$, will cancel out with the one with opposite sign appearing in $\mathcal F_7$ in \eqref{mF7}. Finally, we deal with $\mathcal F_{9,3}$. Using \eqref{decoeqU},
\[
\begin{aligned}
& \left\langle  \partial_x \begin{pmatrix} a\, \partial_x^2 U_2 + U_2 + U_2 (U_1 + h (\tau))  \\ c\, \partial_x^2 U_1 + U_1  + \frac12 U_2^2 \end{pmatrix} , \begin{pmatrix} u_2 \\ \eta_2 \end{pmatrix}  \right\rangle
\\
& =  \omega  \left\langle (1-\partial_x^2) \partial_x \begin{pmatrix} R_\omega  \\ Q_\omega \end{pmatrix} ,\begin{pmatrix} u_2 \\ \eta_2 \end{pmatrix}  \right\rangle
\\
&\quad +   \left\langle \partial_x \begin{pmatrix} a\, \partial_x^2 W_2 + W_2 + W_1 Q_\omega + W_2 R_\omega  +W_1W_2  \\ c\, \partial_x^2 W_1 + W_1 + Q_\omega W_2  + \frac12 W_2^2 \end{pmatrix} ,   \begin{pmatrix} u_2 \\ \eta_2 \end{pmatrix}  \right\rangle.
\end{aligned}
\]
Similar to \eqref{cotaF3}, it is not difficult to see that
\[
\begin{aligned}
& \left|   \left\langle \partial_x \begin{pmatrix} a\, \partial_x^2 W_2 + W_2 + W_1 Q_\omega + W_2 R_\omega  +W_1W_2  \\ c\, \partial_x^2 W_1 + W_1 + Q_\omega W_2  + \frac12 W_2^2 \end{pmatrix} ,  \begin{pmatrix} u_2 \\ \eta_2 \end{pmatrix} \right\rangle \right|
\\
& \quad \leq \left|   \left\langle \begin{pmatrix} a\, \partial_x^3 W_2 +  \partial_x W_2    \\ c\, \partial_x^3 W_1 +  \partial_x W_1    \end{pmatrix} ,   \begin{pmatrix} u_2 \\ \eta_2 \end{pmatrix} \right\rangle \right|
\\
& \quad \quad + \left|   \left\langle \begin{pmatrix}   \partial_x W_1 Q_\omega + W_1 \partial_x Q_\omega  +  \partial_xW_2 R_\omega+W_2 \partial_x R_\omega   \\   \partial_xQ_\omega W_2  + Q_\omega \partial_x W_2    \end{pmatrix} ,   \begin{pmatrix} u_2 \\ \eta_2 \end{pmatrix} \right\rangle \right|
\\
& \quad \quad + \left|   \left\langle \begin{pmatrix}  \partial_x W_1W_2+  W_1\partial_x W_2  \\  W_2  \partial_x W_2 \end{pmatrix} ,   \begin{pmatrix} u_2 \\ \eta_2 \end{pmatrix} \right\rangle \right|
\\
& \quad =: \mathcal F_{9,3,1} +\mathcal F_{9,3,2} +\mathcal F_{9,3,3}.
\end{aligned}
\]
Now we shall use \eqref{estW}:
\[
|\mathcal F_{9,3,1}| \lesssim \|\partial_x \bd{W}^\sharp(\tau)\|_{H^2\times H^2} \|\bd{\eta}_2(\tau)\|_{L^2\times L^2} \lesssim  \varepsilon e^{-k_0 \varepsilon |\tau|} \|\bd{\eta}_2(\tau)\|_{L^2\times L^2}.
\]
\[
\begin{aligned}
|\mathcal F_{9,3,2}| \lesssim &~{} \|  \bd{W}^\sharp(\tau)\|_{L^\infty \times L^\infty} \| \bd{Q}_\omega \|_{L^2\times L^2} \|\bd{\eta}_2(\tau)\|_{L^2\times L^2}\\
&~{} + \|  \partial_x \bd{W}^\sharp(\tau)\|_{L^2 \times L^2} \| \bd{Q}_\omega \|_{L^\infty\times L^\infty} \|\bd{\eta}_2(\tau)\|_{L^2\times L^2}
\\
 \lesssim &~{}  \varepsilon e^{-k_0 \varepsilon |\tau|} \|\bd{\eta}_2(\tau)\|_{L^2\times L^2}.
\end{aligned}
\]
Finally,
\[
\begin{aligned}
|\mathcal F_{9,3,3}| \lesssim &~{}  \|\bd{W}^\sharp(\tau)\|_{L^\infty \times L^\infty} \|\partial_x \bd{W}^\sharp(\tau)\|_{L^2\times L^2} \|\bd{\eta}_2(\tau)\|_{L^2\times L^2}\\
 \lesssim &~{}  \varepsilon^2 e^{-2k_0 \varepsilon |\tau|} \|\bd{\eta}_2(\tau)\|_{L^2\times L^2}.
\end{aligned}
\]
We conclude that
\begin{equation}\label{mF93}
\begin{aligned}
 \mathcal F_{9,3} = &~{} - \rho_2'(\tau)  \left\langle  \partial_x \begin{pmatrix} a\, \partial_x^2 U_2 + U_2 + U_2 (U_1 + h (\tau))  \\ c\, \partial_x^2 U_1 + U_1  + \frac12 U_2^2 \end{pmatrix} , J \bd{\eta}_2 \right\rangle \\
 =&~{}   - \rho_2'(\tau)  \omega  \left\langle (1-\partial_x^2)J\bd{Q}'_\omega ,  \bd{\eta}_2 \right\rangle + O\left( \varepsilon e^{-k_0 \varepsilon |\tau|} |\rho_2'(t)| \left\| \bd{\eta}_2 \right\|_{H^1\times H^1} \right).
\end{aligned}
\end{equation}
 Gathering estimates \eqref{mF9conj}, \eqref{mF91} and \eqref{mF93}, we obtain
\begin{equation}\label{mF9}
\begin{aligned}
\mathcal F_{9} =&~{}  -\omega \left\langle  (1-\partial_x^2) \partial_\tau \bd{\eta}_2 , J \bd{\eta}_2\right\rangle \\
=  &~{}  \omega  (1+ \rho_2'(\tau))  \int \left( \frac12R_\omega' u_2^2 + Q_\omega' \eta_2u_2 \right) +\rho_2'(\tau)  \omega^2  \left\langle (1-\partial_x^2)J\bd{Q}'_\omega ,  \bd{\eta}_2 \right\rangle \\
&~{} +\omega \left( 1+ \rho_2'(\tau)\right) \varepsilon^2   \rho_2(\tau)  \int_0^1   (\partial_s h_0) (\varepsilon(\tau+\sigma \rho_2(\tau)) , \varepsilon \rho(\tau)) d \sigma \left\langle  Q_\omega'   ,  u_2 \right\rangle\\
&~{} +O\left( \left( \varepsilon^{\frac32} e^{-\frac12k_0 \varepsilon |\tau|} +\varepsilon^{10} \right)\| \bd{\eta}_2 \|_{L^2\times L^2} +\varepsilon  e^{- \frac12 k_0 \varepsilon |\tau|}   \| {\bd\eta}_2 \|_{L^2\times L^2}^2 \right) \\
& ~{} + O\left(  \varepsilon e^{-k_0 \varepsilon |\tau|} |\rho_2'(t)| \left\| \bd{\eta}_2 \right\|_{H^1\times H^1} \right).
\end{aligned}
\end{equation}
Notice that the first term in the last computation above cancels with a term coming from $\mathcal F_7$ in \eqref{mF7} and $\mathcal F_3$ in \eqref{mF3}. Additionally, the second term in the last computation above cancels with a term related to $\mathcal F_{3,1}$ in \eqref{mF31}. Adding the $O(\varepsilon^2)$ terms in $\mathcal F_2$ and $\mathcal F_9$ we get the problematic term
\begin{equation}\label{muerete1}
\begin{aligned}
&- \left( 1+ \rho_2'(\tau)\right) \varepsilon^2   \rho_2(\tau)  \int_0^1   (\partial_s h_0) (\varepsilon(\tau+\sigma \rho_2(\tau)) , \varepsilon \rho(\tau)) d \sigma \\
& \quad \times  \left( - \left\langle  Q_\omega'   ,  (1- \partial_x^2)^{-1}M_1 \right\rangle +\omega \left\langle  Q_\omega'   ,  u_2 \right\rangle  \right) \\
& \quad =: m_{-1}(\tau) \langle \widetilde{\bf{Q}}_\omega, \bd{\eta}_2\rangle + O\left( |\rho_2'(\tau)| |m_{-1}(\tau) | \|\bd{\eta}_2\|_{H^1\times H^1} +  |m_{-1}(\tau) | \|\bd{\eta}_2\|_{H^1\times H^1}^2\right).
\end{aligned}
\end{equation}
Indeed, to be precise (see also \eqref{def_m0}),
\[
m_{-1} := -  \varepsilon^2   \rho_2(\tau)  \int_0^1   (\partial_s h_0) (\varepsilon(\tau+\sigma \rho_2(\tau)) , \varepsilon \rho(\tau)) d \sigma = m_0 .
\]
Also, if $\widetilde Q_\omega:= (1-\partial_x^2)^{-1} Q_\omega,$
\[
\ba
 \widetilde{\bf{Q}}_\omega:= &~{}  \partial_x \begin{pmatrix} (1+c\partial_x^2) \widetilde Q_\omega  \\  -\omega Q_\omega+  \frac12  \widetilde Q_\omega^2 -\frac12(\partial_x  \widetilde Q_\omega)^2 \end{pmatrix} \\
 = &~{}  \begin{pmatrix} (1+c\partial_x^2)  (1-\partial_x^2)^{-1} \partial_x Q_\omega  \\  -\omega(1-\partial_x^2) (1-\partial_x^2)^{-1}  \partial_x Q_\omega+ Q_\omega   (1-\partial_x^2)^{-1} \partial_x Q_\omega  \end{pmatrix}.
\ea
\]
Notice also that
\begin{equation}\label{LtQ}
\begin{aligned}
& \mathcal L (1-\partial_x^2)^{-1}  \partial_x Q_\omega \begin{pmatrix} 1\\0  \end{pmatrix}
\\ &~{} =
 \begin{pmatrix}
 c \partial_x^2 +1  &   -\omega (1-\partial^2_x) +Q_\omega   \\
  -\omega (1-\partial^2_x)  +Q_\omega & a \partial_x^2 +1 +R_\omega
\end{pmatrix} (1-\partial_x^2)^{-1}  \partial_x Q_\omega \begin{pmatrix} 1\\0  \end{pmatrix}
=  \widetilde{ \bd{Q}}_\omega.
\end{aligned}
\end{equation}
Now we treat $\mathcal F_{10}$ in \eqref{boussinesq_final_23}. Using \eqref{cotas_m0},
\begin{equation}\label{mF10}
|\mathcal F_{10}| \lesssim \left| m_0'(\tau) \int Q_\omega u_2 \right| \lesssim  K_2\varepsilon^{\frac52-\delta} e^{- \frac12k_0 \varepsilon |\tau|} \| \bd{\eta}_2 \|_{H^1\times H^1}. 
\end{equation}
Now we treat $\mathcal F_{11}$. We have from \eqref{cotas_m0} and \eqref{DS},
\begin{equation}\label{mF11}
\begin{aligned}
\mathcal F_{11} = &~{} -m_0(\tau) \int \partial_\tau Q_\omega u_2 = -m_0(\tau) \omega' \int  \Lambda Q_\omega u_2 + m_0(\tau)  \rho' \int Q_\omega' u_2\\
= &~{}  m_0(\tau)\omega \int Q_\omega' u_2 +O\left(  K_2\varepsilon^{\frac52-\delta}  e^{-k_0 \varepsilon |\tau|} \| \bd{\eta}_2\|_{L^2\times L^2} \right).
\end{aligned}
\end{equation}
Finally, we treat $\mathcal F_{12}$. Using again \eqref{boussinesq_final_22}, and denoting by $(\cdot)_2$ the second component of a vector,  we have
\[
\begin{aligned}
\mathcal F_{12} = &~{} -m_0(\tau) \int  \partial_\tau u_2 Q_\omega \\
= &~{}    (1+ \rho_2'(\tau)) m_0(\tau) \left\langle \left( (1- \partial_x^2)^{-1}\partial_x  JM \right)_2,  Q_\omega \right\rangle    \\
&~{} + \rho_2'(\tau) m_0(\tau) \left\langle  (1- \partial_x^2)^{-1}\partial_x \left( c\, \partial_x^2 U_1 + U_1  + \frac12 U_2^2\right) , Q_\omega \right\rangle  \\
&~{} - (1+  \rho_2'(\tau))  c_1  m_0(\tau) \left\langle  (1- \partial_x^2)^{-1} \left(  \partial_\tau^2 \partial_x h (\tau +  \rho_2(\tau)) - \partial_\tau^2 \partial_x h (\tau)\right) , Q_\omega \right\rangle
\\
&~{} +  m_0(\tau) \left\langle  (1- \partial_x^2)^{-1}\bd{R}^\sharp_2 , Q_\omega \right\rangle \\
= :&~{} \mathcal F_{12,1} +\mathcal F_{12,2} +\mathcal F_{12,3}+\mathcal F_{12,4}.
\end{aligned}
\]
Notice that $\mathcal F_{12,3}$ and $\mathcal F_{12,4}$ can be quickly estimated: using \eqref{hypoH}, \eqref{cotas_m0} and \eqref{Cota_R},
\begin{equation}\label{mF123mF124}
\ba
|\mathcal F_{12,3}|+|\mathcal F_{12,4}| \lesssim &~{} C K_2\varepsilon^{4+\frac32-\delta} e^{- \frac12k_0 \varepsilon |\tau|}+ C K_2\varepsilon^{\frac32-\delta} e^{- \frac12k_0 \varepsilon |\tau|} (\varepsilon^{\frac32} e^{-k_0 \varepsilon |t|} +\varepsilon^{10})\\
\lesssim &~{} C K_2\varepsilon^{3-\delta} e^{- \frac12k_0 \varepsilon |\tau|}.
\ea
\end{equation}
Now we deal with $\mathcal F_{12,1}$ and $\mathcal F_{12,2}$. First, using \eqref{def_M},
\[
\ba
\mathcal F_{12,1} =&~{}  (1+ \rho_2'(\tau)) m_0(\tau) \left\langle (1- \partial_x^2)^{-1}\partial_x  \left( c\, \partial_x^2 \eta_2  + \eta_2  + U_2 u_2 + \frac12u_2^2 \right),  Q_\omega \right\rangle \\
=&~{} - (1+ \rho_2'(\tau)) m_0(\tau) \left\langle   c\, \partial_x^2 \eta_2  + \eta_2  + U_2 u_2 + \frac12u_2^2 ,   (1- \partial_x^2)^{-1}\partial_x Q_\omega \right\rangle \\
=&~{} - (1+ \rho_2'(\tau)) m_0(\tau) \left\langle   c\, \partial_x^2 \eta_2  + \eta_2  + Q_\omega u_2  ,   (1- \partial_x^2)^{-1}\partial_x Q_\omega \right\rangle \\
 &~{} - (1+ \rho_2'(\tau)) m_0(\tau) \left\langle    W_2 u_2 + \frac12u_2^2 ,   (1- \partial_x^2)^{-1}\partial_x Q_\omega \right\rangle\\
 =: &~{} \mathcal F_{12,1,1} + \mathcal F_{12,1,2}.
\ea
\]
First, from \eqref{cotas_m0}, \eqref{estW} and \eqref{ModulationTesis2},
\begin{equation}\label{F1212}
\ba
|\mathcal F_{12,1,2}| \lesssim &~{} K_2^2 \varepsilon^{3-\delta} e^{- \frac12k_0 \varepsilon |\tau|} + K_2^3 \varepsilon^{\frac52-\delta} e^{- \frac12k_0 \varepsilon |\tau|} \lesssim K_2^3 \varepsilon^{\frac52-\delta} e^{- \frac12k_0 \varepsilon |\tau|}.
\ea
\end{equation}
Second, from \eqref{def_L} and \eqref{LtQ},
\begin{equation}\label{F1211}
\ba
\mathcal F_{12,1,1}= &~{}  - (1+ \rho_2'(\tau)) m_0(\tau) \left\langle   (\mathcal L \bd{\eta}_2)_1 +\omega  (1- \partial_x^2)u_2 ,   (1- \partial_x^2)^{-1}\partial_x Q_\omega \right\rangle \\
= &~{} - (1+ \rho_2'(\tau)) m_0(\tau) \left\langle   \mathcal L \bd{\eta}_2  ,   (1- \partial_x^2)^{-1}\partial_x Q_\omega \bp 1 \\ 0 \ep \right\rangle \\
&~{} - \omega (1+ \rho_2'(\tau)) m_0(\tau) \left\langle    u_2 ,  \partial_x Q_\omega \right\rangle \\
= &~{} - (1+ \rho_2'(\tau)) m_0(\tau) \left\langle  \bd{\eta}_2  ,  \bd{\widetilde{Q}}_\omega \right\rangle  - \omega (1+ \rho_2'(\tau)) m_0(\tau) \left\langle    u_2 ,  \partial_x Q_\omega \right\rangle.
\ea
\end{equation}
Notice that the second term above cancels out at first order with the first term in \eqref{mF11}. Also, given the choice of $m_0$ in \eqref{def_m0}, the first term above cancels at first order with the dangerous term \eqref{muerete1}.
Finally, using \eqref{decoeqU},
\begin{align*}
\mathcal F_{12,2} =&~{} \rho_2'(\tau) m_0(\tau) \left\langle  (1- \partial_x^2)^{-1}\partial_x \left( c\, \partial_x^2 U_1 + U_1  + \frac12 U_2^2\right) , Q_\omega \right\rangle \\
=&~{} \rho_2'(\tau) \omega m_0(\tau) \left\langle  \partial_x   Q_\omega , Q_\omega \right\rangle \\
&~{} +  \rho_2'(\tau) m_0(\tau) \left\langle  (1- \partial_x^2)^{-1}\partial_x \left( c\, \partial_x^2 W_1 + W_1 + Q_\omega W_2  + \frac12 W_2^2 \right) , Q_\omega \right\rangle \\
= &~{} \rho_2'(\tau) m_0(\tau)   \left\langle   \left( c\, \partial_x^3 W_1 + \partial_x W_1 + \partial_x(Q_\omega W_2)  +  W_2\partial_x W_2 \right) , (1- \partial_x^2)^{-1} Q_\omega \right\rangle .
\end{align*}
Now we estimate using \eqref{estimationForRho1}-\eqref{ModulationTesis2}, \eqref{cotas_m0}, and \eqref{estW}:
\begin{align}
|\mathcal F_{12,2} | \lesssim &~{} |\rho_2'(\tau) || m_0(\tau)| \nonumber \\
&~{} \qq  \times \left| \left\langle   \left( c\, \partial_x^3 W_1 + \partial_x W_1 + \partial_x(Q_\omega W_2)  +  W_2\partial_x W_2 \right) , (1- \partial_x^2)^{-1} Q_\omega \right\rangle \right| \label{mF122} \\
 \lesssim &~{} K_2^2 \varepsilon^{3-\delta} e^{- \frac12k_0 \varepsilon |\tau|}.\nonumber 
\end{align}
It is not difficult to see using \eqref{mF123mF124}, \eqref{F1212}, \eqref{F1211} and \eqref{mF122} that
\begin{align}
\mathcal F_{12} = &~{} 
- (1+ \rho_2'(\tau)) m_0(\tau) \left\langle  \bd{\eta}_2  ,  \bd{\widetilde{Q}}_\omega \right\rangle  - \omega (1+ \rho_2'(\tau)) m_0(\tau) \left\langle    u_2 ,  \partial_x Q_\omega \right\rangle \nonumber \\
& ~{} +O\left( K_2\varepsilon^{3-\delta} e^{- \frac12k_0 \varepsilon |\tau|}
+ K_2^3 \varepsilon^{\frac52-\delta} e^{- \frac12k_0 \varepsilon |\tau|}
+ K_2^2 \varepsilon^{3-\delta} e^{- \frac12k_0 \varepsilon |\tau|} \right) \nonumber \\
= &~{} 
-  m_0(\tau) \left\langle  \bd{\eta}_2  ,  \bd{\widetilde{Q}}_\omega \right\rangle  - \omega   m_0(\tau) \left\langle    u_2 ,  \partial_x Q_\omega \right\rangle \label{mF12} \\
& ~{} + O\left( |\rho_2'(\tau)|  K_2\varepsilon^{\frac32-\delta} e^{- \frac12k_0 \varepsilon |\tau|} \| \bd{\eta}_2\|_{L^2\times L^2} \right)\nonumber \\
&~{}  +O\left( K_2\varepsilon^{3-\delta} e^{- \frac12k_0 \varepsilon |\tau|} + K_2^3 \varepsilon^{\frac52-\delta} e^{- \frac12k_0 \varepsilon |\tau|}+ K_2^2 \varepsilon^{3-\delta} e^{- \frac12k_0 \varepsilon |\tau|} \right) . \nonumber 
\end{align}
Finally, gathering \eqref{mF1}, \eqref{mF2}, \eqref{mF3}, \eqref{mF4}, \eqref{mF5}, \eqref{mF6}, \eqref{mF7}, \eqref{mF8}, \eqref{mF9}, \eqref{mF10}, \eqref{mF11} and \eqref{mF12} in \eqref{boussinesq_final_23}, we obtain
\begin{align*}
 |\mathfrak F_2'(\tau) |\lesssim &~{} \left| \sum_{j=1}^{12} \mathcal F_j(\tau) \right|
 \\
 \lesssim &~{}
 \varepsilon^{2}(1+\varepsilon | \rho_2(\tau)| +| \rho_2(\tau)| \| \bd{\eta}_2 \|_{L^2\times L^2} ) e^{- \frac12k_0 \varepsilon |\tau|} \left( \| \bd{\eta}_2 \|_{L^2\times L^2} + \| \bd{\eta}_2 \|_{H^1\times H^1}^2\right) \\
&~{} +  \varepsilon  |\rho_2'(\tau)|  e^{-k_0 \varepsilon |\tau |}\|\bd{\eta}_2\|_{L^2\times L^2}
\\
&~{} + \varepsilon  e^{- \frac12 k_0 \varepsilon |\tau|}  \left( \| \bd{\eta}_2 \|_{L^2\times L^2}^2 + \| \bd{\eta}_2 \|_{H^1\times H^1}^3\right) \\
& ~{} + \varepsilon^{\frac72}  e^{- \frac12 k_0 \varepsilon |\tau|}  \left(  \| \bd{\eta}_2 \|_{L^2\times L^2} + \| \bd{\eta}_2 \|_{H^1\times H^1}^2 \right)\\
&~{} +\left( \varepsilon^{\frac32} e^{-k_0 \varepsilon |\tau|} +\varepsilon^{10} \right)\left(\| \bd{\eta}_2 \|_{L^2\times L^2} + \| \bd{\eta}_2 \|_{H^1\times H^1}^2 \right)\\
&~{} +\varepsilon e^{-k_0 \varepsilon |\tau| }  \| \bd{\eta}_2(\tau) \|_{H^1\times H^1}^2\\
&~{} + \varepsilon e^{-k_0 \varepsilon |\tau| -l_0 \varepsilon |\rho(\tau)|}  \| \bd{\eta}_2(\tau) \|_{H^1\times H^1}^2\\
&~{} +  \left( \varepsilon^{\frac32} e^{-\frac12k_0 \varepsilon |\tau|} +\varepsilon^{10} \right)\| \bd{\eta}_2 \|_{L^2\times L^2} +\varepsilon  e^{- \frac12 k_0 \varepsilon |\tau|}   \| {\bd\eta}_2 \|_{L^2\times L^2}^2
%
\\ &~{}+  \varepsilon e^{-k_0 \varepsilon |\tau|} |\rho_2'(t)| \left\| \bd{\eta}_2 \right\|_{H^1\times H^1}  \\
&~{} + K_2\varepsilon^{\frac52-\delta} e^{- \frac12k_0 \varepsilon |\tau|} \| \bd{\eta}_2 \|_{H^1\times H^1} + K_2\varepsilon^{\frac52-\delta}  e^{-k_0 \varepsilon |\tau|} \| \bd{\eta}_2\|_{L^2\times L^2}\\
& ~{} +  |\rho_2'(\tau)|  K_2\varepsilon^{\frac32-\delta} e^{- \frac12k_0 \varepsilon |\tau|} \| \bd{\eta}_2\|_{L^2\times L^2}  \\
&~{}  + K_2\varepsilon^{3-\delta} e^{- \frac12k_0 \varepsilon |\tau|} + K_2^3 \varepsilon^{\frac52-\delta} e^{- \frac12k_0 \varepsilon |\tau|}+ K_2^2 \varepsilon^{3-\delta} e^{- \frac12k_0 \varepsilon |\tau|} .
\end{align*}
Simplifying, we get
\begin{equation*}
\begin{aligned}
 |\mathfrak F_2'(\tau) |  \lesssim &~{}
 \varepsilon^{2}(1+\varepsilon | \rho_2(\tau)| +| \rho_2(\tau)| \| \bd{\eta}_2 \|_{L^2\times L^2} )  e^{- \frac12k_0 \varepsilon |\tau|} \left( \| \bd{\eta}_2 \|_{L^2\times L^2} + \| \bd{\eta}_2 \|_{H^1\times H^1}^2\right) \\
&~{} +  \varepsilon  |\rho_2'(\tau)|  e^{-k_0 \varepsilon |\tau |}\|\bd{\eta}_2\|_{H^1\times H^1}  + \varepsilon  e^{- \frac12 k_0 \varepsilon |\tau|}  \left( \| \bd{\eta}_2 \|_{H^1\times H^1}^2 + \| \bd{\eta}_2 \|_{H^1\times H^1}^3\right) \\
&~{} +\left( \varepsilon^{\frac32} e^{-k_0 \varepsilon |\tau|} +\varepsilon^{10} \right)\left(\| \bd{\eta}_2 \|_{L^2\times L^2} + \| \bd{\eta}_2 \|_{H^1\times H^1}^2 \right) .
\end{aligned}
\end{equation*}
With this estimate, we finally get \eqref{boussinesq_final_25}.
\end{proof}

\subsection{Description of the interaction region} Taking into account Lemma \ref{ModulationInTime}, in particular \eqref{estimationForRho1}, we conclude from \eqref{boussinesq_final_25} that
\begin{equation*}
\begin{aligned}
 |\mathfrak F_2'(\tau) |\lesssim &~{}
  \varepsilon^{2} (1+\varepsilon | \rho_2(\tau)| +| \rho_2(\tau)| \| \bd{\eta}_2 \|_{L^2\times L^2} )  e^{- \frac12k_0 \varepsilon |\tau|} \left( \| \bd{\eta}_2 \|_{L^2\times L^2} + \| \bd{\eta}_2 \|_{H^1\times H^1}^2\right) \\
&~{}+ \varepsilon  e^{- \frac12 k_0 \varepsilon |\tau|}  \left( \| \bd{\eta}_2 \|_{H^1\times H^1}^2 + \| \bd{\eta}_2 \|_{H^1\times H^1}^3\right) \\
&~{} +\left( \varepsilon^{\frac32} e^{-k_0 \varepsilon |\tau|} +\varepsilon^{10} \right)\left(\| \bd{\eta}_2 \|_{H^1\times H^1} + \| \bd{\eta}_2 \|_{H^1\times H^1}^2 \right) .
\end{aligned}
\end{equation*}
Now we use \eqref{ModulationTesis2} in linear and cubic terms:
\begin{equation}\label{boussinesq_final_24}
\begin{aligned}
 |\mathfrak F_2'(\tau) |\lesssim &~{}
  \varepsilon^{2}(1+\varepsilon | \rho_2(\tau)| +| \rho_2(\tau)| \| \bd{\eta}_2 \|_{L^2\times L^2} ) e^{- \frac12k_0 \varepsilon |\tau|} \left( \| \bd{\eta}_2 \|_{L^2\times L^2} + \| \bd{\eta}_2 \|_{H^1\times H^1}^2\right) \\
&~{}+ \varepsilon  e^{- \frac12 k_0 \varepsilon |\tau|}  \left( 1 +K_2  \varepsilon^{\frac12}  \right)\| \bd{\eta}_2 \|_{H^1\times H^1}^2  \\
&~{} +\varepsilon^{10}  \| \bd{\eta}_2 \|_{H^1\times H^1}^2 + K_2 \varepsilon^{2}  \left( e^{-k_0 \varepsilon |\tau|} +\varepsilon^{8} \right) .
\end{aligned}
\end{equation}
Now we use \eqref{Coer_F2_new}, \eqref{Bound_F2} and \eqref{ModulationTesis3} to conclude that
\begin{equation}\label{boussinesq_final_26}
\begin{aligned}
 c_2\| \bd{\eta}_2(\tau)\|_{H^1\times H^1}^2 \leq  &~{}  \mathfrak F_2(\tau) +C(\varepsilon + (K_2+K_2^2)^2\varepsilon^2) +K_2^2\varepsilon^{2-\delta} \\
 \leq &~{}\mathfrak F_2(\tau) -\mathfrak F_2(-\hat T_\varepsilon) + C \varepsilon + C (K_2+K_2^2)^2\varepsilon^{2-\delta} \\
 \leq &~{}  C(  \varepsilon +  (K_2+K_2^2)^2\varepsilon^{2-\delta} )+ \int_{-\hat T_\varepsilon}^{\tau} |\mathfrak F_2'(\sigma)| d\sigma .
\end{aligned}
\end{equation}
Now we shall use estimates \eqref{cotas_rho2} in some easy parts in the estimate \eqref{boussinesq_final_24}. We get
\begin{equation}\label{boussinesq_final_27}
\begin{aligned}
 |\mathfrak F_2'(\tau) |\lesssim &~{}
  \varepsilon^{\frac52-\delta} K_2^2  e^{- \frac12k_0 \varepsilon |\tau|} \\
 &~{} +  K_2 \varepsilon^{\frac32-\delta} e^{- \frac12k_0 \varepsilon |\tau|} \| \bd{\eta}_2 \|_{H^1\times H^1}^2 + \varepsilon  e^{- \frac12 k_0 \varepsilon |\tau|} \| \bd{\eta}_2 \|_{H^1\times H^1}^2  \\
&~{} +\varepsilon^{10}  \| \bd{\eta}_2 \|_{H^1\times H^1}^2 + K_2 \varepsilon^{2}  \left( e^{- \frac12k_0 \varepsilon |\tau|} +\varepsilon^{8} \right) .
\end{aligned}
\end{equation}
Now we integrate in time using \eqref{boussinesq_final_27}:
\[
\begin{aligned}
\int_{-\hat T_\varepsilon}^{\tau} |\mathfrak F_2'(\sigma)| d\sigma \leq  &~{} K_2^2 \varepsilon^{\frac32-\delta} + K_2 \varepsilon  \\
&~{}  + \int_{-\hat T_\varepsilon}^{\tau} \left( K_2\varepsilon^{10} + \left(\varepsilon + K_2 \varepsilon^{\frac32-\delta} \right) e^{- \frac12 k_0 \varepsilon |\sigma|}  \right) \| \bd{\eta}_2 (\sigma)\|_{H^1\times H^1}^2 d\sigma
\end{aligned}
\]
Notice that $\int_{-\hat T_\varepsilon}^{\tau}   \left(K_2\varepsilon^{10} + \left(\varepsilon + K_2 \varepsilon^{\frac32-\delta} \right) e^{- \frac12 k_0 \varepsilon |\sigma|} \right) d\sigma \leq C $, independent of $K_2$, provided $K_2$ is chosen large and then $\varepsilon$ is chose sufficiently small. Coming back to \eqref{boussinesq_final_26},
\begin{equation}\label{boussinesq_final_28}
\begin{aligned}
 c_2\| \bd{\eta}_2(\tau)\|_{H^1\times H^1}^2 \leq  &~{}  C(  \varepsilon +  (K_2+K_2^2)^2\varepsilon^{2-\delta} )
+ \int_{-\hat T_\varepsilon}^{\tau} m(\sigma)  \| \bd{\eta}_2 (\sigma)\|_{H^1\times H^1}^2 d\sigma,
\end{aligned}
\end{equation}
with $\int_{-\hat T_\varepsilon}^{\tau} |m(\sigma)| \leq C$.
Notice that we can bound $\vert \mathfrak F_2(-T_\varepsilon)\vert$ by $K\varepsilon$  ($K$ independent of $K_2$) using \eqref{Bound_F2} and \eqref{ModulationTesis3}. Finally, thanks to Gronwall's inequality applied to \eqref{boussinesq_final_28}, taking $K_2$ greater if necessary, and then $\varepsilon$ smaller, 
\[
\Vert {\bd \eta}_2\Vert_{H^1\times H^1}^2\le \frac14 K_2^2 \varepsilon. 
\]
Therefore, in the variable $\tau$, we obtain $\Vert {\bd \eta}_2(\tau(T_\varepsilon))\Vert_{H^1\times H^1} \le \frac12 K_2 \varepsilon^{\frac12}$. Coming back to the $t$ variable, $\Vert {\bd \eta}_2(t=T_\varepsilon )\Vert_{H^1\times H^1}^2\le \frac12 K_2 \varepsilon^{\frac12}$, a contradiction with the definition of $T_2$ in \eqref{T2}. Therefore, $T_2= T_\varepsilon$ and $T^* >T_\varepsilon $. Now we prove estimate \eqref{Interaction} in Theorem \ref{MT}. From \eqref{T2}, \eqref{DS} and \eqref{estW} we have
\[
 \|  \bd{\eta}(T_\varepsilon) - \bd{Q}_{\omega} (\cdot -\rho_\varepsilon) \|_{H^1\times H^1} \leq K_2 \varepsilon^{\frac12}.
\]
Here $\omega$ is the original speed of the solitary wave and $\rho_\varepsilon := \rho (T_\varepsilon + \widetilde \rho_2(T_\varepsilon))$. This proves estimate \eqref{Interaction} in Theorem \ref{MT}.

\section{End of proof of Main Theorem}\label{Sec:5}

Now we are ready to finish the proof of Theorem \ref{MT}, specifically the long-time stability estimate \eqref{Exit}.  Recall that \eqref{Interaction} is satisfied at time $t=T_\varepsilon.$ For $K_3>1$ to be fixed later, let us define
\begin{equation}\label{T3}
\begin{aligned}
T_3(K_3) := &~{}\sup \Big\{ T>T_\varepsilon ~ : ~ \hbox{for all } t\in [T_\varepsilon, T], ~ \hbox{ there exists $\tilde\rho_3(t)\in\mathbb R$} \\
&~{} \qquad  \hbox{ such that } ~   \|  \bd{\eta}(t) - {\bd Q}_{\omega_2}(\cdot - \tilde \rho_3(t)) \|_{H^1\times H^1} \leq K_3 \varepsilon^{\frac32}\Big\}.
\end{aligned}
\end{equation}
The objective is to show that for $K_3$ large but fixed, $0<\varepsilon_3<\varepsilon_2$ sufficiently small and $0<\varepsilon<\varepsilon_3$, we have $T_3=+\infty.$ Let us assume, by contradiction, that for all $K_3>0$ large, $\varepsilon>0$ small, we have $T_3<+\infty.$

\subsection{Modulation} Since $T_3<+\infty$, we have from \eqref{T3} that
\[
 \|  \bd{\eta}(t) - {\bd Q}_{\omega_2}(\cdot - \tilde \rho_3(t)) \|_{H^1\times H^1} \leq K_3 \varepsilon^{\frac32}
\]
is valid for all $T_\varepsilon <t \leq T_3$, and some  $\tilde \rho_3(t) \in\mathbb R$. In particular, the solution $ \bd{\eta}(t)$ is well-defined up to time $T_3$. Notice that the maximal time of existence of the solution is bounded below by $T_3$, thanks to \eqref{T3}. Using this boundedness, we can find a particular shift $\rho_3(t)$ satisfying an additional orthogonality condition.

\begin{lemma}\label{mod3}
There exists $C_3,\mu_3,\varepsilon_3>0$ such that, for all $0<\varepsilon<\varepsilon_3$ the following is satisfied. Let $(\eta,u)\in C(I_n, H^1\times H^1)$ be the solution to \eqref{Cauchy} constructed in Theorem \ref{MT} such that \eqref{Construction}, \eqref{PreInteraction} and \eqref{Interaction} are satisfied. Assume that $T_3$ in \eqref{T3} is finite. Then, for all $t \in [T_\varepsilon,T_3]$, there exists a $C^1$ modulation shift $\rho_3: [T_\varepsilon,T_3] \to \mathbb R$ such that
\begin{equation}\label{z3}
\bd{\eta}_3(t) := (\eta_3,u_3)(t)=\bd{\eta}(t) - {\bd Q}_{\omega_3}(\cdot - \rho_3(t)),
\end{equation}
satisfies
\begin{equation}\label{orthoz3}
\langle \bd{\eta}_3, (1- \partial_x^2){\bd Q}'_{\omega_3}(\cdot - \rho_3(t)) \rangle =0, \quad \|\bd{\eta}_3(t) \|_{H^1\times H^1}\leq  K_3 \varepsilon^{\frac32}.
\end{equation}
Moreover, one has
\begin{equation}\label{eq_z3}
\begin{cases}
&(1- \partial_x^2)\partial_t \eta_{3}  + \partial_x \! \left( a\, \partial_x^2 u_{3} +u_{3} + u_{3} R_\omega + \eta_3 Q_\omega + (Q_\omega + u_3) h \right) \\
& \qquad = (-1 +a_1 \partial_x^2) \partial_th  + \rho_3' (1- \partial_x^2)R_\omega' , \\
& (1- \partial_x^2)\partial_t u_{3}  + \partial_x \! \left( c\, \partial_x^2 \eta_3 + \eta_3 +  Q_\omega u_3 + \frac12 u_3^2 \right)
\\
& \qquad = c_1  \partial_t^2 \partial_x h + \rho_3' (1- \partial_x^2)Q_\omega' ,
\end{cases}
\end{equation}
and
\begin{equation}\label{der_rho3}
| \rho_3'(t)| \leq C \|\bd{\eta}_{3}\|_{H^1\times H^1} + C \varepsilon e^{-\varepsilon ( k_0 t + \frac9{10}l_0 t )}.
\end{equation}
\end{lemma}

\begin{proof}
The proof of Lemma \ref{mod3} is very similar to the proof of Lemma \ref{mod1}, and we omit the details. The proof of \eqref{der_rho3} is also obtained from \eqref{eq_z3} and \eqref{orthoz3}.
\end{proof}

\subsection{Energy and momentum estimates} Recall \eqref{z3}. Based on \eqref{Energy_new_new} and \eqref{exact0}, one has 
\[
\begin{aligned}
& H_h[{\bd Q}_\omega +  \bd{\eta}_3 ](t) \\
 &~{} =  H_h[{\bd Q}_\omega] + \omega \int \left( (1-\partial^2_x)R_{\omega }  u_3  + (1-\partial^2_x)Q_{\omega }  \eta_3  \right)  +  \int  Q_\omega u_3 h  \\
&~{}  \quad + \frac12\int \left( -a (\partial_x u_3)^2 -c (\partial_x \eta_3)^2  + u_3^2+ \eta_3^2 + 2Q_\omega \eta_3 u_3 + u_3^2(R_\omega + \eta_3 + h) \right) .
\end{aligned}
\]
Notice that $ \int \left( (1-\partial^2_x)R_{\omega }  u_3  + (1-\partial^2_x)Q_{\omega }  \eta_3  \right) =  \langle \bd{\eta}_3, J(1-\partial_x^2)\bd{Q}_\omega \rangle (t)$. Recall that $\omega>0$. Consequently,
\[
\begin{aligned}
& \left| \langle \bd{\eta}_3 , J(1-\partial_x^2)\bd{Q}_\omega \rangle (t) \right| \\
&~{} \leq  \left| \langle \bd{\eta}_3 , J(1-\partial_x^2)\bd{Q}_\omega \rangle (T_\varepsilon) \right|  + \left| H_h[{\bd Q}_\omega +  \bd{\eta}_3 ](t) -H_h[{\bd Q}_\omega +  \bd{\eta}_3 ](T_\varepsilon)\right|\\
&~{} \quad   +\left|\int Q_\omega u_1 h(t)\right|+\left|\int Q_\omega u_1 h(T_\varepsilon)\right|+ \left| H_h[{\bd Q}_\omega](T_\varepsilon)-  H_h[{\bd Q}_\omega](t)\right| \\
&~{} \quad  + \left|   \frac12\int \left( -a (\partial_x u_3)^2 -c (\partial_x \eta_3)^2  + u_3^2+ \eta_3^2 + 2Q_\omega \eta_3 u_3 + u_3^2(R_\omega + \eta_3 + h) \right)(t) \right| \\
&~{} \quad  + \left| \frac12\int \left( -a (\partial_x u_3)^2 -c (\partial_x \eta_3)^2  + u_3^2+ \eta_3^2 + 2Q_\omega \eta_3 u_3 + u_3^2(R_\omega + \eta_3 + h) \right)( T_\varepsilon) \right|.
\end{aligned}
\]
Using that $\| \bd{\eta}_3(T_\varepsilon)\|_{H^1\times H^1} \leq C \varepsilon^{\frac32}$, $\| \bd{\eta}_3(t)\|_{H^1\times H^1} \leq K_3 \varepsilon^{\frac32}$, and $\|\eta_3(t)\|_{L^\infty} \leq C K_3 \varepsilon^{\frac32}$, $t>T_\varepsilon$,
\begin{equation}\label{intermedia_3}
\begin{aligned}
& \left| \langle \bd{\eta}_3, J(1-\partial_x^2)\bd{Q}_\omega \rangle (t) \right| \\
&~{} \leq  \left| H_h[{\bd Q}_\omega +  \bd{\eta}_3 ](t) -H_h[{\bd Q}_\omega +  \bd{\eta}_3 ](T_\varepsilon)\right|\\
&~{} \quad + \left| H_h[{\bd Q}_\omega](T_\varepsilon)-  H_h[{\bd Q}_\omega](t)\right| + C K_3^2 \varepsilon^3 +  C K_3^3 \varepsilon^{9/2}+C\varepsilon^3.
\end{aligned}
\end{equation}
Now we apply Lemma \ref{Lem2p3}, more specifically, \eqref{est_E} and \eqref{est_E_Qc} in the particular case of $t_2=t$, $t_1=T_\varepsilon$, and $\bd{\eta}_3$ replaced by ${\bd Q}_\omega +  \bd{\eta}_3$. We have
\begin{lemma}
Assume that $t\geq T_\varepsilon$. Then one has
\be\label{est_E_final}
 \left| H_h[{\bd Q}_\omega +  \bd{\eta}_3 ](t) -H_h[{\bd Q}_\omega +  \bd{\eta}_3 ](T_\varepsilon)\right| \leq  CK_3^2 \varepsilon^4,
\ee
\be\label{est_P_final}
 \left| P[{\bd Q}_\omega +  \bd{\eta}_3 ](t) - P[{\bd Q}_\omega +  \bd{\eta}_3 ](T_\varepsilon)\right| \leq  CK_3^2 \varepsilon^4,
\ee
and
\be\label{est_E_Qc_final}
\left|  H_h[{\bd Q}_\omega](T_\varepsilon)-  H_h[{\bd Q}_\omega](t) \right|  \leq C\varepsilon^3.
\ee
\end{lemma}
\begin{proof}
Let $t\geq T_\varepsilon$. From \eqref{est_E} one has
\[
\begin{aligned}
\left| \frac{d}{dt} H_h[ {\bd Q}_\omega +  \bd{\eta}_3 ](t) \right| \lesssim  &~{}  \varepsilon^2  e^{-k_0\varepsilon |t| }   \int (u_3^2+\eta_3^2) e^{-l_0\varepsilon |x| }  \\
&~{} +  \varepsilon^{2}  e^{-k_0\varepsilon |t| } \int ( |\eta_3 | +|u_3| ) e^{-l_0\varepsilon |x| } \lesssim  K_3^2 \varepsilon^2  e^{-k_0\varepsilon t }.
\end{aligned}
\]
Therefore \eqref{est_E_final} is obtained after integration in time since $\varepsilon  e^{-k_0\varepsilon T_\varepsilon } \ll \varepsilon^4$. Similarly, from \eqref{est_P} it holds
\[
\begin{aligned}
\left| \frac{d}{dt} P[{\bd Q}_\omega +  \bd{\eta}_3 ](t) \right| \lesssim  &~{}  \varepsilon^2  e^{-k_0\varepsilon |t| }   \int u_3^2 e^{-l_0\varepsilon |x| }  \\
&~{} +  \varepsilon^{2}  e^{-k_0\varepsilon |t| } \int ( |\eta_3 | +|u_3| ) e^{-l_0\varepsilon |x| } \lesssim  \varepsilon^2 K_3^2 e^{-k_0\varepsilon t }.
\end{aligned}
\]
One easily concludes \eqref{est_P_final}.  Now we prove \eqref{est_E_Qc_final}. From \eqref{est_E_Qc} one has
\[
\begin{aligned}
& \left| H_h[\bd{Q}_\omega ](t) -H_h[\bd{Q}_\omega ](T_\varepsilon)  \right|
\\
& \quad \lesssim \varepsilon  \int  Q_\omega^2(x-\rho_3(t)) | h_0 (\varepsilon t, \varepsilon x)-h_0 (\varepsilon T_\varepsilon , \varepsilon x) |dx
\\
& \quad \lesssim \varepsilon  e^{-k_0 \varepsilon T_\varepsilon } \int  Q_\omega^2(x-\rho_3(t)) e^{-l_0 \varepsilon |x|} dx \ll \varepsilon^3.
\end{aligned}
\]
The proof is complete.
\end{proof}

Now we return to \eqref{intermedia_3}. From \eqref{est_E_final} and \eqref{est_E_Qc_final} we get
\begin{equation}\label{intermedia_3_final}
\begin{aligned}
& \left| \langle \bd{\eta}_3, J(1-\partial_x^2)\bd{Q}_\omega \rangle (t) \right| \leq C K_3^2 \varepsilon^3 +  C K_3^3 \varepsilon^{9/2}+C\varepsilon^3.
\end{aligned}
\end{equation}
Now, for a fixed constant $C>0$, \eqref{coer_1} leads to
\[
\begin{aligned}
&\| \bd{\eta}_3(t) \|_{H^1\times H^1}^2  \\
& \quad \leq C \left\langle  \bd{\eta}_3 ,  \mathcal L \bd{\eta}_3  \right\rangle (t) + C \left| \langle \bd{\eta}_3 , J(1-\partial_x^2)\bd{Q}_\omega \rangle (t) \right|^2\\
& \quad \leq C \left( \left\langle  \bd{\eta}_3   ,  \mathcal L \bd{\eta}_3  \right\rangle(t)  - \left\langle  \bd{\eta}_3  ,  \mathcal L \bd{\eta}_3  \right\rangle (T_\varepsilon )  \right)  +C \left\langle  \bd{\eta}_3 ,  \mathcal L \bd{\eta}_3  \right\rangle (T_\varepsilon )    +C \left| \langle \bd{\eta}_3 , J(1-\partial_x^2)\bd{Q}_\omega \rangle (t)\right|^2.
\end{aligned}
\]
Using  \eqref{Energy_new2},
\[
\begin{aligned}
&\| \bd{\eta}_3(t) \|_{H^1\times H^1}^2  \\
& \quad \leq C\left| H_h[{\bd Q}_\omega + \bd{\eta}_3 ](t) -\omega P[{\bd Q}_\omega +  \bd{\eta}_3](t)  -H_h[{\bd Q}_\omega + \bd{\eta}_3 ](T_\varepsilon) + \omega P[{\bd Q}_\omega +  \bd{\eta}_3](T_\varepsilon)  \right| \\
& \quad \quad + C\left| H_h[{\bd Q}_\omega](T_\varepsilon ) -\omega P[{\bd Q}_\omega] (T_\varepsilon ) -H_h[{\bd Q}_\omega](t) + \omega P[{\bd Q}_\omega] (t) \right|  \\
& \quad \quad + C\left|     \int  Q_\omega u_3 h (T_\varepsilon )  -  \int  Q_\omega u_3 h (t) \right|  \\
& \quad \quad  + C \left| \int u_3^2  \left( \eta_3 + h\right)(T_\varepsilon ) -\int u_3^2  \left( \eta_3 + h\right)(t) \right| \\
& \quad \quad  +C \left\langle  \bd{\eta}_3 ,  \mathcal L \bd{\eta}_3  \right\rangle (T_\varepsilon )    +C \left| \langle \bd{\eta}_3 , J(1-\partial_x^2)\bd{Q}_\omega \rangle (t)\right|^2.
\end{aligned}
\]
From \eqref{est_E_final}, \eqref{est_P_final}, \eqref{est_E_Qc_final} and \eqref{P_ind},
\[
\begin{aligned}
\| \bd{\eta}_3(t) \|_{H^1\times H^1}^2   \leq &~{}  C\left|     \int  Q_\omega u_3 h (T_\varepsilon )  -  \int  Q_\omega u_3 h (t) \right|  \\
& ~{}  + C \left| \int u_3^2  \left( \eta_3 + h\right)(T_\varepsilon ) -\int u_3^2  \left( \eta_3 + h\right)(t) \right| \\
& ~{} +C \left\langle  \bd{\eta}_3 ,  \mathcal L \bd{\eta}_3  \right\rangle (T_\varepsilon )    +C \left| \langle \bd{\eta}_3 , J(1-\partial_x^2)\bd{Q}_\omega \rangle (t)\right|^2 + C\varepsilon^3 +CK_3^2\varepsilon^4.
\end{aligned}
\]
Using \eqref{intermedia_3_final} and $| \left\langle  \bd{\eta}_3 ,  \mathcal L \bd{\eta}_3  \right\rangle (T_\varepsilon )| \leq C\varepsilon^3$, we get
\[
\begin{aligned}
\| \bd{\eta}_3(t) \|_{H^1\times H^1}^2   \leq &~{}  C\left|     \int  Q_\omega u_3 h (T_\varepsilon )  -  \int  Q_\omega u_3 h (t) \right|  \\
& ~{}  + C \left| \int u_3^2  \left( \eta_3 + h\right)(T_\varepsilon ) -\int u_3^2  \left( \eta_3 + h\right)(t) \right|  + C\varepsilon^3 +CK_3^2\varepsilon^4.
\end{aligned}
\]
We also have, for $t\geq T_\varepsilon,$
\[
\left| \int  Q_\omega u_3 h (t) \right| \lesssim \varepsilon e^{-k_0 \varepsilon t }  \int  Q_\omega |u_3| e^{-l_0 \varepsilon |x| } \lesssim K_3 \varepsilon e^{-k_0 \varepsilon t } \ll K_3 \varepsilon^4.
\]
Hence
\[
\begin{aligned}
\| \bd{\eta}_3(t) \|_{H^1\times H^1}^2   \leq &~{}  C \left| \int u_3^2  \left( \eta_3 + h\right)(T_\varepsilon ) -\int u_3^2  \left( \eta_3 + h\right)(t) \right|  + C\varepsilon^3 +CK_3^2\varepsilon^4.
\end{aligned}
\]
On the other hand,
\[
 \left| \int u_3^2  \eta_3 (T_\varepsilon ) -\int u_3^2  \eta_3 (t) \right| \lesssim K_3^3 \varepsilon^{9/2},
\]
and
\[
\left| \int u_3^2  h\right| \lesssim K_3^2 \varepsilon^4  e^{-k_0 \varepsilon t } \ll K_3 \varepsilon^4.
\]
We conclude that
\[
 \left| \int u_3^2  \left( \eta_3 + h\right)(T_\varepsilon ) -\int u_3^2  \left( \eta_3 + h\right)(t) \right|  \lesssim K_3^2 \varepsilon^4 +K_3^3 \varepsilon^{9/2}.
\]
Hence
\[
\begin{aligned}
\| \bd{\eta}_3(t) \|_{H^1\times H^1}^2   \leq &~{}    C\varepsilon^3 +CK_3^2\varepsilon^4 +CK_3^3 \varepsilon^{9/2}.
\end{aligned}
\]
Finally, by making $K_3$ larger if necessary, and $\varepsilon$ smaller, we obtain
\[
\| \bd{\eta}_3(t) \|_{H^1\times H^1}^2 \le \frac14 K_3^2 \varepsilon^{3},
\]
improving the previous estimate \eqref{T3}. Therefore, $T_3=+\infty$ and we conclude the proof of Theorem \ref{MT}, specifically the long time stability estimate \eqref{Exit}, by defining $\rho(t)=\rho_3(t)$. Finally, \eqref{est_rho} follows from \eqref{der_rho3} with $\rho(t)=\rho_3(t)$.

\subsection*{Data availability statement} The authors certify that data sharing is not applicable to this article as no new data were created or analyzed in this study.

\subsection*{Conflict of interest} The authors declare no conflict of interest in any of the procedures related to the construction and publication of this manuscript.

\providecommand{\bysame}{\leavevmode\hbox to3em{\hrulefill}\thinspace}
\providecommand{\MR}{\relax\ifhmode\unskip\space\fi MR }
\providecommand{\MRhref}[2]{%
  \href{http://www.ams.org/mathscinet-getitem?mr=#1}{#2}
}
\providecommand{\href}[2]{#2}

%
%
%
%
%
%
%
%

\end{document}